%% file: mod-p-odd.tex
\documentclass[12pt,reqno]{amsart}
\usepackage{amsmath,a4wide}
\usepackage{xfrac}
\usepackage{stmaryrd,mathrsfs,bm,amsthm,mathtools,yfonts,amssymb,color}
\usepackage{esint}
\usepackage{float}
\usepackage{xcolor}
\usepackage[pagebackref,citecolor=black,linkcolor=blue,colorlinks=true]{hyperref}
\usepackage{tikz,tikz-3dplot}

\newenvironment{itemizeb}
{\begin{itemize}\itemsep=2pt}{\end{itemize}}

\begingroup
\newtheorem{theorem}{Theorem}[section]
\newtheorem{lemma}[theorem]{Lemma}
\newtheorem{proposition}[theorem]{Proposition}
\newtheorem{corollary}[theorem]{Corollary}
\endgroup

\theoremstyle{definition}
\newtheorem{definition}[theorem]{Definition}
\newtheorem{remark}[theorem]{Remark}
\newtheorem{ipotesi}[theorem]{Assumption}

\numberwithin{equation}{section}

\newcommand{\N}{\mathbb{N}} 
\newcommand{\Z}{\mathbb{Z}} 
\newcommand{\R}{\mathbb{R}} 

\newcommand{\bA}{\mathbf{A}} 
\newcommand{\bC}{\mathbf{C}} 
\newcommand{\bB}{\mathbf{B}} 
\newcommand{\bS}{\mathbf{S}} 

\newcommand{\p}{\mathbf{p}} 

\newcommand{\bG}{\mathbf{G}} 

\newcommand{\mass}{\mathbf{M}} 
\newcommand{\Rc}{\mathscr{R}} 
\newcommand\res{\mathop{\hbox{\vrule height 7pt width .3pt depth 0pt\vrule height .3pt width 5pt depth 0pt}}\nolimits}

\renewcommand{\flat}{\mathscr{F}}

\newcommand{\F}{\mathscr{F}} 

\newcommand{\M}{\mathbf{M}} 

\newcommand{\reg}{\mathrm{Reg}} 
\newcommand{\sing}{\mathrm{Sing}} 

\newcommand{\bE}{\mathbf{E}} 

\newcommand{\cS}{\mathcal{S}} 

\newcommand{\modp}{{\rm mod}(p)} 

\newcommand{\Ha}{\mathcal{H}} 

\newcommand{\eps}{\varepsilon} 
\newcommand{\spt}{\mathrm{spt}} 
\newcommand{\dist}{\mathrm{dist}} 
\newcommand{\diam}{\mathrm{diam}} 
\newcommand{\Lip}{\mathrm{Lip}} 
\renewcommand{\epsilon}{\varepsilon}

\def\XXint#1#2#3{{\setbox0=\hbox{$#1{#2#3}{\int}$ }
		\vcenter{\hbox{$#2#3$ }}\kern-.6\wd0}}

\newcommand{\mres}{\mathbin{\vrule height 1.6ex depth 0pt width 
		0.13ex\vrule height 0.13ex depth 0pt width 1.3ex}}

\newcommand{\pa}{\partial}

\newcommand{\ssubset}{\subset\joinrel\subset}

\newcommand{\weakstar}{\stackrel{*}{\rightharpoonup}}

\DeclareMathOperator{\Div}{div} 

\def\a#1{\left\llbracket{#1}\right\rrbracket}
\newcommand{\abs}[1]{\lvert#1\rvert} 
\newcommand{\Abs}[1]{\left\lvert#1\right\rvert} 
\newcommand{\norm}[1]{\left\lVert#1\right\rVert} 

\newcommand{\bH}{\mathbf{H}}

\title[Singularities of area minimizing hypersurfaces modulo $2k+1$]{Area minimizing hypersurfaces modulo $p$: a geometric free-boundary problem}

\author[C. De Lellis]{Camillo De Lellis}
\address{School of Mathematics, Institute for Advanced Study, 1 Einstein Dr., Princeton NJ 05840, USA}
\email{camillo.delellis@math.ias.edu}
\author[J. Hirsch]{Jonas Hirsch}
\address{Mathematisches Institut, Universit\"at Leipzig, Augustusplatz 10, D-04109 Leipzig, Germany}
\email{hirsch@math.uni-leipzig.de}
\author[A. Marchese]{Andrea Marchese}
\address{Dipartimento di Matematica, Universit\`a degli Studi di Trento, Via Sommarive 14, I-38123 Povo (TN), Italy}
\email{andrea.marchese@unitn.it}
\author[L. Spolaor]{Luca Spolaor}
\address{Department of Mathematics, UC San Diego, AP\&M, La Jolla, California, 92093, USA}
\email{lspolaor@ucsd.edu}
\author[S. Stuvard]{Salvatore Stuvard}
\address{Dipartimento di Matematica, Universit\`a degli Studi di Milano, Via Saldini 50, I-20133 Milano (MI), Italy}
\email{salvatore.stuvard@unimi.it}

\begin{document}
\usetikzlibrary{arrows.meta, positioning}
\begin{abstract}
We consider area minimizing $m$-dimensional currents $\modp$ in complete $C^2$ Riemannian manifolds $\Sigma$ of dimension $m+1$. For odd moduli we prove that, away from a closed rectifiable set of codimension $2$, the current in question is, locally, the union of finitely many smooth minimal hypersurfaces coming together at a common $C^{1,\alpha}$ boundary of dimension $m-1$, and the result is optimal. For even $p$ such structure holds in a neighborhood of any point where at least one tangent cone has $(m-1)$-dimensional spine. These structural results are indeed the byproduct of a theorem that proves (for any modulus) uniqueness and decay towards such tangent cones. The underlying strategy of the proof is inspired by the techniques developed by Simon in \cite{Simon} in a class of \emph{multiplicity one} stationary varifolds. The major difficulty in our setting is produced by the fact that the cones and surfaces under investigation have arbitrary multiplicities ranging from $1$ to $\lfloor \frac{p}{2}\rfloor$. 
\end{abstract}

	\maketitle

\tableofcontents

%

\input{intro_v2}

\input{preliminaries}

\input{graph_v2}

\input{leone-simone}

\input{scoppiato}

\input{decaduto}

\input{struttura}

\input{appendix}

\bibliographystyle{plain}
\bibliography{references}

\end{document}

%% file: intro_v2.tex
\section{Introduction} \label{sec:intro}

In this paper we consider currents $\modp$ (where $p\geq 2$ is a fixed integer), for which we follow the definitions and the terminology of \cite{Federer69} and \cite{DLHMS}. In particular, given an open set $\Omega \subset \R^{m+n}$ and a relatively closed subset $C \subset \Omega$ which is a Lipschitz neighborhood retract, we denote by $\Rc_m (C)$  (resp. $\F_m (C)$) the space of those $m$-dimensional integer rectifiable currents $T\in \Rc_m (\Omega)$ (resp. $m$-dimensional integral flat chains $T \in \F_m (\Omega)$) with compact support $\spt (T)$ contained in $C$.
Currents modulo $p$ in $C$ are defined introducing an appropriate family $\F^p_K$ of pseudo-distances on $\F_m(C)$, indexed by $K\subset C$ compact, see \cite[Section 1.1]{DLHMS} and Appendix \ref{appendix-flat}.
Two flat chains $T$ and $S$ in $C$ are then congruent modulo $p$ if there is a compact set $K\subset C$ with $\spt (T-S)\subset K$ and such that $\F^p_K (T-S) =0$. The corresponding congruence class of a fixed flat chain $T$ will be denoted by $[T]$, whereas if $T$ and $S$ are congruent we will write
$T \equiv S\, \modp$ or $T=S\,\modp$.
The symbols $\Rc_m^p(C)$ and $\F_m^p(C)$ will denote the quotient groups obtained from $\Rc_m(C)$ and $\F_m(C)$ via the above equivalence relation. The boundary $\partial^p$ is defined accordingly as an operator on equivalence classes.  
In what follows the closed set $C$ will always be (a subset of) a sufficiently smooth submanifold, more precisely a complete  submanifold $\Sigma$ of $\mathbb R^{m+n}$ without boundary.

\begin{definition} \label{def:am_modp}
	Let $p\geq 2$, $\Omega \subset \R^{m+n}$ be open, and let $\Sigma \subset \R^{m+n}$ be a complete submanifold without boundary of dimension $m+\bar{n}$ and class $C^{2}$. We say that an $m$-dimensional integer rectifiable current $T \in \Rc_{m}(\Sigma)$ is \emph{area minimizing} $\modp$ in $\Sigma \cap \Omega$ if
	\begin{equation}\label{e:am_mod_p}
	\mass (T) \leq \mass (T + W) \qquad \mbox{for any $W \in \Rc_{m}(\Omega\cap \Sigma)$ which is a boundary $\modp$}.
	\end{equation}
\end{definition}

Recalling \cite{Federer69}, it is possible to introduce a suitable notion of mass and support $\modp$ for classes $[T]$ $\modp$.
With such terminology we can talk about mass minimizing classes $[T]$, because \eqref{e:am_mod_p} can be rewritten as
\begin{equation} \label{e:am_classes}
\mass^p ([T]) \leq \mass^p ([T] + \partial^p [Z]) \qquad \mbox{for all $[Z]$ with $\spt^p (Z)\subset \Omega \cap \Sigma$.} 
\end{equation}

The set of interior regular points, denoted by $\reg (T)$, is then the relatively open set of points $x\in \spt^p (T)$ in a neighborhood of which $T$ can be represented by a regular oriented submanifold of $\Sigma$ with constant multiplicity, cf. \cite[Definition 1.3]{DLHMS}. Its ``complement'', i.e.
\begin{equation} \label{e:singular set}
\sing (T) := (\Omega \cap  \spt^p (T)) \setminus (\reg (T) \cup \spt^p (\partial T))\, ,
\end{equation}
is the set of interior singular points.

\subsection{Structural results}
In the work \cite{DLHMS} (building upon its companion paper \cite{DLHMS_linear}) we have shown that (when $\Sigma$ is of class $C^{3,\alpha}$) $\sing (T)$ can have Hausdorff dimension at most $m-1$. For odd $p$, we have proved the stronger conclusion that $\sing (T)$ is countably $(m-1)$-rectifiable and has locally finite $\mathcal{H}^{m-1}$ measure in $\Omega\setminus\spt^p(\partial T)$. In fact, there exists a representative, not renamed, such that
\begin{itemize}
\item $T$ is a locally integral current in $\Omega$ with $\spt (\partial T)\cap \Omega\setminus\spt^p(\partial T) \subset \sing (T)$;
\item $\partial T = p \a{\sing (T)}$ in $\Omega\setminus\spt^p(\partial T)$ for some suitable orientation of $\sing (T)$. \footnote{Note that, while in \cite{DLHMS} we state that the multiplicity is an integer multiple of $p$, in fact Proposition \ref{lem:structure_cones} implies that the multiplicity is precisely $p$, up to choosing the orienting vector field $\tau$ appropriately, cf. Remark \ref{r:multiplicity=p}.}
\end{itemize}
Roughly speaking, at $\sing (T)$ $p$ sheets of the smooth submanifold $\reg (T)$ come together: $\sing (T)$  is ``optimally placed'' to minimize the mass of $T$ and the problem of mass minimization $\modp$ can be thought of as a ``geometric free boundary problem''. A classical free boundary is however a more regular object, motivating the following definition.

\begin{definition}\label{def:free-boundary}
Given an open set $U$ we say that $\sing (T)$ is a classical free boundary in $U$ if the following holds for some positive $\alpha$. 
\begin{itemize}
    \item[(i)] $\sing (T)\cap U$ is an orientable $C^{1,\alpha}$ $(m-1)$-dimensional submanifold of $U\cap \Sigma$; 
    \item[(ii)] $\reg (T) \cap U$ consists of $N\leq p$ connected $C^{1,\alpha}$ orientable submanifolds $\Gamma_i$ with $C^{1,\alpha}$ boundary $\partial \Gamma_i \cap U = \sing (T)\cap U$;
    \item[(iii)] There are $k_i\in \{1, \ldots , \lfloor \frac{p}{2}\rfloor\}$ such that, after suitably orienting $\sing (T) \cap U$ and $\Gamma_i$, 
\begin{align*}
&S \res U := \sum_i k_i \a{\Gamma_i} \equiv T\res U\, \modp\\
&\partial S \res U= \sum_i k_i \a{\sing (T)\cap U} = p \a{\sing (T)\cap U}\, .
\end{align*}
\end{itemize}
A subset $A \subset \sing (T)$ is locally a classical free boundary if for every $q \in A$ there exists an open neighborhood $U \ni q$ such that $\sing (T)$ is a classical free boundary in $U$. We let $\mathcal{S}\subset \sing (T)$ be the smallest (relatively) closed set such that $\sing (T)\setminus \mathcal{S}$ is locally a classical free boundary.
\end{definition}

Note that in the case of hypersurfaces ($\bar n = \dim (\Sigma) - m =1$) the Hopf maximum principle implies that the sheets $\Gamma_i$ ``join transversally'', i.e. if $\nu_i$ are tangent fields to $\Gamma_i$ orthogonal to $\sing (T)$ and pointing ``inward'', then $\{\nu_i (y)\}$ are all distinct at every $y\in \sing (T)\cap U$. Furthermore a simple first variation argument shows the following balancing condition:
\begin{equation}\label{e:balance}
\sum_i k_i \nu_i (y) =0\, .
\end{equation}

The first main result of the present paper is the following.

\begin{theorem}\label{t:odd}
Let $p$ be odd and $\Sigma, T$, and $\Omega$ as in Definition \ref{def:am_modp}. If $\dim (\Sigma)=\dim (T)+1 = m+1$, then $\sing (T)$ is locally a classical free boundary outside of a relatively closed $\mathcal{S}$ which is countably $(m-2)$-rectifiable and has locally finite $\Ha^{m-2}$ measure.
\end{theorem}

\begin{remark}
    While we have set $C^{1,\alpha}$ to be the regularity requirement for classical free boundaries when working with general ambient manifolds of class $C^2$, the regularity of the classical free boundary part of $\sing (T)$ can be improved to $C^\infty$ or analytic, if the ambient manifold allows to, by using, for instance, the techniques of \cite{KNS}.
\end{remark}

Since (a representative $\modp$ of) an area-minimizing current $T$ $\modp$ induces a stable varifold outside of $\spt^p(\partial T)$, under the assumption that $\dim (\Sigma) = \dim (T) +1$ and for every moduli $p$, the groundbreaking theorem proved in \cite{Wic} for stable varifolds of codimension $1$ gives the following:
\begin{itemize}
    \item[(a)] either a nontrivial portion of $\sing (T)$ is a classical free boundary;
    \item[(b)] or the Hausdorff dimension of $\sing (T)$ is at most $m-7$.
\end{itemize}
The latter statement, however, still leaves the possibility that, in case (a), the complement in $\sing (T)$ of the classical free boundary is pretty large (in fact for stable varifolds of dimension $m$ and codimension $1$ it is not yet known that the singular set has zero $\mathcal{H}^m$ measure!).
After this work was completed, it was pointed out to us by Minter and Wickramasekera that in fact Theorem \ref{t:odd} (except for the countable $(m-2)$ rectifiability and local finiteness of ${\mathcal H}^{m-2}$ measure of $\mathcal{S}$), as well as Theorems \ref{t:even} and \ref{t:main} below, follow from the theory developed in \cite{Wic} in a relatively direct way, in particular from the decay results of \cite[Section 16]{Wic}, once our Proposition \ref{lem:structure_cones} below is known. The crucial point is that, though the statements of the theorems of \cite[Section 16]{Wic} do not literally apply to our case because the ``$\alpha$-structural hypothesis'' is not satisfied (cf. the introduction of \cite{Wic} for the precise statement of the latter), a closer inspection of the inductive arguments given there shows that the $\alpha$-structural hypothesis is only used in a suitably weaker form that is implied by our Proposition \ref{lem:structure_cones}. The details of this alternative approach are contained in \cite{MW}.

The exact counterpart to Theorem \ref{t:odd} in the case $p$ is even is proved in our other paper \cite{DLHMSS-even}. The main stumbling block is the existence of \emph{flat} singular points, namely singular points having a tangent cone supported in an $m$-dimensional plane (see \cite[Section 7]{DLHMS} for the terminology). Such points {\em do} exist, as can be shown already for $m=2$ and $p=4$ using the structure result of White \cite{White79} (cf. \cite[Example 1.6]{DLHMS}). The main contribution of \cite{DLHMSS-even} is precisely that, under the same assumptions of Theorem \ref{t:odd} with $p$ even, the set of flat singular points has Hausdorff dimension at most $m-2$, and it is countable when $m=2$. Additionally, the $(m-2)$-rectifiability of the set of flat singular points was proved by Skorobogatova in \cite{Sko}, and more recently the fourth- and fifth-named authors together with Skorobogatova obtained in \cite{SSS25} the precise description of the local structure of $T$ in the neighborhood of a flat singular point when the dimension is $m=2$. In the context of the present paper, we limit ourselves to observe that, for arbitrary $p \geq 3$, odd or even, the condition that $\sing(T)$ is a classical free boundary in the neighborhood of some $q \in \sing(T)$ is \emph{equivalent} to a condition which can be checked at the level of the linearization, namely a suitable property of (a priori, one of) the tangent cone(s) to $T$ at $q$. We record this fact in the following theorem.

\begin{theorem}\label{t:even}
Let $p\geq 3$ be arbitrary (i.e. odd or even) and $\Sigma, T$, and $\Omega$ as in Definition \ref{def:am_modp}. If $\dim (\Sigma)=m+1$ then $\sing (T)$ is a classical free boundary in a neighborhood of $q\in \sing (T)$ if and only if one tangent cone $\bC$ to $T$ at $q$ is not flat \footnote{We say that a tangent cone $\bC$ to $T$ at $q$ is \emph{flat} if $\spt (\bC)$ is contained in an $m$-dimensional linear subspace of $T_q\Sigma$. Moreover it is well known that the directions in which the cone is translation-invariant form a vector space, commonly called the spine of $\bC$.} and it is invariant with respect to translations along $(m-1)$ linearly independent directions.
\end{theorem}

We anticipate that an important starting point of the theory developed in \cite{DLHMSS-even} consists of establishing the uniqueness of the tangent plane at a flat singular point. The latter was not yet known upon completion of the present manuscript, and the proof was later obtained by Minter and Wickramasekera in \cite{MW}. In fact, the decay estimates to the unique tangent plane stated in \cite{MW} are not quite sufficient to carry on the arguments in \cite{DLHMSS-even} when the ambient $\Sigma$ is not Euclidean space $\R^{m+1}$: in that case, the needed estimates are deduced in our paper \cite{DLHMSS-uniqueflat}. In the present paper, we show instead preliminarily the validity of the following corollary, whose content is, of course, superseded by the uniqueness theory developed in \cite{MW} and \cite{DLHMSS-uniqueflat}. 

\begin{corollary}\label{c:flat-singular-points}
Let $p=2Q$ be even and $\Sigma,T$, and $\Omega$ be as in Definition \ref{def:am_modp}. Assume one tangent cone to $T$ at $q$ is of the form $Q\a{\pi}$ for some $m$-dimensional plane $\pi$. Then {\em every} tangent cone to $T$ at $q$ is of the form $Q\a{\pi'}$ for some $m$-dimensional plane $\pi'$.
\end{corollary}

For ``small moduli'' $p\in \{2,3,4\}$ much stronger conclusions are available. When $p=2$, it is simple to use classical arguments to rule out the existence of cones $\bC$ with $(m-1)$-dimensional and $(m-2)$-dimensional spines in $\mathbb R^{m+1}$. Thus using \cite{SchoenSimon} and \cite{NV}, one can conclude that $\sing (T)$ is $(m-7)$-rectifiable and has locally finite $\mathcal{H}^{m-7}$ measure. Even for minimizers of general uniformly elliptic integrands the dimension of $\sing (T)$ is strictly less than $m-2$, see \cite{ASS}. In higher codimension there are again no cones $\bC$ with $(m-1)$-dimensional spine, but there are cones with $(m-2)$-dimensional spine, and thus one can conclude, using \cite{NV}, that the singular set is $(m-2)$-rectifiable and has locally finite $\mathcal{H}^{m-2}$-measure.

The case $p=3$ is special as well as there is (up to rotations) a unique cone ${\rm mod} (3)$ with $(m-1)$-dimensional spine in $\mathbb R^{m+n}$ for any $n$. Moreover, it follows from \cite{Taylor} that in codimension $1$ there is no cone ${\rm mod} (3)$ with $(m-2)$-dimensional spine. In particular in that pioneering work Taylor proved that,
for $p=3$, $m=2$ and $\Sigma = \mathbb R^3$, the entire singular set is locally a classical free boundary. On the other hand combining \cite{Simon}, \cite{Taylor}, \cite{NV}, and classical regularity theory, it is possible to reach the following. 

\begin{theorem}\label{t:p=3}
Let $p=3$, and let $\Sigma, \Omega$, and $T$ be as in Definition \ref{def:am_modp}. If $\dim (\Sigma)= \dim (T)+1=m+1$, then $\mathcal{S}$ is empty for $m\leq 2$, and it is $(m-3)$-rectifiable with locally finite $\mathcal{H}^{m-3}$ measure for $m\geq 3$.
If $\dim (\Sigma)\geq \dim (T) + 2=m+2$, then $\mathcal{S}$ is $(m-2)$-rectifiable and has locally finite $\mathcal{H}^{m-2}$ measure.  
\end{theorem}

We do not claim any originality in Theorem \ref{t:p=3} and the statement above has been included for completeness, while for the reader's convenience we include the short argument in the Appendix; see Appendix \ref{s:appendimi}. Finally, for $p=4$ (and in codimension $1$) \cite{White79} shows that minimizers of uniformly elliptic integrands are represented by {\em immersed manifolds} outside of a closed set of zero $\mathcal{H}^{m-2}$ measure.  

We finally notice that the structure theorems are optimal, in the sense that a simple modification of a classical example in \cite{Taylor} yields the following

\begin{proposition}\label{p:example}
For each $p\geq 3$ there is a $2$-dimensional integer rectifiable current $T$ in $\bB_2 \subset \mathbb R^3$ which is an area-minimizing representative $\modp$ and whose singular set consists of a $1$-dimensional circle which is a classical free boundary.  
\end{proposition}

Finally, we remark that the theory initiated in the present paper has been later extended to the case when $\dim(\Sigma) \geq m+2$ by the work of the first-named author with Minter and Skorobogatova in \cite{DLMSk}.

\subsection{Uniqueness of tangent cones} Both Theorem \ref{t:odd} and Theorem \ref{t:even} are in fact corollaries of the following quantitative uniqueness/decay which holds for all $p$'s. To simplify our statements, from now on a cone $\bC$ as in Theorem \ref{t:even} will be said to have an $(m-1)$-dimensional spine and we isolate an assumption which will be recurrent throughout the paper. For the notation used in the assumption we refer to Section \ref{sec:notation}, in particular we will use a second notion of flat distance $\hat\flat^p$ which has important technical advantages over the original one used by Federer; see Appendix \ref{appendix-flat}.

\begin{ipotesi}\label{ass:cone}
$p\in \mathbb N\setminus \{0,1,2\}$, and $\bC_0$ is
an $m$-dimensional area-minimizing cone $\modp$ in $\mathbb R^{m+n}$ with $(m-1)$-dimensional spine and supported in an $(m+1)$-dimensional linear space; see Definition \ref{def:a.m.cones}. 
$\Sigma, T$, and $\Omega$ are as in Definition \ref{def:am_modp} with $\dim (\Sigma)=m+1$. $\eta>0$ and $q\in \spt^p (T)$ are such that $\bB_1 (q)\subset \Omega\setminus \spt^p (\partial T)$ and, setting $T_{q,1} := (\eta_{q,1})_\sharp T$ for $\eta_{q,\lambda} (\bar q) := \lambda^{-1} (\bar q-q)$,
\begin{align}
&\spt (\bC_0) \subset T_q \Sigma\\ \label{multi ass}
& \Theta_T(q) \geq \frac{p}{2}\\
& \hat\flat^p_{\bB_1}(T_{q,1} - \bC_0) \leq  \eta \\
&\bA + \bE_0:= \|A_{\Sigma}\|_{L^\infty (\bB_1 (q))} + \int_{\bB_1} \dist^2 (\bar q, \spt (\bC_0))\, d\|T_{q,1}\| (\bar q) \leq  \eta\, .
\end{align}
\end{ipotesi}

\begin{theorem}[Uniqueness of cylindrical blow-ups]\label{t:main}
Let $p\in \mathbb N\setminus \{0,1,2\}$ and $\bC_0$ be as in Assumption \ref{ass:cone}. There are constants $\bar \eta>0$ and $C$ depending only on $p,m,n$, and $\bC_0$ with the following property. If $T, \Sigma, \Omega$ and $q$ are as in Assumption \ref{ass:cone} with $\eta=\bar \eta$,
then the tangent cone $\bC$ to $T$ at $q$ is unique, has $(m-1)$-dimensional spine, and moreover the following decay estimates hold for every radius $r\leq 1$
    \begin{equation}\label{e:decay-always}
    \frac{1}{r^{m+2}}\int_{\bB_r (q)}\dist^2(\bar q - q, \spt(\bC)) \, d\|T\| (\bar q) \leq C\, (\bE_0 + \bA^{1/2}) r^{\frac{1}{2}} \, ,
    \end{equation}
    \begin{equation}\label{e:decay-flat}
    \hat\flat^p_{\bB_1} ((\eta_{q,r})_\sharp T - \bC) \leq C (\bE_0^{1/2}+\bA^{1/4}) r^\frac{1}{4}\, .
    \end{equation}
    In particular $\hat\flat^p_{\bB_1} (\bC-\bC_0) \leq C (\bE_0^{1/2}+\bA^{1/4}) + \bar\eta$.
\end{theorem}

Our proof of Theorem \ref{t:main} is influenced by the pioneering work of Simon \cite{Simon}. In fact we use many of the tools developed there, but the main difference is that in our case the current $T$ as well as the cones $\bC$ and $\bC_0$ in the above statement come with multiplicities. This major issue forces us to significantly modify the arguments of \cite{Simon} by developing new tools and ideas: for this reason, even in the few points where we could have directly adapted the proofs of \cite{Simon}, we have opted for giving all the details from scratch, making the presentation self-contained.



\subsection{Plan of the paper and description of the strategy}

In Section \ref{sec:notation} we introduce some relevant notation, recall the relation between area minimizing currents $\modp$ and varifolds with bounded mean curvature, deal with some technical assumptions about the ambient manifold $\Sigma$ and finally introduce the modified flat distance $\hat{\flat}^p$. As already mentioned above, the latter has some technical advantages over the flat distance originally introduced by Federer, as it behaves in a much better way with respect to the operation of restriction. Section \ref{sec:cones} recalls the definition of tangent cones, their spines, and the usual stratification of $\spt (T)$ according to the maximal dimension of the spines of the tangent cones at the given point. We then analyze the tangent cones $\bC$ with $(m-1)$-dimensional spine. Two elementary facts will play an important role. First of all, any such $\bC$ can be described as the union of finitely many, but at least $3$, half-hyperplanes $\bH_i$ meeting at a common $(m-1)$-dimensional subspace $V$ and counted with appropriate multiplicities $\kappa_i$. Secondly, if the modulus $p$ is odd, then the angle formed by a pair $(\bH_i, \bH_j)$ of consecutive half-hyperplanes is necessarily smaller than $\pi - \vartheta_0 (p)$, where $\vartheta_0 (p)$ is a positive geometric constant depending only on $p$. This is effectively the reason why for odd moduli our conclusion is stronger.

In Section \ref{sec:excess} we state the most important result of the paper, namely the Decay Theorem \ref{t:decay}: the latter states, roughly speaking, that if the current $T$ is sufficiently close, at a given scale $\rho$, to a cone $\bC$ as above around a point $q$ where $T$ has density at least $\frac{p}{2}$, at a scale $\delta \rho$ the distance to a suitable cone $\bC'$ with the same structure will decay by a constant factor. This is the counterpart of a similar decay theorem proved by Simon in his pioneering work on cylindrical tangent cones \cite{Simon} of multiplicity $1$ under the assumption that the cross section satisfies a suitable integrability condition, which in turn is a far-reaching generalization of the work of Taylor in \cite{Taylor} for the specific case of $2$-dimensional area-minimizing cones $\mod$ $3$ in $\mathbb R^3$ with $1$-dimensional spines (to our knowledge, the first theorem of the kind ever proved in the literature for a ``singular cone'').
While our paper builds on the foundational work of Simon \cite{Simon} on cylindrical tangent cones, substantial work is needed to deal with the fact that the multiplicities are allowed to be larger than $1$. In order to perform our analysis, the theorem is proved for cones $\bC$ which in turn are sufficiently close to a fixed reference cone $\bC_0$. While $\bC_0$ is assumed to be area-minimizing $\modp$, both $\bC$ and $\bC'$ are not. Theorem \ref{t:main} is then proved by iterating Theorem \ref{t:decay} (accomplished in Section \ref{sec:iterazione}), while Theorems \ref{t:odd} and \ref{t:even} are a relatively simple consequence of Theorem \ref{t:main}, and their proofs are given in Section \ref{sec:structure}. The latter also contains the proof of Corollary \ref{c:flat-singular-points} and Proposition \ref{p:example}. The remaining part of the paper is devoted to prove Theorem \ref{t:decay}. 

As in many similar regularity proofs (starting from the pioneering work of De Giorgi \cite{DG}) the main argument is a ``blow-up'' procedure: after scaling, we focus on a sequence of area-minimizing currents $T_k$ which are close at scale $1$ to cones $\bC_k$, which in turn converge to a reference cone $\bC_0$. $\bC_k$ and $\bC_0$ are assumed to share the same spine $V$. The distance between $T_k$ and $\bC_k$ (which is measured in an $L^2$ sense) is the relevant parameter and will be called {\em excess}, cf. Definition \ref{d:excess}, and denoted by $\bE_k$. The distance between $\bC_k$ and $\bC_0$ is not assumed to be related to $\bE_k$. The overall idea is then to approximate the currents $T_k$ and $\bC_k$ with Lipschitz graphs over the halfplanes $\bH_{0,i}$ forming $\bC_0$, consider the differences between these graphs, renormalize them by $\bE_k^{-\sfrac 12}$, and study their limits. These are proved to be harmonic (an idea that dates back to De Giorgi), while the remarkable insight of Simon's work \cite{Simon} is used to prove suitable estimates (and compatibility relations) at the spine $V$. This blow-up procedure is accomplished in Section \ref{sec:bu} (cf. Definition \ref{d:scoppia}, Corollary \ref{c:scoppia}, and Proposition \ref{p:scoppia}), while in Section \ref{sec:harm} we use the elementary properties of harmonic functions to prove a suitable decay of the blow-up limits, cf. Proposition \ref{p:decad-lineare}. A fundamental realization of Simon is that, in order to accomplish the above program, one needs to introduce an additional object, for which we propose the term {\em binding function}, and whose role will be explained in a moment.

As already mentioned, the biggest source of complication is that the multiplicities $\kappa_{0,i}$  of the halfplanes $\bH_{0,i}$ forming the support of $\bC_0$ are typically larger than $1$. In particular it is necessary to use $\kappa_{0,i}$ (not necessarily all distinct) functions to approximate the portions of the current $T_k$ which are close to $\bH_{0,i}$. Likewise, it is necessary to use $\kappa_{0,i}$ functions to describe the portions of $\bC_k$ which are close to $\bH_{0,i}$. Notice that while we know that the number $N_i$ of {\em distinct} functions needed in the representation ranges between $1$ and $\kappa_{0,i}$ and that the multiplicities of the corresponding graphs are positive integers $\kappa_{i,j}$ which sum up to $\kappa_{0,i}$, any choice of coefficients respecting these conditions is possible, and moreover the choice might be different for $T_k$ and $\bC_k$ and depend on $k$. 

In order to produce graphical parametrizations of the current $T_k$ at appropriate scales, we take advantage of the $\eps$-regularity result proved by White in \cite{White86}, but we also need to show that each such parametrization is close to one of the linear functions describing the cone $\bC_k$. This major issue is absent in Simon's work \cite{Simon} thanks to the multiplicity one assumption, and we address it in the three separate Sections \ref{sec:graph}, \ref{sec:local}, and \ref{sec:global}. The relevant ``graphical approximation theorems'' which follow from this analysis and will be used later are Theorem \ref{thm:graph v2} and Corollary \ref{cor:reparametrization}.

First of all in Section \ref{sec:graph} we show how to use \cite{White86} to gain a graphical parametrization of $T= T_k$, cf. Theorem 
\ref{thm:graph_v1}. Inspired by \cite{Simon} we subdivide the support of the current in regions $Q$ of size comparable to their distance $d_Q$ from the spine of $\bC_0$. For practical reasons, $d_Q$ will range in a dyadic scale and we will put an order relation on all the regions according to whether a region $Q'$ is lying ``above'' the region $Q$, cf. Definition \ref{def:whitney} for the precise definition. We then apply the regularity theorem of \cite{White86} on any ``good'' region, i.e. any $Q$ with the property that at $Q$ and at every region above $Q$ the current $T$ is sufficiently close to $\bC = \bC_k$. A simple argument (which uses heavily the fact that the codimension of $T$ in $\Sigma$ is $1$), allows to ``patch'' together the graphical approximations across different regions to achieve $p= \sum_i \kappa_{0,i}$ ``sheets'' which approximate efficiently the current.

In section \ref{sec:local} we show that on each region $Q$ each ``graphical sheet'' of $T$ is close to some sheet of $\bC$, cf. Lemma \ref{l:local_selection}: the main ingredient is an appropriate Harnack-type estimate for solutions of the minimal surface equation, cf. Lemma \ref{lem.harmonic rigidity}. While at this stage the choice might depend on the region $Q$, in Section \ref{sec:global} an appropriate selection algorithm allows to bridge across different regions and show that there is a single sheet of $\bC$ to which each single graphical sheet of $T$ is close on every region $Q$, cf. Lemma \ref{l:global_selection}. The latter selection algorithm will in fact be used again twice later on. An important thing to be noticed is that, since we use a one-sided excess, there might be some sheets of $\bC$ which are not close to any of the graphical sheets of $T$: this phenomenon, which is not present in \cite{Simon}, is due to the fact that the multiplicities $\kappa_{0,i}$ might be higher than $1$, and forces us to introduce an intermediate cone $\tilde{\bC}$ which consists of those sheets of $\bC$ which are close to at least one graphical sheet of $T$. 
 
We next appropriately modify the key idea of \cite{Simon} that the remainder in the classical monotonicity formula can be used to improve the estimates near the spine of the cone $\bC_0$. In Section \ref{sec:Hardt-Simon} this is first done to estimate the distance of $T$ to suitable shifted copies of $\tilde{\bC}$, centered at points of high density of $T$, cf. Theorem \ref{t:Hardt-Simon-main}. It is in this section that we exploit crucially a reparametrization of the graphical sheets of $T$ over $\tilde{\bC}$ (cf. Corollary \ref{cor:reparametrization}) and, in particular, the fact that $\tilde{\bC}$ does not contain any ``halfspace of $\bC$ far from $T$''. In Section \ref{sec:binding} the $\modp$ structure allows us to prove the so-called ``no-hole condition'', namely some point of high density of $T$ must be located close to any point of the spine of $\tilde\bC$ (which, we recall, is the same as the spine of $\bC$ and $\bC_0$), cf. Proposition \ref{prop:no-holes}. The latter is combined with Theorem \ref{t:Hardt-Simon-main} to prove that, upon subtracting some suitable piecewise constant functions with a particular cylindrical structure (the {\em binding functions} of Definition \ref{def:binding functions}), the graphical sheets enjoy good estimates close to the spine, cf. Theorem \ref{thm:est spine}. However, again caused by multiplicities $\kappa_{0,i}$, unlike in \cite{Simon}, we need to introduce a suitable correction to the binding functions, and a crucial point is that the size of the latter can be estimated by the product of the excess $\bE^{\sfrac 12} = \bE_k^{\sfrac 12}$ and the distance of $\bC = \bC_k$ to $\bC_0$.\\ 

\noindent\textbf{Acknowledgments.} C.D.L. acknowledges support from the National Science Foundation through the grant FRG-1854147. J.H. was partially supported by the German Science Foundation DFG in context of the Priority Program SPP 2026 “Geometry at Infinity”. L.S. acknowledges the support of the NSF grant DMS-1951070. A.M. and S.S. acknowledge support of the project PRIN 2022PJ9EFL "\textit{Geometric Measure Theory: Structure of Singular Measures, Regularity Theory and Applications in the Calculus of Variations}", funded by the European Union under NextGenerationEU and by the Italian Ministry of University and Research.

%% file: preliminaries.tex
\section{Preliminaries} \label{sec:notation}

In this section we fix the main notation in use throughout the paper and recall some preliminary facts

\subsection{Notation}

The following notation is of standard use in Geometric Measure Theory; see e.g. \cite{Simon83,Federer69}. More notation will be introduced in the main text when the need arises.

\begin{itemizeb}\leftskip 0.8 cm\labelsep=.3 cm	

\item[$\bB_R(q)$] open ball in $\R^{m+n}$ centered at $q \in \R^{m+n}$ with radius $R > 0$;

\item[$B^k_r(x)$] open disc in $\R^k$ (or in a $k$-dimensional linear subspace of $\R^{m+n}$) centered at $x \in \R^k$ with radius $r>0$;

\item[$\dist(\cdot,E)$] distance function from a subset $E \subset \R^{m+n}$, defined by $\dist(q,E) := \inf\lbrace\abs{q-\bar q}\,\colon\, \bar q \in E\rbrace$;

\item[$\omega_k$] Lebesgue measure of the unit disc in $\R^k$;

\item[$|E|$] Lebesgue measure of $E \subset \R^{m+n}$;

\item[$\Ha^k$] $k$-dimensional Hausdorff measure in $\R^{m+n}$;

\item[$\mu \mres E$] restriction of the measure $\mu$ to the Borel set $E$: it is defined by $(\mu \mres E) (F) := \mu(E \cap F)$ for all Borel sets $F$;

\item[$A_\Sigma$] second fundamental form of a submanifold $\Sigma \subset \R^{m+n}$ of class $C^2$;

\item[$\mathscr{F}_m$, ($\mathscr{F}_{m}^p$)] integral flat chains (modulo $p$) of dimension $m$;

\item[$\Rc_m$, ($\Rc_{m}^p$)] integer rectifiable currents (modulo $p$) of dimension $m$; we write $T = \llbracket M, \vec{\tau}, \theta \rrbracket$ if $T$ is defined by integration with respect to $\vec{\tau} \, \theta \, \Ha^m \res M$ for a countably $m$-rectifiable set $M$ with locally finite $\Ha^m$ measure, oriented by the Borel measurable unit $m$-vector field $\vec{\tau}$ with locally integrable (with respect to $\Ha^m \mres M$) multiplicity $\theta$;

\item[$k \a{\Gamma}$] integer rectifiable current $\a{\Gamma,\vec{\tau},k}$ defined by integration over an oriented embedded submanifold $\Gamma \subset \R^{m+n}$ of class $C^1$ (or a rectifiable set with locally finite Hausdorff measure) with orientation $\vec{\tau}$;

\item[${[ T ]}$] $\modp$ equivalence class of $T \in \mathscr{F}_m$;

\item[$\mass$, ($\mass^p$)] mass functional (mass modulo $p$);

\item[$\|T\|$, ($\|T\|_p$)] Radon measure associated to a current $T$ (to a class $\left[ T \right]$) with locally finite mass (mass modulo $p$); $\|T\|$ will also be used for the corresponding integral varifold if $T\in \Rc_m$;

\item[$\partial T$, $\partial^p{[T]}$] boundary of the current $T$, boundary $\modp$ of the class $[T]$. The latter is defined by $\partial^p[T] := [\partial T]$; 

\item[$\spt(T)$, $\spt^p(T)$] support and support $\modp$ of $T$. The latter is defined as the intersection of the supports of all chains $\tilde T \in [T]$;



\item[$H_V$] generalized mean curvature of a varifold $V$ with locally bounded first variation;

\item[$H_T$] same as $H_{\|T\|}$ for an integer rectifiable current $T$ whose associated varifold $\|T\|$ has locally bounded first variation;

\item[$f_\sharp T$, $f_\sharp V$] push-forward of the current $T$, of the varifold $V$, through the map $f$;

\item[$\langle T, f, z \rangle$] slice of the current $T$ with the function $f$ at the point $z$;

\item[$\eta_{q,R}$] the map $\R^{m+n} \to \R^{m+n}$ defined by $\eta_{q,R}(\bar q) := R^{-1}(\bar q-q)$;

\item[$\Theta^k(\mu,q)$] $k$-dimensional density of the measure $\mu$ at the point $q$, defined by $\Theta^k(\mu,q) := \lim_{r\to0^+} \frac{\mu(\bB_r(q))}{\omega_k\,r^k}$ whenever the limit exists;

\item[$\Theta_T(q)$] same as $\Theta^m(\|T\|,q)$ if $T$ is an $m$-current with locally finite mass.

\end{itemizeb}

\subsection{Varifolds and currents}

We follow \cite{DLHMS}, and recall that an integer rectifiable current $T$ which is area minimizing $\modp$ in $\Omega \cap \Sigma$ as in Definition \ref{def:am_modp} is always a \emph{representative $\modp$} in $\Omega$. This means that if $T = \llbracket M, \vec{\tau}, \theta \rrbracket$ then 
\[
\|T\| \mres \Omega \leq \frac{p}{2} \, \Ha^m \mres (M\cap \Omega) \qquad \mbox{in the sense of Radon measures}\,,
\]
or equivalently that
\[
\abs{\theta} \leq \frac{p}{2} \qquad \mbox{$\Ha^m$-a.e. on $M\cap\Omega$}\,.
\]
We also recall (cf. \cite[Lemma 5.1]{DLHMS}) that the varifold $\|T\|$ induced by $T$ is stationary in the open set $\Omega \setminus \spt^p (\pa T)$ with respect to variations that are tangent to $\Sigma$, that is 
\begin{equation} \label{e:stationarity_tangent}
\delta\|T\| (\chi) = 0 \qquad \mbox{for all $\chi \in C^1_c(\Omega \setminus \spt^p(\pa T),\R^{m+n})$ tangent to $\Sigma$}\,,
\end{equation}
and more generally that 
\begin{equation} \label{e:H in L infty}
\delta \|T\| (\chi) = - \int \chi \cdot H_T \, d\|T\| \qquad \mbox{for all $\chi \in C^1_c(\Omega \setminus \spt^p(\pa T),\R^{m+n})$}
\end{equation}
with $\|H_T\|_{L^\infty} \leq \|A_\Sigma\|_{L^\infty(\Omega \cap \spt(T))}$. 

\subsection{The ambient manifold and preliminary reductions} 

Since the results of this paper depend on a local analysis of the current $T$ at its interior singular points, we can always assume to be working in a small ball $\bB_{\rho}(q)$ centered at some point $q \in \Sigma$. The regularity of $\Sigma$ guarantees that if $\rho$ is sufficiently small then $\Sigma \cap \bB_\rho(q)$ is the graph of a $C^2$ function of $m+1$ variables, and that $\overline{\Sigma \cap \bB_\rho(q)}$ is a Lipschitz deformation retract of $\R^{m+n}$. As observed in \cite{DLHMS}, equivalence classes $\modp$ in $\Sigma \cap \bB_\rho(q)$ do not depend on $\Sigma$. In other words, in these circumstances it does not matter what the shape of $\Sigma$ is outside of $\bB_\rho(q)$, and thus we can assume without loss of generality that $\Sigma$ is in fact an entire graph of $m+1$ variables. Also, since we are only interested in \emph{interior} singularities of $T$, we can assume that $(\pa T) \mres \bB_\rho(q) = 0\;\modp$. Furthermore, since the singularities we are interested in belong to the top-dimensional stratum $\sing^*(T)$ (see \eqref{def:top dimensional} below), we can always assume as a consequence of Proposition \ref{lem:structure_cones} that the $m$-dimensional density of $T$ at $q$ is $\Theta_T(q) = \sfrac{p}{2}$, and in fact $\Theta_T (q) \geq \sfrac{p}{2}$ is sufficient for our purposes. By a standard blow-up procedure, we shall work on the rescaling $T_{q,\rho} := (\eta_{q,\rho})_\sharp T$, and thus also the ambient manifold will be translated and rescaled to $\Sigma_{q,\rho}:=\rho^{-1}(\Sigma-q)$. Notice that, as $\rho \to 0^+$, the manifolds $\Sigma_{q,\rho}$ approach the tangent space $T_q\Sigma$: for this reason, we can further assume without loss of generality that the second fundamental form $A_\Sigma$ of $\Sigma$ satisfies a global bound of the form $\|A_\Sigma\|_{L^\infty(\Sigma)} =: \mathbf{A} \leq c_0$ for a (small) dimensional constant $c_0$. We summarize these reductions in an assumption, which will be taken as a hypothesis in most of our subsequent statements.  

\begin{ipotesi}\label{assumption:manifolds and currents}
We establish the following set of assumptions.
\begin{itemize}
    \item[($\Sigma$)]  $\Sigma$ is an entire $(m+1)$-dimensional graph of a function $\Psi \colon T_0\Sigma \to (T_0\Sigma)^\perp$ of class $C^2$, and $\mathbf{A}:=\|A_\Sigma\|_{L^\infty(\Sigma)} \leq c_0$.
    
    \item[($T$)] $T$ is a current in $\Rc_m(\Sigma)$ such that:
    \begin{itemize}
        \item[($T1$)] $T$ is area minimizing $\modp$ in $\Sigma \cap \Omega$ for $\Omega=\bB_{2R_0} (0)$, where $R_0$ is a geometric constant;
        \item[($T2$)] $0 \in \spt(T)$ and $\Theta_T(0) \geq \sfrac{p}{2}$;
        \item[($T3$)] $\partial T =0 \, \modp$ in $\Omega=\bB_{2R_0}(0)$.
    \end{itemize}
    
\end{itemize}

\end{ipotesi}

\subsection{The modified \texorpdfstring{$p$}{p}-flat distance} Throughout the paper, we will often make use of a modified version of the flat distance $\modp$ which is better behaved than $\flat^p$ with respect to localization, and thus well suited to applications in regularity statements. More precisely, if $\Omega \subset \R^{m+n}$ is open, $C \subset \Omega$ is a relatively closed subset, $T,S \in \mathscr{F}_m(C)$, and $W \ssubset \Omega$ is open, we set

\begin{equation} \label{flat-simon}
\begin{split}
    \hat\flat^p_{W}(T-S) := \inf\Big\lbrace & \|R\|(W) + \|Z\|(W) \, \colon \, R \in \mathscr{R}_m(\Omega),\, Z \in \mathscr{R}_{m+1}(\Omega) \\  
    & \qquad \mbox{such that $T-S = R + \partial Z + pP$ for some $P \in \mathscr{F}_m(\Omega)$} \Big\rbrace\,. 
\end{split}
\end{equation} 

A complete discussion on the necessity of this alternative notion of $p$-flat distance and its relationship with the classical $\flat^p$ is contained in Appendix \ref{appendix-flat}.

\section{Tangent cones} \label{sec:cones}

Using the fact that $\|T\|$ has locally bounded variation we can define (see \cite[Sections 7 and 8]{DLHMS}) the set of tangent cones to $\|T\|$ at every point $q\in \spt (T)\setminus \spt^p (\partial T)$, and stratify 
$\spt (T)\setminus \spt^p (\partial T)$ as 
\[
\mathcal{S}^0 \subset \mathcal{S}^1 \subset \ldots \subset \mathcal{S}^{m-1} \subset \mathcal{S}^m = \spt (T) \setminus \spt^p(\pa T)\, ,
\]
according to the maximal dimension of their \emph{spines}: more precisely, $\mathcal{S}^k$ is the subset of points $q\in \spt (T) \setminus \spt^p (\partial T)$ with the property that no tangent cone to $\|T\|$ at $q$ has spine of dimension $k+1$. 

Of particular importance is the set
\begin{equation} \label{def:top dimensional}
\sing^*(T) := \cS^{m-1} \setminus \cS^{m-2}\, .
\end{equation}
All points $q \in \sing^*(T)$ are characterized by the following two properties:
\begin{itemize}
	\item[(a)] no tangent cone to $\|T\|$ at $q$ is supported in an $m$-dimensional subspace of $T_q\Sigma$;
	\item[(b)] there is at least one tangent cone to $\|T\|$ at $q$ with spine of dimension $m-1$.
\end{itemize}

The following is an obvious corollary of Theorem \ref{t:main}:

\begin{corollary}\label{c:sing^*}
Let $p\geq 2$, $T$, $\Omega$, and $\Sigma$ be as in Definition \ref{def:am_modp} and assume $\dim (\Sigma) = \dim (T)+1$.
Then a point $q$ belongs to $\sing^* (T)$ if and only if (b) above holds.
\end{corollary}

More importantly, by \cite{White86} when $p$ is odd and $\dim (\Sigma) = \dim (T)+1$ the existence of one flat tangent cone at $q$ guarantees the regularity of the point $q$: in that case we thus have $\reg (T) = \cS^m\setminus \cS^{m-1}$ and $\sing^* (T) = \sing (T)\setminus \cS^{m-2}$. Moreover \cite{NV} implies that $\cS^{m-2}$ is $(m-2)$-rectifiable; in Appendix \ref{app:NV} we will show that, as a consequence of the theory to be developed, it also has locally finite $\mathcal{H}^{m-2}$ measure. Therefore Theorem \ref{t:odd} and Theorem \ref{t:even} can be unified in the following single statement

\begin{theorem}\label{t:unified}
Let $p\geq 2$, $T$, $\Omega$, and $\Sigma$ be as in Definition \ref{def:am_modp} and assume $\dim (\Sigma) = \dim (T)+1$. Then $\sing^* (T)$ is locally a classical free boundary.
\end{theorem}

An important point in our analysis is that the tangent cones to the varifold $\|T\|$ are in fact induced by corresponding tangent cones to the  current $T$, cf. \cite[Theorem 5.2]{DLHMS}. We establish next the following terminology.

\begin{definition}[Area minimizing cones $\modp$] \label{def:a.m.cones}
An integer rectifiable current \footnote{We remark explicitly that, following our definitions, integer rectifiable cones are, in fact, only \emph{locally} integer rectifiable currents, in the sense that their restriction to any compact set in $\R^{m+n}$ is integer rectifiable. Nonetheless, in what follows we will avoid being too pedantic on the distinction between being integer rectifiable and being locally integer rectifiable whenever that property refers to a cone.} $\bC$ will be called an \emph{area minimizing cone} $\modp$ in $\R^{m+n}$ provided the following conditions hold:
\begin{itemize}
\item[(a)] $\bC$ is locally area minimizing $\modp$ in $\R^{m+n}$, i.e. it is area minimizing $\modp$ in every open set $W \ssubset \R^{m+n}$;
\item[(b)] $\pa \bC = 0 \, \modp$;
\item[(c)] the associated varifold $\|\bC\|$ is a cone, namely $(\eta_{0,\lambda})_{\sharp} \|\bC\| = \|\bC\|$ for all $\lambda > 0$. 
\end{itemize}
The spine of an area minimizing cone $\modp$ is the spine of the associated varifold $\|\bC\|$.
\end{definition}

As shown in \cite[Corollary 7.3]{DLHMS} we then have 

\begin{proposition}
Let $p$, $\Omega$, $\Sigma$, and $T$ be as in Definition \ref{def:am_modp}. Let $q\in \spt^p (T)\setminus \spt^p (\partial T)$ and consider any sequence $r_k\downarrow 0$. Up to subsequences $(\eta_{q,r_k})_\sharp T$ converges locally to a cone $\bC$ which is area minimizing $\modp$ and is supported in the plane $\pi := T_q \Sigma$. Moreover $(\eta_{q,r_k})_\sharp \|T\|$ converges to $\|\bC\|$ in the sense of varifolds.
\end{proposition}

The following proposition is the starting point of our analysis and gives the geometric structure of $m$-dimensional area minimizing cones $\modp$ with spine of dimension $m-1$.

\begin{proposition} \label{lem:structure_cones}
	Let $p \geq 2$ be an integer, and let $\bC$ be an $m$-dimensional area minimizing cone $\modp$ in $\R^{m+n}$ with spine $V$ of dimension $m-1$. Let $V^\perp$ be the orthogonal complement of $V$. Then:
	\begin{itemize}
		\item[(i)] $\bC = \bC' \times \a{V} $ for some $1$-dimensional area minimizing cone $\modp$ in $V^\perp$ (where we assume to have fixed a choice of a constant orientation $\vec{\tau}$ on $V$);
		\item[(ii)] There exist $N\geq 3$ distinct vectors $v(1), \ldots, v(N)\in \mathbb{S}^n \subset V^\perp$ and $N$ positive integers $\kappa(1), \ldots, \kappa(N) \in \left[1,\sfrac{p}{2}\right) \cap \mathbb{Z}$ such that
		\begin{itemize}
			\item[•] if $\ell(i) := \{t \,v(i) : t\geq 0\} \subset V^\perp$ is oriented in such a way that $\partial \a{\ell(i)} = - \a{0}$, then either $\bC' = \sum_{i=1}^N \kappa(i) \a{\ell(i)} $ or $\bC' =
			- \sum_{i=1}^N \kappa(i) \a{\ell(i)}$;
			\item[•] $\sum_{i=1}^N \kappa(i) v(i) = 0$. 
		\end{itemize}
	\item[(iii)] $\sum_{i=1}^N \kappa(i) = p$, and hence the $m$-dimensional density of $\bC$ at $0$ is $\Theta_{\bC}(0) = \frac{p}{2}$.
	\end{itemize}
	Moreover, when $ p$ is odd, the set of area minimizing cones $\modp$ with spine of dimension $m-1$ is compact with respect to the flat topology.
\end{proposition}

\begin{remark}\label{r:multiplicity=p}
 As already observed in the Introduction, combined with
\cite[Corollary 1.10]{DLHMS}, Proposition \ref{lem:structure_cones} allows to conclude that when $p$ is odd the multiplicity of $\partial T$ and $\sing (T)$
is precisely $p$, up to a suitable choice of the orientation of the rectifiable set $\sing (T)$.
\end{remark}

\begin{remark}
We notice explicitly that the compactness claimed at the end of the proposition fails if $p$ is an even integer, as in that case there area area minimizing cones $\modp$ with spine of dimension $m-1$ that are arbitrarily close, with respect to the flat distance in $\bB_1$, to $m$-planes with multiplicity $\sfrac{p}{2}$. 
\end{remark}

\begin{proof}

First, observe that as a direct consequence of \cite[Lemma 8.5]{DLHMS} we can conclude that
\begin{equation} \label{e:splitting_modp}
\bC = \bC' \times \a{V} \, \modp\,,
\end{equation}
where $\bC'$ is a one-dimensional area minimizing cone $\modp$ in $V^\perp \simeq \R^{n+1}$ with trivial spine. In particular, the associated varifold $\|\bC'\|$ is stationary in $\R^{n+1}$ and singular. Hence, $\|\bC'\|$ consists of the union of $N \geq 3$ distinct half-lines $\ell(i)$ with multiplicities $\kappa(i)$ such that, if $\ell(i) = \left\lbrace t\,v(i) \, \colon \, t \geq 0 \right\rbrace$ for some $v(i) \in \mathbb{S}^n \subset \R^{n+1}$ then $\sum_i\kappa(i)\,v(i) = 0$. Furthermore, $1 \leq \kappa(i) \leq \sfrac{p}{2}$ for all $i$ due to $\bC'$ being area minimizing $\modp$. Observe that we cannot exclude the case $\kappa(i) = \sfrac{p}{2}$ yet. Next, let $\a{\ell(i)}$ be the multiplicity one integral current supported on $\ell(i)$ and oriented so that $\pa \a{\ell(i)} = - \a{0}$. Since $\pa^p[\bC']=0$, we may then conclude that $\bC' = \sum_{i} \tilde\kappa(i) \, \a{\ell(i)}\,\modp$ with $|\tilde\kappa(i)| = \kappa(i)$ for every $i$, with the equality holding in the sense of classical currents if $\kappa(i) \ne \sfrac{p}{2}$ for every $i$. In order to complete the proof of (ii), we will show that it is
\begin{eqnarray*}
&\mbox{either} \qquad \tilde\kappa(i) = \kappa(i) \qquad & \mbox{for every $i \in \{1,\ldots,N\}$}\\
&\mbox{or} \qquad \tilde\kappa(i) = - \kappa(i) \qquad & \mbox{for every $i \in \{1,\ldots,N\}$}\,.
\end{eqnarray*}

Indeed, suppose towards a contradiction and w.l.o.g. that $\tilde \kappa(1) < 0 < \tilde \kappa(2)$. Since $N \ge 3$, we can also assume w.l.o.g. that the vectors $v(1)$ and $v(2)$ do not lie on the same line, and thus $v(1) \neq - v(2)$. We consider then the current 
\[
S := \a{\left[0, v(1)\right]} - \a{\left[ 0, v(2) \right]} + \a{\left[v(1),v(2)\right]}\,,
\] 
where $[x,y]$ denotes the segment $\left[x,y \right] := \{ x + t(y-x) \, \colon \, t \in \left[0,1\right]  \}$ in $\R^{n+1}$ and $\a{\left[x,y \right]}$ is the associated current with multiplicity one endowed with the natural orientation. Since $\partial S = 0$, the current 
\[
W := \bC' + S
\]
has the same boundary of $\bC'$, but
\[
\|W\|(B_1^{n+1}) - \|\bC'\|(B_1^{n+1}) = |v(1) - v(2)| - 2 < 0\,,
\]
thus contradicting the minimality of $\bC'$, cf. Figure \ref{figura-0}.
\begin{figure}[H]
\begin{tikzpicture}[
    scale=1.6,
    point/.style={circle, fill=black, inner sep=0.8pt},
    ray/.style={solid, thick},
    seg/.style={solid, thick, -{Stealth[scale=1.0]}},
    label/.style={font=\footnotesize},
    mult/.style={font=\scriptsize, midway}
]

\begin{scope}[local bounding box=figA]
    \coordinate (O) at (0,0);
    \coordinate (v1) at (40:1);   
    \coordinate (v2) at (100:1);  
    \coordinate (v3) at (180:1);  
    \coordinate (v4) at (240:1);  
    \coordinate (v5) at (330:1);  
    
    \draw[thin] (O) circle (1);
    
    \draw[seg] (O) -- (v1) node[mult, below right=-3pt] {$\kappa(1)=-1$};
    \draw[seg] (O) -- (v2) node[mult, below left=-5pt ] {$\kappa(2)=1$};
    \draw[seg] (O) -- (v3);
    \draw[seg] (O) -- (v4);
    \draw[seg] (O) -- (v5);
    
    \node[point, label={[label]below left:$0$}] at (O) {};
    \node[point] at (v1) {} node[above right=42pt] {$v(1)$};
    \node[point] at (v2) {} node[above=46pt] {$v(2)$};
    \node[point, label={[label]left:$v(3)$}] at (v3) {};
    \node[point, label={[label]below:$v(4)$}] at (v4) {};
    \node[point, label={[label]right:$v(5)$}] at (v5) {};
    
    \node[below=0.1 of figA.south] {\small $\mathbf{C}'$};
\end{scope}

\begin{scope}[local bounding box=figB, shift={(3.5,0)}]
    \coordinate (O) at (0,0);
    \coordinate (v1) at (40:1);
    \coordinate (v2) at (100:1);
    \coordinate (v3) at (180:1);
    \coordinate (v4) at (240:1);
    \coordinate (v5) at (330:1);
    
    \draw[thin] (O) circle (1);
    
    \draw[dashed, thick] (O) -- (v3);
    \draw[dashed, thick] (O) -- (v4);
    \draw[dashed, thick] (O) -- (v5);
    
    \draw[ray] (O) -- (v1);
    \draw[ray] (O) -- (v2);
    
    \draw[seg] (O) -- (v1) node[mult, right] {$+1$};
    \draw[seg] (v2) -- (O) node[mult, left] {$+1$};
    \draw[seg] (v1) -- (v2) node[mult, below] {$+1$};
    
    \node[point, label={[label]below left:$0$}] at (O) {};
    \node[point, label={[label]above right:$v(1)$}] at (v1) {};
    \node[point, label={[label]above left:$v(2)$}] at (v2) {};
    \node[point, label={[label]left:$v(3)$}] at (v3) {};
    \node[point, label={[label]below:$v(4)$}] at (v4) {};
    \node[point, label={[label]right:$v(5)$}] at (v5) {};
    
    \node[below=0.1 of figB.south] {\small $S = [\![0,v(1)]\!] - [\![0,v(2)]\!] + [\![v(1),v(2)]\!]$};
\end{scope}

\begin{scope}[local bounding box=figC, shift={(7,0)}]
    \coordinate (O) at (0,0);
    \coordinate (v1) at (40:1);
    \coordinate (v2) at (100:1);
    \coordinate (v3) at (180:1);
    \coordinate (v4) at (240:1);
    \coordinate (v5) at (330:1);
    
    \draw[thin] (O) circle (1);
    
    \draw[seg] (O) -- (v3);
    \draw[seg] (O) -- (v4);
    \draw[seg] (O) -- (v5);
    \draw[dashed, thick] (O) -- (v1); 
    \draw[dashed, thick] (O) -- (v2); 
    \draw[seg] (v1) -- (v2) node[mult, below] {$+1$};
    
    \node[point, label={[label]below left:$0$}] at (O) {};
    \node[point, label={[label]above right:$v(1)$}] at (v1) {};
    \node[point, label={[label]above left:$v(2)$}] at (v2) {};
    \node[point, label={[label]left:$v(3)$}] at (v3) {};
    \node[point, label={[label]below:$v(4)$}] at (v4) {};
    \node[point, label={[label]right:$v(5)$}] at (v5) {};
    
    \node[below=0.1 of figC.south] {\small $W: = \mathbf{C}' + S$};
\end{scope}
\end{tikzpicture}
\caption{All the multiplicities $\kappa(i)$ have the same sign}\label{figura-0}
\end{figure}
\medskip

		
	 Next, we prove (iii). Of course, $\Theta_{\bC}(0) = \Theta_{\bC'}(0)$, so it suffices to work on $\bC'$. Assume w.l.o.g. that $\tilde\kappa(i) > 0$ for every $i$, so that $\tilde\kappa(i) = \kappa(i)$. Also observe that since $\pa^p[\bC'] = 0$ it must be
	 \[
	2\, \Theta_{\bC'}(0) = \sum_{i} \kappa(i) = \nu\,p
	 \]
	 for some positive integer $\nu$. The proof will be complete once we show that $\sum_{i} \kappa(i) \leq p$, so that it must necessarily be $\nu = 1$. By contradiction, assume that $\sum_{i} \kappa(i) \geq 2\,p$. Let $\pi_0$ be a hyperplane in $V^\perp$ having the property that $\pi_0 \cap \ell(i) = \{0\}$ for every $i$. The hyperplane $\pi_0$ divides $\mathbb{S}^n$ into two relatively open hemispheres $S^\pm$ such that $v(i) \in S^+ \cup S^-$ for every $i$. Let
	 \[
	 \{ v^+(1), \ldots, v^+(J) \} = \{v(i)\} \cap S^+\,, \qquad \{ v^-(1), \ldots, v^-(L) \} = \{v(i)  \} \cap S^-\,,
	 \]
	and define accordingly the positive integers $\kappa^+(j)$ for $1 \leq j \leq J$ and $\kappa^-(l)$ for $1 \leq l \leq L$ and the halflines $\ell^+ (j)$ and $\ell^- (l)$, cf. Figure \ref{figura-1}. 
	
	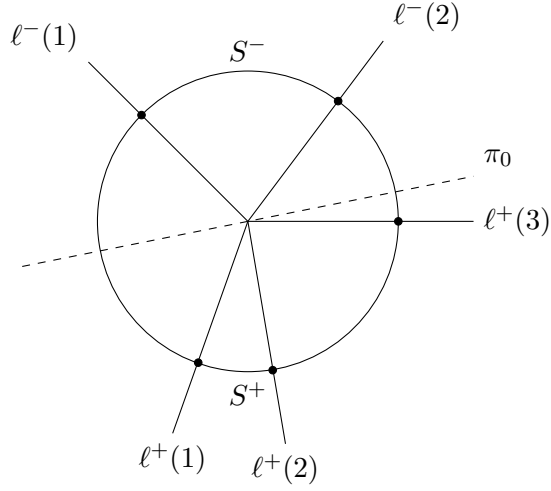
\begin{figure}
\begin{tikzpicture}
\draw (0,0) circle [radius = 2];
\draw[dashed] (-3, -0.6) -- (3, 0.6);
\node[above right] at (3,0.6){$\pi_0$};
\node[above] at (0,2){$S^-$};
\node[below] at (0,-2){$S^+$};
\draw (0,0) -- (-2.12,2.12);
\node[above left] at (-2.12,2.12){$\ell^- (1)$};
\draw (0,0) -- (3,0);
\node[right] at (3,0){$\ell^+ (3)$};
\draw (0,0) -- (1.8,2.4);
\node[above right] at (1.8,2.4){$\ell^- (2)$};
\draw (0,0) -- (0.5, -2.96);
\node[below] at (0.5,-2.96){$\ell^+ (2)$};
\draw (0,0) -- (-1,-2.82);
\node[below] at (-1, -2.82){$\ell^+ (1)$};
\draw[fill] (-1.414,1.414) circle [radius=0.05];
\draw[fill] (2,0) circle [radius=0.05];
\draw[fill] (1.2, 1.6) circle [radius=0.05];
\draw[fill] (0.33,-1.98) circle [radius=0.05];
\draw[fill] (-0.66,-1.88) circle [radius=0.05];
\end{tikzpicture}
\caption{The two hemispheres $S^+$ and $S^-$ and the rays $\ell^+ (j)$ and $\ell^- (l)$. The points $v^+ (j)$ and $v^- (l)$ are the intersections of the rays with the appropriate hemisphere.}\label{figura-1}
\end{figure}

	Since $\sum_i \kappa(i) \geq 2\, p$, the hyperplane $\pi_0$ can be chosen so that
	\[
	\sum_{j=1}^J \kappa^+(j) \geq p\,,
	\]
	 and in fact we can select integers $1 \leq m^+(j) \leq \kappa^+(j)$ such that
	 \begin{equation} \label{boundary splitting}
	  \sum_{j=1}^J m^+(j) = p\,.
	 \end{equation}
	 If we define $W= \sum_{j=1}^J m^+ (j) \a{\ell^+ (j)}$ and $S= \bC'-W$, we can observe that $\|\bC'\|= \|W\|+\|S\|$, $W+S = \bC'\; \modp$ and $\partial^p [W]=0$ and thus conclude that both $W$ and $S$ are area-minimizing currents $\modp$. 
	 
Now, let $z \in \R^{n+1}$ be such that
\[
\sum_{j=1}^J m^+(j) \, \abs{v^+(j) - z} = \min_{y} \sum_{j=1}^J m^+(j) \, \abs{v^+(j) - y}\,.
\]	 
The point $z$ lies in the convex hull of $\{v^+(1)\,, \ldots\,, v^+(J)\}$; therefore, since all vectors $v^+(j)$ belong to $S^+$, it is necessarily $z \neq 0$. In particular, 
\begin{equation} \label{mass drop}
\sum_{j=1}^J m^+(j) \abs{v^+(j) - z} < \sum_{j=1}^J m^+(j)\,.
\end{equation}
We define the integral current
\[
Z := \sum_{j=1}^J m^+(j) \a{[z,v^+(j)]}\, ,
\]
cf. Figure \ref{figura-2}.
Because of \eqref{boundary splitting}, $\pa Z \mres B_1 = \pa W \mres B_1 \, \modp$, but $\mass(Z \mres B_1) < \mass(W \mres B_1)$ due to \eqref{mass drop}. This contradicts the minimality of $W$, thus completing the proof of (iii).

\begin{figure}
\begin{tikzpicture}
\draw (0,0) circle [radius = 2];
\draw[dashed] (-3, -0.6) -- (3, 0.6);
\node[above right] at (3,0.6){$\pi_0$};
\node[below] at (-1.7,-1.42){$S^+$};
\draw[thick] (0,0) -- (2,0);
\draw[thick] (0,0) -- (0.33,-1.98);
\draw[thick] (0,0) -- (-0.66,-1.88);
\draw[fill] (2,0) circle [radius=0.05];
\node[right] at (2,0){$v^+ (3)$};
\draw[fill] (0.33,-1.98) circle [radius=0.05];
\node[below right] at (0.33,-1.98){$v^+ (2)$};
\draw[fill] (-0.66,-1.88) circle [radius=0.05];
\node[below] at (-0.66,-1.88){$v^+(1)$};
\draw (2,0) -- (0.553,-1.286) -- (-0.66,-1.88);
\draw (0.33,-1.98) -- (0.553,-1.286);
\node[below right] at (0.553,-1.286){$z$};
\draw[fill] (0.553,-1.286) circle [radius=0.05];
\end{tikzpicture}
\caption{The restriction of currents $W$ and $Z$ in $B_1^{n+1}$: in this example, the modulus is $p=3$ and the multiplicities $m^+ (j)$ are all equal to $1$. The current $W$ is represented by the thicker lines connecting the points $v^+ (j)$ to the center of the circle, while the current $Z$ is represented by the lighter lines connecting the points $v^+ (j)$ to the point $z$.}\label{figura-2}
\end{figure}
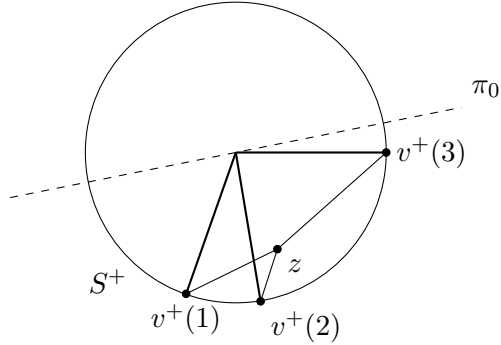

Next, we can use (iii) to prove that necessarily $\kappa(i) < \sfrac{p}{2}$ for every $i$. Indeed, since $\sum_{i} \kappa(i) = p$, the existence of $i$ such that $\kappa(i) = \sfrac{p}{2}$ is incompatible with the stationarity condition $\sum_{i} \kappa(i) v(i) = 0$, unless $\spt(\bC')$ is contained in a line. Since $\bC'$ is singular, the latter condition cannot hold, and the proof of (ii) is complete.

Now we can use (ii) to upgrade \eqref{e:splitting_modp} to the equality in the sense of rectifiable currents
\begin{equation} \label{e:splitting}
\bC = \bC' \times \a{V}
\end{equation}
claimed in (i). Indeed, the proof of \cite[Lemma 8.5]{DLHMS} shows that once a constant orientation is chosen on $V$ so that
\[
\bC' \times \a{V} = \sum_{i=1}^N \tilde\kappa(i) \, \a{{\bf H}(i)}
\] 
for $m$-dimensional half-planes $\bH(i)$ with boundary $V$ and $\tilde\kappa(i) < 0$ for all $i$ or $\tilde\kappa(i) > 0$ for all $i$, then we can represent $\bC$ as
\begin{equation} \label{e:final representation cones}
\bC =  \sum_{i=1}^N \theta(i) \, \a{{\bf H}(i)}
\end{equation}
with $\abs{\theta(i)} = \abs{\tilde\kappa(i)}$ for every $i$. On the other hand, since $\abs{\tilde\kappa(i)} < \sfrac{p}{2}$ for every $i$ and $\bC = \bC' \times \a{V} \, \modp$, it has to be $\theta(i) = \tilde\kappa(i)$ for every $i$, which completes the proof of (i).

Finally, observe that the class of area minimizing cones $\modp$ is compact by \cite[Proposition 5.2]{DLHMS}. Hence, the compactness of the subset of area minimizing cones $\modp$ with spine of dimension $m-1$ when $p$ is an odd integer is a consequence of the following elementary observation: there exists a constant $c(p) > 0$ such that
\begin{equation} \label{far from planes}
\inf\left\lbrace\flat^p_{\overline{\bB}_1} (\bC - \mathbf{Q}) \, \colon \, \mbox{$\mathbf{Q}$ is a multiple $\modp$ of an $m$-plane} \right\rbrace \geq c(p)
\end{equation}
for every such cone $\bC$. Indeed, should \eqref{far from planes} fail there would be an oriented $m$-dimensional plane $\varpi \subset \R^{m+n}$, an integer $q \in \{1,\ldots,\frac{p-1}{2}\}$, and a sequence $\{\bC_l\}_{l=1}^\infty$ of  area minimizing $m$-cones $\modp$ with $(m-1)$-dimensional spines such that $\flat^p_{\overline{\bB}_1}(\bC_l,q\,\a{\varpi}) \to 0$ as $l \to \infty$. By \cite[Proposition 5.2]{DLHMS}, the varifolds induced by $\bC_l$ converge to the varifold induced by $q\a{\varpi}$ and in particular $\omega_m q=\lim_l \|\bC_l\| (\bB_1) = \omega_m \frac{p}{2}$, contradicting that $q$ is an integer.
\end{proof}

\section{Excess-decay}  \label{sec:excess}

The main analytic estimate towards the proof of Theorem \ref{t:main} is an excess decay in the spirit of the work of Simon \cite{Simon}, see Theorem \ref{t:decay} below. Before coming to its statement we introduce some terminology.

\begin{definition}[Open books]\label{d:books}
We call {\em open book} a closed set of $\mathbb R^{m+n}$ of the form
$\bS = \bS' \times V$, where $V$ is an $(m-1)$-dimensional linear subspace and $\bS'$ consists of a finite number of halflines $\ell_1, \ldots , \ell_N$ originating at $0$, all contained in a single $2$-dimensional subspace, and all orthogonal to $V$.
The half spaces $\bH_i = \ell_i \times V$ will be called the {\em pages} of the book $\bS$, and $\sphericalangle(\bS)$ will denote the \emph{minimal} opening angle between two pages of $\bS$. The symbol $\mathscr{B}^p$ will denote the set of open books with at most $p$ pages.

We say that the book is nonflat if it is not contained in a single $m$-dimensional plane. 
\end{definition}

\begin{remark}\label{r:books_and_cones}
Note that the support of a cone $\bC_0$ as in Assumption \ref{ass:cone} is necessarily a nonflat open book in $\mathscr{B}^p$. Moreover, except for the trivial case in which the book consists of a single page, we can assign orientations and multiplicities to $\bS$ so that it becomes the support $\modp$ of a cone $\bC$ with $\partial \bC =0\, \modp$. However, such assignment of multiplicities is clearly not unique and in general there is no choice which would make $\bC$ area minimizing.
\end{remark}

\begin{definition}\label{d:excess}
The excess $\bE$ of a current $T$ with respect to an open book $\bS$ in a ball $\bB_R(q)$ is 
	\begin{equation} \label{e:excess}
	\bE (T, \bS, q, R) := R^{-(m+2)}\int_{\bB_R(q)} {\rm dist}^2 (\bar q,\bS)\, d\|T\| (\bar q)\, \, . 
	\end{equation}
\end{definition}

As anticipated, the main estimate of our paper is contained in the decay Theorem \ref{t:decay} below. 
Before stating it we need to introduce the following quantity. 

\begin{definition}\label{d:opening}
Consider a cone $\bC_0$ which is nonflat, area-minimizing representative $\modp$, has $(m-1)$-dimensional spine $V$ and is contained in $\pi_0:= T_q \Sigma$. Let $\bS_0=\spt (\bC_0)$ be the corresponding open book, $\bH_{0,i}$ its pages, and $\kappa_{0,i}$ positive coefficients so that 
\[
\bC_0 = \sum_{i=1}^{N_0} \kappa_{0,i} \a{\bH_{0,i}}\, .
\]
We say that a representative $\modp$ cone $\bC$ is {\em coherent} with $\bC_0$ if $\spt (\bC)\subset \pi_0$ and there is a rotation $O$ of $\pi_0$ with the following properties:
\begin{itemize}
    \item[(i)] $O_\sharp \bC$ is a nonflat cone with spine $V$;
    \item[(ii)] The pages of the open book $\bS := \spt (\bC)$ can be ordered as $\bH_{i,j}$ with $1\leq i \leq N_0$, $1\leq j \leq J (i)$ so that
    \[
    \bC = \sum_i \sum_j \kappa_{i,j} \a{\bH_{i,j}}
    \]
    for positive coefficients $\kappa_{i,j}$ such that $\sum_j \kappa_{i,j} = \kappa_{0,i}$;
    \item[(iii)] The angles $\theta (O, i,j)$ between the pages $O(\bH_{i,j})$ and $\bH_{0,i}$ are all smaller than $\frac{1}{4}\sphericalangle (\bS_0)$.
\end{itemize}
If $\bC$ is coherent with $\bC_0$ and $\mathcal{O}$ is the collection of rotations of $\pi_0$ which satisfy the conditions above, we denote by $\vartheta (\bC, \bC_0)$ the quantity
\[
\min_{O\in \mathcal{O}} \left(|O-{\rm Id}| + \max_{i,j} \{\theta (O,i,j)\}\right)\, .
\]
\end{definition}

While $\vartheta (\bC_0, \bC)$ is equivalent to the flat distance $\hat\flat_{\bB_1} (\bC -\bC_0)$, it presents some technical advantages as it makes it easier to iterate Theorem \ref{t:decay}.


\begin{theorem}[Decay estimate]\label{t:decay} 
	Let $p \geq 2$ be an integer and let $\bC_0$ be as in Assumption \ref{ass:cone} with $q=0$. There are constants $C$, $\eta>0$, and $\rho > 0$, depending only on $m,n,p$, and $\bC_0$ with the following property. Assume that $\Sigma$ and $T$ are as in the Assumption \ref{assumption:manifolds and currents} with $R_0=1$, that $\hat\flat^p_{\bB_1}(T-\bC_0) \leq \eta$ and that $\bC$ is a representative $\modp$ cycle (not necessarily area-minimizing), with the following properties:
	\begin{itemize}
	\item[(i)] $\bC$ is a cone coherent with $\bC_0$ and $\vartheta (\bC, \bC_0) \leq \eta$;
	\item[(ii)] $\max\{ \bE (T,\spt (\bC), 0, 1), \bA^{\sfrac 12}\}\leq \eta$.
	\end{itemize}
	Then there is a representative $\modp$ cycle $\bC'$ which is a cone coherent with $\bC_0$ and satisfies the following estimates:
	\begin{align}
	&\max \{\bE (T, \spt (\bC'), 0, \rho), (\rho \bA)^{\sfrac 12}\} \leq \rho^{1/2} \max\{\bE (T,\spt (\bC), 0, 1),\bA^{\sfrac 12}\}\, , \label{e:decay}\\
	&\hat\flat^p_{\bB_1} ((\eta_{0, \rho})_\sharp T-\bC') \leq C (\bE (T,\spt (\bC), 0, 1) + \bA)^{\sfrac 12}\, ,\label{e:flat-C'}\\
	&\vartheta (\bC', \bC_0) \leq \vartheta (\bC, \bC_0) + C (\bE (T,\spt (\bC), 0, 1) + \bA)^{\sfrac 12}\, .\label{e:incremento-angolo}
	\end{align}
\end{theorem}

The proof of the above theorem will occupy most of the remaining sections of this paper. Before coming to them, we collect here a series of technical facts about $L^2$, $L^\infty$, and flat distances, and how to compare them in our setting. The proofs of all of them are deferred to the appendix.

\subsection{Flat-excess comparison lemmas}

We start with a qualitative comparison between the excess and the flat distance.

\begin{lemma}
\label{lem:qualitative flat_excess}
Let $\bC$ be a representative of a $\modp$ cycle whose support is a nonflat open book $\bS \in \mathscr{B}^p$ and such that $\Theta_{\bC}(0) = \frac{p}{2}$. Then there is $\eta_1 = \eta_1 (\bS)>0$ with the following property. If $T_j \in \Rc_m(\bB_2)$ are such that
\begin{equation} \label{compactness assumption}
\sup_j \|T_j\|(\bB_{\sfrac{3}{2}}) < \infty\,, \qquad (\pa T_j) \mres \bB_{\sfrac{3}{2}} = 0 \; \modp\,,
\end{equation}
and
\begin{equation} \label{small flat assumption}
\hat\flat^p_{\bB_1} (T_j - \bC) < \eta_1 \qquad \forall j\, ,
\end{equation}
then
\begin{equation} \label{qualitative flat_excess}
\lim_{j\to\infty} \bE (T_j, \bS,0, 1) = 0  \quad \implies \quad \lim_{j\to \infty} \hat\flat^p_{\bB_1}(T_j - \bC) = 0\,.
\end{equation}
\end{lemma}

The following lemma is a quantitative estimate of the (modified) $p$-flat distance between a representative $\modp$ $T$ and a cone $\bC$ in terms of the $L^1$ distance of $T$ from the open book $\spt(\bC)$. By the Cauchy-Schwartz inequality, it implies a corresponding estimate with respect to the $L^2$ distance (which involves also the mass of $T$), and thus it can be regarded as a quantitative version of Lemma \ref{lem:qualitative flat_excess}.

\begin{lemma}\label{lem.flat-L^2 estimate}
If $\bC$ and $\bS$ are as in Lemma \ref{lem:qualitative flat_excess}, then there are $\eta_2=\eta_2(\bS)>0$ and $C=C(\bS)>0$ with the following property. If $T \in \Rc_m(\bB_R)$, $R \leq 1$, with 
\begin{align} \label{hp0:flat_L2 estimate}
&(\pa T) \mres \bB_1 = 0 \; \modp\,,\\ \label{hp1:flat_L2 estimate}
&\dist(q,\bS) < \eta_2\,R \quad \mbox{for all $q \in \spt(\pa^p(T \mres \bB_R))$}\,,\\ \label{hp2:flat_L2 estimate}
& \hat\flat^p_{\bB_R}(T-\bC) < \eta_2 R^{m+1}\,
\end{align}
then 
\begin{equation}\label{eq.flat_L2 estimate} 
\hat\flat^p_{\bB_{\sfrac R2}}(T- \bC) \le C \int_{\bB_R} \dist(\cdot,\bS)\, d\norm{T}  \, . 
\end{equation}
\end{lemma}

The next two lemmas give $L^\infty$-type estimates in the case of area-minimizing currents.

\begin{lemma}\label{lem.L^infty-flat estimate}
Let $\Sigma$ be as in Assumption \ref{assumption:manifolds and currents}. There is a constant $C_0 = C_0(m)$ with the following property. Let $T$ be area minimizing $\modp$ in $\Sigma \cap \bB_{3R}$ with $(\partial T) \mres \bB_{3R} = 0\;\modp$.  If $S\in \Rc_m(\bB_{3R})$ and $\spt(T), \spt(S)$ both intersect $\bB_R$ then, setting $d(\cdot) := \dist(\cdot, \spt(S))$, one has
	\[ \min\{1,d(q)\} \, d^m(q) \le C_0 \,\hat\flat^p_{\bB_{2R}}(T-S) \qquad \mbox{for every $q\in \spt(T) \cap \overline{\bB}_R$}\,.
	\]
\end{lemma}

\begin{lemma}\label{lem.L^2 controls L^infty}
Let $\Sigma$ be as in Assumption \ref{assumption:manifolds and currents}, and let $K \subset \R^{m+n}$ be a compact set with $0 \in K$. If $T$ is area minimizing $\modp$ in $\Sigma \cap \bB_1$ with $(\pa T) \mres \bB_1 = 0\; \modp$ then	
\begin{equation}\label{eq.L^2 controls L^infty}		
\dist^{m+2}(q,K) \le C_0 \int_{\bB_1} \dist^2(\cdot,K) d \norm{T} \qquad \mbox{for every $q \in \spt(T) \cap \overline{\bB}_{\sfrac12}$}\,,	
\end{equation}
for a constant $C_0$ depending on $m$.
\end{lemma}

Finally, we record the validity of the following immediate corollary of Lemmas \ref{lem.flat-L^2 estimate} and \ref{lem.L^infty-flat estimate}.

\begin{corollary} \label{flat-excess improved}
Let $\bC$ and $\bS$ be as in Lemma \ref{lem:qualitative flat_excess}. There are $\eta_2=\eta_2(\bS)>0$ and $C=C(\bS)>0$ with the following property. Let $\Sigma$ be as in Assumption \ref{assumption:manifolds and currents}, and let $T$ be area minimizing $\modp$ in $\Sigma \cap \bB_{3}$ and such that $(\pa T) \mres \bB_{3} = 0 \, \modp$. If $R \leq 1$ is such that
\begin{equation} \label{hp: flat sotto soglia}
    \hat\flat^p_{\bB_{R}}(T-\bC) < \eta_2\, R^{m+1}\,,
\end{equation}
then
\begin{equation} \label{e:flat-excess improved}
    \hat\flat^p_{\bB_{R/2}}(T-\bC) \leq C\, R^{m+1}\, \bE(T,\bS,0,R)^{\frac12}\,.
\end{equation}

\end{corollary}

%% file: graph_v2.tex
\section{Graphical parametrization}\label{sec:graph}

This section is dedicated to construct a ``multigraph'' approximation of $T$ under the assumption that its excess with respect to a nonflat open book is sufficiently small. Before proceeding we recall that the notation $W^\perp$ and ${\mathbf{p}}_W$ will be extensively used for the orthogonal complement of $W$ and the orthogonal projection onto $W$. We start by detailing the assumptions on the current $T$ which will be relevant for the rest of this section.

\begin{ipotesi}\label{a:graphical}
$\Sigma$ and $T$ satisfy the requirements of Assumption \ref{assumption:manifolds and currents} with $\Omega = \bB_{2R_0}(0)$, where $R_0 \gg 1$ is a large constant which depends only on $m$. $\pi_0$ denotes the tangent space $T_0\Sigma$. $\bC_0$ is an $m$-dimensional area minimizing cone as in Assumption \ref{ass:cone}, so that $\bS_0:=\spt(\bC_0)$ is a non-flat open book in $\mathscr{B}^p$. We will assume that $\bS_0 \subset \pi_0$, and we will call $V \subset \pi_0$ the spine of $\bS_0$. $\bC$ is a representative of a $\modp$ cycle whose support is a nonflat open book $\bS \in \mathscr{B}^p$ contained in $\pi_0$, with the same spine $V$ as $\bS_0$, and such that $\Theta_{\bC}(0) = \sfrac{p}{2}$. 
\end{ipotesi}
\begin{ipotesi}\label{a:graphical2}
Furthermore, we assume that
\begin{equation} \label{e:hyp_flat_T_C0}
\hat\flat^p_{\bB_{R_0}}(T-\bC_0) < \eta_{\bS_0}\,, \qquad \hat\flat^p_{\bB_{R_0}}(\bC-\bC_0) < \eta_{\bS_0}\,,
\end{equation}
where, denoting $\theta(\bS_0) := \tan(\sphericalangle(\bS_0)/2)$,
\begin{equation} \label{e:the flat smallness parameter}
\eta_{\bS_0} := \min\left\lbrace\eta_1(\bS_0), \eta_2(\bS_0), \frac{\theta(\bS_0)}{2^{M(m+1)}} \right\rbrace\,,
\end{equation}
for some large number $M$ to be chosen (depending only on $m$).
\end{ipotesi}

It will be useful to set some notation for the rest of this section, more precisely
\begin{equation} \label{e:who are C0 and S0}
    \bC_0 = \sum_{i=1}^{N_0} \kappa_{0,i}\, \a{\bH_{0,i}}\,, \qquad \bS_0 = \bigcup_{i=1}^{N_0} \bH_{0,i}\,,
\end{equation}
where $\bH_{0,i}=\ell_{0,i} \times V$ and $\ell_{0,i}=\{t\,v_{0,i} \, \colon \, t \ge 0\}$, $v_{0,i} \in \mathbb{S}^1 \subset V^\perp \cap \pi_0 =: V^{\perp_0}$, with $v_{0,i}$ pairwise distinct. Moreover, we assume without loss of generality that 
\begin{align*}
\pi_0 &= \{0_{n-1}\} \times \R^{m+1}\\
V &=\{0_{n-1}\} \times \{0_2\} \times \R^{m-1}\\
V^\perp &= \mathbb R^{n+1} \times \{0_{m-1}\}\\ 
V^{\perp_0} &= \{0_{n-1}\}\times \mathbb R^2 \times \{0_{m-1}\}\, .
\end{align*}
Thus every $q \in \R^{m+n}$ will be given canonical coordinates $q=(x,y)$, with $(x,0) \in V^\perp$ and $(0,y) \in V$ and for brevity we shall often identify $x=(x,0)$ and $y=(0,y)$. In particular, $|x|$ will always denote the distance of $q$ from $V$.

\begin{remark} \label{rmk:structure of C and S}

We note explicitly that the hypothesis \eqref{e:hyp_flat_T_C0}, together with the choice of $\eta_{\bS_0}$ specified in \eqref{e:the flat smallness parameter} imply that the cone $\bC$ and its support $\bS$ have the following structure:

\begin{equation} \label{e:who are C and S}
    \bC = \sum_{i=1}^{N_0} \sum_{j=1}^{\kappa_{0,i}} \a{\bH_{i,j}}\,, \qquad \bS = \bigcup_{i=1}^{N_0} \bigcup_{j=1}^{\kappa_{0,i}} \bH_{i,j}\,,
\end{equation}

where $\bH_{i,j}=\ell_{i,j} \times V$ with $\ell_{i,j} \subset V^{\perp_0}$ with possible repetitions.

\end{remark}

\subsection{Multigraphs, Whitney domains, and main approximation}
The Lipschitz approximation of $T$ will be reached through  the following notion of \emph{$p$-multifunction over an open book}.

\begin{definition}\label{def:functions on cones}
Let $V$, $\pi_0$, and $\bS$ be as in Assumptions \ref{a:graphical} and \ref{a:graphical2} and Remark \ref{rmk:structure of C and S}. Given a subset $U \subset \left[ 0, \infty \right) \times V$, a $p$-multifunction $u$ on $U$ over $\bS$ is a collection of functions $\{u_{i,j}\}$ such that:
\begin{itemize}
    \item[(a)] $i\in \{1,\ldots,N_0\}$ and for each $i$ the index $j$ ranges between $1$ and $\kappa_{0,i}$;
    \item[(b)] for every $i$ and $j$ we let $U_{i,j} := \left\lbrace z=(x,y)\in \bH_{i,j}: (|x|, y) \in U   \right\rbrace$ and 
    \begin{equation} \label{e:le funzioni uij}
    u_{i,j} \colon U_{i,j} \subset \bH_{i,j} \to \bH_{i,j}^{\perp}\, .
    \end{equation}
\end{itemize}
For every $k\in \mathbb N$ and $\alpha \in (0,1]$ we say that a $p$-multifunction $u$ on $U$ over $\bS$ is of class $C^{k,\alpha}$ (shortly $u \in C^{k,\alpha}(U)$) if $u_{i,j} \in C^{k,\alpha}(U_{i,j})$ for all $i$ and $j$. Moreover, for $z \in U_{i,j}$ we set
\begin{equation} \label{localized hoelder seminorm}
\left[ D u_{i,j} \right]_{\alpha}(z) := \inf_{R>0} \sup\left\lbrace  \frac{\abs{Du_{i,j}(z_1)-Du_{i,j}(z_2)}}{\abs{z_1-z_2}^\alpha} \, \colon \, z_1\neq z_2 \in U_{i,j} \cap \bB_R(z)\right\rbrace\,.
\end{equation}
Furthermore, for every $\zeta\in U$ we set
\[
\abs{u(\zeta)}:= \max_{i,j} \{|u_{i,j} (z_{i,j})|\}\,, \qquad \mbox{where $z_{i,j}=(x,y) \in U_{i,j}$, with $(|x|,y)=\zeta$}\,,
\] 
and define analogously $\abs{Du(\zeta)}$, and $\left[Du\right]_{\alpha}(\zeta)$. Finally, we define
\begin{equation} \label{holder norm}
\|u\|_{C^{1,\alpha}(U)} := \sup_{\zeta=(t,y) \in U} \left( t^{-1}\abs{u(\zeta)} + \abs{Du(\zeta)} + t^\alpha\left[Du\right]_\alpha(\zeta)  \right)\,,
\end{equation}
If $u \in C^{1,\alpha}(U)$, given any orientation on $V$ which naturally induces orientations on the half spaces $\bH_{i,j}$ for every $i$ and $j$, we set
\begin{equation} \label{def:graph}
\bG_{\bS}(u) := \sum_{i =1}^{N_0}\sum_{j=1}^{\kappa_{0,i}} \bG_{u_{i,j}}\,,
\end{equation}
where $\bG_{u_{i,j}}:=\a{M_{i,j}}$ is the multiplicity-one current on
\begin{equation} \label{poly graphs}
M_{i,j} := \left\lbrace z + u_{i,j}(z) \, \colon \, z \in U_{i,j} \right\rbrace
\end{equation}
with the standard orientation induced by that of $U_{i,j} \subset \bH_{i,j}$.     
\end{definition}

\begin{remark} \label{rmk:quando S e' S0}
In this section, we shall only be working with $p$-multifunctions over the cone $\bS_0$. In this case, Definition \ref{def:functions on cones} applies \emph{verbatim} with the identification $\bH_{i,j}=\bH_{0,i}$ for all $j \in \{1,\ldots,\kappa_{0,i}\}$. In particular, if $u = \{u_{i,j}\}$ is a $p$-multifunction on $U$ over $\bS_0$ then we shall simply denote $U_i$ the common domain of the functions $u_{i,j}$ for $j \in \{1,\ldots,\kappa_{0,i}\}$. 
\end{remark}

The next step before stating the main theorem of the section is to identify the domain on which the $p$-multigraph approximation of $T$ is going to be defined. This will consist of a union of cubes in a \emph{Whitney}-type decomposition of (a subset of) $\left[0, \infty \right) \times V$ with suitably good properties. Preliminarily, consider the half-cube $[0,2]\times [-2,2]^{m-1} \subset \left[ 0, \infty \right) \times V$, and the collection $\mathcal{Q}$ of sub-cubes defined as follows. First, we partition $[0,2]$ into the dyadic intervals $\{ [2^{-k},2^{-k+1}] \}_{k \ge 0}$. Then, we further divide each layer $[2^{-k},2^{-k+1}] \times [-2,2]^{m-1}$ into $2^{mM} \cdot 2^{(m-1)(k+2)}$ congruent sub-cubes of side-length $2^{-(k+M)}$, where $M$ is as in Assumption \ref{a:graphical2}, cf. Figure \ref{figura-5}. Notice that 
\begin{equation} \label{e:dist v diam} 
\frac{2^{M+1}}{\sqrt{m}} \,\diam(Q) \ge \max_{z\in Q} \dist(z,V)\ge \min_{z\in Q} \dist (z,V) \ge \frac{2^M}{\sqrt{m}} \,\diam(Q) \quad \forall Q\in \mathcal{Q}\, .
\end{equation} 

For any $Q \in \mathcal{Q}$, we shall denote $c_Q=(t_Q, y_Q)$ the center of $Q$ and $d_Q$ the diameter of $Q$. 

\begin{figure}
\begin{tikzpicture}
\foreach \x in {1,...,4}
\draw (0, {2+0.5*\x}) -- (6,{2+0.5*\x}) (0,{1+0.25*\x}) -- (6,{1+0.25*\x}) (0,{0.5+0.125*\x}) -- (6,{0.5+0.125*\x}) (0,{0.25+0.0625*\x}) -- (6,{0.25+0.0625*\x}) (0,{0.125+0.03125*\x}) 
-- (6,{0.125+0.03125*\x});
\draw (0,0.125) -- (6,0.125);
\foreach \y in {0,...,12}
\draw ({0.5*\y},2) -- ({0.5*\y},4); 
\foreach \y in {0,...,24}
\draw ({0.25*\y},1) -- ({0.25*\y},2);
\foreach \y in {0,...,48}
\draw ({0.125*\y},0.5) -- ({0.125*\y},1);
\foreach \y in {0,...,96}
\draw ({0.0625*\y},0.25) -- ({0.0625*\y},0.5);
\foreach \y in {0,...,192}
\draw ({0.03125*\y},0.125) -- ({0.03125*\y},0.25);
\fill[black] (0,0) -- (6,0) -- (6,0.125) -- (0,0.125) -- (0,0);
\draw[very thick] (0,0) -- (6,0);
\end{tikzpicture}
\label{figura-5}\caption{The Whitney decomposition of $[0,2] \times [-2,2]^{m-1}$. In the above example the parameter $M$ equals $2$.} 
\end{figure}
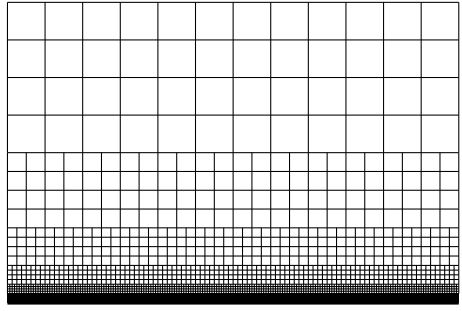

\begin{definition} \label{def:whitney}
We establish the following partial order relation in $\mathcal Q$: if $Q,Q' \in \mathcal Q$, we say that $Q$ \emph{is below} $Q'$, and we write $Q \preceq Q'$, if and only if $\mathbf{p}_V(Q) \subset \mathbf{p}_V(Q')$.
Let $T$ and $\bC$ be as in Assumptions \ref{a:graphical} and \ref{a:graphical2}, with $\bS = \spt(\bC) \in \mathscr{B}^p$, and let $\tau \in \left( 0, 1 \right)$. The {\em Whitney domain} of $[0,2] \times [-2,2]^{m-1}$ associated to $(T,\bS,\tau)$, denoted by $\mathcal{W}=\mathcal{W}(T,\bS,\tau)$, is the subfamily of $Q \in \mathcal{Q}$ such that
\begin{equation}
\bE\left(T,\bS,y_{Q'},\bar M d_{Q'}\right) < \tau^2 \qquad \forall Q \preceq Q'\,,
\end{equation}
where $\bar M = 2^{M+2}/\sqrt{m}$.
\end{definition}

\begin{remark}\label{rem_whitney_houston}
Note that $Q \in \mathcal W$ and $Q \preceq Q'$ imply $Q' \in \mathcal W$, and that every cube $\hat Q \in \mathcal Q$ for which there are no cubes above $\hat Q$ (henceforth called \emph{cubes in the top sub-layer}) belongs to $\mathcal W$ as soon as $\bE(T,\bS,0,R_0)$ is suitably small (depending on $\tau$).
\end{remark}

Since we will often deal with suitable dilations of the cubes in $\mathcal{Q}$ we introduce the following notation. For $\lambda> 0$, $\lambda Q$ is the cube with the same center $c_Q$ as $Q$ and diameter $d_{\lambda Q} = \lambda \, d_Q$. Considering $U= \lambda Q$ as in Definition \ref{def:functions on cones}, we let $\lambda Q_i$ be the corresponding domains $U_i\subset \bH_{0,i}$, as described in Definition \ref{def:functions on cones}(b) and Remark \ref{rmk:quando S e' S0}. Given a Whitney domain $\mathcal W = \mathcal W (T,\bS,\tau)$ and $\lambda > 0$, we shall also denote by $U_{\lambda \mathcal{W}}$ the union 
\begin{equation}
    U_{\lambda \mathcal W} := \bigcup_{Q \in \mathcal W} \lambda Q\,,
\end{equation}
and, setting $U_{\mathcal W} = U_{1\,\mathcal W}$, we define the ``distance'' function $\varrho_{\mathcal W} \colon \left[ -2, 2 \right]^{m-1} \to \left[ 0, 2 \right]$ as
\begin{equation} \label{oh dear vrho}
    \varrho_{\mathcal W}(y) := \inf \left\lbrace t \, \colon \, (t,y) \in U_{\mathcal W}  \right \rbrace\,.
\end{equation}
The {\em graphicality region} $R_{\mathcal W}$ is the rotationally invariant set 
\begin{equation} \label{rotation of Whitney}
    R_{\mathcal W} := \left\lbrace q=(x,y) \in \R^{m+n} \, \colon \, y \in \left[-2,2\right]^{m-1} \quad \mbox{and} \quad \varrho_{\mathcal W}(y)\leq |x|\leq 2     \right\rbrace\,.
\end{equation}

\begin{theorem}[Graphical parametrization] \label{thm:graph_v1}
Let $T,\Sigma, \bC,\bS, \bC_0$, and $\bS_0$ be as in Assumptions \ref{a:graphical} and \ref{a:graphical2}. For any $\beta \in \left(0, \frac{1}{2} \right)$ there are $\tau > 0$ and $\eps_1 > 0$, depending on $(m,n,p,\bS_0, \beta)$ with the following property. If
\begin{align} 
      \bA+\bE(\bC,\bS_0,0,R_0)\; \leq\; & \eps_1^2\,, \label{e:hyp_graph_cone}\\
\mbox{and}\qquad      \bE:=\bE(T,\bS,0,R_0)\; \leq\; & \eps_1^2 \label{e:hyp_graph_current}\,,
\end{align}
then there is a $p$-multifunction $u = \{u_{i,j}\}$ over $\bS_0$ of class $C^{1,\frac12}$ on $U_{4 \mathcal W}$ with $\mathcal W = \mathcal W(T,\bS,\tau)$ and with $u_{i,j} \colon (U_{4\mathcal W})_i \subset \bH_{0,i} \to \bH_{0,i}^{\perp_0}$ for all $i$ and $j$ such that, for some constant $C_2 = C_2 (m,n,p,\bS_0)$,
\begin{itemize}
\item[(i)] every $Q \in \mathcal Q$ with $d_Q \ge C_2\, \frac{\bE^{1/(m+2)}}{\beta} $ belongs to $\mathcal W$;
\item[(ii)] $\|u\|_{C^{1,\frac12}(U_{4\mathcal W})} \leq \beta$;
\item[(iii)] $T \mres R_{\mathcal W} = \mathbf{G}_{\bS_0}(v) \mres R_{\mathcal W} $, where $v$ is the $p$-multifunction on $U_{4\mathcal W}$ over $\bS_0$ defined by 
\begin{equation} \label{e:function to manifold}
v_{i,j}(z) := u_{i,j}(z) + \Psi(z + u_{i,j}(z))
\end{equation}
($\Psi$ is the map detailed in Assumption \ref{assumption:manifolds and currents});
\item[(iv)] the following estimate holds:
\begin{equation} \label{e:L^2 estimate}
    \int_{\bB_2 \setminus R_{\mathcal W}} |x|^2 \, d\|T\|(q) + \int_{\bB_2 \setminus R_{\mathcal W}} |x|^2 \, d\|\bC\|(q) \leq C_2 \, \beta^{-(m+2)}\, \bE\, ,
\end{equation}
where, for $q \in \bB_2 \setminus R_{\mathcal W}$, $|x|$ denotes, as usual, the distance of $q$ from the spine $V$.
\end{itemize}
\end{theorem}

\subsection{White's epsilon-regulary theorem}
In what follows, we will use the shorthand notation ${\rm P}_0$, $\mathbf{p}_i$,
and $\mathbf{p}_i^\perp$ for the orthogonal projections $\mathbf{p}_{\pi_0}$, $\mathbf{p}_{\bH_{0,i}}$ and $\mathbf{p}_{\bH_{0,i}^\perp}$ (where, with a slight abuse of notation, we are identifying $\bH_{0,i}$ with the $m$-dimensional linear plane containing it).
We shall then set $T' := ({\rm P}_0)_\sharp T$; furthermore, given a cube $Q \in \mathcal Q$, letting $Q_i$ denote the corresponding cube in $\bH_{0,i}$ and $\lambda Q_i$ its dilation with center $c_{\lambda Q_i}$ and diameter $d_{\lambda Q_i}$ we define
\begin{equation}\label{e:tilted_cylinders}
    {\rm C}(\lambda Q_i,\delta) := \left\lbrace q \in \pi_0 \, \colon \, \mathbf{p}_i(q) \in \lambda Q_i \quad \mbox{and} \quad |\mathbf{p}_i^\perp (q)| \leq \delta \bar M \, d_{\lambda Q_i}     \right\rbrace\,,
\end{equation}
and
\begin{equation} \label{e:family of tilted cylinders}
    \mathcal{C}(\lambda Q,\delta) := \bigcup_{i=1}^{N_0} {\rm C}(\lambda Q_i,\delta)\,.
\end{equation}
In other words, for each $i$ the set ${\rm C}(\lambda Q_i,\delta)$ is a cylinder in $\pi_0$ with cross section $\lambda Q_i$, axis orthogonal to $\bH_{0,i}$ and height $2\,\bar M \delta\,d_{\lambda Q_i}$, while $\mathcal{C}(\lambda Q,\delta)$ is the union of all such cylinders. With these notation in place, we can state and prove the following lemma, which is the key technical step towards the proof of the main theorem of this section, Theorem \ref{thm:graph_v1}. 

\begin{lemma}\label{lem:White}
Let $\bC_0$ and $\bS_0$ be as in Assumptions \ref{a:graphical} and \ref{a:graphical2}. There exists $\delta_0=\delta_0(m,p,\bS_0)$ with the following property. Let $\delta \in \left( 0, \delta_0\right]$ be arbitrary, and set
\begin{equation} \label{e:first upper bound on tau}
\bar{\eps}_1^2 = \bar{\tau}^2 := \left( \frac{\delta}{16 C_0}\right)^{m+2} \,,
\end{equation}
where $C_0$ is the constant from Lemma \ref{lem.L^2 controls L^infty}. If $T$ and $\bC$ are as in Assumptions \ref{a:graphical} and \ref{a:graphical2}, $\bC$ satisfies \eqref{e:hyp_graph_cone} for some $\eps_1 < \bar{\eps}_1$, and 
\begin{equation} \label{e:ipotesi di cubo}
    Q \in \mathcal{W}(T, \bS, \tau) \qquad \mbox{for some $\tau < \bar\tau$}\,,
\end{equation}
then:
\begin{itemize}
    \item[(a)] $\spt (T)\cap \{(x,y): (|x|,y)\in Q\} \subset {\rm P}_0^{-1} (\mathcal{C} (4Q, \delta))$;
    \item[(b)] there exists a $p$-multifunction $u^Q = \{u^Q_{i,j}\} \in C^{1,\frac12}(4Q)$ over $\bS_0$ with $u^Q_{i,j} \colon 4 Q_i \subset \bH_{0,i} \to \bH_{0,i}^{\perp_0}$ for all $i$ and $j$ and
\begin{equation} \label{e:Holder estimate on functions}
     \| u^Q \|_{C^{1,\frac12}(4Q)} \leq C_1 \, \delta
\end{equation}
for some constant $C_1 = C_1(m,p) > 0$, and such that
\begin{equation} \label{T' is graphical}
    T' \mres \mathcal{C}(4Q,\delta) = \mathbf{G}_{\bS_0}(u^Q) \,.
\end{equation}
\end{itemize}
\end{lemma}
\begin{proof}
First, let us observe that since the manifold $\Sigma$ is the graph of $\Psi$ on $\pi_0$, $({\rm Id} + \Psi)\circ {\rm P}_0$ is the identity map on $\Sigma$. Letting $\delta_{\R^{m+n}}$ denote the standard Euclidean metric on $\R^{m+n}$, we can then define $g_\Psi:= ({\rm Id} + \Psi)^\sharp \delta_{\R^{m+n}}$ to be the associated pull-back metric on $\pi_0$. The current $T'=({{\rm P}_0})_\sharp T $ is supported on $\pi_0$ and area minimizing mod(p) in $\pi_0 \cap \bB_{2R_0}$ with respect to the area functional relative to the metric $g_\Psi$, which falls in the class of elliptic functionals in the sense of Almgren.

\smallskip

We are going to divide the proof into several steps.

\smallskip

\textit{Step one.} Let $\delta_0 < \eta_{\bS_0}$. We claim that for every $Q \in \mathcal W(T,\bS,\tau)$ the following holds:
\begin{equation} \label{e:intorno_uncinato}
  \spt \left(  T_{y_Q,\bar{M}d_Q} \mres (\bB_2 \setminus \overline{B}_{1/8}(V)) \right) \subset \left\lbrace q=(x,y) \in \R^{m+n} \, \colon \, \dist(q,\bS_0) < \delta\, |x|\right\rbrace\,.
\end{equation}
Indeed, applying Lemma \ref{lem.L^2 controls L^infty} with $K=\bS$ and $T_{y_Q,\bar{M}d_Q}$ in place of $T$, we find that
\begin{equation} \label{triangle1}
\dist^{m+2}(q,\bS) \leq C_0 \int_{\bB_4} \dist^2(q',\bS) \, d\|T_{y_Q,\bar{M}d_Q}\|(q') \qquad \mbox{for every $q \in \spt(T_{y_Q,\bar{M}d_Q} ) \cap \overline{\bB}_2$}\,.
\end{equation}

On the other hand, since $\bS=\spt(\bC)$, Lemma \ref{lem.L^2 controls L^infty} implies that for any point $w \in \bS \cap \overline{\bB}_4$ 
\begin{equation} \label{triangle2}
\dist(w,\bS_0)^{m+2} \leq C_0 \int_{\bB_8} \dist^2(\cdot, \bS_0)\, d\|\bC\|\,.
\end{equation}

Thus, putting together \eqref{triangle1} and \eqref{triangle2} we deduce that for any $q \in \spt(T_{y_Q,\bar M d_Q} ) \cap \bB_2$

\begin{equation} \label{triangle final}
\begin{split}
    \dist(q,\bS_0) & \leq C_0 \, \left( \bE(T,\bS,y_Q,\bar M d_Q)^{1/(m+2)} + \bE(\bC,\bS_0,0,R_0)^{1/(m+2)}    \right) \\
& \leq C_0\, (\tau^{2/(m+2)} + \eps_1^{2/(m+2)}) < \frac{\delta}{8}\,.    
    \end{split}
\end{equation}
In particular, if $q = (x,y) \notin B_{1/8}(V)$ then $|x| > 1/8$, and thus the above estimate gives
\begin{equation} \label{triangle final final}
     \dist(q,\bS_0) < \delta\, |x| \qquad \mbox{for $q=(x,y) \in \spt(T_{y_Q,\bar M d_Q}) \cap (\bB_2 \setminus \overline{B}_{1/8}(V))$}\,.
\end{equation}
We observe in passing that  \eqref{e:intorno_uncinato} immediately implies conclusion (a).

\smallskip

\textit{Step two.} Since $\delta$ is smaller than $\theta := \tan(\sphericalangle(\bS_0)/2)$, \eqref{e:intorno_uncinato} implies that for every $Q \in \mathcal W (T,\bS, \tau)$ we can decompose
\begin{eqnarray} \label{e:splitting_branches}
    && T_{y_Q,\bar{M}d_Q} \mres (\bB_2 \setminus \overline{B}_{1/8}(V))  =  \sum_{i=1}^{N_0} \tilde{T}_i^Q\,, \\ \label{support of branches}
  &&  \tilde{T}_i^Q := T_{y_Q,\bar{M}d_Q} \mres \left( (\bB_2 \setminus \overline{B}_{1/8}(V)) \cap \{q=(x,y)\, \colon \, {\rm dist}(q,\bH_{0,i}) < \delta\,|x|\} \right)\,,
\end{eqnarray}
where each $\tilde{T}_{i}^Q$ has no boundary $\modp$ in $\bB_2 \setminus \overline{B}_{1/8}(V)$ by \cite[Lemma 6.1]{DLHMS} and
\begin{eqnarray} \label{e:disjoint_pieces}
    \spt(\tilde T_i^Q) \cap \spt(\tilde T_{i'}^Q) & = & \emptyset \qquad \mbox{whenever $i \ne i'$}\,, \\
    \|T_{y_Q,\bar{M}d_Q}\|(\bB_2 \setminus \overline{B}_{1/8}(V)) &=& \sum_{i=1}^{N_0} \mass(\tilde T_i^Q)\,.
\end{eqnarray}
In particular, each $\tilde T_i^Q$ is area minimizing $\modp$ in $\Sigma \cap (\bB_2 \setminus \overline{B}_{1/8}(V))$. Rescaling back, we have an analogous decomposition 
\begin{equation} \label{e:splitting_rescaled}
    T \mres (\bB_{2\bar{M} d_Q}(y_Q) \setminus \overline{B}_{\bar{M}d_Q/8}(V)) = \sum_{i=1}^{N_0} T_i^Q\,,
\end{equation}
where each ${T}_{i}^Q$ has no boundary $\modp$ in $\bB_{2\bar{M}d_Q}(y_Q) \setminus \overline{B}_{\bar{M}d_Q/8}(V)$ and
\begin{eqnarray} \label{e:disjoint_pieces_rescaled}
    \spt(T_i^Q) \cap \spt(T_{i'}^Q) & = & \emptyset \qquad \mbox{whenever $i \ne i'$}\,, \\
    \|T\|(\bB_{2\bar{M}d_Q}(y_Q) \setminus \overline{B}_{\bar{M}d_Q/8}(V)) &=& \sum_{i=1}^{N_0} \mass( T_i^Q)\,.
\end{eqnarray}
In particular, each $T_i^Q$ is area minimizing $\modp$ in $\Sigma \cap (\bB_{2\bar{M}d_Q}(y_Q) \setminus \overline{B}_{\bar{M}d_Q/8}(V))$.

\medskip

\textit{Step three.} From \eqref{e:splitting_rescaled}, we deduce that, for $T'=({\rm P}_0)_\sharp T$, we have 
\begin{equation} \label{splitting_projection}
    T' \mres (\bB_{\bar M d_Q}(y_Q) \setminus \overline{B}_{\bar M d_Q/8}(V)) = \sum_{i=1}^{N_0} (T')_i^Q\,,
\end{equation}
where each $(T')_i^Q := [({\rm P}_0)_\sharp T_i^Q]\mres (\bB_{\bar M d_Q}(y_Q) \setminus \overline{B}_{\bar M d_Q/8}(V))$ satisfies
\begin{equation} \label{support of currents}
\spt((T')_i^Q) \subset \left( \bB_{\bar M d_Q}(y_Q) \setminus \overline{B}_{\bar M d_Q/8}(V) \right) \cap \{q=(x,y) \in \pi_0\, \colon \, {\rm dist}(q,\bH_{0,i}) < \delta\,|x|\}\,,
\end{equation}  
by \eqref{support of branches}, and it is area minimizing $\modp$ in $\pi_0 \cap \left( \bB_{\bar M d_Q}(y_Q) \setminus \overline{B}_{\bar M d_Q/8}(V) \right)$ with respect to the area functional relative to the metric $g_\Psi$ on $\pi_0$. Furthermore, since
\[
\spt^p(\partial T_i^Q) \subset \left(\partial\bB_{2\bar{M} d_Q}(y_Q) \cup  \partial B_{\bar{M}d_Q/8}(V))\right)\cap \{q=(x,y)\, \colon \, {\rm dist}(q,\bH_{0,i}) \leq \delta\,|x|\}\,,
\]
each $(T')_i^Q$ has no boundary $\modp$ in $\bB_{\bar M d_Q}(y_Q) \setminus \overline{B}_{\bar M d_Q/8}(V)$. 

Now, we observe that, due to \eqref{support of currents}, 
\begin{equation} \label{height bound}
    \spt((T')_i^Q) \cap \{q \in \pi_0 \, \colon \, {\bf p}_i(q) \in  32 Q_i \} \subset {\rm C}(32Q_i,\delta)\,,
\end{equation}
and that
\begin{equation} \label{cylinder is good}
{\rm C}(32Q_i,\delta) \subset \pi_0 \cap \left( \bB_{\bar M d_Q}(y_Q) \setminus \overline{B}_{\bar M d_Q/8}(V) \right)
\end{equation}
as soon as $\bar M$ is large enough (depending on $m$).

Finally, notice that the constancy lemma $\modp$ implies that
\begin{equation} \label{constancy in action}
    ({\bf p}_i)_\sharp [(T')_i^Q \mres \{q \in \pi_0 \, \colon \, {\bf p}_i(q) \in  32 Q_i \}] = \kappa_{Q_i}\,\a{32 Q_i} \; \modp
\end{equation}
for some constant $\kappa_{Q_i} \in \mathbb Z \cap \left( - \frac{p}{2}, \frac{p}{2} \right]$.

\medskip

\textit{Step four.} In this step, we prove that 
\begin{equation}\label{molteplicitagiuste}
 \kappa_{Q_i} = \kappa_{0,i},  
\end{equation}
for every $i\in\{1,\ldots,N_0\}$: this will imply, in particular, that $|\kappa_{Q_i}| < p/2$.

Observe that to this aim it is sufficient to prove \eqref{molteplicitagiuste}
with $Q$ replaced by any cube $Q'$ with $Q\preceq Q'$. Indeed, since for any two consecutive cubes $\hat Q$ and $\tilde Q$ the cubes $32\hat Q$ and $32\tilde Q$ overlap on a region of positive area, then by \eqref{constancy in action} the equality \eqref{molteplicitagiuste} would propagate from $Q'$ to $Q$ along a chain which connects them. 

Let us then choose $Q' \in \mathcal W$ such that $Q\preceq Q'$ and there is no cube of $\mathcal Q$ above $Q'$. By Remark \ref{rem_whitney_houston} we infer that the conclusions of steps one, two, and three hold with $Q'$ in place of $Q$.  

We claim now that there exists $\lambda \in \left( 16, 32 \right)$ such that
\begin{equation}\label{flatpiccola}
    \hat\flat^p_{\bB_{R_0}}\left(({\bf p}_i)_\sharp [(T')_i^{Q'} \mres \{q \in \pi_0 \, \colon \, {\bf p}_i(q) \in  \lambda {Q'}_i \}] - \kappa_{0,i}\,\a{\lambda {Q'}_i}\right)<\Ha^m(\lambda Q')\,,
\end{equation}
which, due to \eqref{constancy in action}, implies in particular that \eqref{molteplicitagiuste} holds for the cube $Q'$, since the (modified) $p$-flat distance between two $m$-currents supported on an $m$-dimensional cube is the mass $\modp$ of their difference.

In order to prove \eqref{flatpiccola}, we begin observing that 
$$\hat\flat^p_{\bB_{R_0}}(T'-\bC_0) < \eta_{\bS_0}\,.$$This follows from \eqref{e:hyp_flat_T_C0} and the fact that $T'-\bC_0=({\rm{P_0}})_\sharp(T-\bC_0)$, because ${\rm{P_0}}$ is a 1-Lipschitz map. 

Let $R, S$ and $Z$ be such that
\begin{equation} \label{e:small flat at large scale}
T'-\bC_0=R+\partial S + pZ \qquad \mbox{with} \qquad  \|R\|(\bB_{R_0}) + \|S\|(\bB_{R_0})\leq 2\eta_{\bS_0}\,. 
\end{equation}
For each $i$, let us define the function $f_i \colon \pi_0 \to \left[ 0, \infty \right)$ by
\begin{equation} \label{e:cyl_inducing function}
    f_i(q) := \max \left\lbrace 2^{M+1}\, |{\bf p}_i(q) - c_{Q'_i}|_{\infty} , 64\, \theta^{-1} |{\bf p}_i^\perp(q)| \right\rbrace\,,
\end{equation}
having denoted $|v|_\infty := \max\{|z_h| \, \colon \, h=1,\ldots,m  \}$ if $v=\left(z_1,\ldots,z_{m}\right)$ is a decomposition of $v$ in the orthonormal system of coordinates on $\bH_{0,i}$ having $V$ as a coordinate hyperplane. Using that there is no cube of $\mathcal Q$ above $Q'$, so that the side length of $Q'$ is $2^{-M}$, together with the definition of $\bar M$, it is not difficult to see that the above definition of $f_i$ implies that, for any $16\leq\lambda \leq 32$, the sublevel set $\{f_i\leq\lambda\}$ coincides with the cylinder
\begin{equation} \label{e:cyls are sublevels}
    \{f_i \leq \lambda\} = {\rm C}\left(\lambda Q'_i,\frac{\theta\, \lambda}{64\,\bar M d_{\lambda Q_i'}}\right)\,.
\end{equation}

By the slicing formula \cite[Lemma 28.5]{Simon83}, for almost every $16\leq\lambda\leq 32$ we have from \eqref{e:small flat at large scale} that
\begin{equation} \label{e:slicing formula in action}
(T' - \bC_0) \mres \{f_i \leq \lambda\} = R \mres \{f_i \leq \lambda\} + \partial [S \mres \{f_i \leq \lambda\}] - \langle S, f_i, \lambda \rangle + p\, Z \mres \{f_i \le \lambda\}\,.
\end{equation}

Now, observe that, by the definition of $\theta$, each cylinder 
\[
{\rm C}\left(\lambda Q'_i,\frac{\theta\, \lambda}{64\,\bar M d_{\lambda Q_i'}}\right) \subset {\rm C}\left(\lambda Q'_i,\frac{\theta}{2\,\bar M d_{\lambda Q_i'}}\right)
\]
does not intersect $\bH_{0,j}$ for any $j\neq i$ for $\bar M$ large enough. Hence,
\begin{equation} \label{e:cone is ok}
    \bC_0 \mres \{f_i \leq \lambda \} = \kappa_{0,i}\, \a{\lambda Q'_i}\,.
\end{equation}
Next, we claim that
\begin{equation} \label{e:breaking good}
    T' \mres \{f_i \leq \lambda \} = (T')_i^{Q'} \mres \{f_i \leq \lambda \} = (T')_i^{Q'} \mres \{ q \in \pi_0 \, \colon \, {\bf p}_i(q) \in \lambda Q_i'\}\,.
\end{equation}

Indeed we have $$\{f_i\leq \lambda\}\cap\left\{q=(x,y):{\rm dist}(q,\bH_{0,j})<\frac{\theta}{2}\,|x|\right\}=\emptyset \qquad \mbox{for every $j \neq i$}\,,$$ 
and therefore, since $\delta < \frac{\theta}{2}$, \eqref{splitting_projection} and \eqref{support of currents} imply the first identity in \eqref{e:breaking good}. The second identity follows from \eqref{height bound} and \eqref{e:cyls are sublevels}, since, for $\delta < \frac{\theta}{2^9}$, ${\rm C}(\lambda Q_i',\delta) \subset {\rm C}\left(\lambda Q'_i,\frac{\theta\, \lambda}{64\,\bar M d_{\lambda Q_i'}}\right)$. Now by \eqref{e:cone is ok} and \eqref{e:breaking good}, we can rewrite \eqref{e:slicing formula in action} as
\begin{equation} \label{e:slicing2}
\begin{split}
    &(T')_i^{Q'} \mres \{q \in \pi_0 \, \colon \, {\bf p}_i(q) \in  \lambda Q_i' \} - \kappa_{0,i}\, \a{\lambda Q'_i} \\ 
    & \qquad \qquad = R \mres \{f_i \leq \lambda\} + \partial [S \mres \{f_i \leq \lambda\}] - \langle S, f_i, \lambda \rangle + p\, Z \mres \{f_i \le \lambda\} \,.
\end{split}
\end{equation}
By \cite[Lemma 28.5 (1)]{Simon83}, we can find $\lambda\in[16,32]$ such that 
\begin{equation}
    \mass(\langle S , f_i , \lambda \rangle) \le \frac{1}{16}{\rm Lip}(f_i)\, \|S\|(\{f_i \leq 32\}) \leq \frac{1}{16} {\rm Lip}(f_i)\, \|S\|(\bB_{R_0})\,.
\end{equation}
In turn, using that ${\rm Lip}(f_i) \leq 64\,\theta^{-1}+2^{M+1}$, \eqref{e:small flat at large scale} yields
\begin{align}
    \hat\flat^p_{\bB_{R_0}}\left((T')_i^{Q'} \mres \{q \in \pi_0 \, \colon \, {\bf p}_i(q) \in  \lambda {Q}_i' \} - \kappa_{0,i}\,\a{\lambda {Q'}_i}\right) &\leq 2\eta_{\bS_0}\,(1 + 4\,\theta^{-1}+2^{M-3})\nonumber\\
    & < \Ha^m(\lambda Q')\,,\label{quasiflatpiccola}
\end{align}
as soon as \[\eta_{\bS_0}<\frac{1}{2^{m(M-2)}(2+8\,\theta^{-1}+2^{M-2})}\,,\] 
whose validity is guaranteed by the choice of $\eta_{\bS_0}$ in \eqref{e:the flat smallness parameter}. Lastly, \eqref{flatpiccola} follows from \eqref{quasiflatpiccola} since ${\bf p}_i$ is 1-Lipschitz.

\medskip

\textit{Step five.} By \eqref{height bound}, \eqref{constancy in action}, and \eqref{molteplicitagiuste}, and recalling that $1 \leq \kappa_{0,i} < p/2$, we see now that, for each $i \in \{1,\ldots,N_0\}$, the currents $(T')_i^Q$ satisfy the hypotheses of the regularity theorem in \cite[Theorem 4.5]{White86} in $\{q \in \pi_0 \, \colon \, {\bf p}_i(q) \in  32 Q_i \}$ as soon as $\delta_0$ is chosen such that
\[
32\,\delta_0 \bar M \leq \delta_{BW}\,,
\]
where $\delta_{BW}=\delta_{BW}(m,p)$ denotes the regularity threshold of \cite[Theorem 4.5]{White86}. We can then conclude from \cite[Theorem 1]{SchoenSimon} that there exist precisely $\kappa_{0,i}$ functions $u_{i,j} \colon 4Q_i \to \bH_{0,i}^{\perp_0} \simeq \R$ of class $C^{1,1/2}$ such that 
\begin{itemize}
    \item[(i)] $u_{i,1} \leq u_{i,2} \leq \ldots \leq u_{i,\kappa_{0,i}}$ in $4Q_i$;
    \item[(ii)] given $j,j' \in \{1,\ldots,\kappa_{0,i}\}$ with $j \leq j'$ it is either $u_{i,j} \equiv u_{i,j'}$ in $4Q_i$ or $u_{i,j}(q) < u_{i,j'}(q)$ for every $q \in 4Q_i$;
    \item[(iii)] $\|u_{i,j}\|_{C^{1,\frac12}(4Q_i)} \leq C_1\,\delta$ for every $i$ and $j$, for some constant $C_1=C_1(m,p)>0$;
    \item[(iv)] the current $(T')_i^Q \mres \{q \in \pi_0 \, \colon \, {\bf p}_i(q) \in  4 Q_i \}$ coincides with the multigraph defined by $\{u_{i,j}\}_{j=1}^{\kappa_{0,i}}$.
\end{itemize}

Finally, we let $u^Q$ denote the $p$-multifunction on $Q$ over $\bS_0$ defined by the functions $\{u_{i,j}\}$ as also $i$ is let vary in $\{1,\ldots, N_0\}$. Conclusion (iii) above readily implies \eqref{e:Holder estimate on functions}, whereas \eqref{T' is graphical} follows from (iv) together with \eqref{splitting_projection} and \eqref{height bound}. 
\end{proof}

\subsection{Proof of Theorem \ref{thm:graph_v1}}
Let $Q \in \mathcal Q$, and, with the usual meaning of $y_Q$ and $d_Q$, observe that, whenever $d_Q \ge \sigma$ it holds

\begin{equation} \label{e:excess estimate for far away cubes}
    \bE(T,\bS,y_Q,\bar M d_Q) = \frac{1}{(\bar M d_Q)^{m+2}} \int_{\bB_{\bar M d_Q}(y_Q)} \dist^2(q,\bS) \, d\|T\| \leq \left( \frac{R_0}{\bar M} \right)^{m+2}\, \frac{\bE}{\sigma^{m+2}}\,.
\end{equation}

In particular, choosing $\sigma = C_2\, \frac{\bE^{1/(m+2)}}{\beta}$ guarantees the validity of (i) as soon as 
\begin{equation} \label{e: lower bound on tau}
    \tau^2 \geq \left( \frac{R_0}{C_2\, \bar M}\right)^{m+2} \, \beta^{m+2}\,.
\end{equation}

Next, fix $\delta := \min \left\lbrace \frac{\beta}{C_{1}}, \delta_0 \right\rbrace$. If $\eps_1$ and $\tau$ are sufficiently small, explicitly if $\eps_1 < \bar{\eps}_1$ and $\tau < \bar\tau$ with $\bar\eps_1$ and $\bar\tau$ defined by \eqref{e:first upper bound on tau} in correspondence with this choice of $\delta$, we can apply Lemma \ref{lem:White} to every cube $Q \in \mathcal{W}$. We can therefore guarantee that the conclusion in (i) is satisfied by choosing the constant $C_2$ so that we can find an appropriate $\tau$ satisfying
\begin{equation} \label{e: choice of tau}
    \left( \frac{R_0}{C_2\, \bar M}\right)^{m+2} \, \beta^{m+2} \leq \tau^2 < \bar\tau^2 \leq \left( \frac{1}{16 C_0 C_1} \right)^{m+2} \, \beta^{m+2}\,.
\end{equation}

From Lemma \ref{lem:White} it then follows that for every $Q \in \mathcal W$ there exists a $p$-multifunction $u^Q\in C^{1,\frac12}(4Q)$ over $\bS_0$ such that \eqref{e:Holder estimate on functions} and \eqref{T' is graphical} hold true with $\beta$ replacing $C_1\,\delta$ in the right-hand side of \eqref{e:Holder estimate on functions}. Since, for any two adjacent cubes $\tilde Q$ and $\hat Q$ in $\mathcal W$, the cubes $4\tilde Q$ and $4\hat Q$ intersect on a set of positive measure, each function $u^Q$ is the restriction, on $4Q$, of a \emph{unique} $p$-multifunction $u$ of class $C^{1,\frac12}$ on $U_{4 \mathcal W}=\bigcup_{Q \in \mathcal{W}} 4Q$ over $\bS_0$ satisfying (ii). In particular, it follows from \eqref{T' is graphical} that, setting $\mathcal{C}_{4\mathcal W, \delta}:= \bigcup_{Q\in \mathcal{W}} \mathcal{C}(4Q,\delta)$
\begin{equation} \label{e:T' is globally graphical}
    T' \mres \mathcal{C}_{4\mathcal W, \delta} =  \mathbf{G}_{\bS_0}(u) \,.
\end{equation}

Recalling that $T'=({\rm P}_0)_\sharp T$, and that, on the manifold $\Sigma$, ${\rm P}_0$ is invertible with inverse ${\rm Id} + \Psi$, the $p$-multifunction $v$ on $U_{4 \mathcal W}$ over $\bS_0$ defined by \eqref{e:function to manifold} satisfies
\begin{equation} \label{e:T is globally graphical}
    T \mres {\rm P}_0^{-1}(\mathcal C_{4\mathcal W,\delta}) = {\bf G}_{\bS_0}(v)\,.
\end{equation}
Conclusion (iii) follows then at once from Lemma \ref{lem:White}(a) which immediately implies that $\spt(T) \cap ( R_{\mathcal W} \setminus V ) \subset {\rm P}_0^{-1} (\mathcal{C}_{4\mathcal{W},\delta})$.

We finally come to (iv). Observe first that 
\[
\bB_2\setminus R_{\mathcal{W}} \subset \bigcup \left\lbrace \bB_{\varrho_{\mathcal{W}} (y)} (y) \, \colon \, y \in \left[-2,2\right]^{m-1} \quad \mbox{with} \quad \varrho_{\mathcal W}(y) > 0 \right\rbrace\, .
\]
Next, we observe that, by the definition of $\varrho_{\mathcal W}(y)$ and the properties of cubes in $\mathcal Q$, for each $y\in [-2,2]^{m-1}$ with $\varrho_{\mathcal W}(y) > 0$ there is a $Q\in \mathcal{Q}$ such that $|y-y_Q|\leq C \varrho_{\mathcal{W}} (y) \leq C d_Q$ and $\bE (T, \bS, y_Q, \bar M d_Q) \geq \tau^2$, where the constant $C$ depends only on $m$ and $M$. In particular, for some other positive constant $\bar C (m, M)$,
\[
\bE (T, \bS, y, \bar C\varrho_{\mathcal{W}} (y)) \geq \bar C^{-1} \tau^2\, .
\]
Apply Vitali's covering theorem to find pairwise disjoint balls $\bB_{r_i} (y_i):= \bB_{\bar C \varrho_{\mathcal{W}} (y_i)} (y_i)$ such that $\{\bB_{5r_i} (y_i)\}$ covers $\bB_2 \setminus R_{\mathcal{W}}$. Using the monotonicity formula, we then have
\begin{align*}
& \int_{\bB_2 \setminus R_{\mathcal W}} |x|^2 \, d\|T\| + \int_{\bB_2 \setminus R_{\mathcal W}} |x|^2 \, d\|\bC\|\\
\leq & \sum_i C r_i^2 ( \|T\|(\bB_{5r_i}(y_i)) + \|\bC\|(\bB_{5r_i}(y_i))  ) \leq \sum_i C r_i^{m+2}\\
\leq & \bar C \, \tau^{-2} \sum_i \int_{\bB_{r_i}(y_i)} \dist^2(\cdot,\bS) \, d\|T\| \leq \bar C \,\tau^{-2}\, \bE\, .
\end{align*}

The estimate in \eqref{e:L^2 estimate} then follows from the choice of $\tau$ in \eqref{e: choice of tau}. \qed

\section{Linear selection I: local algorithm}\label{sec:local}

We next observe that the the cone $\bC$ is also a $p$-multigraph over $\bS_0$. For this reason we introduce linear multifunctions over $\bS_0$.

\begin{definition}\label{d:linear}
A $p$-multifunction $l=\{l_{i,j}\}$ over $\bS_0$ will be called linear if, for each $i$ and $j$, $l_{i,j}: \bH_{0,i} \to \bH_{0,i}^\perp$ is linear and vanishes on the spine $V$.
\end{definition}

The following is then an obvious consequence of Theorem \ref{thm:graph_v1}

\begin{corollary}\label{c:ovvio}
Let $T$, $\Sigma$, $\bC$, $\bC_0$, $\bS := \spt (\bC)$, and $\bS_0 := \spt (\bC_0)$ be as in Theorem \ref{thm:graph_v1}, and consider the corresponding map $v$ and $U_{4\mathcal{W}}$ its domain. Then:
\begin{itemize}
    \item[(i)] there is a linear $p$-multifunction $l$ over $\bS_0$ such that $\bC = \bG_{\bS_0} (l)$;
    \item[(ii)] ${\rm dist}\, (q, \bS) = {\rm \dist}\, \left(q, \cup_h\, \spt \left(\bG_{\bH_{0,i}} (l_{i,h}) \right)\right)$ for each $q\in \spt(\bG_{\bS_0} (v_{i,j}) )$;
    \item[(iii)] there is a geometric constant $C$ such that
    \begin{equation}\label{e:L2-graphical-initial}
    C^{-1} \int_{R_\mathcal{W}} \dist (q, \bS)^2 \, d\|T\| (q)
    \leq \sum_{Q \in \mathcal W} \sum_{i=1}^{N_0} \int_{4 Q_i} \sum_{j=1}^{\kappa_{0,i}} \min_{1 \leq h \leq \kappa_{0,i}} |v_{i,j} (z) - l_{i,h} (z)|^2\, dz 
     \leq C \, \bE \,.
    \end{equation}
\end{itemize}
\end{corollary}

The main purpose of this and the next section is, roughly speaking, to take out of the integral the $\min$ in \eqref{e:L2-graphical-initial}.

\begin{theorem}[Improved estimate]\label{thm:graph v2}
Let $T$, $\Sigma$, $\bC$, $\bC_0$, $\bS := \spt (\bC)$, and $\bS_0 := \spt (\bC_0)$ be as in Theorem \ref{thm:graph_v1}, let $u$ and $U_{4 \mathcal{W}}$ be the corresponding map and its domain, and let $l$ be as in Corollary \ref{c:ovvio}(i). There are a geometric constant $C$ and a selection function $h:(i,j)\mapsto h (i,j) \in \{1, \ldots \kappa_{0,i}\}$ such that if $\tilde l$ denotes the linear $p$-multifunction $\{\tilde l_{i,j} = l_{i,h(i,j)}\}$ and $w$ denotes the $p$-multifunction on $U_{4\mathcal W}$ over $\bS_0$ defined by
\begin{equation} \label{e:difference function}
    w_{i,j} := u_{i,j} - \tilde l_{i,j}\,,
\end{equation}
then
\begin{align}\label{eqn:Linfty estimate excess}
	 	\sup_{\zeta = (t,y) \in U_{3 \mathcal W}}t^{\frac{m}{2}+1}\left( t^{-1} \abs{w(\zeta)} + \abs{Dw(\zeta)}+ t^{\sfrac12} [Dw]_{\sfrac12}(\zeta) \right)&\le C\, (\bE+\bA^2)^{\sfrac12}\,,\\ \label{L2 estimate excess}
 \sum_{Q \in \mathcal W} \sum_{i=1}^{N_0} \sum_{j=1}^{\kappa_{0,i}}	\int_{3 Q_i}   (|w_{i,j} (z)|^2 + \abs{x}^2 \abs{Dw_{i,j}(z)}^2) \, dz &\le C\, (\bE+\bA^2) \,,
\end{align}
where, for $z \in 3 Q_i$, $|x|$ denotes, as usual, the distance of $z$ from $V$.
\end{theorem}

The selection function $(i,j) \mapsto h(i,j)$ identifies a new cone $\tilde{\bC}$ and a new open book $\tilde{\bS} \subset \bS$ as follows.

\begin{definition}\label{d:new-cone}
Let $\tilde l$ be the linear $p$-multifunction $\{\tilde l_{i,j}=l_{i,h(i,j)}\}$ from Theorem \ref{thm:graph v2}. We denote by $\tilde{\bC}$ the cone given by $\bG_{\bS_0}(\tilde{l})$. The corresponding open book $\spt (\tilde{\bC})$ is denoted by $\tilde{\bS}$.
\end{definition}

\begin{remark} Observe that 
\begin{equation}\label{e:new_book}
\tilde{\bS} = \bigcup_{i,j} \bH_{i,h(i,j)} =: \bigcup_{i,j} \tilde{\bH}_{i,j}\, . 
\end{equation}
Clearly the halfspaces appearing in \eqref{e:new_book} are not necessarily distinct, namely it might be that $\tilde{\bH}_{i,j} = \tilde{\bH}_{i', j'}$ for distinct pairs $(i,j)$ and $(i',j')$. Moreover
\[
\tilde{\bC} = \sum_{i,j} \a{\bH_{i,h(i,j)}} = \sum_{i,j} \a{\tilde{\bH}_{i,j}}\, ,
\]
and thus each page $\tilde{\bH}_{i,j}$ is counted in $\tilde{\bC}$ with a multiplicity that equals the number of pairs $(i',j')$ such that $h (i',j') = h (i,j)$.

However, since the $\|l_{i,j}\|_\infty$ is suitably small compared to the minimal angle between distinct pages of $\bS_0$ (cf. \eqref{e:hyp_flat_T_C0} and \eqref{e:the flat smallness parameter}), we at least know that $\tilde{\bH}_{i,j} \neq \tilde{\bH}_{i',j'}$ whenever $i\neq i'$. In particular we can conclude that:
\begin{itemize}
\item $\tilde{\bS}$ has at least as many pages as $\bS_0$; 
\item $\tilde{\bS} \subset \bS$, so that in particular $\tilde{\bS}$ has no more pages than $\bS$
\item $\tilde{\bH}_{i,j}$ has a multiplicity in $\tilde{\bC}$ which is at most $\kappa_{0,i}$. 
\end{itemize}
It is however possible that $\tilde{\bS}$ is a {\em strict} subset of $\bS$, i.e. that it has less pages than $\bS$. Likewise, the multiplicities, in the respective cones $\bC$ and $\tilde{\bC}$, of a page $\bH_{i,j}$ which is common to both $\bS$ and $\tilde{\bS}$ are just two, typically unrelated, integer numbers in $\{1, \ldots , \kappa_{0,i}\}$. 
\end{remark}

An important corollary of Theorem \ref{thm:graph v2} is that, in the graphicality region $R_{\mathcal W}$, the current $T$ coincides also with a $p$-multigraph over $\tilde{\bS}$.

\begin{corollary} \label{cor:reparametrization}
    Let $T, \Sigma, \bC, \bC_0, \bS := \spt(\bC)$, and $\bS_0:= \spt(\bC_0)$ be as in Theorem \ref{thm:graph v2}, and let $\tilde{\bS}$ be the open book in Definition \ref{d:new-cone}. There exists a $p$-multifunction $\tilde u = \{\tilde u_{i,j}\}$ over $\tilde{\bS}$ of class $C^{1,\frac12}$ on $U_{2 \mathcal W}$ and with $\tilde u_{i,j} \colon (\tilde U_{2 \mathcal W})_{i,j} \subset \tilde{\bH}_{i,j} \to \tilde{\bH}_{i,j}^{\perp_0}$ for all $i$ and $j$ such that
    \begin{equation} \label{e:rep_Holder}
        \| \tilde u \|_{C^{1,\frac12}(U_{2\mathcal W})} \leq \beta\,,
    \end{equation}
    and, denoting $\tilde v$ the $p$-multifunction on $U_{2 \mathcal W}$ over $\tilde \bS$ defined by
    \begin{equation} \label{e:rep_up to manifold}
        \tilde v_{i,j}(z) := \tilde u_{i,j}(z) + \Psi(z + \tilde u_{i,j}(z))\,,
    \end{equation}
    we have
    \begin{equation}\label{e:rep_parametrization}
        T \mres R_{\mathcal W} = {\bf G}_{\tilde{\bS}}(\tilde v) \mres R_{\mathcal W}
    \end{equation}
   and 
    \begin{equation}\label{e:comparison-tilde-u-w}
\int_{(\tilde U_{2\mathcal W})_{i,j} \cap \bB_2} (|\tilde{u}_{i,j}|^2 + |x|^2 |\nabla \tilde{u}_{i,j}|^2) 
\leq C \int_{(U_{3 \mathcal W})_i\cap \bB_3} (|w_{i,j}|^2 + |x|^2 |\nabla w_{i,j}|^2) \leq C \, (\bE + \bA^2)\,,
\end{equation}
where $w$ is the $p$-multifunction over $\bS_0$ defined in \eqref{e:difference function}. Moreover, combining this with \eqref{e:L^2 estimate} we trivially have
\begin{equation}\label{eq:tildeexcesvsexcess}
    r^{m+2}\bE(T,\tilde\bS, 0, r)\leq C\,(\bE+\bA^2)\,,\qquad 0<r<2\,.
\end{equation}

Finally, for every fixed $\tilde{\eta}>0$, if $\varepsilon_1$ in \eqref{e:hyp_graph_cone}-\eqref{e:hyp_graph_current} and $\eta_{\bS_0}$ in \eqref{e:hyp_flat_T_C0} are chosen sufficiently small, then 
\begin{equation}\label{e:flat_C_tildeC}
\hat\flat^p_{\bB_{R_0}} (\tilde{\bC} - \bC) < \tilde{\eta}\, .
\end{equation}
\end{corollary}

\smallskip

In this section we will show that a suitable selection $h (i,j)$ as in Theorem \ref{thm:graph v2} exists at the level of each cube in the Whitney domain $\mathcal{W}$. Such local selection algorithm depends on an appropriate Harnack-type estimate and will result in Lemma \ref{l:local_selection} below. How to choose the same $h (i,j)$ on all cubes of the Whitney domain so that the conclusions of Theorem \ref{thm:graph v2} and Corollary \ref{cor:reparametrization} hold will instead be explained in the next section.

\begin{lemma}[Local linear selection]\label{l:local_selection}
There is a constant $C$ depending only on $m$ and $p$ such that the following holds.
Let $T$, $\Sigma$, $\bC$, $\bC_0$, $\bS := \spt (\bC)$, and $\bS_0 := \spt (\bC_0)$ be as in Theorem \ref{thm:graph_v1}, let $u$, $v$ and $U_{4 \mathcal{W}} = \cup_{Q \in \mathcal W} 4Q$ be the corresponding maps and their domain, and let $l$ be as in Corollary \ref{c:ovvio}(i). Consider any $i$ and $j$ and any cube $Q\in \mathcal{W}$. Then there is a $\bar h$, which depends on $i$, $j$, and $Q$, such that
\begin{align}
&d_Q^m \| u_{i,j} - l_{i,\bar h}\|^2_{L^\infty (3 Q_i)} + d_Q^{2+m} \, \|D (u_{i,j} -l_{i,\bar h})\|^2_{L^\infty (3 Q_i)} + d_Q^{3+m} \, [D u_{i,j}]^2_{\frac{1}{2}, 3 Q_i}  \nonumber\\
\leq & C \int_{4Q_i} \min_h |u_{i,j} (z) - l_{i,h} (z)|^2\, dz
+ C \bA^2 d_Q^{2+m}\, \label{e:local-selection}
\end{align}
(where we recall that $d_Q$ denotes the diameter of $Q$). 
\end{lemma}

\subsection{A Harnack type inequality}\label{sec.harmonic rigidity} We start with the Harnack-type estimate which, as explained in the paragraph above, is the main tool to obtain the local selection. It turns out that the only property needed on the linear functions $l_i$ is that they solve the minimal surface equation and since the argument would not be any simpler, we state the lemma under this more general assumption. 

\begin{lemma}\label{lem.harmonic rigidity}
For every $N\in \N$, $L>0$, and $s>m$ there is a constant $C=C(m,N,L,s)$ with the following property. Set $\lambda Q := \left[ -\lambda, \lambda \right]^m \subset \R^m$, and assume that
\begin{enumerate}
	\item[(a1)]  $g_1<g_2< \dotsm < g_N$ is an ordered family of solutions to the minimal surfaces equation in $4 Q \subset \R^m$ with $\norm{\nabla g_j}_{L^\infty(4 Q)} <L$ for every $j=1,\ldots,N$;
	\item[(a2)] $u \in C^2(4 Q)$ with $\norm{\nabla u}_{L^\infty(4 Q)} <L$ is a solution to
	\[ \Div\left( \frac{\nabla u}{\sqrt{1+\abs{\nabla u}^2}}\right) = H\,, \]
	for some $H$ of the form $H=\Div(H_1)+H_2$ with $H_1 \in L^s(4Q)$ and $H_2 \in L^{\sfrac{s}{2}}(4 Q)$.
\end{enumerate}
Then
\begin{align}
& \min_{i=1,\ldots, N} \|u-g_i\|_{L^\infty \left(3 Q\right)}
\le\; &  C \left( \int_{4 Q} \min_{i=1, \ldots, N} \abs{u-g_i}^2\, dz+\left(\|H_1\|_{L^s(4 Q)} + \|H_2\|_{L^{\sfrac{s}{2}}(4 Q)}\right)^2 \right)^{\sfrac{1}{2}}\,.\label{eq:estimate} 
\end{align}
\end{lemma}

\begin{proof}

Setting $F(w):=\displaystyle \frac{w}{\sqrt{1+\abs{w}^2}}$ for $w \in \R^m$, assumptions (a1) and (a2) read
\begin{equation}\label{eq.minimal surface equations with rhs} \Div(F(\nabla g_j)) = 0 \;\; \mbox{for every $j$}\,, \quad \text{and} \quad \Div(F(\nabla u)) = H \quad \text{ in } 4Q. \end{equation}
We note that 
\begin{equation}\label{eq:matrix} DF(w)= \frac{1}{\sqrt{1+\abs{w}^2}} \left(\mathrm{Id} - \frac{w \otimes w}{1+\abs{w}^2}\right) \end{equation}
is a symmetric matrix with minimal and maximal eigenvalues $\lambda(w)=(1+\abs{w}^2)^{-\frac32}$ and $\Lambda(w)=(1+\abs{w}^2)^{-\frac12}$, respectively; in particular, $DF(w)$ is positive-definite, and $\Lambda/\lambda$ is bounded uniformly on $\{|w| \leq L\}$. \\


Arguing by induction on $N \geq 1$, we will prove \eqref{eq:estimate} with $3 Q$ replaced by $2^{-N} Q$ in the left-hand side; the estimate in \eqref{eq:estimate} will then follow by a classical covering argument.\\

\emph{Induction base: $N=1$.} In this situation the statement reduces to classical elliptic regularity: using \eqref{eq.minimal surface equations with rhs}, the function $w:=u-g_N$ solves  
\begin{equation}\label{eq:equation for w} \Div( A\nabla w)= H \text{ on } 4Q \end{equation}
where $A=A(z):= \int_0^1 DF(\nabla g_N(z) + t \nabla w(z))\,dt$ is uniformly elliptic by \eqref{eq:matrix}. Hence, \eqref{eq:estimate} follows from \cite[Theorem 8.17]{GilbargTrudinger}.\\

\emph{Induction step: $N-1 \to N$.} We split the proof into two cases, depending on the validity of
\begin{equation} \label{dichotomy}
\inf_{2Q} (g_N-u) \geq 0\,.\\
\end{equation}

\emph{First case: \eqref{dichotomy} holds.}
Set $K(H) := \|H_1\|_{L^s(4Q)} + \|H_2\|_{L^{\sfrac{s}{2}}(4Q)}$, and suppose further that, for a dimensional $\eps > 0$ to be chosen,
\begin{equation}\label{eq:close} 0\le \inf_{2Q} (g_N-u)< \max\{\epsilon\, \inf_{2Q} (g_N-g_{N-1}), \,K(H)\}\,. \end{equation}
Harnack's inequality, see \cite[Theorems 8.17, 8.18]{GilbargTrudinger}, then implies that, for some constant $\hat C=\hat C(m,L,s)$
\[\sup_{2Q} (g_N-u) \le \hat C \left(\inf_{2Q} (g_N- u)  + K(H) \right) \le 2 \hat C \max\{  \epsilon\, \inf_{2Q}( g_N-g_{N-1} ), K(H)\}\,.\]
If $K(H)\ge \epsilon\, \inf_{2Q}(g_N-g_{N-1})$ we conclude $\sup_{2Q} (g_N-u) \le 2 \hat C \,K(H)$, and thus \eqref{eq:estimate} follows. 
Otherwise, we deduce 
\[\sup_{2Q} (g_N-u) \le 2 \hat C \epsilon \,\inf_{2Q}( g_N-g_{N-1} ) < \frac12 \,\inf_{2Q}( g_N-g_{N-1} )\]
for $\epsilon<\frac{1}{4\hat C}$.  Hence, we have $g_N(z)-u(z) = \min_{i=1,\dots, N} \abs{u(z)-g_i(z)}$ for all $z\in 2Q$. The estimate follows now as in the case $N=1$. In conclusion, \eqref{eq:estimate} holds true if \eqref{eq:close} holds. 

Assume now that \eqref{dichotomy} holds but \eqref{eq:close} fails. Consider next a point $\bar z\in 2Q$. Should
\begin{equation}\label{e:N-minimizes-in-B1}
\arg\min_{i=1,\ldots, N} \abs{u(\bar z)-g_i(\bar z)} = N\, ,
\end{equation}
we must necessarily have $u(\bar z) \geq g_{N-1} (\bar z)$, otherwise \[|u(\bar z) - g_{N-1} (\bar z)| = g_{N-1} (\bar z) - u(\bar z) < g_N (\bar z) - u(\bar z) = |g_N (\bar z) - u(\bar z)|\,,\] 
a contradiction to \eqref{e:N-minimizes-in-B1}. Owing again to \eqref{dichotomy}, we then have
\begin{align*}
 \min_{i=1, \dotsc, N-1} \abs{u(\bar z)-g_i(\bar z)} &\leq
|u(\bar z) - g_{N-1} (\bar z)| = u(\bar z)-g_{N-1}(\bar z)  \le g_N(\bar z)-g_{N-1}(\bar z)\\
 &\le \hat C\, \inf_{2Q}(g_N-g_{N-1})
\le  \frac{\hat C}{\epsilon}\, (g_N(\bar z) - u(\bar z))\\ 
&=\frac{\hat C}{\eps}\,\min_{i=1,\dotsc,N} \abs{u(\bar z)-g_i(\bar z)}\, ,
\end{align*}
where we have again used Harnack's inequality to deduce the first inequality in the second line. On the other hand, if \eqref{e:N-minimizes-in-B1} fails at $\bar z$, then
\[
 \min_{i=1, \dotsc, N-1} \abs{u(\bar z)-g_i(\bar z)} = \min_{i=1, \dotsc, N} \abs{u(\bar z)-g_i (\bar z)}\, .
\] 
This implies that for \emph{every} $z\in 2Q$ 
\[ \min_{i=1, \dotsc, N-1} \abs{u(z)-g_i(z)}  \le \frac{\hat C}{\epsilon}\min_{i=1, \dotsc, N} \abs{u(z)-g_i(z)}\,.\]
Hence by the induction step we conclude
\begin{align}
\min_{i=1, \ldots, N} \sup_{2^{-N}Q} \abs{u-g_i}^2 &\leq \min_{i=1,\dotsc, N-1} \sup_{ 2^{-(N-1)} \frac{Q}{2}} \abs{u-g_i}^2\nonumber\\
&\le  C \int_{2Q} \min_{i=1, \dotsc, N-1} \abs{u-g_i}^2 + C\, K(H)^2\nonumber\\
& \le \frac{C}{\epsilon^2} \int_{2Q} \min_{i=1, \dotsc, N} \abs{u-g_i}^2 + C\, K(H)^2\,\label{eq.g_N>u} \end{align}
(observe that the inductive statement has been applied by replacing the outer cube $4Q$ with $2Q$ and the inner cube $2^{-(N-1)}Q$ with $2^{-N}Q = 2^{-(N-1)} \frac{Q}{2}$; this can however be easily achieved by scaling the original statement).\\

\emph{Second case: \eqref{dichotomy} fails.}  We will reduce the proof to the first case. As observed in the induction base, the function $w:=u-g_N$ solves \eqref{eq:equation for w}, hence $w^+:= \max\{w,0\}$ is a sub-solution to the same equation in $4Q$, and therefore by \cite[Theorem 8.17]{GilbargTrudinger}
\[ \sup_{3 Q} (w^+)^2 \le C \left(\int_{4Q} (w^+)^2 + K(H)^2\right) \le C \left(\int_{4Q} \min_{i=1, \dotsc, N} \abs{u-g_i}^2 + K(H)^2\right)\, , \]
where we have used that, by the ordering of the functions $g_i$, $w^+ = \min_i |u-g_i|$ on the set $\{w^+>0\}$. Define $\tilde{g}_N:=g_N+C\mathcal{D}$ where $\mathcal{D}^2 = \int_{4Q} \min_{i=1, \dotsc, N} \abs{u-g_i}^2 + K(H)^2$, in such a way that $\tilde{g}_N -u \geq 0$ in $3 Q$. Since  
\[
\abs{u(z)-\tilde{g}_N(z)}\le \abs{u(z)-g_N(z)} + C\mathcal{D}\,,
\]
the definition of $\mathcal{D}^2$ implies that, setting $\tilde{g}_i = g_i$ for $i<N$, 
\[ \int_{2Q} \min_{i=1, \dotsc, N} \abs{u-\tilde{g}_i}^2 + K(H)^2\le C\mathcal{D}^2\,.\]
Thus, the family $\{\tilde{g}_i\}_{i=1}^N$ and $u$ satisfy the assumptions of the lemma as well as \eqref{dichotomy} with $\tilde{g}_N$ in place of $g_N$, and therefore by \eqref{eq.g_N>u} 
\begin{align*} \min_{i=1,\dotsc, N} \sup_{2^{-N}Q} \abs{u-g_i}^2  &\le 2\min_{i=1,\dotsc, N} \sup_{2^{-N}Q} \abs{u-\tilde{g}_i}^2 + C\mathcal{D}^2\\&\le C \left(\int_{2Q} \min_{i=1, \dotsc, N} \abs{u-\tilde{g}_i}^2 + K(H)^2 \right) + C\mathcal{D}^2\le C\mathcal{D}^2\,.\end{align*}
This closes the induction step, and completes the proof.\end{proof}

\subsection{Proof of Lemma \ref{l:local_selection}} Let us fix $i \in \{1,\dots,N_0\}$ and $j \in \{1,\ldots,\kappa_{0,i}\}$, and consider the corresponding function $u_{i,j}$ as well as all the linear functions $\{l_{i,h}\}$ defined on the page $\bH_{0,i}$ of the book $\bS_0$. For brevity, we will drop the reference to the fixed pair $(i,j)$, and simply write $\kappa_0$, $\bH_0$ and $u$ for $\kappa_{0,i}$, $\bH_{0,i}$ and $u_{i,j}$, respectively. Fix any $m$-dimensional cube $Q$ in the Whitney domain $\mathcal{W}$ and, following the same convention just explained, identify it with $Q_i$. Observe also that, since the estimate \eqref{e:local-selection} is scaling invariant, we might assume, without loss of generality, that the sidelength of the cube $Q$ is $1$. 

By Theorem \ref{thm:graph_v1}, the graph of the function $z\mapsto v (z) =u (z) + \Psi(z + u(z))$ over $4Q$ is stationary in $\Sigma$. Let $g_{\Psi}:= ({\rm Id} + \Psi)^\sharp \delta_{\mathbb R^{m+n}}$ be the metric on $\pi_0$ defined in the proof of Lemma \ref{lem:White}. We fix an orthonormal basis $\{e_1,\ldots,e_m,e_{m+1}\}$ of $\pi_0$ with $e_{m+1}\in  \bH_0^\perp$, and define the function $\bH_0 \times \mathbb R \times \mathbb R^m \ni (z,\bar{u}, \bar{p})\mapsto \Phi (z, \bar{u}, \bar{p})\in \mathbb R^+$ as
\begin{equation} \label{e:integrando}
    \Phi(z,\bar u , \bar p ):=\sqrt{\det\left[ (g_\Psi)_{z+ \bar u e_{m+1}}( e_\alpha + \bar p_\alpha e_{m+1}, e_\beta + \bar p_\beta e_{m+1})\right]}\,.
\end{equation}
Since $\|u\|_{C^1} \leq \beta$ (cf. Theorem \ref{thm:graph_v1}(ii)), $u$ is then a critical point of the energy
\begin{equation}\label{e:funzionale}
\int_{4Q} \Phi (z, u (z), \nabla u (z))\, dz \, .
\end{equation}
Therefore, $u$ is a solution to the Euler-Lagrange equation for \eqref{e:funzionale}, which reads
\begin{equation}\label{e:elliptic_PDE}
{\rm div}\, (D_{\bar p} \Phi (z, u, \nabla u)) = \underbrace{D_{\bar u} \Phi (z,u, \nabla u)}_{=: H_2}\, .
\end{equation}
Since
\begin{align*} 
&D_{\bar p_\alpha}\Phi(z,\bar u, \bar p) - \frac{\bar p_\alpha}{\sqrt{1+\abs{\bar p}^2}}=R_\alpha(z,\bar u,\bar p) \text{ with } \abs{R_\alpha(z,\bar u,\bar p)}\le C \norm{D\Psi}_\infty(1 + \abs{\bar p})\,, \\
&\abs{D_{\bar u}\Phi(z,\bar u,\bar p)}\le C \norm{D^2 \Psi}_\infty
( 1 + \abs{\bar p})\,, 
\end{align*}
we can then conclude that $u$ solves 
\[
{\rm div}\, \left(\frac{\nabla u}{1+|\nabla u|^2}\right) = {\rm div}\,  \underbrace{(- R (z, u, \nabla u))}_{=: H_1} + H_2\, ,
\]
with $\|H_i\|_\infty \leq C \|D^2\Psi\|_\infty$ (where we have used that $\|D\Psi\|_\infty \leq C \|D^2 \Psi\|_\infty$ given that $D\Psi (0) =0$). On the other hand, the functions $l_h$, being linear, solve the minimal surface equation. We can thus apply (the rescaled version of) Lemma \ref{lem.harmonic rigidity} with $s=\infty$ to conclude the estimate
\[
\min_h \| u - l_h\|^2_{L^\infty (3 Q)} \leq C \int_{4Q} \min_h |u (z) - l_h (z)|^2\, dz + C d_Q^2\,\|D^2 \Psi\|_{L^\infty(4Q)}^2\, .
\]
We let then $\bar h$ be the index such that $\|u-l_{\bar h}\|_{L^\infty (3 Q)} = \min_h \| u - l_h \|_{L^\infty (3 Q)}$, and we estimate $\|D (u-l_{\bar h})\|_{L^\infty (3 Q)}$ and $[D (u- l_{\bar h})]_{\frac{1}{2}, 3 Q}$ using standard Schauder theory
for \eqref{e:elliptic_PDE}: since $Dl_h$ is a constant, and therefore $[D (u- l_h)]_{\frac{1}{2}, 3 Q} = [Du]_{\frac{1}{2}, 3 Q}$, \eqref{e:local-selection} is achieved by observing that $\|D^2\Psi\|_\infty \leq C \|A_\Sigma\|_\infty = C\, \bA$.


\section{Linear selection II: global algorithm}\label{sec:global}

In this section we complete the proof of Theorem \ref{thm:graph v2} and Corollary \ref{cor:reparametrization}. The key of the proof of Theorem \ref{thm:graph v2} is to show an analogue of Lemma \ref{l:local_selection} where the choice of $\bar h$ is independent of the cube $Q$. The relevant statement is thus the following.

\begin{lemma}[Global linear selection]\label{l:global_selection}
There is a constant $C$ depending only on $m$ and $p$ such that the following holds.
Let $T$, $\Sigma$, $\bC$, $\bC_0$, $\bS := \spt (\bC)$, and $\bS_0 := \spt (\bC_0)$ be as in Theorem \ref{thm:graph_v1}, let $u$, $v$ and $U_{4 \mathcal{W}} = \cup_{Q \in \mathcal W} 4Q$ be the corresponding maps and their domain, and let $l$ be as in Corollary \ref{c:ovvio}(i). Consider any $i$ and $j$, let $\hat Q\in \mathcal{W}$ be any cube which does not have any element above (cf. the partial order relation of Definition \ref{def:whitney}) and let $\hat h$ be the index $\bar h$ of Lemma \ref{l:local_selection} corresponding to $i,j$, and $\hat Q$. Then the following two estimates hold
\begin{align}
& \sup_{Q\in \mathcal{W}} \left(d_Q^m \| u_{i,j} - l_{i,\hat h}\|^2_{L^\infty (3 Q_i)} + d_Q^{2+m} \, \|D (u_{i,j} -l_{i,\hat h})\|^2_{L^\infty (3 Q_i)} + d_Q^{3+m} \, [D u_{i,j}]^2_{\frac{1}{2}, 3 Q_i} \right)\nonumber\\
\leq & C \sum_{Q\in \mathcal{W}} \int_{4Q_i} \min_h |u_{i,j} (z) - l_{i,h} (z)|^2\, dz
+ C \bA^2\label{e:global_selection_1}
\end{align}
\begin{align}
& \sum_{Q\in \mathcal{W}} d_Q^m \left( \| u_{i,j} - l_{i,\hat h}\|^2_{L^\infty (3 Q_i)} + d_Q^2 \, \|D (u_{i,j} -l_{i,\hat h})\|^2_{L^\infty (3 Q_i)} \right)\nonumber\\
\leq & C \sum_{Q\in \mathcal{W}} \int_{4Q_i} \min_h |u_{i,j} (z) - l_{i,h} (z)|^2\, dz
+ C \bA^2\, .\label{e:global_selection_2}
\end{align}
\end{lemma}

\begin{proof}[Proof of Theorem \ref{thm:graph v2}]

Setting $h(i,j):=\hat h$ from Lemma \ref{l:global_selection}, and recalling the definition for $w$ given in \eqref{e:difference function}, the estimates \eqref{eqn:Linfty estimate excess} and \eqref{L2 estimate excess} follow immediately from \eqref{e:global_selection_1} and \eqref{e:global_selection_2}, together with the simple observation that
\begin{equation*}
\sum_{Q\in \mathcal{W}} \int_{4Q_i} \min_h |u_{i,j} (z) - l_{i,h} (z)|^2\, dz 
\leq \sum_{Q \in \mathcal W} \int_{4 Q_i} \min_h |v_{i,j} (z) - l_{i,h} (z)|^2\, dz \leq C \bE\, ,
\end{equation*}
where the last inequality is \eqref{e:L2-graphical-initial}.
\end{proof}

\begin{proof}[Proof of Corollary \ref{cor:reparametrization}]

 Consider each $u_{i,j}$, defined on its respective domain $(U_{4 \mathcal W})_i \subset \bH_{0,i}$ as specified in Theorem \ref{thm:graph_v1}. For any $z \in (U_{4 \mathcal W})_i$, we let $\tilde z$ denote the orthogonal projection of $(z + u_{i,j}(z))$ on the page $\tilde{\bH}_{i,j}$ of $\tilde{\bS}$. It is easy to see that, if $\beta$ (which controls the Lipschitz constant of $u$, cf. Theorem \ref{thm:graph_v1}(ii)) is sufficiently small, the map $z\mapsto \tilde{z}$ is biLipschitz on its image, so that, in particular, we can define the maps
\[
\tilde u_{i,j}(\tilde z) := (z + u_{i,j}(z)) - \tilde z \in \tilde{\bH}_{i,j}^{\perp_0}\,, \qquad \tilde v_{i,j}(\tilde z) := \tilde u_{i,j}(\tilde z) + \Psi (\tilde z + \tilde u_{i,j}(\tilde z))\,.
\]
 
 Moreover, for every $\sigma > 0$ a suitable choice of $\beta$ (depending on $\sigma$) entails, for $z= (x,y)$ and $\tilde{z}=(\tilde{x},\tilde{y})$, that
\begin{equation} \label{la bilipschitz tilde}
\tilde y = y \,, \qquad (1-\sigma) |x|\leq |\tilde{x}|\leq (1+\sigma) |x|\, .
\end{equation}

 We denote by $\tilde{U}_{i,j}$ the corresponding domain of $\tilde{u}_{i,j}$ and $\tilde{v}_{i,j}$. Observe that, differently from $u,v$, and $w$, a domain $\tilde{U}_{i',j'}$ with $(i',j') \neq (i,j)$ cannot be recovered from $\tilde{U}_{i,j}$ through a rotation around the spine $V$. Nonetheless, by restricting each $\tilde u_{i,j}$ to a suitable subset of $\tilde U_{i,j}$, we can regard the collection $\{\tilde u_{i,j}\}$ (and thus, analogously, $\{\tilde v_{i,j}\}$) as a $p$-multifunction over $\tilde{\bS}$ on a domain which, in view of \eqref{la bilipschitz tilde} (and for an appropriate choice of $\beta$), contains $U_{2\mathcal W}$. In turn, the latter implies
\begin{equation}\label{e:cosa-manca-al-grafico}
\spt (T - \bG_{\tilde{\bS}} (\tilde{v})) \cap R_{\mathcal W} = \emptyset\,.
\end{equation}

This proves \eqref{e:rep_parametrization}. Concerning \eqref{e:comparison-tilde-u-w}, the first inequality follows readily from the definitions of the maps $w_{i,j}$ and $\tilde u_{i,j}$, whereas the second inequality is an immediate consequence of \eqref{L2 estimate excess}. Finally, \eqref{e:flat_C_tildeC} is a simple compactness argument: fix $\tilde\eta$ and assume $T_k, \Sigma_k, \bC_k$ satisfy the corresponding assumptions with vanishing $\varepsilon_1 (k)$ and $\eta_{\bS_0} (k)$. In particular, it follows readily that $\bC_k$ and $\tilde{\bC}_k$ converge to $\bC_0$ as well, and for $k$ sufficiently large we must satisfy \eqref{e:flat_C_tildeC}. 
\end{proof}

We are then only left with the proof of Lemma \ref{l:global_selection}.

\begin{proof}[Proof of Lemma \ref{l:global_selection}]
Fix $i,j$, and a cube $\hat Q$ as in the statement. As in the proof of Lemma \ref{l:local_selection}, we drop the subscripts $i,j$ and we identify $\bH_0=\bH_{0,i}$
with $[0, \infty)\times V$. 
Let also $\bar h: \mathcal{W}\ni Q \to \{1, \dotsc, \kappa_0\}$
be a map which selects, for each cube $Q \in \mathcal{W}$, the index $\bar h (Q)$ of an $L^\infty(3Q)$-optimal linear function in $\{l_h\}_{h=1}^{\kappa_0}$ as in Lemma \ref{l:local_selection}.

In order to simplify the estimates, let us introduce the monotone function
\begin{equation}\label{eq.set valued function}
	\mu(E):=\int_E \min_h |u - l_h|^2 + \abs{E}^{\frac{m+2}{m}} \bA^2 \quad \mbox{for $E \subset U_{4 \mathcal W}$ Borel}\,,
\end{equation}
so that \eqref{e:local-selection} can be re-written as 
\begin{equation}\label{eq.good selection on a single cube2}
	d_Q^m \left(\|u-l_{\bar h(Q)}\|_{L^\infty (3 Q)}^2 + 
	d_Q^2 \|D(u-l_{\bar h (Q)})\|_{L^\infty (3 Q)}^2 + d_Q^3 [Du]^2_{\frac{1}{2}, 3 Q} \right)\leq C\, \mu(4Q)\,.
\end{equation}
Recall the partial order relation $\preceq$ of Definition \ref{def:whitney} and Remark \ref{rem_whitney_houston}. For every $Q_0\in \mathcal{W}$, let $\mathscr{W}(Q_0)$ be the family of all cubes that are above $Q_0$ together with a shortest path of adjacent cubes in the top sub-layer that connects this family to $\hat{Q}$ (cf. Figure \ref{figura-6}). We index the elements $\mathscr{W}(Q_0) = \{Q_0, \ldots, Q_{\hat N}\}$ with $Q_{\hat N} = \hat Q$, and $Q_i$ either immediately above $Q_{i-1}$ (when $Q_{i-1}$ does not belong to the top sub-layer) or adjacent (if $Q_{i-1}$ belongs to the top sub-layer), and we say that $Q_{i}$ comes right after $Q_{i-1}$. 

Next we select $\kappa_0+1$ elements $\phi (Q_0, s)$ (with ${s\in \{0, \ldots, \kappa_0\}}$) from $\mathscr{W} (Q_0)$, where the function $\phi (Q_0, \cdot)\colon \{0, \ldots, \kappa_0\} \to \mathscr{W} (Q_0)$ is defined through
the following recursive algorithm: 
\begin{itemize}
\item[(a)] $\phi(Q_0, \kappa_0):=Q_{\hat N} = \hat Q$;
\item[(b)] for $0 \leq s \leq \kappa_0-1$ we define $\phi (Q_0, s)$ by:
\begin{itemize}
    \item[$\bullet$] $Q_0$ if $\bar h (\phi(Q_0,s+1)) = \bar h (Q_0)$,
    \item[$\bullet$] otherwise $\phi (Q_0, s)$ is the cube $Q_i$ such that $i+1$ is the smallest index such that $\bar h (Q_{i+1}) = \bar h (\phi (Q_0, s+1))$. 
\end{itemize}
\end{itemize}

\begin{figure}\label{figura-6}
\begin{tikzpicture}
\fill[gray!60] ({25*0.125}, {0.5+2*0.125}) -- ({26*0.125}, {0.5+2*0.125}) -- ({26*0.125}, {0.5+3*0.125}) -- ({25*0.125}, {0.5+3*0.125}) -- ({25*0.125}, {0.5+2*0.125});
\fill[gray!60] ({25*0.125}, {0.5+3*0.125}) -- ({26*0.125}, {0.5+3*0.125}) -- ({26*0.125}, {0.5+4*0.125}) -- ({25*0.125}, {0.5+4*0.125}) -- ({25*0.125}, {0.5+3*0.125});
\foreach \l in {1,...,4}
\fill[gray!60] ({12*0.25}, {1+(\l-1)*0.25}) -- ({12*0.25}, {1+\l*0.25}) -- ({13*0.25}, {1+\l*0.25}) -- ({13*0.25}, {1+(\l-1)*0.25}) -- ({12*0.25}, {1+(\l-1)*0.25});
\foreach \l in {1,...,4}
\fill[gray!60] ({6*0.5}, {2+(\l-1)*0.5}) -- ({6*0.5}, {2+\l*0.5}) -- ({7*0.5}, {2+\l*0.5}) -- ({7*0.5}, {2+(\l-1)*0.5}) -- ({6*0.5}, {2+(\l-1)*0.5});
\fill[gray!60] (1,4) -- (3,4) -- (3,3.5) -- (1,3.5) -- (1,4);
\foreach \x in {1,...,4}
\draw (0, {2+0.5*\x}) -- (6,{2+0.5*\x}) (0,{1+0.25*\x}) -- (6,{1+0.25*\x}) (0,{0.5+0.125*\x}) -- (6,{0.5+0.125*\x}) (0,{0.25+0.0625*\x}) -- (6,{0.25+0.0625*\x}) (0,{0.125+0.03125*\x}) 
-- (6,{0.125+0.03125*\x});
\draw (0,0.125) -- (6,0.125);
\foreach \y in {0,...,12}
\draw ({0.5*\y},2) -- ({0.5*\y},4); 
\foreach \y in {0,...,24}
\draw ({0.25*\y},1) -- ({0.25*\y},2);
\foreach \y in {0,...,48}
\draw ({0.125*\y},0.5) -- ({0.125*\y},1);
\foreach \y in {0,...,96}
\draw ({0.0625*\y},0.25) -- ({0.0625*\y},0.5);
\foreach \y in {0,...,192}
\draw ({0.03125*\y},0.125) -- ({0.03125*\y},0.25);
\fill[black] (0,0) -- (6,0) -- (6,0.125) -- (0,0.125) -- (0,0);
\draw[very thick] (0,0) -- (6,0);
\node[above] at (1.25,4) {$\hat{Q}$};
\end{tikzpicture}
\caption{An example of $\mathscr{W}(Q_0)$.}
\end{figure}
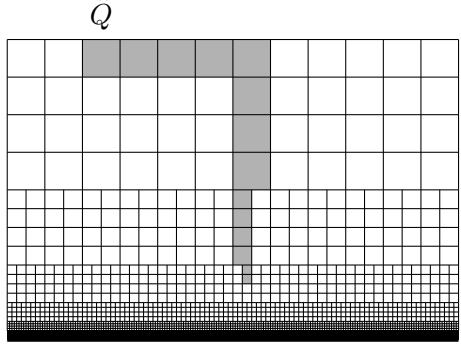

In particular, the map $\phi$ enjoys the following properties:
\begin{enumerate}
\item[(p1)] $\phi(Q_0,\kappa_0)=\hat{Q}$ and $\phi(Q_0,0)=Q_0$ for all $Q_0\in \mathcal{W}$,
\item[(p2)] $\phi(Q_0,s)\in \mathscr{W}(Q_0)$ for all $s\in \{0,\dotsc, \kappa_0\}$, 
\item[(p3)] If $s\leq \kappa_0-1$ and $Q^\top$ denotes the cube that comes right after $Q$, then 
\[
\bar h((\phi (Q_0, s))^\top)=\bar h(\phi(Q_0,s+1))\, .
\]
\end{enumerate}
Since $l_h$ are linear functions of the distance to the spine $V$, by \eqref{e:dist v diam} we have 
\begin{equation} \label{linear estimate}
\|l_h-l_{h'}\|_{L^\infty (3 Q_1)}  \le C\,  \frac{d_{Q_1}}{d_{Q_2}} \|l_h-l_{h'}\|_{L^\infty ( Q_2)}
\end{equation}
for all $h,h' \in \{1,\ldots,\kappa_0\}$, and for all cubes $Q_1,Q_2 \in \mathcal{W}$. Hence, recalling $\hat{h} = \bar h (\hat{Q})$ we can estimate
\begin{eqnarray*}
&& \|u-l_{\hat h}\|_{L^\infty (3 Q_0)} \le  \|u-l_{\bar h(Q_0)}\|_{L^\infty (3 Q_0)} + \sum_{s=0}^{\kappa_0-1} \|l_{\bar h(\phi(Q_0,s+1))}-l_{\bar h(\phi(Q_0,s))}\|_{L^\infty (3 Q_0)}\\
& \overset{\eqref{linear estimate}}{\le} & \|u-l_{\bar h(Q_0)}\|_{L^\infty (3 Q_0)} + C\, \sum_{s=0}^{\kappa_0-1}	\frac{d_{Q_0}}{d_{\phi(Q_0,s)}} \|l_{\bar h(\phi(Q_0,s+1))}-l_{\bar h(\phi(Q_0,s))}\|_{L^\infty ( \phi (Q_0, s))}\\
&\le & \|u-l_{\bar h(Q_0)}\|_{L^\infty (3 Q_0)} + C\,\sum_{s=0}^{\kappa_0-1}	\frac{d_{Q_0}}{d_{\phi(Q_0,s)}}
\Bigg(\|l_{\bar h((\phi(Q_0,s))^\top)}-u\|_{L^\infty (3 (\phi(Q_0,s))^{\top})}\\
&&\qquad\qquad\qquad\qquad\qquad\qquad\qquad\qquad\qquad+\|u-l_{\bar h(\phi(Q_0,s))}\|_{L^\infty(3\phi(Q_0,s))}\Bigg)\,,
\end{eqnarray*}
where in the last inequality we have used that $Q \subset 3 \, Q^\top$. In turn, \eqref{eq.good selection on a single cube2} allows to conclude
\begin{align}
	d_{Q_0}^m \, \|u-l_{\hat h}\|_{L^\infty (3 Q_0)}^2
	 \le & C \left( \mu(4Q_0)+\sum_{s=0}^{\kappa_0-1}
	 \left(\frac{d_{Q_0}}{d_{\phi(Q_0,s)}}\right)^{m+2}\left(\mu(4\,\phi(Q_0,s))+ \mu(4\,(\phi(Q_0,s))^{\top})\right)\right)\nonumber\\
	 \leq & C \sum_{Q\in \mathcal{W}} \mu (4Q)\, ,\label{eq.good selection on a single cube3}
\end{align}
which gives the $L^\infty$ bound of \eqref{e:global_selection_1}.
Arguing similarly for the first derivative we conclude the whole estimate.

Next, summing over $Q_0$ the inequality \eqref{eq.good selection on a single cube3}, the left hand side of \eqref{e:global_selection_2} is bounded by
\begin{align*}
&\sum_{Q_0 \in \mathcal{W}}  \left(	\mu(4Q_0)+\sum_{s=0}^{\kappa_0-1}
	 \left(\frac{d_{Q_0}}{d_{\phi(Q_0,s)}}\right)^{m+2}\left(\mu(4\,\phi(Q_0,s))+ \mu(4\,\phi(Q_0,s)^{\top})\right)\right)
	 &=: I + II + II^{\top}\,. 
\end{align*}
Thus, we just need to show that $II, II^{\top} \leq C \sum_{Q\in \mathcal{W}} \mu (4Q)$.
The two cases are analogous and we just argue for $II$. Interchanging the summation, we have 
\begin{equation} \label{stima per II}
II = \sum_{Q \in \mathcal{W}} \mu(4Q) \sum_{ Q_0 \,\colon\, \exists\, s \text{ s.t. } \phi(Q_0,s)=Q} \left(\frac{d_{Q_0}}{d_Q}\right)^{m+2}\,.
\end{equation}
For fixed $Q \in \mathcal{W}$, we aim at estimating the inner sum in \eqref{stima per II} by analyzing separately the contributions coming from each layer $\left[ 2^{-k}, 2^{-k+1} \right] \times \left[-2,2\right]^{m-1}$. First of all observe that if $Q_0$ belongs to the same layer of $Q$, so that $d_{Q_0}=d_Q$, then either $Q$ is above $Q_0$ and $d_{Q_0}=d_{Q}$, or both must belong to the top sub-layer: in particular the number of such $Q_0$ is at most $C(m,M)$. Moreover, for $d \ge 1$ there are precisely $2^M \cdot 2^{d(m-1)}$ cubes $Q_0 \in \mathcal{Q}$ such that $Q$ is above $Q_0$ and $d_{Q_0}=2^{-d}\,d_{Q}$. Hence,
\begin{align*} \sum_{ Q_0 \colon \exists\, s \text{ s.t. } \phi(Q_0,s)=Q} \left(\frac{d_{Q_0}}{d_Q}\right)^{m+2} &\le C(m,M) + \sum_{d \ge 1} 2^{-d(m+2)} \cdot 2^{M} \cdot 2^{d(m-1)} \leq C(m, M)\, .
\end{align*}
Inserting the latter inside \eqref{stima per II} we conclude $II\leq C \sum_{Q\in \mathcal{W}} \mu (4Q)$.
\end{proof}

%% file: leone-simone.tex
\section{Hardt-Simon type estimates}\label{sec:Hardt-Simon}

This section implements one of the crucial ideas of Simon's work \cite{Simon} (cf. also \cite{CoEdSp}): close to points of high density we can use the monotonicity formula to give an improved $L^2$ estimate, see \eqref{e:Hardt-Simon}; in particular, such points of high density are bound to lie close to the spine $V$ at the scale of the excess $\bE$. 

\begin{theorem}\label{t:Hardt-Simon-main}
There exists a constant $\beta_1 > 0$ depending only on $\bS_0$ with the following property. Let $T,\Sigma,\bC,\bC_0,\bS$, and $\bS_0$ be as in Assumptions \ref{a:graphical} and \ref{a:graphical2}. For any $\beta \in \left( 0, \beta_1 \right)$ there are constants $C$, $\eta_3$, and $0 <\varepsilon_2 \leq \varepsilon_1$ (where $\varepsilon_1$ was given in Theorem \ref{thm:graph_v1}) depending upon $(m,n,p,\bS_0,\beta)$ such that the following conclusion holds. Assume that:
\begin{itemize}
\item[(a)] \eqref{e:hyp_flat_T_C0}-\eqref{e:the flat smallness parameter}-\eqref{e:hyp_graph_cone}-\eqref{e:hyp_graph_current} are satisfied with $\eps_2$ and $\eta_3$ in place of $\eps_1$ and $\eta_{\bS_0}$;
\item[(b)] $\{\tilde{l}\} = \{l_{i, h(i,j)}\}$ is the linear $p$-multifunction of Theorem \ref{thm:graph v2} and $\tilde{\bS}$ denotes the open book induced by it as in Definition \ref{d:new-cone};
\item[(c)] $q_0=(x_0,y_0)\in (V^\perp \times V)\cap \bB_{\sfrac{3}{4}}$ is a point with $\Theta_T (q_0) \geq \Theta_{\bC} (0) = \frac{p}{2}$. 
\end{itemize}
Then
\begin{equation} \label{e:Hardt-Simon}
|x_0|^2 + \int_{\bB_1} \frac{{\rm dist}^2\, (q-q_0, \tilde{\bS})}{|q-q_0|^{m+\frac74}}\, d\|T\| (q) \leq C (\bE + \bA)\, .     
\end{equation}
\end{theorem}

\subsection{Corollaries of the monotonicity formula} We summarize in the following lemma two consequences of the stationarity of the varifold $\|T\|$.

\begin{lemma}\label{l:monot-Jonas}
Let $T$ and $\bC$ be as in Theorem \ref{t:Hardt-Simon-main}, and assume that $g (q) = |q|^k \,\hat g (\frac{q}{|q|})$ for some $k\geq 1$ and some Lipschitz nonnegative function $\hat g$ on the unit sphere. Then, for every $2>\alpha>0$ and $R_1 \leq R_0$ we have
\begin{align}\label{e.h_k monotonicity}
	\frac{\alpha}{2} \int_{\bB_{R_1}} \frac{g^2 (q)}{\abs{q}^{m  + 2k -\alpha}}\, d\norm{T} (q) \leq & \frac{m+2k}{R_1^{m+2k-\alpha}} \int_{\bB_{R_1}} g^2 \,d\norm{T} + \frac{2}{\alpha} \int_{\bB_{R_1}}  \frac{\abs{\nabla g (q)}^2\abs{q^\perp}^2}{\abs{q}^{m+2k-\alpha}}\, d\norm{T} (q)\nonumber\\
	&+ C \bA \|\hat g\|_\infty^2 \frac{\norm{T}(\bB_{R_1})}{R_1^{m-\alpha}}\,,
\end{align}
where $q^\perp := q - \mathbf{p}_{\vec{T}}(q)$ at $\Ha^m$-a.e. $q \in \spt(T)$ (here, $\mathbf{p}_{\vec{T}} = \mathbf{p}_{\vec{T}(q)}$ is the orthogonal projection onto ${\rm span}(\vec{T}(q))$).
Moreover, for every nonnegative $f\in C^1 (\mathbb R)$, upon setting $F (t) := -\int_t^{R_1} f'(s) s^m\, ds$, we have
\begin{align}\label{e.classical monotonicity with f} 
 & \int_{\bB_{R_1}} f(|q|)\, d\norm{\bC} (q) - \int_{\bB_{R_1}} f(\abs{q})\, d\norm{T} (q) +  \int_{\bB_{R_1}} \frac{F(\abs{q})\,\abs{q^\perp}^2}{\abs{q}^{m+2}}\, d\norm{T} (q)\nonumber\\
 \leq & -\frac{1}{m
} \int_{\bB_{R_1}} \frac{F(\abs{q})\, q^\perp\cdot H_T (q)}{\abs{q}^{m}}\, d\norm{T} (q)\,.
\end{align}
\end{lemma}

The proof, which follows standard computations, is given in the Appendix. This is the point where one crucially uses the assumption that $\Theta_T(0) \geq \frac{p}{2} = \Theta_{\bC}(0)$.

The rest of the section will be devoted to the proof of Theorem \ref{t:Hardt-Simon-main}.

\subsection{Preliminary estimates} In the sequel we denote by $\partial_r$ the derivative in the radial direction $\frac{q}{|q|}$ and, given a $p$-multifunction $u$ as in Theorem \ref{thm:graph v2}, we use the shorthand notation $\left|\partial_r \frac{u_i (z)}{|z|}\right|^2$ for the functions
\[
\sum_j \left|\partial_r \frac{u_{i,j} (z)}{|z|}\right|^2
\]
on the respective domains $U_i = R_{\mathcal{W}}\cap \bH_{0,i}$.

\begin{proposition}\label{prop:decay-growth-est}
There exists a geometric constant $\beta_1 > 0$ such that for any $\beta \in \left( 0, \beta_1 \right)$ there are constants $C$ and $\varepsilon_1$ depending on $(m,n,p,\bC_0,\beta)$ with the following property. Let $T,\Sigma,\bC,\bC_0$, and $\bS_0$ be as in Theorem \ref{thm:graph_v1}, let $u$ 
be the map defined in Theorem \ref{thm:graph_v1}. 
Then:
	\begin{align}
	&\int_{\bB_{\sfrac{11}{6}}} \frac{\dist(q,\bS)^2}{\abs{q}^{m+\frac74}} \,d\|T\|+ \sum_i \int_{\bB_{\sfrac{11}{6}} \cap U_i} \abs{z}^{2-m} \left|\partial_r \frac{u_i (z)}{\abs{z}}\right|^2 \,dz\nonumber\\
	\leq & C \int_{\bB_{\sfrac{11}{6}}} \frac{|q^\perp|^2}{|q|^{m+2}}\,d\|T\|  + C(\bE + \bA)\, .\label{eq:est_prelim}
	\end{align}
	
\end{proposition}

\begin{proof}
We apply \eqref{e.h_k monotonicity} with $R_1 = \sfrac{11}{6}$, $g(q)=\dist(q,\bS)$, and $\alpha=\sfrac{1}{4}$. Since $g$ is $1$-homogeneous,
and $1$-Lipschitz, the fist integral in \eqref{eq:est_prelim} can be bounded by the right hand side. 

\smallskip

To deduce the bound on the second element in the sum, first observe that
\[
\left|\partial_r \frac{u_{i,j} (z)}{\abs{z}}\right|\le 
\left|\partial_r\frac{v_{i,j} (z)}{\abs{z}}\right|\, . 
\]
Hence, since for $q = z + v_{i,j}(z)$ it holds
\[
|q| \leq C \, (1 + \beta) \, |z|\,,
\]
the inequality will follow if we can show that the pointwise estimate
	\begin{equation}\label{eq.point wise bound}
		\left| \partial_r \frac{v_{i,j}(z)}{\abs{z}}\right|^2 \le 2  \frac{|(z+ v_{i,j}(z))^\perp|^2}{\abs{z}^4} \qquad \mbox{for every $z\in \bB_{\sfrac{11}{6}} \cap U_i$}
	\end{equation}
holds for every $i \in \{1,\ldots,N_0\}$ and $j \in \{1,\ldots,\kappa_{0,i}\}$.

Using that $\partial_r (z/|z|) = 0$, we readily calculate
\begin{equation} \label{e:to be estimated}
    \partial_r \frac{v_{i,j}(z)}{\abs{z}} = \partial_r \frac{z + v_{i,j}(z)}{\abs{z}} = \frac{z + |z|\,\partial_r v_{i,j}(z)}{\abs{z}^2} - \frac{z + v_{i,j}(z)}{\abs{z}^2}\,,
\end{equation}
so that, since $z + |z|\,\partial_r v_{i,j}(z)$ is tangent to the graph of $v_{i,j}$ (and thus to $\spt(T)$) at $z + v_{i,j}(z)$, we have
\begin{equation} \label{e:normal component}
    \left| \left( \partial_r \frac{v_{i,j}(z)}{|z|}  \right)^\perp   \right|^2 = \frac{\abs{(z + v_{i,j}(z))^\perp}^2}{\abs{z}^4}\,.
\end{equation}
Then, to conclude the validity of \eqref{eq.point wise bound} we only have to estimate the tangential component. To this aim, we recall the notation $\mathbf{p}_i$ for the orthogonal projection onto (the $m$-plane containing) $\bH_{0,i}$, and using that $\mathbf{p}_i\, \left(\partial_r \frac{v_{i,j}(z)}{\abs{z}}\right) = 0$ since $v_{i,j}(z) \in \bH_{0,i}^\perp$, we deduce that, for $\mathbf{p}_{\vec T} = \mathbf{p}_{\vec{T}(q)}$ with $q=z+v_{i,j}(z)$,
\[
\left| \mathbf{p}_{\vec{T}} \,\left( \partial_r \frac{v_{i,j}(z)}{\abs{z}}\right) \right| \leq \| \mathbf{p}_{\vec{T}} - \mathbf{p}_i \|_O \, \left| \partial_r \frac{v_{i,j}(z)}{\abs{z}} \right|\,,
\]
where $\| \cdot \|_O$ denotes operator norm. In particular, a suitable choice of $\beta > 0$ yields, due to conclusion (ii) in Theorem \ref{thm:graph_v1},
\begin{equation} \label{e:tangential component}
    \left| \mathbf{p}_{\vec{T}} \, \left(\partial_r \frac{v_{i,j}(z)}{\abs{z}}\right) \right|^2 \leq \frac{1}{2} \, \left| \partial_r \frac{v_{i,j}(z)}{\abs{z}} \right|^2\,,
\end{equation}
so that we can conclude \eqref{eq.point wise bound} from \eqref{e:tangential component} and \eqref{e:normal component}.
\end{proof}

The goal of the next proposition is to show that, in fact, also the first addendum in the right-hand side of \eqref{eq:est_prelim} can be estimated by $C\,(\bE + \bA)$, which can be thought as a Caccioppoli-type inequality.

\begin{proposition}\label{prop:density-est}
Under the same assumptions of Proposition \ref{prop:decay-growth-est} (up to possibly choosing a smaller value for $\beta_1$) the following estimate holds.
Denote by $\mathbf{p}_V$ the orthogonal projection on the spine $V$ of $\bC_0$, and for $\|T\|$-a.e. $q$ denote by $\mathbf{p}_{\vec{T} (q)^\perp}$ the projection on the orthogonal complement of the tangent plane to $T$ at $q$. Then
\begin{gather}\label{e.reverse poincare}
\int_{\bB_{\sfrac{11}{6}}}\left(\left|\mathbf{p}_V\cdot \mathbf{p}_{\vec{T} (q)^\perp}\right|^2 \, + \frac{|q^\perp|^2}{\abs{q}^{m+2}}\right) \,d\|T\| (q) \leq C \, (\bE +\bA^2)\, ,
\end{gather}
where $|\cdot|$ is the Hilbert-Schmidt norm and the constant $C$ depends upon $(m,n,p, \bS_0,\beta)$.
\end{proposition}

\begin{proof}
Let $g \in C^\infty_c(\bB_{R_0})$, and, denoting $\mathbf{p}_{V^\perp}$ the orthogonal projection onto the complement $V^\perp$ to the spine $V$ of $\bC_0$, test the first variation formula \eqref{e:H in L infty} with the vector field $\chi(q) = \chi(x,y) := g^2(q)\, \mathbf{p}_{V^\perp}(q) = g^2(q)\,x$ to obtain
\begin{equation} \label{e:first variation in action}
    - \int g^2 \, x \cdot H_T \, d\|T\| = \int {\rm div}_{\vec{T}}\, (g^2 x) \, d\|T\|\,.
\end{equation}
In order to calculate ${\rm div}_{\vec{T}}\, (g^2 x)$, where $\vec{T}=\vec{T}(q)$ and $x = \p_{V^\perp}(q)$, let us denote $(\tau_1,\ldots,\tau_m)$ and $(\nu_{m+1}, \ldots, \nu_{m+n})$ orthonormal bases of $\vec{T}$ and $\vec{T}^\perp$ respectively, so that
\[
{\rm div}_{\vec{T}}\, (g^2 x) = \sum_{i=1}^m \tau_i \cdot \nabla_{\tau_i}(g^2 x) = 2\, g \, \p_{\vec{T}}(x) \cdot \nabla g + g^2\, \sum_{i=1}^m \tau_i \cdot \p_{V^\perp}(\tau_i)\,.
\]
Concerning the first addendum, we see that
\[
\p_{\vec{T}}(x) \cdot \nabla g = \p_{\vec{T}}(x) \cdot (\nabla_{V^\perp}g + \nabla_V g) = \p_{\vec{T}}(x) \cdot \nabla_{V^\perp}g - \p_{\vec{T}^\perp}(x) \cdot \nabla_V g\,,
\]
since $x \cdot \nabla_V g = \p_{V^\perp}(q) \cdot \nabla_V g = 0$. Concerning the second addendum, instead, we write
\[
\sum_{i=1}^m \tau_i \cdot \p_{V^\perp}(\tau_i) = m - \sum_{i=1}^m \tau_i \cdot \p_V(\tau_i)\,,
\]
and since
\[
\sum_{i=1}^m \tau_i \cdot \p_V(\tau_i) + \sum_{j=1}^n \nu_{m+j} \cdot \p_V(\nu_{m+j}) = {\rm tr}\,(\p_V) = m-1\,,
\]
we deduce
\[
\sum_{i=1}^m \tau_i \cdot \p_{V^\perp}(\tau_i) = 1 + \sum_{j=1}^n \nu_{m+j} \cdot \p_V(\nu_{m+j}) = 1 + {\rm tr}\,(\p_V \cdot \p_{\vec{T}^\perp}) = 1 + \left| \p_V \cdot \p_{\vec{T}^\perp} \right|^2\,.
\]
Hence, Young's inequality allows to estimate from \eqref{e:first variation in action}
\begin{align}
  & - \int g^2 \, x \cdot H_{T} \, d\|T\|\nonumber\\
  = & \int (1 + \left| \p_V \cdot \p_{\vec{T}^\perp} \right|^2) \, g^2 \, d\|T\| + \int 2g\, \left( \p_{\vec{T}}(x) \cdot \nabla_{V^\perp}g - \p_{\vec{T}^\perp}(x) \cdot \nabla_V g \right) \, d\|T\| \nonumber \\
  \geq & \int (1 + {\textstyle{\frac12}} \left| \p_V \cdot \p_{\vec{T}^\perp} \right|^2 )\, g^2 \, d\|T\| -2 \int \left(  \abs{x^\perp}^2\, \abs{\nabla_V g}^2 - g\, (\p_{\vec{T}}(x) \cdot \nabla_{V^\perp}g)  \right) \, d\|T\|\,.
\end{align}
In particular we infer
\begin{align}
&\int (1 + {\textstyle{\frac12}} \left| \p_V \cdot \p_{\vec{T}^\perp} \right|^2 )\, g^2 \, d\|T\|\nonumber\\
\leq & - \int g^2 \, x \cdot H_{T} \, d\|T\|
+ 2 \int \left(  \abs{x^\perp}^2\, \abs{\nabla_V g}^2 - g\, (\p_{\vec{T}}(x) \cdot \nabla_{V^\perp}g)  \right) \, d\|T\|\,.\label{e:FV-1}
\end{align}
We next consider the linear $p$-multifunction $\{l_{i, h (i,j)}\}$ of Theorem \ref{thm:graph v2} and the corresponding cone $\tilde{\bC}$. Since $\tilde{\bC}$ has spine $V$, it is invariant with respect to scaling in the $V^\perp$ direction, so that, if we define $\iota_r (x,y):= (\frac{x}{r}, y)$, then $(\iota_r)_\sharp \tilde{\bC} = \tilde{\bC}$ for all $r>0$. Hence, if we differentiate in $r$ the identity
\[
\int g^2\, d\|\tilde{\bC}\| = \int g^2\, d\|(\iota_r)_\sharp \tilde{\bC}\|
\]
and evaluate for $r=1$ we conclude
\[
0 = - \int \left( 2 g \left(x\cdot \nabla_{V^\perp} g\right) + g^2 \right)\, d\|\tilde{\bC}\|\,,
\]
that is
\begin{equation}\label{e:omogenea}
\int g^2\, d\|\tilde{\bC}\| = - \int 2 g \left(x\cdot \nabla_{V^\perp} g\right) d \|\tilde{\bC}\|\,.
\end{equation}
Subtracting \eqref{e:omogenea} from \eqref{e:FV-1} we infer
\begin{align}
&\int \left| \p_V \cdot \p_{\vec{T}^\perp} \right|^2 g^2 d\|T\|
+ 2 \left(\int g^2 d\|T\| - \int g^2 d\|\tilde{\bC}\|\right)
\leq \underbrace{- 2 \int g^2 \, x \cdot H_{T} \, d\|T\|}_{=: {\rm(A)}}\nonumber\\
& \qquad\qquad+ \underbrace{4 \int \abs{x^\perp}^2\, \abs{\nabla_V g}^2\, d\|T\|}_{=: {\rm (B)}}
 + \underbrace{4 \int g \left(x\cdot \nabla_{V^\perp} g\right) d \|\tilde{\bC}\|
 - 4 \int g\, (\p_{\vec{T}}(x) \cdot \nabla_{V^\perp}g)\, d\|T\|}_{=: {\rm (C)}} \,. \label{e:FV-2}
\end{align}
Choose next $g (q) := \gamma (|q|)$, where $\gamma$ is a smooth, nonnegative, and nonincreasing function which equals $1$ on $[0,\sfrac{11}{6}]$ and is supported in $[0, 2)$. With this choice, and using \eqref{e.classical monotonicity with f}, the left-hand side of \eqref{e:FV-2} dominates the left-hand side of \eqref{e.reverse poincare} up to a summand $C\bA$. Moreover, it is easy to bound (A) with $C\bA$ and thus it remains to bound (B) and (C) with $C (\bE+\bA)$. We first use $\|\nabla g\|_\infty \leq C$ and \eqref{e:cosa-manca-al-grafico} to achieve the bound
\begin{align*}
{\rm (B)} &\leq C \int_{\bB_2\setminus R_{\mathcal W}} |x|^2 \, d\|T\| +
C \sum_{i,j} \underbrace{\int_{\bB_2} |x^\perp|^2 \,d \|\bG_{\tilde{v}_{i,j}}\|}_{=:{\rm (B}_{i,j}{\rm)}}\, ,
\end{align*}
where $\tilde v = \{\tilde v_{i,j}\}$ is the $p$-multifunction over $\tilde{\bS}$ introduced in Corollary \ref{cor:reparametrization}. The first summand is bounded by $C (\bE + \bA^2)$ because of \eqref{e:L^2 estimate}. As for the second summand, we use the graphical structure to write it as an integral over $\tilde{U}_{i,j}$. To that end, we write every $z\in \tilde{U}_{i,j}\subset \bH_{1,i,j}$ as $z=(\xi,\zeta)\in V^\perp\times V$ and denote by $\mathbf{p}^\perp_z$ the orthogonal projection onto the normal space $(T_{z+\tilde{v}_{i,j} (z)} \bG_{\tilde{v}_{i,j}})^\perp$. Using the fact that the Lipschitz constant of $\tilde{v}_{i,j}$ is bounded by $C\beta$, we then infer
\begin{align*}
    ({\rm B}_{i,j}) &\leq C \int_{\tilde{U}_{i,j}\cap \bB_2} \left|\mathbf{p}^\perp_z (\xi+\tilde{u}_{i,j} (z)+\Psi (z+\tilde{u}_{i,j} (z)))\right|^2\, dz\,.
\end{align*}
Next, consider that $\|\mathbf{p}^\perp_z - \mathbf{p}_{\bH_{1,i,j}^\perp}\|_O \leq C |\nabla \tilde{v}_{i,j} (z)|\leq C\|D\Psi\|_0 + C |\nabla \tilde{u}_{i,j} (z)|$ and since 
$\mathbf{p}_{\bH_{1,i,j}^\perp} (\xi) =0$, we conclude
\begin{equation}
\left|\mathbf{p}^\perp_z (\xi)\right|\leq C \bA + C |\xi||\nabla \tilde{u}_{i,j} (z)|
\, .
\end{equation}
On the other hand,
\begin{equation}
\left|\mathbf{p}^\perp_z (\tilde{u}_{i,j} (z) + \Psi (z+\tilde{u}_{i,j} (z)))\right|
\leq C |\tilde{u}_{i,j} (z)| + C \bA\, .
\end{equation}
In particular we conclude
\begin{align*}
({\rm B}_{i,j}) &\leq C \int_{\tilde{U}_{i,j}} \left(|\xi|^2 |\nabla \tilde{u}_{i,j} (z)|^2 +
|\tilde{u}_{i,j} (z)|^2\right)\, dz + C \bA^2\\
&\stackrel{\eqref{e:comparison-tilde-u-w}}{\leq} C \int_{(U_{3\mathcal W})_i\cap \bB_3} \left(|\xi|^2 |\nabla w_{i,j} (z)|^2 +
|w_{i,j} (z)|^2\right)\, dz + C\bA^2
\end{align*}
and thus $({\rm B}_{i,j})\leq C (\bE + \bA^2)$ because of \eqref{L2 estimate excess}.

\medskip

We now come to estimating (C). To this aim, we first compute the two integrands, namely
\begin{align*}
g (q) \, x\cdot \nabla_{V^\perp} g (q) &= \frac{\gamma' (|q|) \gamma (|q|)}{|q|} \, |x|^2 
=: \lambda (|q|)\, |x|^2 \\
g (q) \, \mathbf{p}_{\vec{T} (q)} (x) \cdot \nabla_{V^\perp} g (q)
&= \frac{\gamma' (|q|) \gamma (|q|)}{|q|}\, \mathbf{p}_{\vec{T} (q)} (x)\cdot x
=: \lambda (|q|)\, \mathbf{p}_{\vec{T} (q)} (x)\cdot x\, .
\end{align*}
In both cases the integrands are bounded by $C|x|^2$ due to the fact that $|q|^{-1} \gamma' (|q|)$ is bounded. In particular, arguing as for (B) we can estimate
\begin{align}
{\rm (C)} & \leq C (\bE + \bA^2) + 4\, \sum_{i,j} \left|\int_{\tilde{U}_{i,j}} \lambda (|z|)\,
|\xi|^2\, dz - \int \lambda (|q|)\, \mathbf{p}_{\vec{T} (q)} (x)\cdot x\, 
d\|\bG_{\tilde{v}_{i,j}}\| (q)\right|\, .\label{e:bound-for-(C)}
\end{align}
If we introduce the projection $\mathbf{p}_z$ onto the tangent $T_{z+\tilde{v}_{i,j} (z)}\bG_{\tilde{v}_{i,j}}$ and the Jacobian $J \tilde{v}_{i,j} (z)$, we can then use the graphicality to express the second integral as 
\begin{equation}\label{e:ugly-formula}
\int_{\tilde{U}_{i,j}} \underbrace{\lambda (|z+\tilde{v}_{i,j} (z)|)\,  
\mathbf{p}_{z} (\xi+\tilde{v}_{i,j} (z)) \cdot (\xi +\tilde{v}_{i,j} (z))}_{=: f(z)}\,  J\tilde{v}_{i,j} (z)\, dz\, .
\end{equation}
Recall first the classical Taylor expansion
\[
|J\tilde{v}_{i,j} (z) -1| \leq C |\nabla \tilde{v}_{i,j} (z)|^2
\leq C |\nabla\tilde{u}_{i,j} (z)|^2 + C \bA^2\, .
\]
Since $|f(z)| \leq C|\xi|^2 + |\tilde{v}_{i,j} (z)|^2 \leq C |\xi|^2 + C\bA^2$, up to an error term $C |\xi|^2 |\nabla\tilde{u}_{i,j} (z)|^2 + C \bA^2$ the integrand in \eqref{e:ugly-formula} can be treated as 
\begin{equation}\label{e:less-ugly}
\lambda (|z+\tilde{v}_{i,j} (z)|)\,  
\mathbf{p}_{z} (\xi+\tilde{v}_{i,j} (z)) \cdot (\xi +\tilde{v}_{i,j} (z))\,.
\end{equation}
Next note that $\lambda$ vanishes on $[0,1]$, so that we can regard it as a smooth function of $|q|^2$. Since $|z+\tilde{v}_{i,j} (z)|^2-|z|^2 = |\tilde{u}_{i,j} (z)|^2 + |\Psi (z+\tilde{u}_{i,j} (z))|^2 \leq |\tilde{u}_{i,j} (z)|^2 + C \bA^2$, up to an error term $C|\tilde{u}_{i,j} (z)|^2 + C \bA^2$, the expression in \eqref{e:less-ugly} can be treated as 
\begin{equation}\label{e:less-ugly-2}
\lambda (|z|)\,  
\mathbf{p}_{z} (\xi+\tilde{v}_{i,j} (z)) \cdot (\xi +\tilde{v}_{i,j} (z))\, .
\end{equation}
Next observe that
\[
\mathbf{p}_{z} (\xi+\tilde{v}_{i,j} (z)) \cdot (\xi +\tilde{v}_{i,j} (z))
= |\mathbf{p}_{z} (\xi+\tilde{v}_{i,j} (z))|^2 = |\p_z(\xi)|^2 + |\p_z(\tilde v_{i,j}(z))|^2 + 2 \, \p_z(\xi) \cdot \p_z(\tilde v_{i,j}(z))\,.
\]
Now, we clearly have
\[
|\p_z(\tilde v_{i,j}(z))|^2 \leq C |\tilde u_{i,j}(z)|^2 + C \bA^2\,;
\]
furthermore, from the definition of $\p_z$ one gets that
\[
\begin{split}
2 \, |\p_z(\xi) \cdot \p_z(\tilde v_{i,j}(z))| &= 2 \, |\xi \cdot \p_z (\tilde v_{i,j}(z))| \leq C \, |\xi| \, |\nabla \tilde v_{i,j}(z)| \, |\tilde v_{i,j}(z)| \\
& \leq C \, |\xi|^2 \, |\nabla \tilde u_{i,j}(z)|^2 + C \, |\tilde u_{i,j}(z)|^2 + C \bA^2\,.
\end{split}
\]
Thus, up to an error term of type $C \, |\xi|^2 \, |\nabla \tilde u_{i,j}(z)|^2 + C \, |\tilde u_{i,j}(z)|^2 + C \bA^2$, \eqref{e:less-ugly-2} can be treated as 
\[
\lambda (|z|) |\mathbf{p}_z (\xi)|^2\, .
\]
In turn we can write
\[
|\xi|^2 - |\mathbf{p}_z (\xi)|^2 \leq C |\xi|^2 |\nabla \tilde{v}_{i,j} (z)|^2
\leq C |\xi|^2 |\nabla \tilde{u}_{i,j} (z)|^2 + C \bA^2\, .
\]
Since $\lambda (|z|) |\xi|^2$ is the integrand in the first integral of \eqref{e:bound-for-(C)}, summarizing our considerations we achieve
\begin{align*}
&\left|\int_{\tilde{U}_{i,j}} \lambda (|z|)\,
|\xi|^2\, dz - \int \lambda (|q|)\, \mathbf{p}_{\vec{T} (q)} (x)\cdot x\, 
d\|\bG_{\tilde{v}_{i,j}}\| (q)\right|\\
\leq & C\bA^2 + C\int_{\tilde{U}_{i,j}\cap \bB_2} (|\xi|^2 |\nabla \tilde{u}_{i,j} (z)|^2 + |\tilde{u}_{i,j} (z)|^2)\, dz\, .
\end{align*}
Hence, using again \eqref{e:comparison-tilde-u-w} and \eqref{L2 estimate excess} we conclude the desired estimate ${\rm (C)} \leq C (\bE+\bA^2)$.
\end{proof}

\subsection{Proof of Theorem \ref{t:Hardt-Simon-main}} 

Before coming to the proof we isolate the following simple remark:


\begin{lemma}\label{l:shifting-points}
Under the assumptions of Theorem \ref{t:Hardt-Simon-main} the following holds provided $\varepsilon_2$ and $\eta_3$ are chosen sufficiently small.
Let $\lambda \in \left(\sfrac{1}{2R_0},\frac{1}{R_0}\right)$, $q_0$ be as in Theorem \ref{t:Hardt-Simon-main} and $\tilde{\bC}$ be the cone in Definition \ref{d:new-cone}. Let $O$ be an orthogonal linear transformation of $\mathbb R^{m+n}$ mapping $T_0 \Sigma$ onto $T_{q_0} \Sigma$ so that $|O-{\rm Id}|$ is minimal. Then the assumptions of Proposition \ref{prop:density-est}
hold when we replace $T$ with $T_{q_0,\lambda} = (\eta_{q_0,\lambda})_\sharp T$, $\Sigma$ with $\Sigma_{q_0,\lambda}:= (\Sigma-q_0)/\lambda$, and the cones $\bC_0$ and $\bC$ with cones $O( \bC_0)$ and $O(\tilde{\bC})$.
\end{lemma} 
\begin{proof} Observe first that $|O-{\rm Id}|\leq C_0 \bA$ for some geometric constant $C_0$.
Next recall that $\spt(\tilde{\bC})=\tilde{\bS}\subset \bS=\spt(\bC)$ and that the multiplicities of $\tilde{\bC}$ are a reordering of the multiplicities of $\bC$. We thus deduce that 
\[\hat\flat^p_{{\bB}_{R_0}}(O(\tilde{\bC})-\bC_0)\le C|O-{\rm Id}|+ \hat\flat^p_{{\bB}_{R_0}}(\tilde{\bC}-\bC_0)< \eta_{\bS_0}\,.\]
Next note that, upon choosing $\varepsilon_2$ very small, we can assume as a consequence of (i) in  Theorem \ref{thm:graph_v1} that $q_0=(x_0,y_0)$ is sufficiently close to $V$ i.e. $|x_0|\le \delta(\varepsilon_2)$. Hence by the using the invariance of $\bC_0$ along $V$ and scaling we have 
\[\hat\flat^p_{{\bB}_{R_0}}(T_{q_0,\lambda}-\bC_0)\le C \hat\flat^p_{{\bB}_{R_0}}(T-\bC_0) + C|x_0| \le \eta_{\bS_0}\,. \]

It remains to check that the excess with respect to $O(\tilde{\bS})$ is small. 
\begin{align*} \dist\left(\frac{q-q_0}{\lambda}, O(\tilde{\bS})\right)&\le \dist \left( \frac{q-q_0}{\lambda}, \tilde{\bS}\right) +|O-{\rm Id}|\left|\frac{q-q_0}{\lambda}\right| \\
&= \frac1\lambda\dist \left(q-q_0, \tilde{\bS}-y_0\right) + |O-{\rm Id}|\left|\frac{q-q_0}{\lambda}\right|\\
&\le \frac1\lambda \dist(q,\tilde{\bS}) + \frac{|x_0|}{\lambda} + |O-{\rm Id}|\left|\frac{q-q_0}{\lambda}\right|\,.
\end{align*}
Hence we conclude
\[\bE(T_{q_0,\lambda},O(\tilde{\bS}), 0, R_0) \le C \bE(T, \tilde{\bS},q_0, \lambda R_0) + C\left(|x_0|^2 + \bA^2 \right)\,.
\]
Now we can appeal to \eqref{eq:tildeexcesvsexcess} and conclude the lemma.
\end{proof}

\begin{proof}[Proof of Theorem \ref{t:Hardt-Simon-main}] 
By Lemma \ref{l:shifting-points}, we can apply Proposition
\ref{prop:decay-growth-est} and Proposition \ref{prop:density-est} with $T_{q_0,\lambda}$ and $O(\tilde{\bC})$ replacing $T$ and $\bC$. Choosing $\lambda =\frac{1}{2R_0}$, 
we conclude
\begin{equation}\label{e:LS1}
\int_{\bB_{\lambda}(q_0)} \frac{{\rm dist}\, (q-q_0, O(\tilde{\bS}))^2}{|q-q_0|^{m+\frac74}} \, d\|T\|(q) \leq C (\bE (T, O(\tilde{\bS}) +q_0, q_0, \sfrac12) + \bA)\, .
\end{equation}
Next, observe that $\dist(q-q_0, \tilde{\bS}) \le \dist(q-q_0, O(\tilde{\bS})) + |O-{\rm Id}| |q-q_0|$. Since $|q-q_0|^{2-m-\frac74}$ is integrable with respect to $d||T||$ and $|O-{\rm Id}| \leq C_0 \bA$ we can replace $O ({\tilde{\bS}})$ in the left hand side with $\tilde{\bS}$ at the price of a larger constant $C$ in the right hand side. 
Next, observe that by the invariance of $\tilde{\bS}$ along the spine $V$ we can estimate 
\begin{equation}\label{e:LS2}
\dist (q, \tilde{\bS} + q_0) \leq \dist (q, \tilde{\bS}) + |x_0|\, .
\end{equation}
Hence 
\begin{align*}
\bE (T, \tilde{\bS} +q_0, q_0, \sfrac12) &\leq C |x_0|^2 + C \bE (T, \tilde{\bS}, 0, 2)
\leq C |x_0|^2 + C(\bA + \bE)\, ,
\end{align*}
where in the last inequality we have used \eqref{eq:tildeexcesvsexcess}. 
Combining with \eqref{e:LS1} we achieve
\begin{equation}\label{e:LS3}
\int_{\bB_{\lambda}(q_0)} \frac{{\rm dist}\, (q-q_0, \tilde{\bS})^2}{|q-q_0|^{m+\frac74}}\, d\|T\| (q) \leq C (\bE + \bA) + C |x_0|^2\,.
\end{equation}
In $\bB_1 \setminus \bB_\lambda(q_0)$ we have $|q-q_0|>\lambda$ hence
\begin{align*}&\int_{\bB_1\setminus \bB_{\lambda}(q_0)} \frac{{\rm dist}\, (q-q_0, \tilde{\bS})^2}{|q-q_0|^{m+\frac74}}\, d\|T\| (q) \le C \bE(T, \tilde{\bS}+q_0,0, 1)\\
&\qquad\le C|x_0|^2 + \bE(T,\tilde{\bS},0,1) \le C|x_0|^2 +C(\bA + \bE)\,, \end{align*}
where we have used once again \eqref{eq:tildeexcesvsexcess}.

We claim the existence of $\beta_1$ and $C_1$ depending only upon $\bS_0$ such that, if 
\begin{equation}\label{e:condition-on-rho}
\rho\geq \max \{C_2 \bE^{\sfrac{1}{(m+2)}} \beta^{-1}, \bar{C} |x_0|\}\, ,
\end{equation}
where $C_2$ is the constant of Theorem \ref{thm:graph_v1}(i) and $\bar C$ depends only on $\bS_0$, then
\begin{equation}\label{e:LS-claim}
|x_0|^2 \leq C_1 \rho^{\frac74} \int_{\bB_{1}} \frac{{\rm dist}\, (q-q_0, \tilde{\bS})^2}{|q-q_0|^{m+\frac74}}\, d\|T\| (q) + C_1 \rho^{-m} (\bE + \bA)\, . 
\end{equation}
Using \eqref{e:LS-claim} with a fixed appropriately small $\rho$ we then get from \eqref{e:LS3}
\begin{equation}\label{e:LS4}
\int_{\bB_{1}} \frac{{\rm dist}\, (q-q_0, \tilde{\bS})^2}{|q-q_0|^{m+\frac74}}\, d\|T\| (q) \leq C (\bE + \bA) \, ,
\end{equation}
which in turn we can combine again with \eqref{e:LS-claim} to achieve the desired estimate \eqref{e:Hardt-Simon}. Note that in order to ensure that $\rho$ can be chosen sufficiently small, we need $\bE$ sufficiently small, which in turn dictates a sufficiently small choice of $\varepsilon_2$, depending on $\beta$, and $|x_0|$ smaller than a constant depending on $\beta$ and $\bS_0$, which in turn requires $\eta_3$ to be chosen sufficiently small.

We now come to the proof of \eqref{e:LS-claim}. We first choose a half plane $\bH_{0,i}$ which is furthest away from $q_0$, i.e. 
\begin{equation}\label{e:i-min}
|\mathbf{p}_{\bH_{0,i}^\perp} (q_0)| = |\mathbf{p}_{\bH_{0,i}^\perp} (x_0)|
= \max_j \big\{|\mathbf{p}_{\bH_{0,j}^\perp} (q_0)|\big\} 
= \max_j \big\{|\mathbf{p}_{\bH_{0,j}^\perp} (x_0)|\big\}\, .
\end{equation}
Note that, since the open book $\bS_0$ is nonflat, there is a positive constant $c$ depending only on $\bC_0$ such that
\begin{equation}\label{e:angolo}
4c |x_0| \leq |\mathbf{p}_{\bH_{0,i}^\perp} (q_0)|\,.
\end{equation}

Consider now the orthogonal complement of $\bH_{0,i}$ in $\pi_0$, namely $\bH_{0,i}^{\perp_0} = \bH_{0,i}^\perp \cap \pi_0$. The latter is a line and we can identify it with $\{(t, 0, \ldots, 0): t\in \mathbb R\}$. We now look at the projection of $q_0$ on this line, which is given by $(t_0,0,\ldots 0)$. Observe that, $|\mathbf{p}_{\bH_{0,i}^\perp} (q_0)|\leq |t_0| + C_0 \bA |x_0|$, because $q_0\in \spt (T)\subset \Sigma$, where $C_0$ is a geometric constant. In particular, choosing $\varepsilon_2$ sufficiently small, we can assume that $|t_0|\geq 2c |x_0|$. If $t_0=0$, it follows that $x_0=0$ and there is nothing to prove. 
We can thus assume, without loss of generality, that $t_0>0$. We now consider $\{\tilde{\bH}_{i,j}\}_j$ as graphs over $\bH_{0,i}$ of functions taking values in  $\bH_{0,i}^{\perp_0}$. We then choose the $j$ whose graph is lowest in the natural ordering induced by the variable $t$.
We then have
\[
|\mathbf{p}_{\tilde{\bH}_{i,j}^\perp} (q_0)|\geq
|\mathbf{p}_{\bH_{0,i}^\perp} (q_0)| - \big| \mathbf{p}_{\tilde{\bH}_{i,j}^\perp}-
\mathbf{p}_{\bH_{0,i}^\perp}| |x_0|\, .
\]
Choosing $\eta_3$ sufficiently small we can thus ensure
\begin{equation}\label{e:angolo-2}
c |x_0| \leq |\mathbf{p}_{\tilde{\bH}_{i,j}^\perp} (q_0)|\, .
\end{equation}
On the other hand for any point $x$ in $\tilde{\bH}_{i,j}$ with $\dist (x,V) \geq \bar C |x_0|$, where the constant $\bar C$ depends only upon $\bS_0$, it follows that 
\begin{equation}\label{e:distanza}
\dist (x-q_0, \tilde{\bS}) = |\mathbf{p}_{\tilde{\bH}_{i,j}^\perp} (q_0)| \geq c |x_0|\, .
\end{equation}
In order to prove the latter claim we first observe that it suffices to show it for the point $\mathbf{p}_{\pi_0} (q_0)$. Secondly, using the invariance of the cone along the spine $V$, we can assume as well that $q_0, x\in V^{\perp_0}$, thus reducing the claim to a simple $2$-dimensional geometric consideration. An illustration of why the latter holds is given in Figure \ref{f:jonas}.

\begin{figure}
\begin{tikzpicture}
\draw[dashed] (-7,0) -- (-3,0);
\node[right] at (-3,0){$\bH_{0,i}$}; 
\draw[dashed] (-7,0) -- ({-7+4 * cos 100}, {4* sin 100});
\draw[dashed] (-7,0) -- ({-7+4 * cos 230}, {4* sin 230});
\draw (-7,0) -- ({-7+3* cos 10}, {3* sin 10});
\draw (-7,0) -- ({-7+3* cos 4}, {3* sin 4});
\draw (-7,0) -- ({-7+3* cos (-3)}, {3* sin (-3)});
\node[below] at ({-7+3* cos (-3)}, {3* sin (-3)}){$\bH_{i,j}$};
\draw (-7,0) -- ({-7+3* cos 107}, {3* sin 107});
\draw (-7,0) -- ({-7+3* cos 112}, {3* sin 112});
\draw (-7,0) -- ({-7+3* cos 235}, {3*sin 235});
\draw (-7,0) -- ({-7+3* cos 225}, {3*sin 225});
\draw (-7,0) -- ({-7+3* cos 240}, {3*sin 240});
\draw (-7,0) -- ({-7+3* cos 220}, {3*sin 220});

\draw (0,0) -- ({4* cos (-3)}, {4* sin (-3)});
\node[below] at ({4* cos (-3)}, {4* sin (-3)}){$\bH_{i,j}$};
\draw (1,1) -- ({1+3* cos 10}, {1+3* sin 10});
\draw (1,1) -- ({1+3* cos 4}, {1+3* sin 4});
\draw (1,1) -- ({1+3* cos (-3)}, {1+3* sin (-3)});
\draw (1,1) -- ({1+3* cos 107}, {1+3* sin 107});
\draw (1,1) -- ({1+3* cos 112}, {1+3* sin 112});
\draw (1,1) -- ({1+3* cos 235}, {1+3*sin 235});
\draw (1,1) -- ({1+3* cos 225}, {1+3*sin 225});
\draw (1,1) -- ({1+3* cos 240}, {1+3*sin 240});
\draw (1,1) -- ({1+3* cos 220}, {1+3*sin 220});
\draw[fill] ({3* cos (-3)}, {3* sin (-3)}) circle [radius = 0.05];
\node[below] at ({3* cos (-3)}, {3* sin (-3)}){$x$};
\draw[dashed] ({3* cos (-3)}, {3* sin (-3)}) -- ({3* cos (-3) + (sin 3 + cos 3) * sin 3}, {3* sin (-3) + (sin 3 + cos 3)* cos 3});
\draw[fill] ({3* cos (-3) + (sin 3 + cos 3) * sin 3}, {3* sin (-3) + (sin 3 + cos 3)* cos 3}) circle [radius=0.05];
\node[below right] at ({3* cos (-3) + (sin 3 + cos 3) * sin 3}, {3* sin (-3) + (sin 3 + cos 3)* cos 3}){$\bar x$};
\end{tikzpicture}
\caption{The picture on the left shows the books $\bS_0$ (dashed lines) and the book $\tilde{\bS}$ (solid lines).The page $\bH_{0,i}$ is pictured horizontal and the page $\bH_{i,j}$ is the ``lowest page'' of $\tilde{\bS}$ among those close to $\bH_{0,i}$. The angle formed between $q_0$ and $\bH_{0,i}$ is larger than a geometric constant (depending only on $\bS_0$) and much larger than the angle between $\bH_{0,i}$ and $\bH_{i,j}$.
The picture on the right shows $\bH_{i,j}$, the translated book $q_0 + \tilde{\bS}$ and a point $x\in \bH_{i,j}$ with the property that $|x-q_0|\geq C |x_0|$ for a suitable constant. Observe that ${\rm dist}\, (x-q_0, \tilde{\bS}) = {\rm dist}\, (x, q_0 + \tilde{\bS}) = |\bar x - x|$, where $\bar x$ is the point on $q_0 + \tilde{\bS}$ closest to $x$. Note that $\bar x$ must belong to $q_0 + \bH_{i,j}$ and $\bar x-x$ must be orthogonal it, in particular $|x-\bar x| = |\mathbf{p}_{\bH_{i,j}^\perp} (q_0)|$.} \label{f:jonas}
\end{figure}
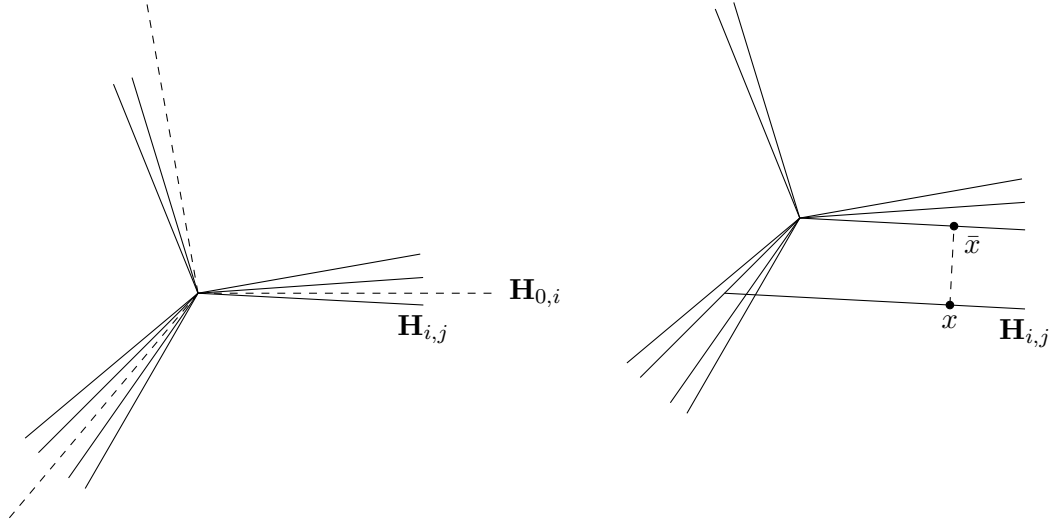

Fix a system of coordinates so that $\tilde{\bH}_{i,j} = \{(x_1, 0, \ldots, 0, v): v\in V, x_1\geq 0\}$.
Consider now $\beta< \beta_1$ fixed and let $\varepsilon_2$ be chosen so small that the domain of the function $\tilde{v}_{i,j}$ contains $\Omega := [\rho, 2\rho]\times \{0\} \times B_{\rho} (y_0)$. This is possible by choosing $\varepsilon_2$ sufficiently small because of \eqref{e:condition-on-rho} and Theorem \ref{thm:graph_v1}. We also require that any point in $\Omega$ satisfies $\dist (x,V) \geq \rho\geq \bar C |x_0|$, so that \eqref{e:distanza} holds.    

The graph of $\tilde{v}_{i,j}$ over $\Omega$ belongs to $T$, and if $q$ is a point on it and $x$ is its projection onto $\tilde{\bH}_{i,j}$ we can combine \eqref{e:distanza} and the triangle inequality to get
\[
c |x_0| \leq \dist (x - q_0, \tilde{\bS})
\leq \dist (q-q_0, \tilde{\bS}) + |\tilde{v}_{i,j} (x)|
\leq \dist (q-q_0, \tilde{\bS}) + |\tilde{u}_{i,j} (x)| + C \bA\, .
\]
Squaring the latter inequality, integrating it over the domain $\Omega$, and using that $\Omega \subset \bB_1$ if $\rho$ is small enough we reach
\begin{align*}
|x_0|^2 &\leq C \rho^{\frac74} \int_{\bB_1} \frac{\dist (q-q_0, \tilde{\bS})^2}{|q-q_0|^{m+\frac74}} d\|T\| (q)
+ C \rho^{-m} \int_{\bB_{1}\cap \tilde{\bH}_{i,j}} |\tilde{u}_{i,j}|^2 + C  \bA^2\nonumber\\
& \leq C \rho^{\frac74}\int_{\bB_{1}}\frac{\dist (q-q_0, \tilde{\bS})^2}{|q-q_0|^{m+\frac74}} d\|T\| (q)
+ C \rho^{-m} (\bE + \bA^2)\, ,
\end{align*}
where we have used \eqref{e:comparison-tilde-u-w}.
\end{proof}

\section{No-hole condition, binding functions, and estimates on the spine}\label{sec:binding}

We start by summarizing the assumptions on the various currents and parameters. 

\begin{ipotesi}\label{a:leone-simone} We let $T, \Sigma, \bC_0, \bC, \bS_0$, and $\bS$ be as in Assumption \ref{a:graphical}. $\beta_1$ is the constant of Theorem \ref{t:Hardt-Simon-main}, which depends only upon $\bS_0$. For any fixed $\beta<\beta_1$ we choose $\eta_3$ and $\varepsilon_2$, depending upon $(m,n,p,\bS_0, \beta)$, as in Theorem \ref{t:Hardt-Simon-main} and we assume that
\begin{align} 
&\hat\flat^p_{{\bB}_{R_0}}(T-\bC_0) + \hat\flat^p_{{\bB}_{R_0}}(\bC-\bC_0) < \eta_3\,,\label{e:hyp_flat_T_C0-aggiornata}\\
&\bA+ \bE (\bC, \bS_0, 0, R_0) + \bE (T, \bS, 0, R_0)\leq \eps_2^2\, . \label{e:hyp_graph_aggiornata}
\end{align}
\end{ipotesi}

In this section, we are going to adopt the following notation. Recall that $\mathcal{Q}$ defines a collection of cubes in $\left[ 0, 2 \right] \times \left[-2,2\right]^{m-1} \subset \left[0,\infty\right) \times V$, so that
\[
\bigcup_{Q \in \mathcal Q} Q = \left( 0, 2 \right] \times \left[-2,2\right]^{m-1} = \left[ 0, 2 \right] \times \left[-2,2\right]^{m-1} \setminus V\,.
\]
Recalling that 
\begin{align*}
\pi_0 &= \{0_{n-1}\} \times \R^2 \times \R^{m-1}\,,\\
V &= \{0_{n-1}\} \times \{0_2\} \times \R^{m-1}\,,
\end{align*}
we will set 
\begin{equation}
    R_{\mathcal Q} := \left\lbrace q=(0,x,y) \in \pi_0 \, \colon \, 0 < |x| \leq 2 \quad \mbox{and} \quad y \in \left[-2,2\right]^{m-1} \right\rbrace\,.
\end{equation}
Notice that $R_{\mathcal Q}$ is invariant with respect to rotations around $V$ in $\pi_0$. 

\begin{definition} \label{def:binding functions}

A \emph{binding function} is any Borel measurable function $\xi \colon R_{\mathcal Q} \to V^\perp$ with the property that
     $\xi(q) = \xi(q')$ for all $q=(0,x,y)$ and $q'=(0,x',y')$ such that $(|x|,y)$ and $(|x'|,y')$ belong to the interior of the same $Q \in \mathcal Q$.
\end{definition}

The following is the main theorem of this section. In the statement, we will use the notation 
\begin{equation}\label{e:def_rho_inf}
\varrho_\infty := \| \varrho_{\mathcal W} \|_{L^\infty(\left[-2,2\right]^{m-1})}\, ,
\end{equation}
where $\varrho_{\mathcal W}$ is the function defined in \eqref{oh dear vrho} corresponding to a Whitney domain $\mathcal W$.

\begin{theorem} \label{thm:est spine}
    Let $\bC_0$ be as in Assumption \ref{ass:cone}, with $\bS_0 = \spt(\bC_0) \subset \pi_0$ and let $V$ be the spine of $\bC_0$. Let $\beta_1$ be as in Assumption \ref{a:leone-simone}. For any $\beta < \beta_1$ there exist positive constants $\eta_4 \leq \eta_3$, $\eps_3\leq \eps_2$, and $C$, depending upon $(m,n,p,\bS_0,\beta)$ with the following property. Let $T,\Sigma,\bC$, and $\bS$ be as in Assumption \ref{a:leone-simone} such that \eqref{e:hyp_flat_T_C0-aggiornata}-\eqref{e:hyp_graph_aggiornata} hold with $\eta_4$ and $\eps_3$ in place of $\eta_3$ and $\eps_2$. Then, setting as usual $|x|(q)=\dist(q,V)$, it holds
    
    \begin{equation}\label{e:est spine}
        \int_{\bB_{\sfrac12}} \frac{\dist(q,\tilde\bS)^2}{\max\{\varrho_\infty, |x|\}^{1/2}} \, d\|T\|(q) \leq C (\bE + \bA)\,,
    \end{equation}
    where $\tilde\bS$ is the open book $\spt(\tilde\bC) = {\bG}_{\bC_0}(\tilde l)$ from Definition \ref{d:new-cone}. Furthermore, let $u$, $U_{4\mathcal W}$, $\tilde l$, and $w$ be as in Theorem \ref{thm:graph v2}. There exist a binding function $\xi$ and a $p$-multifunction $\varpi$ on $U_{\mathcal W}$ over $\bS_0$ such that
    
    \begin{align} \label{e:bf infty}
       & \| \xi \|_{\infty}^2  \leq C (\bE + \bA)\,, \\ \label{bf:multierror}
       & \sup_{\zeta \in U_{\mathcal W}} \abs{\varpi(\zeta)}^2 \leq C\, (\bE + \bA) \, \hat\flat^p_{\bB_{R_0}}(\bC-\bC_0)^2\,, \\  \label{e:bf L2}
       & \sum_{i=1}^N \sum_{j=1}^{\kappa_{0,i}} \int_{(U_\mathcal W)_i \cap \bB_{\sfrac12}} \frac{\abs{w_{i,j} - \varpi_{i,j} - \mathbf{p}_{\bH_{0,i}^{\perp_0}}(\xi)}^2}{|x|^{5/2}} + \frac{\abs{\nabla w_{i,j}}^2}{\abs{x}^{1/2}} \, dz \leq C(\bE + \bA)\,.
        \end{align}
    
\end{theorem}

\subsection{No-holes property}

The following ``no-holes property'' is the crucial tool towards the proof of Theorem \ref{thm:est spine}.

\begin{proposition}[No-holes property]\label{prop:no-holes}
Let $\bC_0$ be as in Assumption \ref{ass:cone}, with 
\[
\bS_0 = \spt(\bC_0) \subset \pi_0
\]
and let $V$ be the spine of $\bC_0$. For every $\delta \in \left(0, \frac18 \right)$, there exists $\eps_{{\rm NH}} = \eps_{\rm NH}(m,p,\bS_0,\delta) > 0$ with the following property. Let $\Sigma$ be as in Assumption \ref{assumption:manifolds and currents}, and let $T$ be area minimizing $\modp$ in $\Sigma \cap \bB_2$ with $(\pa T) \mres \bB_1 = 0\, \modp$. If $\hat\flat^p_{{\bB}_1}(T-\bC_0) < \eta_1(\bS_0)$ (where $\eta_1(\bS_0)$ is the parameter defined in Lemma \ref{lem:qualitative flat_excess}), and if
\begin{equation} \label{e:no_holes}
   \bA + \bE(T,\bS_0,0,1) \leq \eps_{{\rm NH}}^2\,,
\end{equation}
then $T$ satisfies the following \emph{$\delta$-no-holes condition w.r.t. $\bC_0$ in $B^{m-1}_{\sfrac{1}{2}} \subset V$}: 
\begin{itemize}
\item[(NH)] for any $y \in B_{\sfrac{1}{2}}^{m-1}$, there exists $q \in \bB_\delta((0,y))$ such that $\Theta_T(q) \geq \Theta_{\bC_0}(0) = \sfrac{p}{2}$. 
\end{itemize}
\end{proposition}

We first prove the following lemma.

\begin{lemma} \label{lem:technical_no_hole}
Let $\Sigma$ be as in Assumption \ref{assumption:manifolds and currents}, and let $T$ be area minimizing $\modp$ in $\Sigma \cap \bB_2$ with $(\pa T) \mres \bB_1 = 0 \, \modp$. If $\Theta_T(q) < \sfrac{p}{2}$ for every $q \in \bB_1$, then $(\pa T) \mres \bB_1 = 0$. In particular, $T$ is an area minimizing integral current without boundary in $\bB_1$.
\end{lemma}

\begin{proof}
First observe that, since $\{q \in \bB_1 \,\colon\,\Theta_T(q)=\sfrac{p}{2}\}=\emptyset$, $\pa T$ is a flat chain whose support $\spt(\pa T) \cap \bB_1$ is contained in the singular set $\sing(T) \cap \bB_1$. By the standard stratification of $\sing(T)$, given a point $q \in \spt(\pa T) \cap \bB_1$, one and only one of the following (mutually exclusive) cases may occur:
\begin{itemize}
\item[(a)] $q \in \cS^{m-2}$,
\item[(b)] $q \in \cS^{m-1} \setminus \cS^{m-2}$,
\item[(c)] $q \in [\cS^m \setminus \cS^{m-1}] \cap \sing(T)$.
\end{itemize}
By Proposition \ref{lem:structure_cones}, the assumption that $\Theta_T(q) < \sfrac{p}{2}$ prevents case (b) to occur, and, since $T$ has codimension one in $\Sigma$, White's regularity \cite[Theorem 4.5]{White86} implies that the set of points in (c) is empty. Thus, the only possible alternative is (a), whence $\dim_{\Ha}(\spt(\pa T) \cap \bB_1) \leq m-2$. Since $\pa T$ is a flat chain of dimension $m-1$, this implies that necessarily $\pa T \mres \bB_1 = 0$ (see e.g. \cite[Theorem 3.1]{White_deformation}).
\end{proof}

\begin{remark}
We observe explicitly that if $\Sigma$ is of class $C^{3,a_0}$ for some $a_0\in \left(0,1\right)$ then Lemma \ref{lem:technical_no_hole} holds true also when the codimension of $T$ in $\Sigma$ is larger than one. The proof is the same, modulo the fact that White's regularity theory cannot be invoked, and that the set of points in (c) may in fact be not empty. Nonetheless, we can still bound its Hausdorff dimension by $m-2$ using \cite[Theorem 1.7]{DLHMS}.
\end{remark}

\begin{proof}[Proof of Proposition \ref{prop:no-holes}]

Given the structure of area minimizing $m$-cones $\modp$ with $(m-1)$-dimensional spine as detailed in Proposition \ref{lem:structure_cones}, we can assume that $\partial \bC_0 = p \, \a{V}$ in $\R^{m+n}$, for a suitable choice of a constant orientation on $V$. In particular, it holds
\begin{equation} \label{e:bdry_comparison1}
(\pa \bC_0) \mres \bB_{\delta}((0,y)) \neq 0 \qquad \mbox{for every $\delta > 0$ and for every $y \in V$}\,. 
\end{equation}

Now, suppose towards a contradiction that the proposition is false. Then, there are $0 < \delta < \sfrac{1}{8}$, a sequence $\eps_j \downarrow 0^+$, and currents $T_j$ area minimizing $\modp$ in $\Sigma_j \cap \bB_2$ with 
\begin{equation} \label{contradiction to no holes}
    \spt^p(\pa T_j) \cap \bB_1 = \emptyset\,, \quad \hat\flat^p_{{\bB}_1}(T_j-\bC_0) < \eta_1(\bS_0)\,, \quad \bA_j + \bE (T_j, \bS_0,0,1) \leq \eps_j^2\,,
\end{equation}
which do not satisfy (NH). That is, there are points $y_j \in B_{1/2}^{m-1} \subset V$ such that $\Theta_{T_j}(q) < \sfrac{p}{2}$ for all $q \in \bB_\delta ((0,y_j))$. Lemma \ref{lem:technical_no_hole} then yields 
\begin{equation} \label{zero classical bdry}
(\pa T_j) \mres \bB_{\delta}((0,y_j)) = 0\,.
\end{equation}

First observe that, by a classical argument, it is easy to see that the second condition in \eqref{contradiction to no holes} together with the minimality $\modp$ of $T_j$ imply that the masses of $T_j$ in, say, $\bB_{4/5}$ are uniformly bounded by a constant $C(m,p)$. Moreover, since $\eps_j \downarrow 0^+$, Lemma \ref{lem:qualitative flat_excess} implies that
\begin{equation} \label{convergence to cone}
\hat\flat^p_{{\bB}_{\sfrac{3}{4}}}(T_j-\bC_0) \to 0 \qquad \mbox{as $j \to \infty$}\,.
\end{equation}

 Next, let $y \in \overline{B}^{m-1}_{1/2} \subset V$ be a subsequential limit of the points $y_j$. By slicing theory, if we denote ${\rm d}_y(q) = |q-(0,y)|$, we have that
\[
\int_{\frac{\delta}{2}}^{\delta} \mass(\langle T_j , {\rm d}_y , \sigma \rangle) \, d\sigma \leq \|T_j\|(\bB_\delta((0,y))) \leq \|T_j\|(\bB_{3/4})\,,
\]
so that there exist a (not relabeled) subsequence of $T_j$ and $\sigma \in \left( \frac{\delta}{2}, \delta \right)$ with the property that 
\[
\lim_{j \to \infty} \mass(\langle T_j , {\rm d}_y , \sigma \rangle) \leq C(m,p) \, \delta^{-1}\,.
\]
In particular, since
\[
\partial[T_j \mres \bB_\sigma((0,y))] = \langle T_j, {\rm d}_y, \sigma \rangle
\]
by \eqref{zero classical bdry} for all sufficiently large $j$, the sequence $\{T_j \mres \bB_\sigma((0,y))\}_{j}$ satisfies the hypotheses of the Federer-Fleming compactness theorem for integral currents, so that a further subsequence converges, in the sense of currents and with respect to the classical flat distance $\flat_{\overline{\bB}_\sigma((0,y))}$, to an integral current $\hat T$, and \eqref{zero classical bdry} guarantees that
\begin{equation} \label{e:bdry_comparison2}
(\pa \hat T) \mres \bB_{\sfrac{\delta}{2}}((0,y)) = 0\,.
\end{equation}

By \cite[Proposition 5.2]{DLHMS}, $\hat T$ is area minimizing $\modp$, and by Proposition \ref{flat:same topology} it holds $\lim_{j \to \infty}\hat\flat^p_{\bB_\sigma((0,y))}(T_j \mres \bB_\sigma((0,y)) - \hat T)=0$. In turn, using \eqref{convergence to cone}, the monotonicity of $\hat\flat^p$ with respect to the localizing set, and Proposition \ref{p:who cares about what is outside}, we conclude that $\hat\flat^p_{\bB_{\sigma}((0,y))}((\hat T - \bC_0) \mres \bB_{\sigma}((0,y))) = 0$, so that, in particular,
\begin{equation}\label{limit current is the cone}
    \hat T \mres \bB_{\delta/2}((0,y)) = \bC_0 \mres \bB_{\delta/2}((0,y)) \quad \modp
\end{equation}
by Corollary \ref{cor:flat hat and congruence}. In fact, since the multiplicities on $\bC_0$ are all strictly less than $\sfrac{p}{2}$, the identity in \eqref{limit current is the cone} holds in the sense of classical currents. The conditions \eqref{e:bdry_comparison1} and \eqref{e:bdry_comparison2} are then incompatible, and we have reached a contradiction. 
\end{proof}

\subsection{Proof of Theorem \ref{thm:est spine}}

\textit{Step one.} Recall the notation $\mathcal{Q}$ for the collection of cubes $Q \subset \left[0,2\right] \times \left[-2,2\right]^{m-1}$ defined in Section \ref{sec:graph}. Select, thanks to \eqref{e:dist v diam}, a number $\delta \in \left( 0, \frac{1}{4} \right)$ such that 
\begin{equation}\label{the choice of delta}
 \dist(4Q,V) \geq 2\delta d_Q \qquad \mbox{for every $Q \in \mathcal{Q}$}\,, 
\end{equation}
and then let $\eps_{{\rm NH}}$ be given by Proposition \ref{prop:no-holes} in correspondence with this choice of $\delta$. Let $y \in B_{\sfrac{1}{2}}^{m-1} = \bB_{\sfrac{1}{2}} \cap V$ be arbitrary, and let $2 > R > \varrho(y)$. By definition of $\varrho(y)$ and the structure of $U_{\mathcal W}$, then, there exists a cube $Q \in \mathcal W(T,\bS,\tau)$ such that $\zeta=(R,y) \in Q$. As usual, let $c_Q = (x_Q,y_Q)$ be the center of $Q$, $y_Q=(0,y_Q)$ the projection of $c_Q$ onto $V$, and $d_Q$ the diameter of $Q$. Notice, in passing, that $\abs{y-y_Q} < d_Q/2$. Also observe that, by \eqref{e:dist v diam} and our choice of $\bar{M}$, it holds
\begin{equation} \label{again distance v diameter}
    \frac{1}{4}\bar{M}d_Q \leq R \leq \frac{1}{2} \bar{M} d_Q\,.
\end{equation}
We claim now that, modulo possibly choosing a smaller value for $\tau$, the current $T_{y,R} := (\eta_{y,R})_\sharp T$ satisfies the hypotheses of Proposition \ref{prop:no-holes}. It is clear, by the choice of $R_0$, that $T_{y,R}$ is area minimizing $\modp$ in $\Sigma_{y,R} \cap \bB_2$, where $\Sigma_{y,R}: = \frac{\Sigma-y}{R}$, and that $(\partial T_{y,R}) \mres \bB_1 = 0 \; \modp$. Next, $\bA_{y,R} = \|A_{\Sigma_{y,R}}\|_0 \leq R \, \|A_\Sigma\|_0 \leq 2\,\bA$, and thus $\bA_{y,R} \leq \eps_{{\rm NH}}^2$ as soon as $\eps_3 \leq \eps_{{\rm NH}}^2/2$. Hence, we only have to check that
\begin{equation} \label{no holes hypotheses}
\hat\flat^p_{{\bB}_1}(T_{y,R} - \bC_0) < \eta_1(\bS_0)\,, \qquad \bE(T_{y,R},\bS_0,0,1) \leq \eps_{{\rm NH}}^2\,.
\end{equation}
For the excess estimate, using that $\bE(T_{y,R},\bS_0,0,1) = \bE(T,\bS_0,y,R)$ together with \eqref{again distance v diameter} we deduce that
\[
\bE(T_{y,R},\bS_0,0,1) \leq C\, \bE(T,\bS,y_Q, \bar{M} d_Q) + C\, \dist_{\Ha}(\bS \cap \bB_1, \bS_0 \cap \bB_1)^2 \leq C\, \tau^2 + C\, \eta_4^2\,.
\]
Therefore, the second inequality in \eqref{no holes hypotheses} is satisfied for suitable choices of $\tau$ and $\eta_4$. In this regard, notice that the quantity $\tau$ defining the Whitney domain was previously chosen so that \eqref{e: choice of tau} is satisfied: the smallness condition of $\tau$ with respect to $\eps_{{\rm NH}}$ forces, therefore, $\beta$ to be sufficiently small with respect to $\eps_{{\rm NH}}$, which translates into a smallness requirement on the constant $\beta_1$ of Theorem \ref{t:Hardt-Simon-main}, depending on $m,p$, and $\bS_0$.

Next, we prove the estimate on the modified flat distance. Of course, the estimate is trivial (provided $\eta_4$ is chosen small enough depending on the constant $M$) if $R$ is comparable to $1$, so we can assume without loss of generality that $Q$ is not contained in the top stratum $\left[1,2\right] \times \left[-2,2\right]^{m-1}$, so that $\log_{\sfrac{1}{2}}(\bar M d_Q) + 1 \geq 0$. The estimate is a simple consequence of the following claim: for any integer $0 \leq \ell \leq  \log_{\sfrac{1}{2}}(\bar M d_Q) +1$ it holds
\begin{equation} \label{flat iteration}
    \hat\flat^p_{\bB_{1}}(T_{y,2^{-(\ell+1)}} - \bC_0) \leq 2^{m+1}C(\bS_0)\, \bE(T,\bS_0,y,2^{-\ell})^{\frac12}\,,
\end{equation}
where $C(\bS_0)$ is the constant of Corollary \ref{flat-excess improved}. To prove \eqref{flat iteration}, first, notice that, by assumption, $\hat\flat^p_{\bB_1}(T_{y,1} - \bC_0) < \eta_4$. Thus, as long as $\eta_4 \leq \eta_2(\bS_0)$, Corollary \ref{flat-excess improved} implies that
\[
\hat\flat^p_{\bB_{\sfrac{1}{2}}}(T_{y,1} - \bC_0) \leq C(\bS_0)\, \bE(T,\bS_0,y,1)^{\frac12}\,,
\]
and thus, by rescaling,
\begin{equation}\label{first iteration}
    \hat\flat^p_{\bB_1}(T_{y,2^{-1}} - \bC_0) \leq 2^{m+1}C(\bS_0) \, \bE(T,\bS_0,y,1)^{\frac12}\,,
\end{equation}
which is \eqref{flat iteration} when $\ell=0$. Next, suppose that \eqref{flat iteration} is true for $\ell-1 \geq 0$, namely that
\begin{equation}\label{ind hyp}
    \hat\flat^p_{\bB_{1}}(T_{y,2^{-\ell}} - \bC_0) \leq 2^{m+1}C(\bS_0)\, \bE(T,\bS_0,y,2^{-(\ell-1)})^{\frac12}\,,
\end{equation}
and let us prove it for $\ell$. Since $\ell-1 \leq \log_{\sfrac{1}{2}}(\bar M d_Q)$, there exists a cube $Q'$ with diameter $d_{Q'} = 2^{-(\ell-2)}$ and such that $Q \preceq Q'$ (see Definition \ref{def:whitney}). By the definition of $\mathcal W$, we then have that
\[
\bE(T,\bS,y_{Q'}, \bar M d_{Q'}) \leq \tau^2\,,
\]
which in turn yields
\[
\begin{split}
\bE(T,\bS_0,y, 2^{-(\ell-1)}) &\leq C\, \bE(T,\bS_0,y_{Q'},\bar M d_{Q'})\\ 
&\leq \bE(T,\bS,y_{Q'}, \bar{M} d_{Q'}) + C\, \dist_{\Ha}(\bS \cap \bB_1, \bS_0 \cap \bB_1)^2 \\ 
&\leq C\, \tau^2 + C\, \eta_4^2\,.
\end{split}
\]
If $\tau$ and $\eta_4$ are sufficiently small, depending only on $m,p$, and $\bS_0$, \eqref{ind hyp} then implies that $\hat\flat^p_{\bB_1}(T_{y,2^{-\ell}} - \bC_0) < \eta_2(\bS_0)$, so that Corollary \eqref{flat-excess improved} applies and gives
\[
\hat\flat^p_{\bB_{\sfrac{1}{2}}}(T_{y,2^{-\ell}} - \bC_0) \leq C(\bS_0) \, \bE(T,\bS_0,y,2^{-\ell})^{\frac12}\,,
\]
that is, by rescaling, \eqref{flat iteration}. When \eqref{flat iteration} is applied with $\ell =  \log_{1/2}(\bar M d_Q) +1$, and keeping \eqref{again distance v diameter} into account, we deduce the first inequality in \eqref{no holes hypotheses} as soon as $\tau$ and $\eta_4$ are sufficiently small.

\medskip

\textit{Step two.} As a first, immediate consequence of \textit{Step one}, we see that for every $y \in B_{\sfrac{1}{2}}^{m-1}$ the hypotheses of Proposition \ref{prop:no-holes} are satisfied when $T$ is replaced by $T_{y,R}$ with $R = \varrho_\infty$. In particular, for every $y \in B_{\sfrac{1}{2}}^{m-1}$ there exists a point $\xi \in \bB_{\delta\varrho_\infty}((0,y))$ with $\Theta_T(\xi) \geq \Theta_{\bC_0}(0)=\sfrac{p}{2}$. Let us write the left-hand side of \eqref{e:est spine} as  
\begin{equation}\label{est spine split}
\begin{split}
    \int_{\bB_{\sfrac{1}{2}}} \frac{\dist^2(q,\tilde\bS)}{\max\{\varrho_\infty, |x|\}^{1/2}} \, d\|T\|(q) &=  \int_{\bB_{\sfrac{1}{2}} \cap B_{\varrho_\infty}(V)} \frac{\dist^2(q,\tilde\bS)}{\varrho_\infty^{1/2}} \, d\|T\|(q) \\
    &\qquad \qquad+ \int_{\bB_{\sfrac{1}{2}} \setminus B_{\varrho_\infty}(V)} \frac{\dist^2(q,\tilde\bS)}{|x|^{1/2}} \, d\|T\|(q)\,,
\end{split}
\end{equation}
and let us discuss here the first summand. For any $y \in B_{\sfrac{1}{2}}^{m-1}$, letting $\xi$ be a ``no-hole'' point as above in $\bB_{\delta \varrho_\infty}((0,y))$, we can apply Theorem \ref{t:Hardt-Simon-main} to estimate
\[
\begin{split}
\int_{\bB_{\varrho_\infty}((0,y))} \frac{\dist^2(q,\tilde\bS)}{\varrho_\infty^{1/2}} \, d\|T\|(q) &\leq C \, \varrho_{\infty}^{m+\sfrac{7}{4}-\sfrac{1}{2}} \int_{\bB_{\varrho_\infty}((0,y))} \frac{\dist^2(q-\xi, \tilde \bS)}{\abs{q-\xi}^{m+\frac74}} + C \, \varrho_\infty^{m-\sfrac{1}{2}}\,\abs{\p_{V^\perp}(\xi)}^2\\
&\leq C \varrho_\infty^{m-\sfrac{1}{2}} (\bE + \bA)\,.
\end{split}
\]

We can then cover $\bB_{\sfrac{1}{2}} \cap B_{\varrho_\infty}(V)$ with $N \leq C \varrho_\infty^{-(m-1)}$ balls $\{\bB_{\varrho_\infty}((0,y_j))\}_{j=1}^N$, and using the Besicovitch covering theorem to arrange such balls in $C_B$ subfamilies each consisting of pairwise disjoint balls we finally conclude that 
\begin{equation}\label{first piece comes home}
    \int_{\bB_{\sfrac{1}{2}} \cap B_{\varrho_\infty}(V)} \frac{\dist^2(q,\tilde\bS)}{\varrho_\infty^{1/2}} \, d\|T\|(q) \leq C (\bE + \bA)\,.
\end{equation}

\medskip

\textit{Step three.} Concerning the second term in the sum \eqref{est spine split}, we first notice that $\bB_{\sfrac{1}{2}} \setminus B_{\varrho_\infty}(V) \subset \bB_{\sfrac{1}{2}} \cap R_{\mathcal W}$. Then, we write $R_\mathcal W = \bigcup_{Q \in \mathcal W} A_Q$, where $A_Q$ is the set of all points $q=(x,y) \in \R^{m+n}$ such that $(0,|x|,y) \in Q$. \textit{Step one} shows that for each cube $Q \in \mathcal W$ whose center $c_Q$ has a projection $y_Q$ onto $V$ in $B_{\sfrac{1}{2}}^{m-1}$ the hypotheses of Proposition \ref{prop:no-holes} hold for the current $T_{y_Q,R_Q}$ when $R_Q:=\min_{\zeta \in Q}\dist(\zeta,V)$. As a consequence, for each $Q \in \mathcal W$ there exists a point $\xi_Q \in \bB_{\delta R_Q}((0,y_Q))$ with $\Theta_T(\xi_Q)\geq \sfrac p2$, and Theorem \ref{t:Hardt-Simon-main} gives 
\begin{equation}\label{hs at Qnh}
    \abs{\p_{V^\perp}(\xi_Q)}^2 + \int_{\bB_1} \frac{\dist^2(q-\xi_Q,\tilde{\bS})}{\abs{q-\xi_Q}^{m+\frac74}} \, d\|T\|(q) \leq C (\bE + \bA)\,.
\end{equation}
We can then argue as in \textit{Step two}, using in addition that $|x|\geq c(m) d_Q$ when $q \in A_Q$ to estimate 
\begin{equation}\label{second piece one at a time}
\begin{split}
\int_{\bB_{\sfrac{1}{2}}\cap A_Q} \frac{\dist^2(q,\tilde\bS)}{\max\{\varrho_\infty, |x|\}^{1/2}} \, d\|T\|(q) &\leq C \,d_Q^{m+\sfrac{7}{4}-\sfrac{1}{2}} \int_{\bB_{\sfrac{1}{2}}\cap A_Q} \frac{\dist^2(q-\xi_Q,\tilde\bS)}{\abs{q-\xi_Q}^{m+\frac74}} \, d\|T\|(q)\\
&\qquad + C\, d_Q^{m-\sfrac{1}{2}}\, \abs{\p_{V^\perp}(\xi_Q)}^2 \\
&\leq C\, d_Q^{m-\sfrac{1}{2}} (\bE+\bA)\,,
\end{split}
\end{equation}
where we have used \eqref{hs at Qnh} in the last inequality. Consider now that $d_Q$ is comparable to the sidelength of the cube, and that such sidelength is given by $2^{-k}$ for some positive integer $k$. For each fixed $k$, consider the collection $\mathcal{C}_k$ of cubes in $\mathcal W$ which intersect $\bB_{\sfrac{1}{2}}\setminus B_{\varrho_\infty} (V)$ and have sidelength $2^{-k}$. There are at most $C (2^{-k})^{1-m}$ such cubes. Therefore, we can estimate
\begin{align}
& \int_{\bB_{\sfrac{1}{2}} \setminus B_{\varrho_\infty}(V)}  \frac{\dist^2(q,\tilde\bS)}{\max\{\varrho_\infty, |x|\}^{1/2}} \, d\|T\|(q)\leq  \sum_k \sum_{Q\in \mathcal{C}_k} \int_{\bB_{\sfrac{1}{2}}\cap A_Q} \frac{\dist^2(q,\tilde\bS)}{\max\{\varrho_\infty, |x|\}^{1/2}} \, d\|T\|(q)\nonumber\\
\leq & \sum_k C 2^{-k/2} (\bE + \bA) \leq C (\bE + \bA)\,,\label{e:risomma}
\end{align}
thus completing the proof of \eqref{e:est spine}.

\medskip

{\textit {Step four.}} Towards the proofs of \eqref{e:bf infty}-\eqref{e:bf L2}, we will need to repeat the arguments of Lemmas \ref{l:local_selection} and \ref{l:global_selection} leading to the definition of the selection function $h=h(i,j)$ of Theorem \ref{thm:graph v2}. Let us fix $i \in \{1, \ldots, N_0\}$ and $j \in \{1, \ldots, \kappa_{0,i}\}$, drop the corresponding subscripts, and identify $\bH_0 = \bH_{0,i}$ with $\left[0,\infty\right) \times V$. We then write $u=u_{i,j}$, $v = v_{i,j}$, and $w = u - \tilde{l}$ for the functions of Theorem \ref{thm:graph v2} and Definition \ref{d:new-cone} on $U_{4\mathcal{W}}$. Let also $\{l_h\}_{h=1}^{\kappa_0}$, where $\kappa_0 = \kappa_{0,i}$, be the collection of all linear functions $l_{i,h}$ defined on the page $\bH_0=\bH_{0,i}$ and whose graph parametrizes $\bS$. Let us also fix a cube $Q \in \mathcal W$, and let $\xi_Q \in \bB_{\delta R_Q}((0,y_Q))$ be the corresponding point from \textit{Step three}. Setting
\[
q(z):=z+v(z)=z+u(z)+\Psi(z+u(z)) \qquad \mbox{for all $z \in 4Q$}\,,
\]
we see that for every $z \in 4Q$
\begin{align*}
    \dist^2(q(z)-\xi_Q,\bS) &\geq \min_{1 \leq h \leq \kappa_0} \inf_{\tilde z\in \bH_0} \left\lbrace \abs{z - \p_{\bH_0}(\xi_Q) - \tilde z}^2 + \abs{u(z) - \p_{\bH_0^{\perp_0}}(\xi_Q) - l_h(\tilde z)}^2 \right\rbrace \\
    &= \min_{1 \leq h \leq \kappa_0} \inf_{\tilde z\in \bH_0} \left\lbrace |z - \tilde z|^2 + \abs{u(z) - \p_{\bH_0^{\perp_0}}(\xi_Q) - l_h(\tilde z)  + l_h(\p_{\bH_0}(\xi_Q)) }^2 \right\rbrace\,,
\end{align*}
where the first inequality was obtained by projecting on $\pi_0 = \bH_0 \oplus \bH_0^{\perp_0}$ (where, with a slight abuse of notation, we are identifying $\bH_0$ with the linear space containing it) and using that the distance on the left-hand side is realized by pages in $\bS$ parametrized as graphs $l_h=l_{i,h}$ on $\bH_0=\bH_{0,i}$ due to the choice of $\delta$ in \eqref{the choice of delta}; and where the second identity was obtained by simply replacing $\tilde z$ with $\tilde z + \p_{\bH_0}(\xi_Q)$. Notice that, due to the invariance of $\bS$ (and, therefore, of the corresponding functions $l_h$) with respect to translations along $V$, we may assume without loss of generality that $\p_V(z)=\p_V(\xi_Q)=0$, and thus that also $\p_V(\tilde z)=0$ in the above infimum. It is then a simple exercise in planar geometry to show that
\begin{equation} \label{comparing distances}
    \dist^2(q(z)-\xi_Q,\bS) \geq \frac{1}{2}\,\min_{1\leq h \leq \kappa_0} \abs{u(z) - \p_{\bH_0^{\perp_0}}(\xi_Q) - l_h(z)  + l_h(\p_{\bH_0}(\xi_Q))}^2\,.
\end{equation}

Next, we proceed as in Lemma \ref{l:global_selection}, letting $\bar h$ denote a map $Q \in \mathcal W \mapsto \bar h(Q) \in \{1, \ldots, \kappa_0\}$ which selects, for each $Q \in \mathcal W$, the index $\bar h(Q)$ of an $L^\infty(3Q)$-optimal function in the sense of Lemma \ref{lem.harmonic rigidity} when $u(z)$ is replaced by $u(z) - \p_{\bH_0^{\perp_0}}(\xi_Q)$ and $g_h(z)$ is replaced by $l_h(z) - l_h(\p_{\bH_0}(\xi_Q))$. Setting, for the sake of simplicity, 
\begin{equation} \label{un bell'errore}
\varpi_Q := l_{\bar h(Q)}(\p_{\bH_0}(\xi_Q))\,,    
\end{equation} 
we then obtain the following estimate, similar to \eqref{e:local-selection}:
\begin{align}
&d_Q^m \| u - \p_{\bH_0^{\perp_0}}(\xi_Q) - l_{\bar h(Q)} + \varpi_Q\|^2_{L^\infty (3 Q)} + d_Q^{2+m} \, \|D (u -l_{\bar h(Q)})\|^2_{L^\infty (3 Q)} + d_Q^{3+m} \, [D u]^2_{\frac{1}{2}, 3 Q}  \nonumber\\
\leq & C \int_{4Q} \min_h |u (z) - \p_{\bH_0^{\perp_0}}(\xi_Q) - l_{h} (z) + l_h(\p_{\bH_0}(\xi_Q))|^2\, dz
+ C \bA^2 d_Q^{2+m}\,. \label{e:improved-local-selection}
\end{align}

Combining \eqref{comparing distances} with \eqref{hs at Qnh} then yields
\begin{align}
    & \int_{4Q} \min_h |u (z) - \p_{\bH_0^{\perp_0}}(\xi_Q) - l_{h} (z) + l_h(\p_{\bH_0}(\xi_Q))|^2\, dz\nonumber\\ 
    \leq &C\, d_Q^{m+\frac74} \int_{4Q} \frac{\dist^2(q(z)-\xi_Q,\bS)}{\abs{q(z)-\xi_Q}^{m+\frac74}} \, dz \nonumber \\ \label{verso hs}
    &\leq C\,d_Q^{m+\frac74} \, (\bE + \bA)\,,
\end{align}

so that we achieve, through \eqref{e:improved-local-selection}, the estimate

\begin{equation}\label{local selection + hs}
    \| u - \p_{\bH_0^{\perp_0}}(\xi_Q) - l_{\bar h(Q)} + \varpi_Q\|^2_{L^\infty (3 Q)} + d_Q^{2} \, \|D (u -l_{\bar h(Q)})\|^2_{L^\infty (3 Q)} + d_Q^{3} \, [D u]^2_{\frac{1}{2}, 3 Q}    \leq C d_Q^{\sfrac{7}{4}} (\bE + \bA)\,.
\end{equation}

Next, we proceed \emph{verbatim} as in the proof of Lemma \ref{l:global_selection}. Letting $\hat Q \in \mathcal W$ be any cube which does not have any element above, and setting $\hat h := \bar h(\hat Q)$, for any $Q_0 \in \mathcal W$ the recursive algorithm and estimates from the proof of Lemma \ref{l:global_selection} (see the argument leading to formula \eqref{eq.good selection on a single cube3}) yield the estimate
\begin{align}
    &d_{Q_0}^{-1}\,\| u - \p_{\bH_0^{\perp_0}}(\xi_{Q_0}) - l_{\hat h} + \varpi_{Q_0}\|_{L^\infty (3Q_0)} \nonumber \\
    &\leq d_{Q_0}^{-1}\,\| u - \p_{\bH_0^{\perp_0}}(\xi_{Q_0}) - l_{\bar h(Q_0)} + \varpi_{Q_0} \|_{L^\infty (3Q_0)}\nonumber\\
    &\qquad\qquad+C\, \sum_{s=0}^{\kappa_0-1} d_{\phi(Q_0,s)}^{-1} \|l_{\bar h(\phi(Q_0,s+1))} - l_{\bar h(\phi(Q_0,s))} \|_{L^\infty(\phi(Q_0,s))} \nonumber \\
    & \leq d_{Q_0}^{-1}\,\| u - \p_{\bH_0^{\perp_0}}(\xi_{Q_0}) - l_{\bar h(Q_0)}+ \varpi_{Q_0} \|_{L^\infty (3Q_0)}\nonumber\\
    & \qquad\qquad +C\, \sum_{s=0}^{\kappa_0-1} \|D l_{\bar h(\phi(Q_0,s+1))} - D l_{\bar h(\phi(Q_0,s))}\|_{L^\infty(\phi(Q_0,s))} \nonumber \\
    & \leq d_{Q_0}^{-1}\,\| u - \p_{\bH_0^{\perp_0}}(\xi_{Q_0}) - l_{\bar h(Q_0)}+ \varpi_{Q_0} \|_{L^\infty (3Q_0)}\nonumber\\
    &\qquad\qquad +C\, \sum_{s=0}^{\kappa_0-1} \Big( \|D l_{\bar h((\phi(Q_0,s))^{\top})} - D u\|_{L^\infty(3(\phi(Q_0,s))^{\top})}+ \|D u- D l_{\bar h(\phi(Q_0,s))} \|_{L^\infty(3\phi(Q_0,s))} \Big)\nonumber \,,
\end{align}
so that finally \eqref{local selection + hs} gives
\begin{equation} \label{improved Linfty estimate}
    d_{Q_0}^{-2}\,\| u - \p_{\bH_0^{\perp_0}}(\xi_{Q_0}) - l_{\hat h} + \varpi_{Q_0}\|^2_{L^\infty (3Q_0)} \leq C \, d_{Q_0}^{-\sfrac{1}{4}}\, (\bE+\bA)
\end{equation}

Standard elliptic estimates then imply also that
\begin{equation}
    \|D(u - l_{\hat h}) \|^2_{L^\infty(2Q_0)} + d_{Q_0} \left[ Du \right]^2_{\frac12,2Q_0} \leq C \, d_{Q_0}^{-\sfrac{1}{4}} \, (\bE + \bA) \,. \label{improved c_onealfa estimate}
\end{equation}

We set $h^*=h^*(i,j) := \hat h$, and $l^*_{i,j}:=l_{i,h^*(i,j)}$, and we can proceed to compare the linear $p$-multifunction $\{l^*_{i,j}\}$ with the linear $p$-multifunction $\{\tilde l_{i,j}\}$ introduced in Theorem \ref{thm:graph v2} and corresponding to the selection $h=h(i,j)$. By the triangle inequality, and still dropping the subscripts $_{i,j}$, we have for any $Q \in \mathcal W$ that
\begin{align*}
    d_Q^{m+2}\|D(l^*-\tilde l)\|_{L^\infty(2Q)}^2 & \leq C\, d_Q^{m+2}  \|D (l^* - u) \|_{L^\infty(2Q)}^2 + C \, d_Q^{m+2} \| D(\tilde l - u)\|_{L^\infty(2Q)}^2 \\
    & \overset{\eqref{improved c_onealfa estimate}}{\leq} C\, d_Q^{m+\frac{7}{4}} (\bE + \bA) + C \, d_Q^{m+2} \| D(\tilde l - u)\|_{L^\infty(2Q)}^2\,.
\end{align*}
Using that $\|D (l^* - \tilde l)\|_{L^\infty(2Q)}$ is constant with respect to $Q$, we can sum the above inequality over $Q \in \mathcal W$: using \eqref{e:global_selection_2} together with the fact that cubes $Q \in \mathcal W$ have side length $l_Q = 2^{-k}$ for some positive integer $k$, and that for every $k$ the set $\mathcal C_k$ of cubes $Q$ with side length $l_Q = 2^{-k}$ has cardinality $\sharp(\mathcal C_k) = C\, (2^{-k })^{1-m}$, we obtain
\begin{equation} \label{comparing selections}
    | D (l^* - \tilde l) |^2 \leq C \, (\bE + \bA)\,,
\end{equation}
and thus
\begin{equation} \label{comparing selections due}
    \| l^* - \tilde l\|_{L^\infty(3Q)}^2 \leq C\, d_Q^2 \, (\bE + \bA)\,.
\end{equation}

In particular, the estimates in \eqref{improved Linfty estimate} and \eqref{improved c_onealfa estimate} can be rewritten using the multifunction $\{\tilde l\}$ in place of $\{l^*\}$, that is it holds
\begin{align} \label{improved Linfty old sel}
d_Q^{-2} \, \| u - \p_{\bH_0^{\perp_0}}(\xi_Q) - \tilde l + \varpi_Q\|_{L^\infty(3Q)}^2 &\leq C\, d_Q^{-\sfrac{1}{4}}\, (\bE + \bA)\,, \\ \label{improved c_onealfa old sel}
    \|D(u_{i,j} - \tilde l_{i,j})\|_{L^\infty(2Q)}^2 &\leq C \, d_Q^{-\sfrac{1}{4}} (\bE + \bA)\,.
\end{align}

\medskip

\textit{Step five.} Recall the notation $A_Q$ introduced in \textit{Step three}. We define the binding function $\xi$ on $R_{\mathcal Q}$ by:
\begin{equation}\label{def:the binding function}
    \xi(q) = 
    \begin{cases}
    \p_{V^\perp}(\xi_Q) &  \mbox{if $q \in {\rm int}(A_Q)$ for some $Q \in \mathcal Q$ with $y_Q \in B^{m-1}_{\sfrac{1}{2}}$}\,,\\
    0 & \mbox{elsewhere}\,.
    \end{cases}
\end{equation}

In particular, the $L^\infty$ estimate for the binding function $\xi$ appearing in \eqref{e:bf infty} is immediate from \eqref{hs at Qnh}. Next, we define the $p$-multifunction $\varpi$ on $U_{\mathcal W}$ over $\bS_0$ as follows: for every $i \in \{1,\ldots, N \}$ and $j \in \{1,\ldots,\kappa_{0,i}\}$, we let $\varpi_{i,j}$ be given on $(U_\mathcal W)_i$ by
\begin{equation}\label{def:the multierror}
    \varpi_{i,j}(z) := 
    \begin{cases}
    \varpi_{i,j,Q} &  \mbox{if $z=(x,y)$ and $(|x|,y) \in {\rm int}(Q)$ for some $Q \in \mathcal W$ with $y_Q \in B^{m-1}_{\sfrac{1}{2}}$}\,,\\
    0 & \mbox{elsewhere}\,.
    \end{cases}
\end{equation}
where $\varpi_{i,j,Q}$ is the constant $\varpi_Q$ defined in \eqref{un bell'errore}. In particular, if $\zeta=(|x|,y) \in {\rm int}(Q)$ for some $Q \in \mathcal W$, then for $z = (x,y)$ it holds, thanks to \eqref{hs at Qnh},
\begin{equation} \label{estimate on the multierror}
    |\varpi_{i,j}(z)|^2 \leq C \, \|\nabla l_{\bar h(i,j,Q)}\|_\infty^2 \, (\bE + \bA)\,,
\end{equation}
which implies \eqref{bf:multierror}. Finally, we prove \eqref{e:bf L2}. Recalling that $w=u-\tilde l$, using the definitions of $\xi$ and $\varpi$ as in \eqref{def:the binding function} and \eqref{def:the multierror}, and taking into account that, for every $Q \in \mathcal W$, the diameter $d_Q$ is comparable to the distance from the spine, we can use \eqref{improved Linfty old sel} to estimate
\begin{align*}
    &\sum_{i=1}^N \sum_{j=1}^{\kappa_{0,i}} \int_{(U_\mathcal W)_i \cap \bB_{\sfrac12}} \frac{\abs{w_{i,j} - \varpi_{i,j} - \mathbf{p}_{\bH_{0,i}^{\perp_0}}(\xi)}^2}{|x|^{5/2}}\, dz \leq C\, \sum_{Q \in \mathcal W} d_Q^{m-\frac34}\, (\bE + \bA) \\ 
    & \qquad \qquad \qquad \leq C\,  \sum_{k \geq 0} \sum_{Q \in \mathcal C_k} d_Q^{m-\frac34}  \, (\bE + \bA) \leq C \, \sum_{k \geq 0} \left( \frac{1}{2^{\sfrac{1}{4}}} \right)^k \, (\bE + \bA)\,.
\end{align*}
This proves the first part of \eqref{e:bf L2}; the proof of the second part is analogous, using \eqref{improved c_onealfa old sel} in place of \eqref{improved Linfty old sel}. \hfill\qedsymbol

%% file: scoppiato.tex
\section{Blow up}\label{sec:bu}

In this section we consider ``blow-up'' sequences.

\begin{definition}\label{d:scoppia} A blow-up sequence is given by
\begin{itemize}
\item[(a)] submanifolds $\Sigma_k$ as in Assumption \ref{assumption:manifolds and currents} with $T_0 \Sigma_k = \pi_0 = \{0_{n-1}\}\times \mathbb R^{m+1}$; 
\item[(b)] a sequence of currents $T_k$ in $\mathscr{R}_m (\Sigma_k)$ which are area-minimizing $\modp$ in $\Sigma_k\cap \bB_{2R_0}$;
\item[(c)] a sequence of cones $\bC_k$ supported in $\pi_0$; 
\end{itemize}
such that
\begin{itemize}
    \item[(i)] $\Theta_{T_k} (0)\geq\frac{p}{2}$ for every $k$;
    \item[(ii)] the cones $\bC_k$ have the same $(m-1)$-dimensional spine $V = \{0_{n-1}\}\times \{0_2\}\times \mathbb R^{m-1}$;
    \item[(iii)] $\bC_k$ converge in the flat topology to an area-minimizing cone $\bC_0$ with spine $V$;
    \item[(iv)] the currents $T_k$ converge, with respect to $\hat\flat^p_{\bB_{R_0}}$, to $\bC_0$;
    \item[(v)] upon denoting by $\bS_k$ and $\bS_0$ the books corresponding to $\bC_k$ and $\bC_0$, by $\bE_k$ the excesses $\bE (T_k, \bS_k, 0, R_0)$, and by $\bA_k$ the quantities $\|A_{\Sigma_k}\|_{L^\infty}$, we have
    \begin{align}\label{e:blow-up-conditions}
    \bE_k \to 0 \qquad \mbox{and} \qquad \frac{\bA_k}{\bE_k} \to 0\, ,    
    \end{align}
    (where we implicitly assume $\bE_k>0$).
\end{itemize}
\end{definition}

Having fixed the constant $\beta_1$ of Theorem \ref{thm:est spine}, we let $\beta = \frac{\beta_1}{2}$ and assume, without loss of generality, that each pair $(T, \bC) = (T_k, \bC_k)$ satisfies the assumptions of Theorem \ref{thm:est spine}. In particular we denote:
\begin{itemize}
    \item[($\alpha$)] by $w^k$ the corresponding $p$-multifunctions $w$ over $\bS_0$ and by $U^k:= U_{4\mathcal{W}^k}$ their domains (here, $\mathcal{W}^k = \mathcal{W}(T_k,\bS_k, \tau)$ with $\tau$ depending only on $(m,n,p,\bS_0)$);
    \item[($\beta$)] by $h^k$ the selection functions $h$ from Theorem \ref{thm:graph v2};
    \item[($\gamma$)] by $l^k$ the corresponding linear $p$-multifuctions $\tilde l$;
    \item[($\delta$)] by $\xi^k$ the corresponding binding functions $\xi$;
    \item[($\varepsilon$)] by $\varpi^k$ the corresponding $p$-multifunctions $\varpi$.
\end{itemize}
Observe that under \eqref{e:blow-up-conditions}, by Theorem \ref{thm:graph_v1}(i), the domains $U^k$ ``close around'' the spine, in the sense that, for any fixed cube $Q\in \mathcal{Q}$, $Q\subset U^k$ provided $k$ is large enough. For further reference, we let $U^\infty$ be the union of all $2Q$ for $Q$'s in $\mathcal{Q}$. The following is then an easy corollary of Theorem \ref{thm:est spine}, whose proof is left to the reader.  

\begin{corollary}\label{c:scoppia}
Consider a blow-up sequence $(T_k, \bC_k)$ as in Definition \ref{d:scoppia}, set $\beta = \frac{\beta_1}{2}$ for $\beta_1$ as in Theorem \ref{thm:est spine}, consider $w^k, U^k, h^k, l^k$, $\xi^k$, and $\varpi^k$ as in ($\alpha$)-($\varepsilon)$, and set $\bar w^k:= \bE_k^{-\sfrac{1}{2}} w^k$, $\bar\xi^k:= \bE_k^{-\sfrac{1}{2}} \xi^k$, and $\bar\varpi^k := \bE_k^{-\sfrac{1}{2}} \varpi^k$. Up to subsequences, the following holds:
\begin{itemize}
    \item[(i)] $h^k (i,j)$ is constant for every $i$ and $j$;
    \item[(ii)] $\bar w^k$ converges locally in $C^1$ to a $p$-multifunction $\bar w$ on $U^\infty \cap \{\abs{\zeta}< \sfrac{1}{2}\}$ over $\bS_0$ taking values in $\pi_0$; 
    \item[(iii)] $\bar \xi^k$ converges locally uniformly to a binding function $\bar \xi$ defined on $U^\infty \cap \{\abs{\zeta}< \sfrac{1}{2}\}$, whereas $\bar\varpi^k$ converges locally uniformly to zero;
    \item[(iv)] The following estimates hold (for a geometric constant $C$ which depends only on $p$, $m$ and $n$): 
    \begin{align} \label{bu:c_onealpha}
        &\sup_{\zeta = (t,y) \in U^\infty} t^{\frac{m}{2}+1}\left( t^{-1} \abs{\bar w (\zeta)} + \abs{D\bar w (\zeta)}+ t^{\sfrac12} [D\bar w]_{\sfrac12}(\zeta) \right)\le C\\ \label{bu:radial}
        &\sum_{i,j} \int_{\bB_{1/2}\cap U^\infty_i} |z|^{2-m}\left|\partial_r \frac{\bar w_{i,j} (z)}{|z|}\right|^2\, dz \leq C\\ \label{bu:nonconcentration w12}
        &\sum_{i,j}\int_{\bB_{1/2}\cap U^\infty_i} \frac{|\bar w_{i,j} - \mathbf{p}_{\mathbf{H}_{0,i}^{\perp_0}} (\bar \xi)|^2}{|x|^{\frac{5}{2}}} +\frac{|\nabla \bar w_{i,j}|^2}{|x|^{\frac{1}{2}}}\, dz \leq C\, .
    \end{align}
\end{itemize}
\end{corollary}

\begin{remark}\label{r:local}
Observe that the open sets $U^\infty_i\subset \bH_{0,i}$ do not include any portion of the spine $V$, rather $\bB_{1/2}\cap V$ is contained in the boundary of each $U^\infty_i$. When we refer to ``local properties'', we understand them as taking place in compact subsets of the domain, i.e. ``away from $V$''. In particular, precise formulations of points (ii) and (iii) in Corollary \ref{c:scoppia} are the following:
\begin{itemize}
    \item[(ii)] for every $Q\in \mathcal{Q}$ and for every $i$, upon ordering the sheets $\bar w^k_{i,j}$, and $\bar w_{i,j}$ monotonically, each $\bar w^k_{i,j}$ (which is defined on $2Q$ for $k$ large enough), converges in $C^1 (2Q)$ to $\bar w_{i,j}$; 
    \item[(iii)] for every $Q\in \mathcal{Q}$ the constant values taken by the $\bar \xi^k$ on ${\rm int}(A_Q)$ converge to the constant value taken by $\bar \xi$ (and $\bar\varpi^k$ converge to zero on $Q$).
    \end{itemize}
\end{remark}

The main point of this section is to show that the convergence of $\bar w^k$ is strong in $L^2$ and to collect some relevant properties of the pair of functions $\bar w$ and $\bar \xi$ in Corollary \ref{c:scoppia}. One crucial property is \eqref{e:Simon-test} below, which is valid for {\em cylindrical} vector fields.

\begin{definition}\label{d:campi-cilindrici}
Let $V$ be a linear subspace of $\mathbb R^{m+n}$. A vector field $W: \mathbb R^{m+n}\to \mathbb R^{m+n}$ is called {\em cylindrical with respect to $V$} if $W (q)= W (\bar q)$ for any pair of points $q, \bar{q}$ such that $\mathbf{p}_V (q) = \mathbf{p}_V (\bar q)$ and $\dist (q, V) = \dist (\bar q, V)$.
\end{definition}

\begin{proposition}\label{p:scoppia}
Let $T_k,\bar w^k, \bar \xi^k, U^k$, $\bar w, \bar \xi$, and $U^\infty$ be as in Corollary \ref{c:scoppia}. Then:
\begin{itemize}
    \item[(i)] The converge of $\bar w_k$ to $\bar w$ is strong in the sense that, for $|x| (q) = \dist (q,V)$,
    \begin{equation}\label{e:strong-L2}
    \int_{U^\infty \cap \bB_{1/2}} (|\bar w|^2 + |x|^2 |D\bar w|^2)\, d\|\bC_0\| =
    \lim_{k\to\infty} \int_{U^k \cap \bB_{1/2}} (|\bar w^k|^2 + |x|^2 |D\bar w^k|^2)\, d\|\bC_0\| < \infty\, .
    \end{equation}
    \item[(ii)] The following estimate holds (for, we recall, $\bS_k$ the open book $\spt(\bC_k)$)
    \begin{align}
        &\limsup_{k\to\infty} \frac{1}{\bE_k} \int_{\bB_{1/2}} \dist (q, \bS_k)^2\, d\|T_k\|\leq \int_{\bB_{1/2}\cap U^\infty} |\bar w|^2\, d\|\bC_0\|\, .
    \end{align}
    \item[(iii)] $\bar w_{i,j}$ is (locally) smooth in $U^\infty_i \cap \bB_{1/2}$ and $\Delta \bar w_{i,j} =0$, for every $i$ and $j$;
    \item[(iv)] for any $W\in C^\infty_c (\bB_{1/2}, \mathbb R^{m+n})$ cylindrical with respect to $V$ we have
    \begin{equation}\label{e:Simon-test}
    \sum_i \int_{\bH_{0,i}} \sum_j \nabla \bar w_{i,j} : \nabla \frac{\partial W}{\partial y_l}\, dz = 0\, \qquad\forall\, l=1, \ldots , m-1\, .  
    \end{equation}
\end{itemize}
\end{proposition}

\begin{remark}
In \eqref{e:Simon-test} each map $\bar w_{i,j}$, which takes values in $\bH_{0,i}^{\perp_0}$, is regarded as a map taking values in $\mathbb R^{m+n}$, while $W$ is restricted on $\bH_{0,i}$ and regarded thus as a map from $\bH_{0,i}$ to $\mathbb R^{m+n}$.  The corresponding product in \eqref{e:Simon-test} is thus understood as the usual Hilbert-Schmidt product of the Jacobian matrices $\nabla \bar w_{i,j}$ and $\nabla \frac{\partial W}{\partial y_l}$, where in both cases $\nabla$ denotes the differential with respect to the variables in $\bH_{0,i}$. In a few computations we will use the notation $D W$ for the {\em full Jacobian matrix} of $W$, i.e. when the derivatives are taken with respect to all variables. Observe however that, because of the special symmetry assumption on $W$, $\partial_v W (q) =0$ for every $q\in \bH_{0,i}$ and any $v\in \bH_{0,i}^\perp$.  
\end{remark}

\begin{proof}{\bf Proof of (i).} By the uniform convergence on compact subsets of $\bB_{1/2} \cap U^\infty_i$ for all $i$, it suffices to show that there is no ``concentration'' at the spine. To that end, consider a positive radius $r$ and estimate, using the results of Theorem \ref{thm:est spine},
\begin{align*}
& \int_{\bB_{1/2}\cap B_r (V)\cap U^k} (|\bar w^k|^2 + |x|^2 |D\bar w^k|^2)\, d\|\bC_0\|\\
\leq &\bE_k^{-1} \int_{\bB_{1/2}\cap B_r (V)\cap U^k} (|w^k - \varpi^k - \mathbf{p}_{\bH_{0,i}^{\perp_0}} (\xi^k)|^2 + |x|^2 |D w^k|^2)\, d\|\bC_0\|\\
&\qquad\qquad+ C r \bE_k^{-1} \left(\|\xi^k\|_\infty^2 + \|\varpi^k\|_\infty^2 \right)\\
\leq & \bE_k^{-1} r^{5/2}  \int_{\bB_{1/2}\cap B_r (V)\cap U^k} \frac{|w^k - \varpi^k - \mathbf{p}_{\bH_{0,i}^{\perp_0}} (\xi^k)|^2}{|x|^{5/2}} + \frac{|D w^k|^2}{|x|^{1/2}}\, d\|\bC_0\|\\
&\qquad\qquad + C r \bE_k^{-1} \left(\|\xi^k\|_\infty^2 + \|\varpi^k\|_\infty^2 \right) \\ 
\leq &C \left(r^{5/2} + r + r\,\hat\flat^p_{\bB_{R_0}}(\bC_k-\bC_0)^2\right) \left(1+ \bE_k^{-1} \bA_k\right)\, .
\end{align*}
Recalling that both $\bE_k^{-1} \bA_k$ and $\hat\flat^p_{\bB_{R_0}}(\bC_k-\bC_0)$ are infinitesimal, we conclude
\[
\limsup_{k\to\infty} \int_{\bB_{1/2}\cap B_r (V)\cap U^k} (|\bar w^k|^2 + |x|^2 |D\bar w^k|^2)\, d\|\bC_0\| \leq C r\, .
\]

\medskip

{\bf Proof of (ii).} First of all observe that $q= z+ l^k_{i,j} (z)+w^k_{i,j} (z) + \Psi^k(z+ l^k_{i,j} (z)+w^k_{i,j} (z))\in \spt (T_k)$ for every choice of $k,i,j$ and every $z\in U^k_i$, while $z+ l^k_{i,j} (z)\in \spt (\tilde{\bC}_k) = \tilde{\bS}_k \subset \bS_k$. We thus have $\dist (q, \bS_k) \leq |w^k (z)| + C\,\bA_k$. Moreover, the support of the current $T_k$ coincides with the graph of the multifunction $u^k + \Psi^k(\cdot + u^k)$ on $\bB_1\cap \{\dist (q, V)\geq \sigma_k\}$ for some infinitesimal sequence $\sigma_k$. Therefore we can write
\begin{align*}
&\limsup_{k\to\infty} \bE_k^{-1} \int_{\bB_{1/2}\setminus B_r (V)} \dist (q, \bS_k)^2\, d\|T_k\| (q)\\
\leq &\limsup_{k\to\infty} \left\lbrace (1+ C \Lip (u^k) + C\bA_k) \int_{\bB_{1/2}\setminus B_r (V)} \bE_k^{-1} |w^k (z)|^2\, d\|\bC_0\| (z) + Cr\bE_k^{-1}\bA_k \right\rbrace\, ,
\end{align*}
where we have used the area formula to estimate the area element on the graphical parametrization of the current induced by the graph of $u^k + \Psi^k(\cdot + u^k)$ with $1+ C \Lip (u^k) + C\bA_k$. 
Observe now that the Lipschitz constant of $u^k$ converges to $0$ on any compact set in $\bB_1\setminus V$ and we can thus conclude
\begin{align*}
\limsup_{k\to\infty} \bE_k^{-1} \int_{\bB_{1/2}\setminus B_r (V)} \dist (q, \bS_k)^2\, d\|T_k\| (q)
\leq &\limsup_{k\to\infty} \int_{\bB_{1/2}\setminus B_r (V)} |\bar w^k (z)|^2\, d\|\bC_0\| (z)\\
\leq & \int_{\bB_{1/2}} |\bar w|^2\, d\|\bC_0\|\, .
\end{align*}
In order to prove (ii) we then need to show the nonconcentration estimate
\[
\lim_{r\downarrow 0} \limsup_{k\to\infty} \bE_k^{-1} \int_{\bB_{1/2}\cap B_r (V)} \dist (q, \bS_k)^2\, d\|T_k\| (q) = 0\, .
\]
Fix $r$ and assume that $k$ is large enough so that $r>\varrho^k_\infty$ (recall the definition of the latter is given in \eqref{e:def_rho_inf}). We can then use \eqref{e:est spine} in Theorem \ref{thm:est spine} to bound
\begin{align*}
&\bE_k^{-1} \int_{\bB_{1/2}\cap B_r (V)} \dist (q, \bS_k)^2\, d\|T_k\| (q)\\
\leq &\bE_k^{-1} r^{1/2} \int_{\bB_{1/2}} \frac{\dist (q, \bS_k)^2}{\max\{\varrho_\infty^k,|x|\}^{1/2}}\, d\|T_k\| (q)
\leq C (1+\bE_k^{-1} \bA_k) r^{1/2}\, .
\end{align*}

\medskip

{\bf Proof of (iii).} We consider a cube $Q\in U^\infty$ and recall that, on $4Q$, $\bar w^k$ converges to $\bar w$ in $C^1$. Consider a single sheet $u^k_{i,j} = w^k_{i,j}+ l^k_{i,j}$, and recall that it is a critical point of the area functional with respect to the metric $({\rm Id}+\Psi^k)^\sharp\delta_{\R^{m+n}}$ on $\pi_0$; see formula \eqref{e:funzionale}. In particular, $u^k_{i,j}$ is a solution to the corresponding Euler-Lagrange equation, which we can rewrite as 
\[
{\rm div} \left(\frac{\nabla u^k_{i,j} (z)}{\sqrt{1+|\nabla u^k_{i,j} (z)|^2}}\right)
= {\rm div}\, \underbrace{(-R_k (z, u^k_{i,j} (z), \nabla u^k_{i,j} (z))}_{=:f_k (z)} + \underbrace{S_k (z, u^k_{i,j} (z), \nabla u^k_{i,j} (z))}_{=:g_k (z)}\, ,
\]
where the functions $S_k$ and $R_k$ satisfy the bounds $|R_k (z,\bar u,\bar p)| + |S_k (z,\bar u,\bar p)|\leq C \bA_k (1+|\bar u|+|\bar p|)$. On the other hand we have
\[
{\rm div} \left(\frac{\nabla  l^k_{i,j} (z)}{\sqrt{1+|\nabla l^k_{i,j} (z)|^2}}\right)=0\, .
\]
Subtract the two equations, divide by $\bE_k^{-1/2}$ and consider that $c_k:= \sqrt{1+|\nabla l^k_{i,j} (z)|^2}$ is a constant. We can then write 
\begin{align}\label{e:ugly-expansion}
\Delta \bar w^k_{i,j} (z)
= c_k {\rm div} (\bE_k^{-1/2} f_k) + c_k \bE_k^{-1/2} g_k +
{\rm div}\, \Bigg(\bE_k^{-1/2}\underbrace{ 
\Bigg(\frac{c_k}{\sqrt{1+|\nabla u^k_{i,j} (z)|^2}}-1\Bigg)}_{=: h_k}\nabla u^k_{i,j} \Bigg)
\end{align}
Let next $k\to\infty$: clearly the left hand side converges to $\Delta \bar w$ in the sense of distributions. On the other hand we have the estimates $\|f_k\|_{C^0} + \|g_k\|_{C^0}\leq C\bA_k$ and (since $c_k\to 1$) the first two summands in the right hand side converge (distributionally) to $0$. 
We next estimate 
\[
|h_k (z)|\leq C |\nabla u^k_{i,j} (z)-\nabla l^k_{i,j} (z)|= C |\nabla w^k_{i,j} (z)|\leq C
|x|^{-\frac{m}{2}-1} (\bE_k^{1/2}+\bA_k)\, . 
\]
Considering however that we are taking $z\in 4Q_i$, we can estimate
\[
\bE_k^{-1/2} \|h_k\|_{L^\infty (4 Q_i)} \leq C d_Q^{-\frac{m}{2}-1}\, .
\]
Since $\nabla u^k_{i,j}$ converges uniformly to $0$ on $4Q_i$, we conclude that the third summand in \eqref{e:ugly-expansion} converges to $0$ as well. 

\medskip

{\bf Proof of (iv).} Fix $W$ as in the claim. We first observe that each map $w^k_{i,j}$ takes values in the linear subspace $V^{\perp_0}$. We can therefore assume, without loss of generality, that $W$ takes values in $V^{\perp_0}$ as well. 

Fix next $r>0$ and consider a cut-off function $\phi_r$ which is identically equal to $0$ in a neighborhood of $0$, equals $1$ on $[r, \infty)$ and satisfies the bound $\|\phi_r'\|_0\leq C r^{-1}$. Consider then the vector field
\[
W_r (\bar{x}, y) = W(\bar{x}, y) \phi_r (|\bar x|) + W (0, y) (1-\phi_r (|\bar x|))\, .
\]
Obviously $W_r$ depends only on $y$ in a neighborhood of $V$, while we have the estimate 
\[
\left|\nabla \frac{\partial (W-W_r)}{\partial y_l} (\bar x,y)\right|\leq C\|D^2 W\|_0 \qquad \forall\,l=1,\ldots,m-1\,.
\]
Using that $W-W_r = 0$ outside of $B_r(V)$, it thus follows easily that
\begin{align*}
\left|\int_{\bH_{0,i}} \sum_j \nabla \bar w_{i,j} \colon \nabla \left(\frac{\partial (W-W_r)}{\partial y_l}\right)\right|
&\leq C \|D^2 W\|_0 \sum_j \int_{\bH_{0,i}\cap \bB_{1/2}\cap B_r (V)} |\nabla \bar w_{i,j}| \\
& \leq C \|D^2 W\|_0 \, r^{\sfrac{5}{4}} \, \sum_j \left(\int_{\bH_{0,i} \cap \bB_{1/2}} \frac{\abs{\nabla \bar w_{i,j}}^2}{|x|^{\frac12}} \, dz \right)^{\sfrac{1}{2}}\,.
\end{align*}
Since the left hand side converges to $0$ as $r\downarrow 0$ by \eqref{bu:nonconcentration w12}, and since the vector fields $W_r$ are still cylindrical, it suffices to prove (iv) for a cylindrical vector field $W$ which in addition depends only upon the variable $y$ in some neighborhood of $V$. 

We then fix the index $l \in \{1,\ldots,m-1\}$, set $\bar W:= \frac{\partial W}{\partial y_l}$ and, summarizing the above discussion, without loss of generality we assume:
\begin{itemize}
    \item[(S)] $\bar W$ is cylindrical, it depends only on the variable $y\in V$ in $B_{r_0} (V)$, and it takes values on $V^{\perp_0}$ (everywhere).
\end{itemize}
Next consider that:
\begin{itemize}
    \item Since $\bC_k$ is invariant in the direction $y_l$ and $\bar W$ is a derivative along that direction (of a compactly supported smooth vector field) then $\delta \bC_k (\bar W) =0$;
    \item Since $T_k$ is area minimizing $\modp$ in $\Sigma_k$, we have $|\delta T_k (\bar W)|\leq \bA_k \|\bar W\|_0$.
\end{itemize}
In particular 
\begin{equation}\label{e:variation_vanishes}
\lim_{k\to\infty} \bE_k^{-\frac{1}{2}} (\delta T_k (\bar W) - \delta \bC_k (\bar W)) = 0\, .
\end{equation}
Next consider an $r< r_0$ and a $k$ large enough so that $T_k$ is the graph of the multifunction $v^k = u^k + \Psi^k(\cdot + u^k)$ outside $B_r (V)$. We split both currents $T_k$ and $\bC_k$ into two pieces: 
\begin{itemize}
    \item $T_k^{\rm g}$ is the graph of the multifuction $v^k$ over $\bS_0$ on $U_r:= \{\zeta=(t,y) \, \colon \, t> r\}$, while $T_k^{\rm r}$ is the remainder in $\bB_{1/2}$ (i.e. $(T_k - T_k^{\rm g})\res \bB_{1/2}$);
    \item likewise $\bC_k^{\rm g}$ is the graph of the multifunction $l^k$ over $\bS_0$ on $U_r$ and $\bC_k^{\rm r}:= (\bC_k - \bC_k^{\rm g})\res \bB_{1/2}$.
\end{itemize}
We denote by $U_{r,i}$ the sets $(U_r)_i\subset  \bH_{0,i}$ and make the following claims:
\begin{align}
&\limsup_{k\to \infty} \bE_k^{-1/2} \left|\int {\rm div}_{T_k} \bar W\, d\|T_k^{\rm r}\| -
\int {\rm div}_{\bC_k} \bar W\, d\|\bC_k^{\rm r}\|\right| \leq C r^{1/2}\label{e:pezzo-dentro}\\
&\lim_{k\to\infty} \bE_k^{-1/2} \left(\int {\rm div}_{T_k} \bar W\, d\|T_k^{\rm g}\| -
\int {\rm div}_{\bC_k} \bar W\, d\|\bC_k^{\rm g}\|\right) = \sum_i \int_{U_{r,i}} \sum_j \nabla \bar{w}_{i,j} \colon \nabla \bar W\label{e:pezzo-fuori}\, .
\end{align}
Since obviously
\begin{align*}
\delta T_k (\bar W) &= \int {\rm div}_{T_k} \bar W\, d\|T_k^{\rm r}\| +
\int {\rm div}_{T_k} \bar W\, d\|T_k^{\rm g}\|\\
\delta \bC_k (\bar W) &= \int {\rm div}_{\bC_k} \bar W\, d\|\bC_k^{\rm r}\| +\int {\rm div}_{\bC_k} \bar W\, d\|\bC_k^{\rm g}\|\, , 
\end{align*}
the combination of \eqref{e:pezzo-dentro}, \eqref{e:pezzo-fuori}, and \eqref{e:variation_vanishes} implies
\[
\left|\sum_i \int_{U_{r,i}} \sum_j \nabla \bar{w}_{i,j} \colon \nabla \bar W\right| \leq C r^{1/2}\, .
\]
Thus the desired conclusion follows from letting $r\downarrow 0$. 

We now come to the proof of \eqref{e:pezzo-dentro} and \eqref{e:pezzo-fuori}. Concerning \eqref{e:pezzo-dentro}, observe that (S) above implies that
\[
{\rm div}_\pi \bar W =0 \qquad \mbox{in $B_{r_0}(V)$}
\]
for every vector space $\pi$ which contains $V$ and, in particular,
\begin{equation}\label{e:piano-tilta}
|{\rm div}_\pi \bar W|\leq C \|D \bar W\|\left|\mathbf{p}_V\cdot\mathbf{p}_{\pi^\perp}\right|\, ,
\end{equation}
for any arbitrary vector space $\pi$.

Since every tangent plane to $\bC_k$ contains $V$, and since $\bC_k^{\rm r}$ is supported in $B_{r_0}(V)$, we conclude
\[
\int {\rm div}_{\bC_k} \bar W\, d\|\bC_k^{\rm r}\| = 0\, ,
\]
while using \eqref{e:piano-tilta} and the Cauchy-Schwartz inequality we get
\[
\left|\int {\rm div}_{T_k} \bar W\, d\|T_k^{\rm r}\|\right|
\leq C \left(\|T_k\| (\bB_{1/2}\cap B_r (V))\right)^{1/2} 
\left(\int_{\bB_{1/2}} \left|\mathbf{p}_{V}\cdot \mathbf{p}_{\vec{T}_k (q)^\perp}\right|^2d\|T_k\| (q)\right)^{1/2}\, .
\]
The proof of \eqref{e:pezzo-dentro} is then complete once we apply estimate \eqref{e.reverse poincare} in Proposition \ref{prop:density-est} to reach
\[
\left|\int {\rm div}_{T_k} \bar W\, d\|T_k^{\rm r}\|\right| \leq C \bE_k^{1/2} \left(\|T_k\| (\bB_{1/2}\cap B_r (V))\right)^{1/2} \leq C \bE_k^{1/2} r^{1/2}\, .
\]

In order to show \eqref{e:pezzo-fuori} we will decompose $\bC_k^{\rm g}$ and $T_k^{\rm g}$ in the union of the graphs of $l^k_{i,j}$ and of $v^k_{i,j} = l^k_{i,j} + w^k_{i,j} + \Psi^k(\cdot + l^k_{i,j} + w^k_{i,j})$ and  over the corresponding domain $U_{r,i}$. To that end, 
\begin{align*}
\int {\rm div}_{T_k} \bar W\, d\|T_k^{\rm g}\| &= \sum_{i,j} \underbrace{\int_{\mathbf{p}_{\bH_{0,i}}^{-1} (U_{r,i})}
{\rm div}_{\bG_{\bH_{0,i}} (v^k_{i,j})} \bar W\, d\|\bG_{\bH_{0,i}} (v^k_{i,j})\|}_{=: I^{(1)}_{k,i,j}}\, \\
\int {\rm div}_{\bC_k} \bar W\, d\|\bC_k^{\rm g}\| &=\sum_{i,j} \underbrace{\int_{\mathbf{p}_{\bH_{0,i}}^{-1} (U_{r,i})}
{\rm div}_{\bG_{\bH_{0,i}} (l^k_{i,j})} \bar W\, d\|\bG_{\bH_{0,i}} (l^k_{i,j})\|}_{=:I^{(2)}_{k,i,j}}\, .
\end{align*}
Our task is then accomplished once we show that, for every $i$ and $j$,
\begin{equation}\label{e:pezzo-fuori-i,j}
\lim_{k\to\infty} \bE_k^{-1/2} (I^{(1)}_{k,i,j} - I^{(2)}_{k,i,j})
= \int_{U_{r,i}} \nabla \bar w_{i,j} \colon \nabla \bar W\, .
\end{equation}
From now on we fix $i$ and $j$ and, in order to simplify our notation, we drop both of them. Next, fix an orthonormal frame $e_1, \ldots , e_{m+n}$ such that $e_1, \ldots , e_m$ is a basis of $\bH_{0,i}$ and define the $(m+n)\times m$ matrices $A(k) = (A_1 (k), \ldots , A_m(k))$ and $B(k) = (B_1 (k), \ldots , B_m (k))$, where $A_\alpha (k)$ and $B_\alpha (k)$ are the following vectors in $\mathbb R^{m+n}$:
\begin{align*}
A_{\alpha}(k) &:= e_\alpha +\partial_\alpha v^k = e_\alpha + \partial_\alpha l^k + \partial_\alpha w^k + \psi^k_\alpha\\
B_{\alpha}(k) &:= e_\alpha + \partial_\alpha l^k\, ,
\end{align*}
where $\abs{\psi^k_\alpha} \leq C\bA$. Furthermore, given any matrix $A= (A_1, \ldots , A_m)\in \mathbb R^{(m+n)\times m}$ we let 
$M (A)\in \mathbb R^{(m+n)\times (m+n)}$ be the matrix
\[
M (A) = \sum_{\alpha,\beta} \sqrt{\det A^TA} (A^TA)^{-1}_{\alpha\beta} A_\alpha\otimes A_\beta\, .
\]
We can then apply the well known formula for the variation of the area functional, leading to
\begin{equation}\label{e:espressione}
I^{(1)} (k) - I^{(2)} (k) = \int_{U_{r,i}} \big(M (A (k)) - M (B (k))\big): D \bar{W}\, .   
\end{equation}
We next compute
\begin{align*}
&A_\alpha (k)\otimes A_\beta (k) - B_\alpha (k)\otimes B_\beta (k)- (\partial_\alpha w^k \otimes e_\beta + e_\alpha \otimes \partial_\beta w^k) \\&= \partial_\alpha w^k \otimes \partial_\beta l^k + \partial_\alpha l^k \otimes \partial_\beta w^k	 + \partial_\alpha w^k \otimes \partial_\beta w^k + \psi^k_\alpha \otimes A_\beta(k) + A_\alpha(k)\otimes\psi^k_\beta\, .
\end{align*}
Recall next that over the domain $U_{r,i}$ we have the estimate
\[
\|\nabla w^k\|_{L^\infty} \leq C (\bE_k+\bA_k)^{1/2} 
\]
while $\|\nabla l^k\|= {\rm o} (1)$. Furthermore, $\abs{A_\alpha} = {\rm O}(1)$, whereas $\abs{\psi_\alpha} = {\rm o}(\bE_k^{1/2})$ due to \eqref{e:blow-up-conditions}. We therefore conclude
\[
A_\alpha (k)\otimes A_\beta (k) - B_\alpha (k)\otimes B_\beta (k)= \partial_\alpha w^k \otimes e_\beta + e_\alpha \otimes \partial_\beta w^k + {\rm o} (\bE_k^{1/2})\, .
\]
Similarly
\begin{align*}
(A(k)^T A(k))_{\alpha\beta} &=\delta_{\alpha\beta} + \partial_\alpha l^k\cdot \partial_\beta l^k +{\rm o} (\bE_k^{1/2}) = (B (k)^T B (k))_{\alpha\beta} + {\rm o} (\bE_k^{1/2})\\
(B(k)^T B (k))_{\alpha\beta} &= \delta_{\alpha\beta} + {\rm o} (1)\, .
\end{align*}
We then conclude that
\begin{equation}\label{e:espansione-matriciazze}
M (A(k)) - M (B(k)) = \sum_\alpha (\partial_\alpha w^k\otimes e_\alpha + e_\alpha\otimes \partial_\alpha w^k) + {\rm o} (\bE_k^{1/2})\, .
\end{equation}
Recall next that, because of the special structure of $W$, $\partial_v W\equiv 0$ on $\bH_{0,i}$ whenever $v\in \bH_{0,i}^\perp$. Therefore, since $\partial_\alpha w^k\in \bH_{0,i}^\perp$, we conclude
\begin{align*}
(M(A(k) - M (B(k)): D \bar W &= \sum_\alpha (\partial_\alpha w^k\otimes e_\alpha) : D \bar W + {\rm o} (\bE_k^{1/2})\\
&= \nabla w^k: \nabla \bar W + {\rm o} (\bE_k^{1/2}) = \bE_k^{1/2} \nabla \bar w^k: \nabla \bar W + {\rm o} (\bE_k^{1/2})\, .
\end{align*}
Combined with \eqref{e:espressione}, the latter estimate gives \eqref{e:pezzo-fuori-i,j}. 
\end{proof}

\section{Decay for the linear problem}\label{sec:harm}

The aim of this section is to prove the fundamental integral decay property of the blow-up map $\bar w$ of Corollary \ref{c:scoppia}. 

\begin{proposition}\label{p:decad-lineare} There is a constant $C$ (which depends only on $m$) with the following property.
Let $\bar w$ be as in Corollary \ref{c:scoppia}. Then there are a linear map $b: V\to V^{\perp_0}$ and a linear $p$-multifunction $a=\{a_{i,j}\}$ over $\bS_0$ (taking also values in $\pi_0)$ such that $\|a\|_{L^\infty (\bS_0\cap \bB_1)} + \|b\|_{L^\infty (\bS_0\cap \bB_1)} \leq C$ and the following holds for all $0<\rho<r<\frac{1}{2}$:
\begin{equation}\label{e:decad-lineare}
\sum_i \int_{\bH_{0,i}\cap \bB_\rho} \sum_j \big|\bar{w}_{i,j} (x,y) - a_{i,j} (x) - \mathbf{p}_{\bH_{0,i}^{\perp_0}} (b (y))\big|^2 \leq C \left(\frac{\rho}{r}\right)^{m+4} \int_{\bC_0\cap \bB_r} |\bar w|^2 \, .
\end{equation}
\end{proposition}

\subsection{Smoothness and properties of the average} An important step in the proof of Proposition \ref{p:decad-lineare} is showing smoothness for the ``average of the sheets'', which is defined in the following way. First of all, we consider a linear isometry $\iota: \mathbb R^{m+1}\to \pi_0$ with the property that $\iota (0,y) \in V$ for every $y\in \mathbb R^{m-1}$. In particular, by a small abuse of notation we will denote $\{0\}\times \mathbb R^{m-1}$ as well by $V$. For each $i$ we then select an angle $\theta_i$ such that
\[
\bH_{0,i} = \{(t \cos \theta_i, t \sin \theta_i, y): (t,y)\in \mathbb R^+\times V = \mathbb R^m_+\}\, . 
\]
The average of $\bar{w}$ is then the function $\omega: B_{1/2}^+\to \mathbb R^{m+n}$ given by
\begin{equation}\label{e:media}
\omega(\zeta)=\omega (t, y) = \frac{1}{p} \sum_i \sum_j \bar{w}_{i,j} (t \cos \theta_i, t \sin \theta_i, y)\, ,
\end{equation}
where we use the notation $B^+_r := \{\zeta=(t,y)\in \mathbb R^m_+: t^2 +|y|^2 < r^2\}$. The sum in \eqref{e:media} must be understood as a sum of vectors in $\mathbb R^{m+n}$. 

The relevant properties of $\omega$ are collected in the following 

\begin{lemma}\label{l:media}
Let $\bar w$ be as in Corollary \ref{c:scoppia} and define $\omega$ as in \eqref{e:media}. Then
\begin{itemize}
    \item[(i)] $\omega$ is harmonic and can be extended to a harmonic function (still denoted $\omega$) on $B_{1/2}\subset \mathbb R^m$ with the property that $\frac{\partial^2 \omega}{\partial t\partial y_l} =0$ on $V\cap B_{1/2}$ for every $l=1, \ldots, m-1$;
    \item[(ii)] $\omega (0) =0$ and $\omega (0,y)$ takes values in $V^{\perp_0}$;
    \item[(iii)] There is a linear map $L: \mathbb R^{m+n}\to V^{\perp_0}$, which depends only on $\bC_0$, such that each $\bar{w}_{i,j} (0,y)= \mathbf{p}_{\bH_{0,i}^{\perp_0}} (L (\omega (0,y)))$ for every $y\in V \cap \bB_{1/2}$. In particular, for each $i$ the functions $\{\bar w_{i,j}\}_j$ have the same trace on $V \cap \bB_{1/2}$.
\end{itemize}
\end{lemma}
\begin{proof} {\bf Proof of (i).} The fact that $\omega$ is harmonic is an immediate consequence of Proposition \ref{p:scoppia}(iii). Next, recall that $\omega\in W^{1,2}$ and that, by \eqref{bu:nonconcentration w12},
\begin{equation}\label{e:non-concentration}
\int \frac{|\nabla \omega|^2}{t^{1/2}}<\infty\, , 
\end{equation}
where $\nabla$ denotes the gradient in the coordinates $\zeta=(t,y)$ on $\R^m_+$. Fix now any test function $\varphi\in C^\infty_c (B_{1/2}, \mathbb R^{m+n})$. Clearly, the vector field $W \in C^\infty_c (\mathbb R^{m+n}, \mathbb R^{m+n})$ defined by
\[
W (x,y):= \varphi (|x|, y)
\]
is cylindrical, and therefore an admissible test in \eqref{e:Simon-test}. We thus conclude
\begin{equation}\label{e:Simon-test-2}
\int \nabla \omega : \nabla \frac{\partial\varphi}{\partial y_l} = 0\, \qquad \forall l=1, \ldots , m-1 \, .
\end{equation}

Observe next that $\frac{\partial \omega}{\partial t}$ is an $L^2$ function because of \eqref{e:non-concentration}, and we can thus regard its trace on $V\cap B_{1/2}$ as a distribution in $H^{-1/2}$: the action of the latter on a test function $\psi\in C^\infty_c (V\cap B_{1/2}, \R^{m+n})$ will, by abuse of notation, be denoted by
\[
\int_V \frac{\partial \omega}{\partial t} \cdot \psi\, .
\]
Having fixed any function $\psi\in C^\infty_c (B_{1/2} \cap V,\R^{m+n})$, take a smooth extension to some $\varphi\in C^\infty_c (B_{1/2},\R^{m+n})$. Integrating \eqref{e:Simon-test-2} by parts we then conclude
\begin{equation}
\int_V \frac{\partial \omega}{\partial t} \cdot \frac{\partial \psi}{\partial y_l} = 0 \qquad \forall l\in \{1, \ldots , m-1\}\, .
\end{equation}
The latter identity implies that the distribution $\frac{\partial \omega}{\partial t}$ is a constant, which we can denote by $c$. But then $\omega - ct$ is an harmonic function which satisfies the Neumann boundary condition and clearly we can extend it to an harmonic function on $B_{1/2}$ using the Schwarz reflection principle.

\medskip

{\bf Proof of (ii).} First of all recall that, by Corollary \ref{c:scoppia},
\begin{equation}\label{e:integrabilita}
\int_{B_{1/2}^+} |\zeta|^{-m} \left|\zeta\cdot \nabla \frac{\omega (\zeta)}{|\zeta|}\right|^2 < \infty\, .
\end{equation}
Using the $C^2$ regularity of $\omega$, we write
\[
\omega (\zeta) = \omega (0) + \nabla \omega (0) \cdot \zeta {+D^2\omega(0)[\zeta]^2} + \bar\omega (\zeta)\, , 
\]
{having used the notation $A[\zeta]^2$ for the quadratic form $\sum_{\alpha,\beta}A_{\alpha\beta}\zeta_\alpha\zeta_\beta$}. Since $\bar\omega$ is $C^2$ by construction, and its first and second derivatives vanish at $0$, $\frac{\bar\omega (\zeta)}{|\zeta|}$ is a Lipschitz function, which means that its derivative is bounded. On the other hand, $\nabla \omega (0) \cdot \frac{\zeta}{|\zeta|}$ is a $0$-homogeneous function, so that $\zeta\cdot \nabla \frac{\nabla \omega (0)\cdot \zeta}{|\zeta|} = 0$. {Instead, $\frac{D^2\omega(0)[\zeta]^2}{|\zeta|}$} is $1$-homogeneous, and thus $\zeta\cdot\nabla \frac{D^2\omega(0)[\zeta]^2}{|\zeta|} = \frac{D^2\omega(0)[\zeta]^2}{|\zeta|}$, which is bounded.

We therefore conclude that
\[
\zeta\cdot \nabla \frac{\omega (\zeta)}{|\zeta|} - \omega (0) \zeta\cdot \nabla |\zeta|^{-1}  
\]
is a bounded function. Since $\zeta\cdot \nabla |\zeta|^{-1}=-|\zeta|^{-1}$, if $\omega (0)$ were different from $0$ then the integral in the left hand side of \eqref{e:integrabilita} would be infinite.

Consider next the linear map
\begin{equation}\label{e:somma-proiezioni}
P := \frac{1}{p} \sum_i \kappa_{0,i}\, \mathbf{p}_{\bH_{0,i}^{\perp_0}}
\end{equation}
and observe that the image of $P$ is contained in $V^{\perp_0}$.

Consider next the binding function $\bar\xi$ as a function on $\mathbb R^m_+$ and note that, due to \eqref{bu:nonconcentration w12}
\begin{equation}\label{e:binding-media}
\int_{B_{1/2}^+} t^{-{5/2}} |\omega (t,y) -  P (\bar\xi (t,y))|^2\, dy\, dt < \infty\, . \end{equation}
In particular 
\[
\int_{B_{1/2}^+} t^{-{5/2}} |\mathbf{p}_{V} \omega (t,y)|^2\, dy\, dt < \infty\, ,
\]
which clearly implies $\omega (0,y)\in V^{\perp_0}$ for every $y$.

\medskip

{\bf Proof of (iii).} Observe that $P$ is self-adjoint: in particular its image $Z$ coincides with its cokernel, and it is mapped into itself. We denote by $P|_Z^{-1}$ the inverse of the restriction $P|_Z:Z\to Z$ and let $L:= P|_Z^{-1} \circ \mathbf{p}_Z$. Observe in particular that, since $Z\subset V^{\perp_0}$, $L$ maps the whole space on $V^{\perp_0}$. Moreover, if $v\in Z$, then $P(L(v))=v$. In particular, since $\omega(0,y) \in Z$ as a consequence of \eqref{e:binding-media}, it holds $P (L (\omega (0,y)) = \omega (0,y)$. We therefore conclude from the regularity of $\omega$ that 
$|P(L (\omega (t,y)))-\omega (t,y)|\leq C t$, so that from \eqref{e:binding-media} we obtain
\begin{equation}\label{e:binding-media-2}
\int_{B_{1/2}^+} t^{-{5/2}}\, \left|P \Big(L (\omega (t,y)) -  \bar\xi (t,y)\Big)\right|^2\, dy\, dt < \infty\, . 
\end{equation}
Next observe that, since $\mathbf{p}_{\bH_{0,i}^{\perp_0}}$ is an orthogonal projection,
\begin{align*}
&\sum_i \kappa_{0,i} |\mathbf{p}_{\bH_{0,i}^{\perp_0}} (L (\omega (t,y)) -  \bar\xi (t,y))|^2\\
=&\sum_i \kappa_{0,i} (\mathbf{p}_{\bH_{0,i}^{\perp_0}} (L (\omega (t,y)) -  \bar\xi (t,y))) \cdot (L (\omega (t,y)) -  \bar\xi (t,y))\,.
\end{align*}
In particular, using the boundedness of $\bar\xi$ and the local boundedness of $\omega$, for every $r<1/2$ we conclude
\[
\sum_i \kappa_{0,i} |\mathbf{p}_{\bH_{0,i}^{\perp_0}} (L (\omega (t,y)) -  \bar\xi (t,y))|^2
\leq C\, |P (L (\omega (t,y)) - \bar\xi (t,y))| \qquad \forall (t,y)\in B_r^+\, .
\]
We can thus estimate
\[
\int_{B_r^+}\frac{|\mathbf{p}_{\bH_{0,i}^{\perp_0}} (L (\omega (t,y)) -  \bar\xi (t,y))|^2}{t^{5/4}} <\infty
\]
using \eqref{e:binding-media-2} and the Cauchy-Schwarz inequality. In turn, this easily implies, for every $i$ and $j$, that
\[
\int_{\bB_r\cap \bH_{0,i}} \frac{|\bar w_{i,j} (x,y) - \mathbf{p}_{\bH_{0,i}^{\perp_0}} (L (\omega (|x|, y)))|^2}{|x|^{5/4}}<\infty\, .
\]
The claim that $\mathbf{p}_{\bH_{0,i}^{\perp_0}} (L (\omega (0, y)))$ is the trace of $\bar{w}_{i,j}$ on $V$ follows then at once.
\end{proof}

\subsection{An elementary lemma on harmonic functions} In order to prove Proposition \ref{p:decad-lineare} we will appeal to classical decay lemmas for harmonic functions. On the other hand, since our objects are actually defined on ``half balls'', we will require the following estimate.

\begin{lemma}\label{l:elementare-Watson}
There is a constant $C= C(m) > 0$ with the following property.
Let $B_r \subset \mathbb R^m$ and $\omega\in L^2(B_r, \mathbb R^N)$ be a harmonic function such that $\frac{\partial \omega}{\partial t}$ is constant on $V\cap B_r$. Then
\begin{equation}\label{e:controllo-da-sinistra}
\int_{B_r} |\omega|^2 \leq C \int_{B_r^+} |\omega|^2\, .
\end{equation}
\end{lemma}
\begin{proof}
First of all, since we can argue componentwise, we can assume that $N=1$ and, by scaling, we can also assume that $r=1$. Furthermore, the Schwarz reflection principle shows that, denoting $c := \frac{\partial \omega}{\partial t}$ on $V \cap B_r$, the function $\omega_e (t,x):= \omega (t,x)-ct$ is even in $t$, which in turn implies that 
\begin{align*}
\int_{B_1} \omega^2 = \int_{B_1} (\omega_e^2 + c^2 t^2)\, .
\end{align*}
Since 
\begin{align*}
\int_{B_1} \omega_e^2 &= 2 \int_{B_1^+} \omega_e^2\, ,\\
\int_{B_1} t^2 &= 2 \int_{B_1^+} t^2\, ,
\end{align*}
it suffices to show the existence of a positive constant $\delta (m)>0$ such that
\begin{equation}\label{e:giusta}
2 \left|\int_{B_1^+} c \omega_e t\right| \leq (1-\delta) \int_{B_1^+} (\omega_e^2 + c^2 t^2)
\end{equation}
for any $\omega_e$ which belongs to the space $H_e$ of even harmonic functions. 

Observe that, for any fixed $\omega_e$, the inequality is equivalent to the nonnegativity of the two quadratic polynomials
\[
P^\pm (c):= (1-\delta) \|\omega_e\|_{L^2 (B_1^+)}^2 + (1-\delta) c^2 \|t\|_{L^2 (B_1^+)}^2 \pm
2 c \langle \omega_e, t \rangle\, ,
\]
where we denote by $\langle\cdot, \cdot \rangle$ the $L^2 (B_1^+)$ scalar product.
Since both polynomials have a positive coefficient in the quadratic monomial, their nonnegativity is equivalent to the nonpositivity of their (common) discriminant, which in turn is the inequality  
\begin{equation}\label{e:giusta-2}
\left|\langle \omega_e, t \rangle \right|
\leq (1-\delta) \|\omega_e\|_{L^2 (B_1^+)} \|t\|_{L^2 (B_1^+)}\, .
\end{equation}
Since the latter inequality is homogeneous in $\omega_e$, we can prove it under the additional assumption that $\|\omega_e\|_{L^2 (B_1^+)}=1$. So assume by contradiction that a sequence $\{\omega_k\}\subset H_e$ satisfies $\|\omega_k\|_{L^2}=1$ and violates the inequality \eqref{e:giusta-2} with $\delta= \frac{1}{k}$, i.e. (upon changing a sign)
\begin{equation}\label{e:giusta-3}
\int_{B_1^+} \omega_k\, t \geq \frac{k-1}{k} \|t\|_{L^2 (B_1^+)}\, . 
\end{equation}
Upon extraction of a subsequence we can assume that $\omega_k$ converges to some $\omega_\infty\in H_e$ weakly in $L^2$. By lower semicontinuity, $\|\omega_\infty\|_{L^2 (B_1^+)}\leq 1$, and by weak convergence we have
\begin{equation}\label{e:CS-reverse}
\int_{B_1^+} \omega_\infty\, t \geq \|t\|_{L^2 (B_1^+)}\, .
\end{equation}
Using the Cauchy-Schwarz inequality we conclude that, on the other hand,
\begin{equation}\label{e:CS}
\int_{B_1^+} \omega_\infty t \leq \|\omega_\infty\|_{L^2 (B_1^+)} \|t\|_{L^2 (B_1^+)}\, .
\end{equation}
In particular $\|\omega_\infty\|_{L^2 (B_1^+)}=1$ and thus \eqref{e:CS-reverse} and \eqref{e:CS} hold with equality. But this would mean that $\omega_\infty$ is collinear with $t$, i.e. $\omega_\infty$ is a nontrivial even harmonic function with $\omega_\infty (0,y) =0$ for every $y$. By Schwarz reflection, $\omega_\infty$ must be odd in $t$ as well, which implies that $\omega_\infty$ vanishes identically, contradicting $\|\omega_\infty\|_{L^2 (B_1^+)}=1$.
\end{proof}

\subsection{Proof of Proposition \ref{p:decad-lineare}} First of all, consider the average $\omega$ as defined in \eqref{e:media}, and extended to the whole ball $B_{1/2}$ as in Lemma \ref{l:media}. Then, define a linear function $b_1(y)$ by
\[
y \mapsto \nabla_y \omega (0) \cdot y\, .
\]
Consider additionally $c:= \frac{\partial \omega}{\partial t} (0)$. Since $\omega (0) =0$ by Lemma \ref{l:media}, classical estimates on harmonic functions and Lemma \ref{l:elementare-Watson} imply
\begin{align}
\int_{B_\rho^+} |\omega (t,y)- b_1 (y) - ct|^2 \leq & C
\left(\frac{\rho}{r}\right)^{m+4} \int_{B_r} |\omega|^2 
\leq C
\left(\frac{\rho}{r}\right)^{m+4} \int_{B^+_r} |\omega|^2\nonumber \\
\leq & C \left(\frac{\rho}{r}\right)^{m+4} \int_{\bB_r\cap \bC_0} |\bar w|^2\, . \label{e:media-decade}
\end{align}
Next fix $i$ and $j$ and, with a slight abuse of notation, identify $\bH_{0,i}$ with $\mathbb R^m_+$. Define then the map $\hat{w}_{i,j}$ as
\[
\hat{w}_{i,j} (t,y) := \bar{w}_{i,j} (t,y) - \mathbf{p}_{\bH_{0,i}^{\perp_0}} (L(\omega (t,y)))\, .
\]
By Lemma \ref{l:media}, the trace of $\hat{w}_{i,j}$ on $V$ is zero and we can therefore extend it to $B_{1/2}$ as a harmonic function, which is odd on $t$. The $y$ derivative of $\hat{w}$ at $0$ vanishes, and if we let $d_{i,j}:= \frac{\partial \hat{w}_{i,j}}{\partial t} (0)$ we get, from classical estimates on harmonic functions,
\begin{equation}\label{e:foglio-decade}
\int_{B_\rho^+} |\hat{w}_{i,j} (t,y)- d_{i,j} t|^2 \leq  C
\left(\frac{\rho}{r}\right)^{m+4} \int_{B^+_r} |\hat{w}_{i,j}|^2
\leq  C \left(\frac{\rho}{r}\right)^{m+4} \int_{\bB_r\cap \bC_0} |\bar w|^2
\end{equation}
If we now set $b (y):= L (b_1 (y))$ and $a_{i,j} (x):= d_{i,j} |x| + \mathbf{p}_{\bH_{0,i}^{\perp_0}} (L(c)) |x|$, combining \eqref{e:media-decade} and \eqref{e:foglio-decade} we reach
\begin{equation}\label{e:decadimento-combinato}
\int_{\bB_\rho\cap \bH_{0,i}} |\bar w_{i,j} (x,y)- \mathbf{p}_{\bH_{0,i}^{\perp_0}} (b (y)) - a_{i,j} (x) |^2
\leq  C \left(\frac{\rho}{r}\right)^{m+4} \int_{\bB_r\cap \bC_0} |\bar w|^2\,
\end{equation}
Summing the latter inequality over $i$ and $j$ we reach the desired conclusion. \hfill\qedsymbol

%% file: decaduto.tex
\section{Proofs of Theorem \ref{t:decay} and of Theorem \ref{t:main}}\label{sec:iterazione}

In this section we prove Theorem \ref{t:decay} and obtain Theorem \ref{t:main} as a corollary. 

\subsection{The new cone} We start with a simple corollary of the analysis that we carried on thus far.

\begin{corollary}\label{c:decaduto-A<<E}
Let $\bC_0$ be as in Assumption \ref{ass:cone}, with $\bS_0 = \spt(\bC_0) \subset \pi_0$ and let $V$ be the spine of $\bC_0$.
There are a threshold $\rho^+ \in \left(0,1\right)$ and a constant $\bar C$, depending only on $\bC_0$, $m$, $n$, and $p$, such that, if $\rho^-<\rho^+$ is a second positive number, then the following properties hold, provided $\eta_5 = \eta_5 (\rho^-)>0$ is chosen sufficiently small. Assume $T, \Sigma, \bC$, and   $\bS$ are as in Theorem \ref{thm:est spine} with $\eta_5$ replacing $\eta_3$ and $\varepsilon_2$ in \eqref{e:hyp_flat_T_C0-aggiornata}-\eqref{e:hyp_graph_aggiornata}. Assume in addition that $\bA \leq \eta_5 \bE$, and let $\tilde l = \tilde l_{i,j}$ be the linear $p$-multifunction of Theorem \ref{thm:est spine}. Then there are a rotation $O$ of $\pi_0$ and a linear $p$-multifunction $l^+_{i,j}$ with the following properties:
\begin{itemize}
    \item[(i)] $|O-{\rm Id}| + \|l^+\|_{L^\infty (\bS_0\cap \bB_1)} \leq \bar C \bE^{1/2}$; 
    \item[(ii)] If $\bC^+$ is the cone realized as $p$-multigraph over $\bS_0$ of $\tilde l_{i,j} + l^+_{i,j}$ and $\bC' := O_\sharp \bC^+$, then
    \begin{equation}\label{e:decaduto}
    \rho^{-m-2} \int_{\bB_\rho} \dist (q, \spt (\bC'))^2 d\|T\| (q) \leq \frac{\rho}{R_0} \bE \qquad \forall \rho\in [\rho^-, \rho^+]\, .
    \end{equation}
\end{itemize}
\end{corollary}

Before coming to the proof, we observe that $\bC'$ is coherent with $\bC_0$ and that in addition
\begin{align}
\vartheta (\bC', \bC_0)\leq & \vartheta (\bC, \bC_0) + \tilde{C} \big(|O - {\rm Id}| + \|l^+\|_{L^\infty (\bS_0 \cap \bB_1)}\big)\leq
\vartheta (\bC, \bC_0) + \tilde{C} \bE^{1/2}\, ,\label{e:controllo-angolo}
\end{align}
for an appropriate constant $\tilde{C}$, which depends only on $\bC_0$

\begin{proof} First of all, the parameters $\bar C$ and $\rho^+$ will be chosen, respectively, sufficiently large and sufficiently small, depending only on the constants $R_0$ of Assumption \ref{a:graphical} and $C$ of Theorem \ref{thm:est spine}. We fix them for the moment and will specify their choices later. Hence we fix any $\rho^-<\rho^+$ and in order to find the threshold $\eta_5$ we argue by contradiction. If the statement is false, we then have a blow-up sequence $\Sigma_k, T_k, \bC_k$ as in Definition \ref{d:scoppia} which is violating the claim of the Corollary. In particular no matter how we choose $l^+= l^{k,+}$ and $O= O_k$ with $|O-{\rm Id}| + \|l^+\|_{L^\infty(\bS_0 \cap \bB_1)}$ satisfying the bound (i), \eqref{e:decaduto} (with $T=T_k$) will fail for some $\rho=\rho_k\in [\rho^-, \rho^+]$. After extraction of a subsequence (not relabeled) we can then apply Corollary \ref{c:scoppia} and Proposition \ref{p:scoppia}. If we consider the corresponding blow-up map $\bar w$, we can then apply Proposition \ref{p:decad-lineare}. Let $a$ and $b$ be the corresponding maps. We then set $l^{k,+} := \bE_k^{1/2} a$. Consider next the vector field $b$, let $b^*\colon V^{\perp_0}\to V$ be its adjoint, and define $\tilde b\colon V^{\perp_0}\oplus V \simeq\pi_0 \to \pi_0$ by $\tilde b(x,y):=b(y)-b^*(x)$. The vector field $\tilde b$ is the infinitesimal generator of a one-parameter family of linear transformations $O (t, \cdot)$ of $\pi_0$, i.e. 
\[
\left\{
\begin{array}{ll}
O (0,x,y) &= (x,y)\\
\partial_t O(0,x,y) &= \tilde b(x,y)\, . 
\end{array}\right.
\]
Observe that, if we take the matrix $B$ which represents the linear transformation $\tilde b$, then $O (t, \cdot)$ is the linear transformation whose matrix is the exponential $\exp (t B)$. On the other hand, the definition of $\tilde b$ implies that $B$ is antisymmetric, and in particular $\exp (tB)$ is orthogonal. This means that $O$ is a one-parameter family of rotations.
We then define the rotation $O_k := O (\bE_k^{1/2}, \cdot)$. Observe that the bound (i) is satisfied with this choice of $l^{k,+}$ and $O_k$.

Next let $\bC^+_k$ be the graph over $\bS_0$ of the multifunction $l^k_{i,j} + l^{k,+}_{i,j}$ and $\bC'_k := (O_k)_\sharp \bC_k^+$.
We claim that, combining Corollary \ref{c:scoppia}, Proposition \ref{p:scoppia}, and 
Proposition \ref{p:decad-lineare}, for every fixed $\sigma > 0$, we conclude
\begin{equation}\label{e:sfava}
\limsup_{k\to \infty} \sup_{\rho\in [\rho^-, \rho^+]} \underbrace{\rho^{-m-3} \bE_k^{-1} \int_{\bB_\rho\setminus B_\sigma (V)} \dist (q, \spt (\bC'_k))^2 d \|T_k\| (q)}_{=: I_k} \leq C \rho^+\, .
\end{equation}
Note first that each $q\in \spt (T_k)\cap \bB_\rho\setminus B_\sigma (V)$ is contained in the graph of a function $v^k_{i,j}$ for all $k$ sufficiently large. Fix $q$ and let therefore $z=(x,y)\in \bH_{0,i}$ be such that $z+v^k_{i,j} (z) = q$. Consider now the following points
\begin{align*}
z^+ &:= z- \bE_k^{1/2}\mathbf{p}_{\bH_{0,i}} (b(y))+\bE_k^{1/2} b^*(x)\\
q^+ &:= z^++l^k_{i,j} (z^+) + l^{k,+}_{i,j} (z^+)\\
q' &:= O (\bE_k^{1/2}, q^+)\, .
\end{align*}
Observe that $q'\in \spt(\bC'_k)$ and thus 
\[
\dist (q, \spt(\bC'_k)) \leq |q-q'|\, .
\]
On the other hand $q-q'= (q-q^+) + (q^+-q')$. First of all notice that: 
\begin{align*}
q-q^+ &= \bE_k^{1/2} \mathbf{p}_{\bH_{0,i}} (b(y)) -\bE_k^{1/2} b^*(x) + u^k_{i,j} (z) - l^k_{i,j} (z^+) - l^{k,+}_{i,j} (z^+) + {\rm O} (\bA)\nonumber\\
&= \bE_k^{1/2} \mathbf{p}_{\bH_{0,i}} (b(y)) -\bE_k^{1/2} b^*(x) + u^k_{i,j} (z)- l^k_{i,j} (z) - l^{k,+}_{i,j} (z) + ({\rm o} (1)  + {\rm O (\bE_k^{1/2})}) |z-z^+| + {\rm O} (\bA)\, ,
\end{align*}
where we have used that $\|\nabla l^k_{i,j}\|_\infty \leq C \hat \flat^p_{\bB_1} (\bC_k, \bC_0) = {\rm o} (1)$ and $\|\nabla l^{k,+}_{i,j}\|_\infty = {\rm O} (\bE_k^{1/2})$.
In particular we conclude 
\begin{align*}
q-q^+ &= \bE_k^{1/2} \mathbf{p}_{\bH_{0,i}} (b(y))-\bE_k^{1/2} b^*(x) + w^k_{i,j} (z) - \bE_k^{1/2} a_{i,j} (x) + {\rm o} (\bE_k^{1/2}) \, .
\end{align*}
On the other hand, using that $\tilde b$ is the generator of $O(t, \cdot)$ we obviously have 
\begin{align*}
q^+-q' &= - \bE_k^{1/2} b (y) +\bE_k^{1/2} b^*(x) + {\rm O} (\bE_k)\, .   
\end{align*}
Summing the last two estimates we conclude that
\[
|q-q'| = |w^k_{i,j} (z) - \bE_k^{1/2} (a_{i,j} (x) + \mathbf{p}_{\bH_{0,i}}^{\perp_0} (b(y)))| + {\rm o} (\bE_k^{1/2})
\]
Thus we can estimate 
\begin{align*}
I_k &\leq \left(\rho^{-m-3} \bE_k^{-1} \sum_i \int_{\bH_{0,i} \cap \bB_\rho \setminus B_\sigma (V)}  \sum_j |w_{i,j}^k (z) - \bE_k^{1/2} (a_{i,j} (x) + \mathbf{p}_{\bH_{0,i}^{\perp_0}} (b(y)))|^2 \right)
+ {\rm o} (1)\, .
\end{align*}
We can now apply Proposition \ref{p:decad-lineare} to estimate the term in the parenthesis, and, using that the $L^2$ norm of $\bar w$ is bounded due to Proposition \ref{p:scoppia}, conclude \eqref{e:sfava}.

In addition, observe that in $\bB_1$ we have $\dist (q, \spt (\bC'_k))\leq \dist (q, \spt (\bC_k)) + C \bE_k^{1/2}$ and thus we can estimate
\begin{align*}
&\limsup_{k\to \infty} \sup_{\rho\in [\rho^-, \rho^+]}\rho^{-m-3} \bE_k^{-1} \int_{\bB_\rho\cap B_\sigma (V)} \dist (q, \spt (\bC'_k))^2 d\|T_k\| (q)\\ 
\leq & C (\rho^{-})^{-m-3} \sigma + \sigma^{1/2} \limsup_{k\to \infty} \sup_{\rho\in [\rho^-, \rho^+]}\rho^{-m-3} \bE_k^{-1} \int_{\bB_\rho\cap B_\sigma (V)} \frac{\dist (q, \spt (\bC_k))^2}{\sigma^{1/2}} d\|T_k\| (q)\\
\leq & C \sigma^{1/2} (\rho^{-})^{-m-3}\, .
\end{align*}
Since $\rho^-$ is fixed, we can now choose $\sigma$ arbitrarily small to conclude
\[
\limsup_{k\to \infty} \sup_{\rho\in [\rho^-, \rho^+]} \rho^{-m-3} \bE_k^{-1} \int_{\bB_\rho} \dist (q, \spt (\bC'_k))^2 d \|T_k\| (q) \leq C \rho^+
\]
Obviously, choosing $\rho^+$ so that $2C\rho^+< R_0^{-1}$, for a sufficiently large $k$ we actually reach a contradiction with $\bC_k=\bC^+$, as $\bC_k, O_k, l^{k,+}$ satisfy both the conclusions (i) and (ii) of the Corollary. 
\end{proof}

\subsection{Proof of Theorem \ref{t:decay}} First of all, by scaling, we can replace the outer radius $1$ by $R_0$, while of course the decay rate has to be replaced by $(\frac{\rho}{R_0})^{1/2}$. Choose now $\rho:=\rho^+$, coming from Corollary \ref{c:decaduto-A<<E}. The proof will distinguish between two regimes. If $\bA \leq \eta_5 \bE$, where $\eta_5$ comes from Corollary \ref{c:decaduto-A<<E}, we will be able to apply Corollary \ref{c:decaduto-A<<E}, while in the other regime we will let $\bC'$ be the cone $\tilde\bC$ obtained as the graph, over $\bC_0$, of the linear $p$-multifunction $\tilde l$ of Theorem \ref{thm:est spine}. Observe that in both cases the claim \eqref{e:incremento-angolo} holds: in the case $\bA\leq\eta_5 \bE$ it follows from \eqref{e:controllo-angolo}, while in the other case it follows because $\spt (\bC')\subset \spt (\bC)$. 

Consider now the case $\bA \leq \eta_5 \bE$, and choose $\bC'$ as in Corollary \ref{c:decaduto-A<<E}: then, the estimate \eqref{e:decay} is obvious, since both the quantities over which we maximize have the right decay. Consider next the situation in which
$\bA > \eta_5 \bE$. Here we will choose $\eta < \eta_5$ (and in fact much smaller than $\eta_5$). We know that
\[
(R_0 \bA)^{1/2} \geq R_0^{1/2}\eta^{-1} \bA 
\geq \eta_5 R_0^{1/2} \eta^{-1} \bE\, .
\]
Hence if we choose $\eta$ sufficiently small, depending only on $R_0$ e $\eta_5$, we conclude 
\begin{equation}\label{e:massimo-giusto}
(R_0 \bA)^{1/2} = \max \{\bE (T, \spt(\bC), 0, R_0), (R_0 \bA)^{1/2}\}\, .
\end{equation}
By our choice of $\bC'$, we also have
\begin{align*}
\bE (T, \spt (\bC'), 0, \rho) &\leq C \rho^{-m-2} (\bE + \bA) \leq C\eta_5^{-1} \rho^{-m-2} \bA\\
&\leq C \eta\, \eta_5^{-1} \rho^{-m-2-1/2}\, (\rho \bA)^{1/2}\, .
\end{align*}
Hence, by choosing $\eta$ so small that 
\[
C \eta \eta_5^{-1} \rho^{-m-2-1/2} \leq 1
\]
(which again is a choice depending only on $\eta_5$ and $\rho$, which have been fixed), we achieve
\begin{equation}\label{e:massimo-2}
(\rho \bA)^{1/2} = \max \{\bE (T, \spt(\bC'), 0, \rho), (\rho \bA)^{1/2}\}\, .
\end{equation}
\eqref{e:massimo-giusto} and \eqref{e:massimo-2} give thus the desired decay in the regime $\bA > \eta_5 \bE$.

We now come to estimate \eqref{e:flat-C'}. Observe first that, since now $\rho$ is fixed and $\hat\flat^p$ behaves nicely under restrictions and rescalings, it suffices to estimate $\hat\flat^p_{\bB_{1/2}} (T - \bC')$. Observe also that it suffices to estimate $\hat\flat^p_{\bB_{1/2}} (T - \tilde{\bC})$, since in one regime we have $\bC'=\tilde\bC$, while in the other regime we can estimate $\hat\flat^p_{\bB_{1/2}} (\bC'-\tilde\bC) \leq \vartheta (\bC', \tilde\bC)\leq C \bE^{1/2}$. Coming to $\tilde\bC$, we first wish to extend the multifunction $u= \{u_{i,j}\}$ so that its domain of definition is $\bS_0 \cap \bB_{1/2}$ and it satisfies the bounds
\[
|x|^{-1} |u_{i,j} (x,y)| + |\nabla u_{i,j} (x)|\leq C\, .
\]
In order to achieve the latter extension, we first claim that $u_{i,j}$ is globally Lipschitz. In fact pick two points $(x,y), (x',y')$ and denote by $Q$ and $Q'$ the corresponding cubes of the Whitney domain which include them. If the two cubes are neighbors then we obviously have
\[
|u_{i,j} (x,y) - u_{i,j} (x',y')|\leq (\|\nabla u_{i,j}\|_{L^\infty (Q)} + \|\nabla u_{i,j}\|_{L^\infty (Q')}) |(x,y)- (x',y')|\, .
\]
We can thus assume that they are not neighbors. In particular this implies that $|(x',y')- (x,y)|\geq c_0 \max \{d_Q, d_{Q'}\}$ for a suitable geometric constant. Thus we can estimate
\[
|u_{i,j} (x,y) - u_{i,j} (x',y')| \leq |u_{i,j} (x,y)|+ |u_{i,j} (x',y')|\leq C (d_Q+ d_{Q'})
\leq C |(x,y)-(x',y')|\, .
\]
Having established the global Lipschitz bound, it suffices to first extend $u_{i,j}$ to $V$ identically $0$ and observe that such extension is still Lipschitz. Hence we can further extend $u_{i,j}$ to a Lipschitz function defined on the whole $\bB_{1/2}\cap \bH_{0,i}$. The estimate $|u_{i,j} (x,y)|\leq C |x|$ follows from the Lipschitz regularity and $u_{i,j} (0,y) =0$.

Next we extend the map $v$ as well by simply setting 
\[
v_{i,j} (z) = u_{i,j} (z) + \Psi (z+ u_{i,j} (z))\, .
\]
Observe that the extension obeys the estimate
\begin{align*}
\int_{B_{\varrho} (V)\cap \bB_{1/2}} (|v|^2 + |\tilde l|^2) \leq & C
\int_{\bB_{1/2} \cap B_\varrho (V)} |x|^2 d\|\bG_{\bS_0} (v)\| +
C \int_{\bB_{1/2} \cap B_\varrho (V)} |x|^2 d\|\bG_{\bS_0} (\tilde{l})\| + {\rm O}(\bA)\\
\leq & C \int_{\bB_{1/2}\cap B_\varrho (V)} |x|^2 d\|T\| +  C \int_{\bB_{1/2}\cap B_\varrho (V)} |x|^2 d\|\bC\| + {\rm O}(\bA) \\ 
\leq & C (\bE + \bA)\, ,
\end{align*}
by Theorem \ref{thm:graph_v1}(iv). Notice that, in the second inequality of the estimate above, we have used, as in the proof of Theorem \ref{thm:graph_v1}(iv), that $\bB_{1/2}\cap B_\varrho (V)$ can be covered by balls $\mathbf B_{5r_i}(y_i)$ where $y_i \in V$, the $\mathbf B_{r_i}(y_i)$ are disjoint, and $r_i=\bar C \varrho_{\mathcal W}(y_i)$. By definition of $\varrho_{\mathcal W}$, each $B_{r_i}(y_i)$ contains the rotation of a cube in $\mathcal Q$ where the excess of $T$ from $\mathbf S$ is large, so that in each $\mathbf B_{5r_i}(y_i)$ the integral of $|x|^2$ on the graph $\bG_{\bS_0} (v)$ is controlled by $r_i^{m+2}$, and the latter is in turn controlled by the integral of $|x|^2$ with respect to $\|T\|$ over the rotated cube.

We now write
\[
T \res \bB_1 - \tilde\bC = \underbrace{T- \bG_{\bS_0} (v) \res \bB_1}_{=:R_1} +
\underbrace{\bG_{\bS_0} (v)\res \bB_1 - \bG_{\bS_0} (\tilde{l})\res \bB_1}_{=:R_2}\, .
\]
Then 
\begin{align*}
\hat\flat^p_{\bB_{1/2}} (R_2) \leq & \sum_i \int_{\bH_{0,i}\cap \bB_{1/2}} \sum_j |v_{i,j}- \tilde{l}_{i,j}|\\
\leq & C \bA + \int_{\bS_0\cap \bB_{1/2}\setminus B_{\varrho} (V)} |w| + \int_{\bS_0\cap \bB_{1/2} \cap B_\varrho (V)} (|v|+ |\tilde{l}|)
\leq C (\bE+\bA)^{1/2}\, .
\end{align*}
As for estimating $\hat\flat^p_{\bB_{1/2}} (R_1)$, observe that $R_1=0$ outside $B_{\varrho} (V)$. Hence consider the homotopy $H (t,x,y):= (tx,y)$, which is retracting $\mathbb R^{m+n}$ onto $V$. We then apply the homotopy formula and 
conclude that
\[
R_1\res \bB_1 = - \partial H_\sharp (\a{[0,1]}\times R_1)\res \bB_1 \quad\modp\, 
\]
and we can estimate
\begin{align*}
\hat\flat^p_{\bB_{1/2}} (R_1) \leq & \mass (H_\sharp (\a{[0,1]}\times R_1)\leq  C\int_{\bB_{1/2}\cap B_{\varrho} (V)} |x|\, d\|T\|
+ C\int_{\bB_{1/2}\cap B_\varrho (V)} |x|\, d\|\bG_{\bS_0} (v)\|\\
\leq & C (\bE + \bA)^{1/2}\, . \tag*{\qed}
\end{align*}

\subsection{Proof of Theorem \ref{t:main}} First of all, without loss of generality we assume $q=0$. Next we assume that $\bar \eta < \eta$ is sufficiently small, where $\eta$ is the constant of Theorem \ref{t:decay}, so that we can apply it setting $\bC = \bC_0$. We then find a new cone $\bC_1$ which satisfies
\begin{align*}
\max \{ \bE ((\eta_{0,\rho})_\sharp T, \spt (\bC_1), 0,1), (\rho \bA)^{1/2}\} &\leq \rho^{1/2} (\bE_0 + \bA^{1/2})\\
\hat\flat^p_{\bB_1} ((\eta_{0,\rho})_\sharp T - \bC_1) &\leq C (\bE_0^{1/2} + \bA^{1/4})\\
\vartheta (\bC_1, \bC_0) &\leq C (\bE_0^{1/2} + \bA^{1/4})\, .
\end{align*}
Assume now that for a certain number of steps $j =1, \ldots , k-1$ we can apply Theorem \ref{t:decay} to the triple $(T_j, \bC_j, \Sigma_j)$ where 
$T_j = (\eta_{0,\rho^j})_\sharp T$ and $\Sigma_j = \eta_{0, \rho^j} (\Sigma)$. Observe that $\|A_{\Sigma_j}\|_{L^\infty (\Sigma_j)}=\rho^j \bA$. Setting 
\begin{align*}
m (j) &:= \max \{ \bE (T_j, \spt (\bC_j), 0,1), (\rho^j \bA)^{1/2}\}
\end{align*}
we then get
\begin{align}
m (k) \leq & \rho^{k/2} (\bE_0 + \bA^{1/2})\, , \notag\\
\hat\flat^p_{\bB_1} (T_k - \bC_k) \leq & C\, m(k-1)^{1/2} \leq C \rho^{(k-1)/4} (\bE_0^{1/2} + \bA^{1/4})\, , \notag\\
\vartheta (\bC_k, \bC_0) \leq & C \sum_{j=0}^{k-1} m (j)^{1/2} \leq C \sum_{j=0}^{k-1} \rho^{j/4} (\bE_0^{1/2} + \bA^{1/4})\, .\label{e:Cauchy}
\end{align}
In particular we conclude that
\begin{align*}
\|A_{\Sigma_k}\|_{L^\infty} = & \rho^k \bA \leq \bar \eta \rho^k\\
\bE (T_k, \spt (\bC_k), 0,1) \leq & m(k) \leq \rho^{k/2} ( \bE_0 + \bA^{1/2}) \leq 2 \rho^{(k-1)/2} \bar \eta \, ,\\
\hat{\flat}^p_{\bB_1} (T_k - \bC_0) \leq & \hat \flat^p_{\bB_1} (T_k - \bC_k) 
+ \hat\flat^p_{\bB_1} (\bC_k - \bC_0)\\
\leq & \hat \flat^p_{\bB_1} (T_k - \bC_k) + C \vartheta (\bC_k, \bC_0) \leq C \sum_{j=0}^{k-1} \rho^{j/4} (\bE_0^{1/2} + \bA^{1/4}) \leq C \bar\eta^{1/2}\, ,
\end{align*}
where the constant $C$ is independent of both $\bar\eta$ and $k$. If $\bar \eta$ is chosen sufficiently small, the latter estimates guarantee that we can keep applying Theorem \ref{t:decay} for all $k\in \mathbb N$. 

The conclusions of Theorem \ref{t:main} thus follow at once, considering that the unique tangent cone to $T$ at $q$ is simply the unique limit of the sequence $\bC_k$ (which is a Cauchy sequence in the flat distance by \eqref{e:Cauchy}). \qed

%% file: struttura.tex
\section{Proofs of the structure theorems}\label{sec:structure}

In this section we give the proof of the two structure Theorems \ref{t:odd} and \ref{t:even}, of Corollary \ref{c:flat-singular-points}, and of Proposition \ref{p:example}. In fact, the two theorems will be corollaries of the following more precise consequence of Theorem \ref{t:main}.

\begin{corollary}\label{c:struttura}
Consider $\Sigma, T$, and $\Omega$ as in Definition \ref{def:am_modp} and assume that $\dim (\Sigma) = \dim (T) +1 = m+1$. Assume $q\in \spt (T)$ is a point where a tangent cone $\bC_0$ has $(m-1)$-dimensional spine $V$, i.e. it takes the form
\[
\bC_0 = \sum_{i=1}^{N_0} \kappa_{0,i} \a{\bH_{0,i}}\, ,
\]
where $\bH_{0,i}$ are the (distinct) pages of the open book $\bS_0 = \spt (\bC_0)$ and $\kappa_{0,i}\in \mathbb N \cap \left[1,\frac{p}{2}\right)$ are such that $\sum_i \kappa_{0,i} = p$. Then there is a neighborhood $U$ of $q$ such that $\sing (T)\cap U$ is a classical free boundary as in Definition \ref{def:free-boundary}, with the additional information that:
\begin{itemize}
    \item[(i)] The coefficients $k_i$ in Definition \ref{def:free-boundary} coincide with $\kappa_{0,i}$;
    \item[(ii)] The tangent to $\sing (T)$ at $q$ is $V$;
    \item[(iii)] The tangent to each $\Gamma_i$ at $q$ is $\bH_{0,i}$.
\end{itemize}
\end{corollary}

The corollary is obviously a stronger version of Theorem \ref{t:even} when $p$ is even. As for Theorem \ref{t:odd}, in the case of odd $p$, observe that, using the terminology of the proof of Lemma \ref{lem:technical_no_hole}, White's regularity theorem implies that $\mathcal{S}^m\setminus \mathcal{S}^{m-1}$ consists of regular points and thus $\mathcal{S}^{m-1}$ is closed. Next observe that any point $q\in \mathcal{S}^{m-1}\setminus \mathcal{S}^{m-2}$ falls in the assumptions of Corollary \ref{c:struttura}, hence $\sing^* (T) = \mathcal{S}^{m-1}\setminus \mathcal{S}^{m-2}$ is locally a classical free boundary of $T$. On the other hand, if $q \in \sing(T)$ has an open neighborhood $U$ such that $\sing(T) \cap U$ is a classical free boundary, then $T$ has, at $q$, a unique tangent cone with $(m-1)$-dimensional spine, which means that $q\in \sing^* (T)$. We conclude therefore that $\mathcal{S}=\mathcal{S}^{m-2}$ is relatively closed in $\mathcal{S}^{m-1}$. Furthermore, it is a simple consequence of Theorem \ref{t:main} and Corollary \ref{c:struttura} that, when $p$ is odd, $\mathcal{S}^{m-2}$ coincides, locally, with the \emph{quantitative stratum} $\mathcal{S}^{m-2}_{\eta}$ for some $\eta > 0$; see Appendix \ref{app:NV} for the terminology and the proof of this fact. Thus, by the Naber-Valtorta rectifiability theorem, cf. \cite{NV}, $\mathcal{S}=\mathcal{S}^{m-2}$ is $(m-2)$-rectifiable and it has locally finite $(m-2)$-dimensional measure, thus giving the conclusions of Theorem \ref{t:odd}. 

\subsection{Stucture} We now come to Corollary \ref{c:struttura}.

\begin{proof} Without loss of generality we let $q=0$. 
We fix a small threshold $\hat\eta$.
First of all, by rescaling, we can assume $\max\{ \bE (T, \bC_0, 0, 4), \bA\} < \bar\eta$ and since we are in the position of applying Theorem \ref{t:main}, we conclude that $\bC_0$ is the unique tangent cone to $T$ at $0$ and that we actually have
   \begin{equation}\label{e:decay-always-2}
    \frac{1}{r^{m+2}}\int_{\bB_r}\dist^2(\bar q, \spt(\bC_0)) \, d\|T\| (\bar q) \leq \hat\eta r^{1/2} \, ,
    \end{equation}
    \begin{equation}\label{e:decay-flat-2}
    \hat\flat^p_{\bB_1} ((\eta_{0,r})_\sharp T - \bC_0) \leq \hat \eta^{1/2} r^\frac{1}{4}\, .
    \end{equation}
    for every $r<4$. 
    We moreover denote by $V_0$ the spine of $\bC_0$.
    
    For every $q\in \bB_4 (0)\cap \Sigma$ consider $\pi_q := T_q \Sigma$ and let $O_q$ be a rotation of $\mathbb R^{m+n}$ which maps $T_q \Sigma$ onto $T_0\Sigma$. $O_q$ can be chosen to have a $C^1$ dependence on $q$ and to satisfy $O_q = {\rm Id}$ at $q=0$. Consider now the currents $(O_q)_\sharp (T_{q,1})$. Observe that the map $q\mapsto (O_q)\sharp (T_{q,1})$ is continuous in the flat topology. Hence, for a sufficiently small $\delta$ and for every $q\in \bB_\delta (0)\cap \spt (T)$, it holds
    \begin{equation}
    \int_{\bB_3} \dist^2 (\bar q, \bC_0)^2\, d\|(O_q)_\sharp T_{q,1}\| (\bar{q}) \leq C \hat \eta\,    
    \end{equation}
    \begin{equation}
    \hat\flat^p_{\bB_3} ((O_q)_\sharp T_{q,1} - \bC_0) \leq C \hat\eta\, .   
    \end{equation}    
    In particular, we are again in the position to apply Theorem \ref{t:main} with $\bC_0$ to the current $(O_q)_\sharp (T_{q,1})$, provided $\Theta_T (q) \geq \frac{p}{2}$. We thus conclude that, for every $q\in \bB_\delta$, the following alternative holds true:
    \begin{itemize}
        \item[(a)] Either $\Theta_T (q)<\frac{p}{2}$;
        \item[(b)] Or $\Theta_T (q) \geq \frac{p}{2}$, in which case we can apply Theorem \ref{t:main} and hence find a unique tangent cone $\bC_q$ to $T$ at $q$, with $m-1$-dimensional spine $V_q$, and the decay properties
        \begin{equation}\label{e:decay-always-3}
    \frac{1}{r^{m+2}}\int_{\bB_r (q)}\dist^2(\bar q - q, \spt(\bC_q)) \, d\|T\| (\bar q) \leq C \hat\eta r^{1/2} \, ,
    \end{equation}
    \begin{equation}\label{e:decay-flat-3}
    \hat\flat^p_{\bB_1} ((\eta_{q,r})_\sharp T - \bC_q) \leq C \hat \eta^{1/2} r^\frac{1}{4}\, ,
    \end{equation}
    for all radii $r<3$
    \end{itemize}
Since we can apply a further rescaling to the current, from now on we assume that the alternative in fact holds for every $q\in \bB_1\cap \spt (T)$.
Consider now two points $q, q'\in \bB_1$ with $\Theta_T (q), \Theta_T (q')\geq \frac{p}{2}$. If we set $|q-q'|=r$ and we denote by $\tau_v$ the translations $\tau_v (z)=z+v$, we conclude easily from \eqref{e:decay-flat-3} that
\[
\hat{\flat}^p_{\bB_1} (\bC_q - (\tau_{r^{-1} (q'-q)})_\sharp \bC_{q'}) \leq C \hat{\eta}^{1/2} r^{\frac{1}{4}}\, .
\]
After scaling, the latter estimate implies
\begin{itemize}
    \item $\dist (q, V_{q'}) + \dist (q', V_q) \leq C \hat{\eta}^{1/2} r^{5/4}$;
    \item $|\mathbf{p}_{V_q}-\mathbf{p}_{V_{q'}}| \leq C \hat{\eta}^{1/2} r^{1/4}$.
\end{itemize}
It thus turns out that the set $\{\Theta_T \geq \frac{p}{2}\}$ is contained in the graph of a $C^{1,1/4}$ map $\psi: V_0 \to V_0^\perp$ with $\|\psi\|_{1,1/4}\leq C \hat{\eta}^{1/2}$. 

Fix $\delta>0$ and consider now a point $q_0\in V_0\cap \bB_{1/2}$. Observe that Proposition \ref{prop:no-holes} implies that (if $\eta$ is chosen sufficiently small), then $\bB_\delta (q_0)$ contains a point $\bar q_1\in \spt (T)$ with $\Theta_T (\bar q_1)\geq \frac{p}{2}$. We now consider the point $q_1\in \bar{q}_1 + V_{\bar q_1}$ such that $\mathbf{p}_{V_0} (q_1) = \mathbf{p}_{V_0} (q_0) = q_0$ and observe that $|q_1 - \bar{q}_1|\leq 2 \delta$. We are thus in the position to apply again Proposition \ref{prop:no-holes} to the current $(\eta_{\bar q_1, 1/2})_\sharp T$, the cone $\bC_{\bar q_1}$, and the spine $V_{\bar q_1}$ to find a point $\bar q_2$ with $\theta_T (\bar q_2) \geq \frac{p}{2}$ such that $|\bar q_2 - q_1|\leq \frac{\delta}{2}$. Hence we consider the unique point $q_2\in \bar q_2 + V_{\bar q_2}$ such that $\mathbf{p}_{V_0} (q_2) = \mathbf{p}_{V_0} (q_1) = q_0$, for which we have $|q_2-\bar{q}_2| \leq 2 \frac{\delta}{2}$. Proceeding inductively we get two sequences of points $\{q_k\}, \{\bar q_k\}$, with the following properties:
\begin{align*}
|\bar q_k - q_{k-1}| \leq & 2^{-k+1}\delta\\
|q_k - \bar{q}_k| \leq & 2^{-k+2} \delta\\
\Theta_T (\bar{q}_k) \geq &\frac{p}{2}\\
\mathbf{p}_{V_0} (q_k) = & q_0\, .
\end{align*}
Both sequences converge to a unique point $q_\infty\in \bB_{4\delta} (q_0)$, with $\mathbf{p}_{V_0} (q_\infty) = q_0$. 
The latter argument implies that $\{\Theta \geq \frac{p}{2}\}\cap \bB_{1/2}$ is indeed the graph $\Lambda$ of a $C^{1,1/4}$ map $\psi: V_0\to V_0^\perp$. Observe moreover that $T_q \Lambda = V_q \subset \pi_q = T_q \Sigma$ and that indeed $\Theta (T, q) = \frac{p}{2}$ for all points $q\in \Lambda \cap \bB_{1/2}$. 

We now choose a second rotation $U_q$ of $\mathbb R^{m+n}$ with the property that $U_q (T_0 \Sigma) = T_q \Sigma$, $U_q (V_0) = V_q$, and $U_q$ has a $C^{1/4}$ dependence on $q$. In fact we can see that
\[
\hat\flat^p_{\bB_1} (\bC_q - (U_q)_\sharp \bC_0) \leq C \hat\eta^{1/2} |q|^{1/4}\, .
\]
We can now apply the approximation Theorem \ref{thm:graph_v1} to the current $(U_q^{-1})_\sharp T_{q,r}$ with $\bC = (U_q^{-1})_\sharp \bC_q$ and $\bC_0$, for every $r<R_0^{-1}$. 

Consider now the halfplanes $\bH_{q,i} := U_q (\bH_{0,i})$. Fix a point $z\in \bB_{(4R_0)^{-1}}\setminus \Lambda$ and let $q (z)\in \Lambda$ be a point such that $r = |q(z)-z|= \dist (z, \Lambda)$.
Theorem \ref{thm:graph_v1} implies then that:
\begin{itemize}
    \item $z$ is a regular point of $T$, and indeed, in a neighborhood of $z$, $T$ is a graph over $q (z) + \bH_{q(z),i}$ for some $i$;
    \item $|\mathbf{p}_{\bH_{q(z), i}^\perp} (z-q(z))|\leq \hat C \hat \eta |z-q(z)|^{5/4}$, where $\hat{C}$ is a geometric constant. 
\end{itemize}
We therefore consider the open sets
\[
\mathcal{U}_i :=\big\{q\in \bB_{(4R_0)^{-1}} : |\mathbf{p}_{\bH^\perp_{q(z), i}} (z-q(z))|\leq \hat C \hat \eta |z-q(z)| \big\}\, .
\]
We now restrict our attention to $T\res \mathcal{U}_i$. Let $\bH$ be the $m$-dimensional space which contains $\bH_{0,i}$, consider the projection $\lambda := \mathbf{p}_\bH (\Lambda)$ and let $e_i$ be the unit normal to $\lambda$ which is contained in $\bH_{0,i}$. Denote by $B_\rho$ the balls of radius $\rho$ in $\bH$ and observe that $\lambda$ divides each sufficiently small $B_\rho$ in two connected regions: we will call $B_\rho^+$ the one such that $e_i$ is the interior unit normal at $0$. We claim that, if $\rho$ is sufficiently small, in the intersection of the set $\mathcal{U}_i$ with the cylinder $\mathbf{p}_{\bH}^{-1} (B^+_\rho)$, the current $T$ is given by $\kappa_{0,i}$ Lipschitz graphs (which are Lipschitz up to the boundary $\lambda$) of functions $v_j$, $j\in \{1, \ldots , \kappa_{0,i}\}$. Observe that, if $\Psi_0: \pi_0\to \pi_0^\perp$ is the graphical parametrization of $\Sigma$ over $\pi_0$, each function will take the form $u_j (\xi) = (\xi, u_j (\xi), \Psi (\xi, u_j (\xi)))$, where $\xi\in \bH$ and $u_j (\xi)\in \bH^\perp\cap \pi_0$. Since $\bH^\perp \cap \pi_0$ is $1$-dimensional, we can identify it with $\mathbb R$ and order the functions $u_j$ from top to bottom. Observe also that there is a function $\psi_0: \lambda \to \bH^\perp \cap \pi_0$ such that $\Lambda$ is the graph of the map $\lambda \ni \xi \mapsto (\xi, \psi_0 (\xi), \Psi (\xi, \psi_0 (\xi)))$. As a consequence of our claim we will conclude that $u_j|_\lambda = \psi_0$ for every $j$.

Consider now the classical Whitney (or Calderon-Zygmund) decomposition of $B_{2\rho}^+$. In particular for each cube $Q$ in the decomposition, we let $\xi (Q)$ be a closest point on $\lambda$ and observe that the distance $d_Q$ of $Q$ to $\xi (Q)$ is comparable to the sidelength $\ell (Q)$ of the cube. We let $q$ be the point $q = (\xi (Q), \psi_0 (\xi (Q)), \Psi (\xi (q), \psi_0 (\xi (Q)))\in \Lambda$. Apply now Theorem \ref{thm:graph_v1} to $(U_q^{-1})_\sharp T_{q,\ell (Q)}$ with $\bC_0$ and $\bC = (U_q^{-1})_\sharp \bC_q$. If we enlarge slightly the cube $Q$ to a concentric cube $Q'$ with slightly larger sidelength (say $9\ell (Q)/8$) and we consider the region $R:= \mathbf{p}_{\bH}^{-1} (Q')$, then $(U_q^{-1})_\sharp T_{q,\ell (Q)}\res \mathcal{U}_i\cap R$ consists of pieces of $\kappa_{0,i}$ graphs over $\bH_{0,i}$ with controlled Lipschitz constant $C$. In rotating back to the original system of coordinates using $U_q$, the graphical representation still holds because $|U_q-{\rm Id}|$ is small, and the Lipschitz constants becomes slightly larger, but they are still controlled by a geometric constant. The claimed graphicality is thus correct over each enlarged cube $Q'$, and since for neighboring cubes the enlarged ones has a nontrivial overlap, the ordering of the sheets shows that the functions $u_j$ (and hence the $v_j$'s) can be defined coherently over the whole region $B^+_\rho$. The Lipschitz constant of the restriction of each $u_j$ (and hence $v_j$) to every cube in the Whitney decomposition is bounded by an absolute constant $C$. Observe however that we have as well the bound 
\[
\|v_j - \psi_0 (\xi (Q))\|_{L^\infty (Q)} \leq C \ell (Q) 
\]
for every cube $Q$, simply because the graphs are contained in the open sets $\mathcal{U}_i$. It is now simple to see that the graph $v_j$ is then globally Lipschitz. In fact consider $z, z'\in B^+_\rho$ and let $Q$ and $Q'$ be the cubes of the Whitney decomposition which contain them. If the two cubes are neighbors, then obviously 
\[
|v_j (z) - v_j (z)|\leq C |z-z'|\, .
\]
If the two cubes are disjoint, notice that $\ell (Q') +\ell (Q) \leq C |z-z'|$ and $|\xi (Q)-\xi (Q')|\leq C |z-z'|$. Hence
\begin{align*}
|v_j (z)-v_j (z')|\leq & |v_j (z)-\psi_0 (\xi (Q))| + |\psi_0 (\xi (Q)) - \psi_0 (\xi (Q'))|+
|v_j (z') - \psi_0 (\xi (Q'))|\\
\leq & C (\ell (Q) + |\xi (Q) - \xi (Q')| + \ell (Q')) \leq C |z-z'|\, .
\end{align*}
Having shown that each $v_j$ is a minimal Lipschitz graph and that $u_j|_\lambda = \psi_0$, the $C^{1,1/4}$ regularity of $v_j$ up to $\lambda$ in $B^+_{\rho/2}$ follows now from standard Schauder estimates. 

However, because of the decay to the cone $\bC_0$ at the point $0$, the normal derivatives of each $v_j$ at $0$ is in fact $0$. Observe that $u_1 \leq u_2 \leq \ldots \leq u_{\kappa_{0,i}}$ on their domain of definition. In particular the Hopf maximum principle implies that all of these functions coincide. Thus $T\res \mathcal{U}_i$ is, in a neighborhood $U$ of $0$, a single $C^{1,1/4}$ graph with boundary $\Lambda\cap U$ and multiplicity $\kappa_{0,i}$. 
\end{proof}

With the obvious modifications of the proof given above, one can prove the following $\eps$-regularity result:

\begin{corollary}[$\eps$-regularity] \label{eps_reg}
Let $p\in \N \setminus \{0,1,2\}$ and $\bC_0$ be  as in  Assumption  \ref{ass:cone}, with $(m-1)$-dimensional spine $V$, i.e. of the form
\[
\bC_0 = \sum_{i=1}^{N_0} \kappa_{0,i} \a{\bH_{0,i}}\, ,
\]
where $\bH_{0,i}$ are the (distinct) pages of the open book $\bS_0 = \spt (\bC_0)$ and $\kappa_{0,i}\in \mathbb N \cap \left[1,\frac{p}{2}\right)$ are such that $\sum_i \kappa_{0,i} = p$. Then there is a constant $\bar\eta>0$ depending  only  on $p,m,n$ and $\bC_0$ with  the  following  property.  If $T,\Sigma, \Omega$ and $q$ are as in  Assumption \ref{ass:cone} with $\eta=\bar{\eta}$, then $\sing (T)\cap B_{\sfrac1{10}}(q)$ is a classical free boundary as in Definition \ref{def:free-boundary}, with the additional information that:
\begin{itemize}
    \item[(i)] The coefficients $k_i$ in Definition \ref{def:free-boundary} coincide with $\kappa_{0,i}$;
    \item[(ii)] The tangent to $\sing (T)$ at $q$ is $V$;
    \item[(iii)] The tangent to each $\Gamma_i$ at $q$ is $\bH_{0,i}$.
\end{itemize}
\end{corollary}

\subsection{Flat singular points for even moduli} We next come to the proof of Corollary \ref{c:flat-singular-points}. Fix thus $T$, $p=2Q$, $\Sigma$, $\Omega$, and $q$ as in the statement. Clearly $\Theta_T (q) = Q$. Consider the set ${\rm Tan}\, (T, q)$ of cones which are tangent to $T$ at $q$ and subdivide it into 
\[
{\rm Tan}_f := \{S\in {\rm Tan}\, (T, q): S = Q \a{\pi} \quad\mbox{for some $m$-plane $\pi$}\}
\]
and 
\[
{\rm Tan}_{nf}:= {\rm Tan}\, (T,q) \setminus {\rm Tan}_{f}\, .
\]
Consider now $Z\in {\rm Tan}_{nf}$ and let $V$ be its spine, which is given by
\[
V:= \{x: \Theta_Z (x)=Q\}\, .
\]
Since $Z$ is not flat, ${\rm dim}\, (V) \leq m-1$. On the other hand if it were $m-1$, then Proposition \ref{lem:structure_cones} and Theorem \ref{t:even} would imply that $Z$ is the unique tangent cone to $T$ at $q$. So we must necessarily have ${\rm dim}\, (V) \leq m-2$. Consider now a point $x\in \spt (Z)$ which is not regular, i.e. $x\in \sing (Z)$. If a tangent cone to $Z$ at $x$ is $Q\a{\pi'}$ for some plane $\pi'$, then $x\in V$. If $x\in \sing (T)\setminus V$ then no tangent cone to $Z$ can be flat, because the multiplicity would have to be an integer $k\in \{1, \ldots , Q-1\}$ and then $x$ would be regular by White's theorem \cite{White86}. Next, by Proposition \ref{lem:structure_cones}, if a tangent cone to $Z$ at $x$ has $(m-1)$-dimensional spine, then $\Theta_Z (x) = Q$ and thus $x\in V$. We conclude that
\begin{itemize}
    \item if $x\in \sing (Z)\setminus V$, then the spine of any tangent cone to $Z$ at $x$ has dimension at most $m-2$.
\end{itemize}
From the Almgren's stratification theorem we then conclude that $\dim_{\Ha} (\sing (Z)\setminus V) \leq m-2$. But then the Hausdorff dimension of the whole $\sing (Z)$ is at most $m-2$. This in turn implies that $Z$ is a classical area-minimizing current without boundary. Indeed, take first a connected component $\mathcal{U}$ of $\reg (Z)$. On $\mathcal{U}$ the multiplicity of $Z$ must be a constant $\modp$. On the other hand, since this regular part is not part of the spine of the cone, such multiplicity cannot be congruent to $\frac{p}{2} = Q$. It is thus simple to see that it can be chosen to be constant as an integer valued function. Recalling that $Z$ is a precise representative, we conclude that the support of $\partial Z$ must be contained in the singular set $\sing (Z)$. Since $\partial Z$ is a flat chain supported in a set of zero $\mathcal{H}^{n-1}$-dimensional measure, a well-known theorem of Federer implies that $\partial Z =0$.  

We are now ready to show that ${\rm Tan}_{nf}=\emptyset$. Assume otherwise and for each $Z\in {\rm Tan}\, (T, q)$ consider now its spherical cross section $\langle Z, |\cdot|, 1\rangle$. Observe that the space of such cross sections is a compact subset of the space of $\modp$ cycles in $\partial \bB_1$ in the topology of $\hat\flat^p_{\bB_2}$. For each $Z\in {\rm Tan}_{nf}$ we consider the function
\[
d (Z):= \min \{ \hat\flat^p_{\bB_2} (\langle Z, |\cdot|, 1) -
\langle S, |\cdot|, 1\rangle): S \in {\rm Tan}_f\}\, .
\]
Now, $d(Z)>0$. We claim that:
\begin{itemize}
    \item[(Con)] if ${\rm Tan}_{nf}\neq \emptyset$ then there is $\sigma_0>0$ such that
    \begin{equation}\label{e:connectedness}
    \forall s\in (0, \sigma_0) \; \exists Z\in {\rm Tan} (T, q) \;\mbox{such that}\; d (Z) = s\, .
    \end{equation}
\end{itemize}
The latter is an easy consequence of the observation that the function
\[
d (r) :=  \min \{ \hat\flat^p_{\bB_2} (\langle T_{q,r}, |\cdot|, 1) -
\langle S, |\cdot|, 1\rangle): S \in {\rm Tan}_f\}
\]
is continuous in $r$ and that, if $T_{q,r_k} \to Z$, then $d (r_k) \to d (Z)$. Having these two properties in mind, we let $\sigma_0 := d (Z_0)$ for some fixed chosen $Z_0 \in {\rm Tan}_{nf}$. Then there is $\rho_k\downarrow 0$ such that $T_{q, \rho_k}\to Z_0$. On the other hand there is also $r_k\downarrow 0$ such that $T_{q, r_k}\to S_0\in {\rm Tan}_f$. W.l.o.g. we can assume $r_k < \rho_k$. Next, $d (\rho_k) \to d (Z_0) = \sigma_0$ and $d (r_k)\to 0$. Fix therefore $s\in (0, \sigma_0)$. Then for every sufficiently large $k$ there must be a $\tau_k\in (r_k, \rho_k)$ such that $d (\tau_k)=s$. Since by extraction of a subsequence we can assume $T_{q, \tau_k}\to Z$ for some $Z\in {\rm Tan}\, (T,q)$, we then conclude 
\[
d (Z) = \lim_{k\to\infty} d (\tau_k) = s\, ,
\]
thus proving \eqref{e:connectedness}. Next, consider $s = \frac{1}{k}$ and let $Z_k$ be the corresponding element of ${\rm Tan}(T,q)$ such that $d (Z_k)=\frac{1}{k}$. Then $Z_k \to Q \a{\pi}$ for some $m$-dimensional $\pi$ in the flat topology $\modp$ on every bounded open set. On the other hand since both $Z_k$ and $Q \a{\pi}$ are integral cycles, the latter convergence takes place in the usual flat topology of integer rectifiable cycles as well. In particular, since $Z_k$ is area-minimizing and has codimension $1$ in $T_q \Sigma$, for a sufficiently large $k$ the regularity theory implies that $Z_k$ is in fact everywhere regular. But since $Z_k$ is a cone, it must then be a flat cone, i.e. $Q \a{\pi_k}$ for some $m$-dimensional plane $\pi_k$. On the other hand the latter conclusion would imply $d (Z_k) =0$, while we know that $d (Z_k) = \frac{1}{k}>0$. \hfill\qedsymbol

\subsection{Proof of Proposition \ref{p:example}}

We fix coordinates $x_1, x_2, x_3$ in $\mathbb R^3$, consider a cycle $\modp$ $S$ which is invariant under rotations around the $x_3$ axis and let $T$ be an area-minimizing current $\modp$ with $\partial^p T = S$. The following is a well-known fact:

\begin{lemma}
$T$ is invariant with respect to rotations around the $x_3$ axis. 
\end{lemma}

\begin{proof} Fix a minimizer $T$ which is a representative $\modp$ and let $r: \mathbb R^3 \to \mathbb R^+$ be given by $r (x_1, x_2, x_3) = \sqrt{x_1^2+x_2^2}$. We denote by $C_\delta$ the closed set $\{r\leq \delta\}$ and observe that, by the monotonicity formula,
\[
\|T\| ( C_\delta) \leq C \delta
\]
for some constant $C$ independent of $\delta$. Moreover, up to a rotation around $x_3$ we can assume that $\|T\| (\{x_2=0\})=0$. Introduce next the function $\theta (x_1, x_2, x_3)$ which gives the angle between $(x_1, x_2, 0)$ and the $x_1$ axis. We will assume that $\theta$ is defined on the complement of $H := \{x_2=0, x_1\leq 0\}$.
Even though $\theta$ is just locally Lipschitz, we can define the current $\langle T, \theta, \alpha\rangle \res r$ by taking the limit of the currents
    \[
    \langle T_\delta, \theta, \alpha \rangle \res r\, ,
    \]
    where $T_\delta = T \res (C_\delta \cup H)^c$. It follows easily that
    \[
    \int_{-\pi}^\pi \mass (\langle T, \theta, \alpha\rangle \res r)\, d\alpha = \mass (T \res rd\theta) \leq \mass (T)\, .
    \]
    So 
    \[
    {\rm ess inf}_\alpha \, \mass (\langle T, \theta, \alpha\rangle \res r) \leq \frac{1}{2\pi} \mass (T)\, .
    \]
    On the other hand, for a set of full measure of $\alpha$, if we construct an integer rectifiable current by rotating the current $\langle T, \theta, \alpha\rangle$ around the $x_3$ axis we find a current $T'$ with $\partial T' = S \modp$ and 
    \[
    \mass (T) \leq \mass (T') = 2\pi \mass (\langle T, \theta, \alpha\rangle \res r)\, .
    \]
    This shows that indeed $\mass (T\res d\theta) = \mass (T)$, which in turn shows that $T$ must be invariant under rotations around the $x_3$ axis.
\end{proof}

Using the notation of the Lemma we observe that $\langle S, \theta, 0\rangle$ is a sum of Dirac masses 
\[
\sum_i \kappa_i \a{P_i}\, .
\]
$T$ is then obtained by rotating around the $x_3$ axis the 
current $T_0$ in $\{x_1 >0\}$ with the property that $\partial T_0 = S$ $\modp$ and $T_0$ minimizes the mass relative to the Riemannian metric $\hat{g} = x_1 (dx_1^2 + dx_2^2)$. It is easy to see that $T_0$ consists of the union of finitely many geodesic arcs (in the metric $\hat{g}$) with integer weights and which meet in a finite number of singular points. In particular the singular set of $T$ consists of finitely many circles $\gamma_i$ contained in planes $\{x_3=c_i\}$ and centered at $(0,0, c_i)$. It suffices to show that for an appropriate choice of the $P_i$'s and of their weights at least one ``singular'' circle must be present. This however can be arranged by choosing $p$ distinct $P_i$'s and multiplicities $\kappa_i =1$, with the additional property that the $P_i$'s are not all contained in a single geodesic. \hfill\qedsymbol

%% file: appendix.tex
\appendix

\section{On two notions of flat distance} \label{appendix-flat}

Recall that, for any (relatively) closed subset $C$ of an open set $\Omega \subset \R^{m+n}$, the group $\mathscr{F}_m(C)$ of \emph{$m$-dimensional integral flat chains in $C$} consists of all $m$-dimensional currents $T$ in $\Omega$ for which there exists a compact set $K \subset C$ such that
\[
T = R + \partial Z
\]
for some integer rectifiable currents $R$ and $Z$ (of the appropriate dimensions) with support $\spt(R), \spt(Z) \subset K$. Given an integer $p \geq 2$, and following Federer \cite{Federer69}, one endows $\mathscr{F}_m(C)$ with a family of pseudo-distances as follows: if $T \in \mathscr{F}_m(C)$ and $K \subset C$ is compact, then one sets
\begin{equation} \label{flat-federer}
\begin{split}
\flat^p_K(T) := \inf\Big\lbrace & \mass(R) + \mass(Z) \, \colon \, R \in \mathscr{R}_m(K),\, Z \in \mathscr{R}_{m+1}(K) \\
& \qquad \mbox{such that $T = R + \partial Z + pP$ for some $P \in \mathscr{F}_m(K)$} \Big\rbrace\,.
\end{split}
\end{equation}

Then, if $T,S \in \mathscr{F}_m(C)$ one defines the \emph{flat distance modulo $p$} between $T$ and $S$ in $K$ to be the quantity $\flat^p_K(T-S)$. Notice that such distance may be infinite. Given $T,S \in \mathscr{F}_m(C)$, we say that $T = S \, \modp$ if there exists a compact set $K \subset C$ such that $\flat^p_K(T-S)=0$.

\medskip

The definition of flat distance $\modp$ proposed in \eqref{flat-federer} is ill-behaved with respect to localization. Consider, as an examples, two integer rectifiable currents $T,S \in \mathscr{R}_m(\bB_2)$ such that $\spt(T-S) \subset \overline{\bB}_1$. The quantity $\flat^p_{\overline{\bB}_1}(T-S)$ is certainly finite, bounded above by $\mass(T-S)$. If, on the other hand, one wanted to measure the localized flat distance $\modp$ between $T$ and $S$ in, say, $\overline{\bB}_{1/2}$, the definition \eqref{flat-federer} would produce $\flat^p_{\overline{\bB}_{1/2}}(T-S) = \infty$ \emph{unless} $\spt(T-S) \subset \overline{\bB}_{1/2}$. An obvious solution to this apparently minor issue would be to modify the definition so that $\flat^p_{\overline{\bB}_{1/2}}(T-S)$ is given by
$\flat^p_{\overline{\bB}_{1/2}}((T-S)\mres \bB_{1/2})$. Although natural, such approach is not completely satisfactory either, since the resulting distances $\flat^p_{\overline{\bB}_1}(T-S)$ and $\flat^p_{\overline{\bB}_{1/2}}(T-S)$ \emph{may not be comparable}, and in particular it is false, in general, that the former controls the latter. To see this, let $R$ and $Z$ be almost optimal for $\flat^p_{\overline{\bB}_1}(T-S)$, so that
\[
T-S = R + \partial Z + pP \quad \mbox{and} \quad \mass(R) + \mass(Z) \leq \flat^p_{\overline{\bB}_1}(T-S) + \varepsilon\,.
\]
An obvious attempt would be to use $R$ and $Z$ as competitors to estimate the localized distance $\flat^p_{\overline{\bB}_{1/2}}(T-S)$. On one hand,
\[
(T-S) \mres \bB_{1/2} = R \mres \bB_{1/2} + (\partial Z) \mres \bB_{1/2} + p \, P \mres \bB_{1/2}\,,
\]
so that a competitor decomposition of $(T-S) \mres \bB_{1/2}$ may be obtained by applying the slicing formula \cite[Lemma 28.5]{Simon83} to get
\[
(T-S) \mres \bB_{1/2} = R \mres \bB_{1/2} - \langle Z, |\cdot|, 1/2 \rangle + \partial (Z \mres \bB_{1/2} ) + p \, P \mres \bB_{1/2}\,.
\]
On the other hand, the slice $\langle Z, |\cdot|, 1/2 \rangle$ at the given radius $r=1/2$ may not even be defined, and, even if it were, its mass may be arbitrarily large. Of course, any fixed neighborhood of $r=1/2$ contains radii $r'$ such that the corresponding slices enjoy good mass estimates (which degenerate as the neighborhoods shrink), but the fact that \eqref{flat-federer} does not allow to gain control at a fixed sub-scale makes its use rather inconvenient when it comes to the regularity statements that are needed in the present paper.

\medskip

In order to overcome these issues, we are going to define here an alternative notion of flat (pseudo)-distance $\modp$, inspired by that proposed by Simon in \cite{Simon83}. While the two definitions are not metrically equivalent in a given compact set, they induce the same notion of convergence on $\mathscr{F}_m(C)$ (see Proposition \ref{flat:same topology} below) and, in particular, the same equivalence classes $\modp$. Let $\Omega$ and $C$ be as above, and assume, for the sake of simplicity, that $C$ is a Lipschitz neighborhood retract of $\R^{m+n}$. For any $T \in \mathscr{F}_m(C)$, and for any open set $W \ssubset \Omega$, we define

\begin{equation} \label{flat-simon-2}
\begin{split}
    \hat\flat^p_{W}(T) := \inf\Big\lbrace & \|R\|(W) + \|Z\|(W) \, \colon \, R \in \mathscr{R}_m(\Omega),\, Z \in \mathscr{R}_{m+1}(\Omega) \\  
    & \qquad \mbox{such that $T = R + \partial Z + pP$ for some $P \in \mathscr{F}_m(\Omega)$} \Big\rbrace\,, 
\end{split}
\end{equation}
and then we let the \emph{modified flat distance modulo $p$} between $T,S \in \mathscr{F}_m(C)$ in $W$ to be $\hat\flat^p_W(T-S)$.

Notice that, since integral currents are dense in the space of integral flat chains with respect to (classical) flat distance, \eqref{flat-simon-2} is unchanged if we replace the condition $P \in \mathscr{F}_m(\Omega)$ with $P \in \mathscr{I}_m(\Omega)$. Furthermore, since $C$ is a Lipschitz neighborhood retract of $\R^{m+n}$, one could require the currents $R,Z$, and $P$ to be (compactly) supported in $C$ rather than in $\Omega$ and obtain a comparable definition of $\hat\flat^p_W(T)$ to the one in \eqref{flat-simon}, with comparison constant depending only on the Lipschitz constant of the retraction. Since we are only interested in the notion of convergence induced by the family $\{\hat\flat^p_W\}_W$ on $\mathscr{F}_m(C)$ and the corresponding equivalence classes $\modp$, we shall not enforce this requirement here (see also \cite[Remark 1.1]{DLHMS}). Observe that the quantity $\hat\flat^p$ is monotone non-decreasing with respect to set inclusion, namely $\hat\flat^p_{W'}(T) \leq \hat\flat^p_{W}(T)$ if $W' \subset W$.

\medskip

The following proposition shows that, when $T$ is integer rectifiable, the value of $\hat\flat^p$ in an open set depends \emph{only} on the restriction of $T$ to the open set itself.

\begin{proposition} \label{p:who cares about what is outside}

Let $\Omega$ and $C$ be as above, and let $T \in \mathscr{R}_m(C)$. For any open set $W \ssubset \Omega$, it holds
\begin{equation} \label{e:extension}
    \hat\flat^p_{W}(T) = \hat\flat^p_{W}(T\mres W)\,.
\end{equation}
\end{proposition}

\begin{proof}

For any $\delta > 0$, let $R^\delta \in \mathscr{R}_m(\Omega)$, $Z^\delta \in \mathscr{R}_{m+1}(\Omega)$, and $P^\delta \in \mathscr{I}_{m}(\Omega)$ be such that
\begin{equation} \label{almost optimal two}
T = R^\delta + \partial Z^\delta + p P^\delta \qquad \mbox{and} \qquad \|R^\delta\|(W) + \|Z^\delta\|(W) \le \hat\flat^p_{W}(T) + \delta\,.
\end{equation}

We can then write
\[
    T \mres W = T - T \mres (\R^{m+n} \setminus W) =R^\delta - T \mres (\R^{m+n} \setminus W) +  \partial Z^\delta + p P^\delta 
\]
so that
\[
\hat\flat^p_{W}(T \mres W) \leq \|R^\delta\|(W) + \|Z^\delta\|(W) \leq \hat\flat^p_{W}(T) + \delta\,,
\]
and thus the inequality
\begin{equation} \label{easy one}
    \hat\flat^p_{W}(T \mres W) \leq \hat\flat^p_{W}(T)
\end{equation}
follows from the arbitrariness of $\delta$.

\medskip

For the converse, for any $\delta > 0$ let now $R^\delta \in \mathscr{R}_m(\Omega)$, $Z^\delta \in \mathscr{R}_{m+1}(\Omega)$, and $P^\delta \in \mathscr{I}_{m}(\Omega)$ be such that
\begin{equation} \label{almost optimal converse}
T \mres W = R^\delta + \partial Z^\delta + p P^\delta \qquad \mbox{and} \qquad \|R^\delta\|(W) + \|Z^\delta\|(W) \le \hat\flat^p_{W}(T \mres W) + \delta\,.
\end{equation}

We can then write
\[
    T = T \mres W + T \mres (\R^{m+n} \setminus W) = R^\delta + T \mres (\R^{m+n}\setminus W) + \partial Z^\delta + p P^\delta\,,
\]
which, since $T \mres (\R^{m+n} \setminus W)$ is integer rectifiable with zero localized mass in $W$, yields
\begin{equation}\label{even easier one?}
    \hat\flat^p_{W}(T) \leq \|R^\delta\|(W) + \|Z^\delta\|(W) \leq \hat\flat^p_{W}(T \mres W) + \delta\,,
\end{equation}
and the conclusion follows again from the arbitrariness of $\delta$.
\end{proof}

The following proposition compares the values of $\flat^p$ and $\hat\flat^p$ for a given flat chain $T$.

\begin{proposition}\label{flat:same topology}
Let $C \subset \Omega$ be a Lipschitz neighborhood retract of $\R^{m+n}$, and let $T$ be in $\mathscr{F}_m(C)$.

\begin{itemize}

\item[(a)] Let $K \subset C$ be a compact set. Then
\begin{equation}\label{f2s}
    \hat\flat^p_W(T) \leq \flat^p_K(T)
\end{equation}
for all open sets $W \ssubset \Omega$. 

\vspace{0.2cm}

\item[(b)] Let $W \ssubset \Omega$ be an open set. For every $\varepsilon > 0$ there exists an open set $U_{\varepsilon} \subset C \cap W$ with $\dist(\overline{U_\eps}, \R^{m+n} \setminus W) < \eps$ such that
\begin{equation}\label{s2f}
    \flat^p_{\overline{U_\varepsilon}}(T \mres U_\varepsilon) \leq C_\varepsilon\, \hat\flat^p_W(T)\,,
\end{equation}
where $C_\varepsilon \to \infty$ as $\varepsilon \to 0^+$.

\end{itemize}
\end{proposition}

\begin{proof} {\bf Proof of (a).} We can assume that $\flat^p_K(T)<\infty$, otherwise the inequality is trivial. In particular, $\spt(T) \subset K$. For any $\delta > 0$, let $R^\delta \in \mathscr{R}_m(K)$, $Z^\delta \in \mathscr{R}_{m+1}(K)$, and $P^\delta \in \mathscr{F}_m(K)$ be such that 
    \[
 T = R^\delta + \partial Z^\delta + p P^\delta \qquad \mbox{and} \qquad \mass(R^\delta) + \mass (Z^\delta)\leq  \mathscr{F}_K^p(T)+\delta\,.
\]

In particular, since $R^\delta, Z^\delta$, and $P^\delta$ are supported in $K$, it holds
\[
\|R^\delta\|(W') + \|Z^\delta\|(W')=\mass(R^\delta) + \mass (Z^\delta) \leq \mathscr{F}_K^p(T)+\delta
\]
for all open sets $W' \ssubset \Omega$ such that $K \subset W'$. Thus, for any $W$ as in the statement, letting $W' \ssubset \Omega$ be any open set containing $W \cup K$ we have
\[
\hat\flat^p_W(T) \leq \|R^\delta\|(W) + \|Z^\delta\|(W) \leq \|R^\delta\|(W') + \|Z^\delta\|(W') \leq \mathscr{F}_K^p(T)+\delta\,,
\]
so that \eqref{f2s} follows by letting $\delta \downarrow 0$.
    
\medskip
    
{\bf Proof of (b).} Let $R_h \in \mathscr{R}_m(\Omega)$, $Z_h \in \mathscr{R}_{m+1}(\Omega)$, and $P_h \in \mathscr{I}_m(\Omega)$ be such that
\begin{equation} \label{s2f1}
T = R_h + \partial Z_h + p P_h\,, \qquad \|R_h\|(W) + \|Z_h\|(W) \leq \hat\flat^p_W(T) + \frac{1}{h}\,.
\end{equation}
Letting $\Pi \colon \R^{m+n} \to C$ be a Lipschitz retraction, we can first replace the currents $R_h,Z_h$, and $P_h$ with $\Pi_\sharp R_h$, $\Pi_\sharp Z_h$, and $\Pi_\sharp P_h$, respectively, in such a way that the first part of \eqref{s2f1} holds with currents $R_h, Z_h$, and $P_h$ supported on $C$, whereas the inequality in the second part still holds with the right-hand side multiplied by $L:=\Lip(\Pi)^{m+1}$ in case $\Lip(\Pi) > 1$. Next, fix $\eps > 0$, and let  ${\rm d}_W$ denote the function ${\rm d}_W(q) := \dist(q, \R^{m+n} \setminus W)$. By \cite[Lemma 28.5]{Simon83}, it holds
\begin{equation}\label{s2f3}
    \int_{0}^\eps \mass(\langle Z_h, {\rm d}_W, \sigma \rangle) \, d\sigma \leq \|Z_h\|(W) \leq L \left( \hat\flat^p_W(T) + \frac{1}{h} \right)\,,
\end{equation}
so that there exists $\sigma \in \left( 0, \varepsilon \right)$ and, for every $\delta > 0$ there exists a subsequence $h(\ell)$ such that
\begin{equation}\label{s2f:controlling the slices}
    \sup_{\ell \geq 1} \mass (\langle Z_{h(\ell)}, {\rm d}_W, \sigma \rangle) \leq \frac{L}{\varepsilon} \hat\flat^p_W(T) + \delta\,.
\end{equation}
Then, let
\begin{equation}\label{s2f5}
    U_\eps := \lbrace q \in C \, \colon \, {\rm d}_W(q) > \sigma \rbrace\,, \qquad K_\eps := \overline{U_\eps}\,.
\end{equation}

Notice that $U_\eps \subset C \cap W$ is open, $K_\eps$ is compact, and $\dist(K_\eps, \R^{m+n} \setminus W) < \eps$ by the choice of $\sigma$. Next, we can write from \eqref{s2f1} and the slicing formula

\[
\begin{split}
T \mres U_\eps &= R_{h(\ell)} \mres U_\eps + (\partial Z_{h(\ell)}) \mres U_\eps + p\, P_{h(\ell)}\mres U_\eps \\
& = R_{h(\ell)} \mres U_\eps + \langle Z_{h(\ell)}, {\rm d}_W, \sigma \rangle + \partial (Z_{h(\ell)}) \mres U_\eps) + p\, P_{h(\ell)}\mres U_\eps\,,
\end{split}
\]

so that combining \eqref{s2f1} and \eqref{s2f:controlling the slices} yields

\begin{equation}\label{s2f6}
\begin{split}
    \flat^p_{K_\eps}(T\mres U_\eps) &\leq \mass(R_{h(\ell)} \mres U_\eps) + \mass(Z_{h(\ell)} \mres U_\eps) + \mass(\langle Z_{h(\ell)}, {\rm d}_W, \sigma\rangle) \\
    &\leq L\, \left( 1 + \frac{1}{\varepsilon} \right) \,\hat\flat^p_W(T) + \frac{L}{h(\ell)} + \delta  \,.
\end{split}    
\end{equation}
and \eqref{s2f} follows by letting first $\ell \to \infty$ and then $\delta \to 0^+$.

\end{proof}

\begin{corollary} \label{cor:flat hat and congruence}
If $T \in \mathscr{R}_m(C)$ and $W \ssubset \Omega$ is such that $\hat\flat^p_W(T)=0$, then $T \mres U = 0 \; \modp$ for every $U \ssubset W$.
\end{corollary}

\begin{proof}
From Proposition \ref{flat:same topology}(b) it follows that there exists $U'$ with $C \cap U \subset U' \subset C \cap W$ such that $\flat^p_{\overline{U'}}(T \mres U') = 0$, so that $T \mres U' = 0 \; \modp$. In particular, since $T$ is integer rectifiable there exists a rectifiable current $R$ such that $T \mres U' = p\,R$, which in turn gives $T \mres U = p\, R \mres U = 0\;\modp$.
\end{proof}

\section{Proof of Lemma \ref{lem:qualitative flat_excess}}
The Lemma will be a simple consequence of a compactness argument and of the following extreme case.

\begin{lemma} \label{l:zero_excess}
Let $\bC$ and $\bS$ be as in Lemma \ref{lem:qualitative flat_excess}. There exists $\eta_1 = \eta_1(\bS) > 0$ with the following property. Let $T$ be a representative $\modp$ in $\bB_1 \subset \R^{m+n}$ with $\pa T \mres \bB_1 = 0 \, \modp$. If \begin{equation} \label{hp:zero excess}
\bE (T, \bS, 0, 1) = 0 \qquad \mbox{and} \qquad \hat\flat^p_{{\bB}_1} (T-\bC) < 2\,\eta_1
\end{equation}
then 
\begin{equation} \label{e:identical currents}
T \mres \bB_1 = \bC \mres \bB_1\,.
\end{equation}
\end{lemma}

\begin{proof}
To fix the notation, let $\kappa_i \in \left[1,\sfrac{p}{2}\right) \cap \Z$ and $\bH_i$ be such that 
\[
\bC = \sum_{i=1}^N \kappa_i \a{\bH_i}\,.
\]
 Let $T$ be as in the statement. The first hypothesis in \eqref{hp:zero excess} implies that $\spt(T) \cap \bB_1 \subset \bS$. Given that $\pa T = 0 \, \modp$, the constancy lemma for currents $\modp$ \cite[Lemma 7.4]{DLHMS} applied on each page $\bH_i$ of the book $\bS$ implies that there are integers $\theta_i$ with $\abs{\theta_i} \in \left[0, \frac{p}{2} \right]$ such that
\begin{equation} \label{e:congruence_modp}
T \mres \bB_1 = \sum_{i=1}^N \theta_i \a{\bH_i} \mres \bB_1 \quad \modp\,.
\end{equation}
Since there are only finitely many classes of integer rectifiable representatives $\modp$ having the structure \eqref{e:congruence_modp} which are not congruent to $\bC$ $\modp$ in $\bB_1$, the minimum of their $\hat\flat^p_{\bB_1}$-distance from $\bC$ is positive by Corollary \ref{cor:flat hat and congruence}. If we let $2\,\eta_1$ be this value, the condition $\hat\flat^p_{\bB_1}(T - \bC) < 2\,\eta_1$ forces  
\begin{equation} \label{e:identical_modp}
T \mres \bB_1 = \bC\mres \bB_1 \quad \modp\,.
\end{equation}
The conclusion in \eqref{e:identical currents} then follows from the fact that $T$ is representative and all multiplicities on $\bC$ satisfy $\kappa_i < \sfrac{p}{2}$.
\end{proof}

\begin{proof}[Proof of Lemma \ref{lem:qualitative flat_excess}]
By \eqref{compactness assumption} and standard slicing and compactness, there exist $\sigma \in \left( 1, 3/2 \right)$ and a rectifiable current $T$ in $\bB_\sigma$ which is representative $\modp$ with $(\pa T) \mres \bB_\sigma = 0  \; \modp$ such that $\flat^p_{\overline{\bB}_\sigma}(T_j \mres \bB_\sigma - T) \to 0$. In particular, $\hat\flat^p_{\bB_1}(T_j-T) \to 0$ by Propositions \ref{p:who cares about what is outside} and \ref{flat:same topology}.

For every $\lambda > 0$, setting $U_\lambda := \left\lbrace q \in \bB_1 \, \colon \, \dist^2(q,\bS) > \lambda \right\rbrace$ we have that
\begin{equation} \label{cebicev}
\|T_j\|(U_\lambda) \leq \lambda^{-1} \, \int_{\bB_1} \dist^2(\cdot, \bS) \, d\|T_j\| = \lambda^{-1} \bE (T_j, \bS, 1)\,.
\end{equation}
Since $T_j,T$ are representatives $\modp$, and the mass $\modp$ is lower semicontinuous with respect to the flat convergence $\modp$ in $\overline{U_\lambda}$ for almost every $\lambda > 0$, we conclude from \eqref{cebicev} that $\spt(T) \cap \bB_1 \subset \bS \cap \bB_1$, and thus $\bE (T, \bS, 1)=0$.
Lemma \ref{l:zero_excess} then implies that $T \mres \bB_1 = \bC \mres \bB_1$, and the proof is complete.
\end{proof}

\section{Proof of Lemma \ref{lem.flat-L^2 estimate}}

We first prove the statement for $R=1$, and then we show how to deduce \eqref{eq.flat_L2 estimate} in full generality. Let $\bC$ be the cone $\bC = \bC' \times \a{V^{m-1}}$, where $\bC'$ is a singular one-dimensional cone in the orthogonal complement $V^\perp$ of $V$ in $\R^{m+n}$ (with $\bS':=\spt(\bC')$ contained in some two-dimensional linear subspace of $V^\perp$). We consider a retraction $F' \colon V^\perp \simeq \R^{n+1} \to \bS'$ satisfying the following properties:
\begin{itemize}
\item[(i)] $F'$ is 1-homogeneous
\item[(ii)] $\left.F'\right|_{\pa B_1^{n+1}}$ agrees with the closest point projection onto $\bS'\cap \partial B_1^{n+1} = \{v(i)\}_{i=1}^N$ in a tubular neighborhood of this set;
\item[(iii)] $F'$ is smooth outside of $0$ and $L$-Lipschitz with Lipschitz constant $L=L(\bS')$; 
\item[(iv)] $\abs{F'(x)-x}\le C\, \dist(x, \bS')$ for some $C=C(\bS')$.
\end{itemize}
For instance $F'$ can be constructed as follows (cf. Figure \ref{figura-3}): 
we fix $\rho \in \left(0,\sfrac{1}{8}\right)$ and a corresponding tubular neighborhood $U_{2\rho}:=\{ x \in \pa B^{n+1}_1 \, \colon \, \dist(x, \{v(i)\}_{i=1}^N)<2\rho\}$ where a closest point projection $F'_0: U_{2\rho} \to \{v(i)\}_{i=1}^N$ is uniquely defined. Then we extend $F'_0$ to $\partial B^{n+1}_1$ by setting \[F'_1(x) := \phi(\abs{x-F'_0(x)}) F'_0(x)\,,\] where $0 \leq \phi(t) \leq 1$ is a smooth cut-off with $\phi(t)=1$ for $t<\rho$ and $\phi(t)=0$ for $t\ge 2\rho$. Finally, we let $F' \colon \R^{n+1} \to \bS'$ be the 1-homogenous extension \[F'(x) = \abs{x} \,F'_1\left(\frac{x}{\abs{x}}\right)\,.\] 
We remark that $\rho$, and thus the Lipschitz constants of $\phi$ and $F'$, depend on the smallest opening angle between two distinct branches of $\bC'$. Now, if $x$ is such that $\frac{x}{\abs{x}} \in U_{\rho}$ we have that \[\abs{F'(x)-x} = \abs{x}\, \Abs{ F_1'\left(\frac{x}{\abs{x}}\right)- \frac{x}{\abs{x}}}\le 2\, \dist(x, \bS')\,.\] If, instead, $\frac{x}{\abs{x}} \notin U_{\rho}$ then $\dist(x,\bS')\ge c_\rho \,\abs{x}$, whereas $\abs{F'(x)-x} \le  \abs{x}$, so that (iv) holds for an appropriate constant $C(\bS')$. 

\begin{figure}
\begin{tikzpicture}
\fill[gray!10] (0,0)--({2*sqrt(2)*cos(100)},{2*sqrt(2)*sin(100)}) arc (100:140:{2*sqrt(2)}) -- (0,0);
\fill[gray!10] (0,0)--({2*sqrt(2)*cos(-120)},{2*sqrt(2)*sin(-120)}) arc (-120:-160:{2*sqrt(2)}) -- (0,0);
\fill[gray!10] (0,0)--({2*sqrt(2)*cos(5)},{2*sqrt(2)*sin(5)}) arc (5:-35:{2*sqrt(2)}) -- (0,0);
\fill[gray!10] (0,0)--({2*sqrt(2)*cos(25)},{2*sqrt(2)*sin(25)}) arc (25:65:{2*sqrt(2)}) -- (0,0);
\fill[gray!30] (0,0)--({2*sqrt(2)*cos(110)},{2*sqrt(2)*sin(110)}) arc (110:130:{2*sqrt(2)}) -- (0,0);
\fill[gray!30] (0,0)--({2*sqrt(2)*cos(-130)},{2*sqrt(2)*sin(-130)}) arc (-130:-150:{2*sqrt(2)}) -- (0,0);
\fill[gray!30] (0,0)--({2*sqrt(2)*cos(-5)},{2*sqrt(2)*sin(-5)}) arc (-5:-25:{2*sqrt(2)}) -- (0,0);
\fill[gray!30] (0,0)--({2*sqrt(2)*cos(35)},{2*sqrt(2)*sin(35)}) arc (35:55:{2*sqrt(2)}) -- (0,0);
\draw (0,0) circle [radius = 2];
\draw[very thick] (0,0) -- (2,2);
\node[right] at ({sqrt(2)+0.2},{sqrt(2)}){$v(1)$};
\draw[very thick] (0,0) -- ({2*sqrt(2)*cos(120)},{2*sqrt(2)*sin(120)});
\node[above right] at ({2.2*cos(120)},{2.2*sin(120)}){$v(2)$};
\draw[very thick] (0,0) -- ({2*sqrt(2)*cos(-140)},{2*sqrt(2)*sin(-140)});
\node[left] at  ({2*cos(-140)-0.2},{2*sin(-140)}){$v(3)$};
\draw[very thick] (0,0) -- ({2*sqrt(2)*cos(15)},{-2*sqrt(2)*sin(15)});
\node[below right] at ({2*cos(-15)-0.2},{2*sin(-15)-0.2}){$v(4)$};
\draw[very thick] ({2*cos(110)},{2*sin(110)}) arc (110:130:2);
\draw[very thick] ({2*cos(-130)},{2*sin(-130)}) arc (-130:-150:2);
\draw[very thick] ({2*cos(-5)},{2*sin(-5)}) arc (-5:-25:2);
\draw[very thick]({2*cos(35)},{2*sin(35)}) arc (35:55:2);
\end{tikzpicture}
\caption{A visual illustration of the map $F$. In the nonshaded areas $F$ takes the constant value $0$. In each of the four darker shaded areas $F(x)$ is the unique point in the central halfline such that $|F(x)|=|x|$ (hence each thick arc is mapped into its middle point $v(i)$). In the remaining lighter shaded areas the map $F$ is extended to be Lipschitz, while still taking values in the nearest thick halfline.}\label{figura-3}
\end{figure}
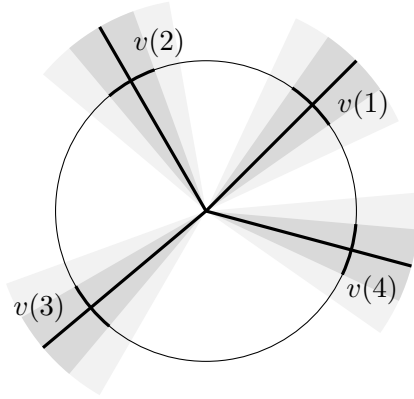

Next, we extend $F'$ to a retraction $F \colon \R^{m+n} \to \bS$ by setting \[F(q) := \left(F'(x),y\right) \qquad \mbox{for $q=(x,y) \in V^\perp \times V$}\,,\]
so that property (iv) implies
\begin{equation} \label{key property}
\abs{F(q)-q} \leq C(\bS) \, \dist(q,\bS) \qquad \mbox{for all $q \in \R^{m+n}$}\,.
\end{equation}

%
%

Consider now the linear homotopy $H: [0,1] \times \R^{m+n} \to \R^{m+n}$ defined by \[H(t,q):=(1-t)\, q + t\, F(q)\,,\] and deduce from \eqref{key property} that choosing $\eta_2=\eta_2(\bC)$ suitably small we can ensure that
\begin{equation} \label{key property 2}
\mbox{$\abs{q} \geq \frac{3}{4}$ and $\dist(q,\bS) < 2\,\eta_2$} \quad\implies\quad \abs{H(t,q)} \geq \frac{1}{2} \quad \mbox{for all $t \in \left[0,1\right]$}\,.
\end{equation}

Let now $T$ be as in the statement, and apply the polyhedral approximation theorem $\modp$ \cite[Theorem 3.4]{MS_a} to the restriction $T \mres \bB_1$ in order to determine, by exploiting the assumptions \eqref{hp0:flat_L2 estimate} and \eqref{hp1:flat_L2 estimate}, a sequence $\{\delta_k\}_{k=1}^\infty$ of positive numbers with $\delta_k \to 0^+$ as $k \to \infty$ and a sequence $\{P_k\}_{k=1}^\infty$ of representative $\modp$ integral polyhedral chains in $\bB_1$ such that
\begin{eqnarray*}
&\flat^p_{\overline{\bB}_1}(T\mres \bB_1-P_k) \leq \delta_k\,,\\ 
&\mass(P_k) \leq \mass^p(T \mres \bB_1) + \delta_k\,,\\ 
&\|P_k\| \weakstar \|T\| \mbox{ in $\bB_1$ as $k \to \infty$}\,,\\
&\mass^p((\pa P_k) \mres \bB_{1-\delta_k}) \leq \delta_k\,,\\ 
&\spt^p(\pa P_k) \setminus \bB_{1-\delta} \subset \{\dist(\cdot,\bS) \leq \eta_2 + \delta_k\}\,.
\end{eqnarray*}
We can then apply the $H$-homotopy formula to each $P_k$, and if $S_k$ is any representative $\modp$ of $\pa P_k$ we can write
\begin{equation} \label{homootopy formula}
F_\sharp P_k - P_k = \pa\left( H_\sharp(\a{(0,1)} \times P_k) \right) - H_\sharp(\a{(0,1)} \times S_k) =: \pa Z_k + W_k \, \modp\,.
\end{equation}
By the properties of $P_k$ and \eqref{key property 2}, $\|W_k\|(\bB_{\sfrac{1}{2}})\leq \delta_k$ for all $k$ sufficiently large. Hence, we can estimate for all such $k$
\begin{equation} \label{homotopy estimate k}
\hat\flat^p_{\bB_{\sfrac12}} (F_\sharp P_k - P_k) \leq \delta_k + \|Z_k\|(\bB_{\sfrac12}) \,,
\end{equation}
so that, letting $T'$ denote a representative $\modp$ of $F_{\sharp}(T\mres \bB_1)$, we have
\begin{equation}\label{eq.homotopy between T and C} 
\begin{split}
\hat\flat^p_{\bB_{\sfrac12}}(T' - T) &\le \liminf_{k \to \infty} \left(\hat\flat^p_{\bB_{\sfrac12}}(T'-F_\sharp P_k) + 2\,\delta_k +C(\bS)\int_{\bB_1} \abs{F(q) - q} \, d\|P_k\|(q) \right)\\
 &\le\liminf_{k \to \infty} \left( \hat\flat^p_{\bB_{\sfrac12}}(T'-F_\sharp P_k) + 2\,\delta_k +C(\bS)\int_{\bB_1} \dist(\cdot,\bS) \, d\|P_k\|  \right) \\ & = C(\bS) \, \int_{\bB_1} \dist(\cdot, \bS) \, d\|T\|\,.
\end{split}
\end{equation}

Finally we note that $T'$ is supported in $\bS$, and that $\spt^p(\pa T') \cap \bB_{\sfrac12} = \emptyset$. Estimating
\begin{equation} \label{flat distance estimate}
\begin{split}
\hat\flat^p_{\bB_{\sfrac12}}(T'-\bC) &= \hat\flat^p_{\bB_{\sfrac12}}(T'-F_\sharp \bC) \\
& \leq \liminf_{k \to \infty} \left(  \hat\flat^p_{\bB_{\sfrac12}}(T' - F_\sharp P_k) + \Lip(F)^{m+1}\, \hat\flat^p_{\bB_{\sfrac12}}(P_k-\bC)  \right)\\
&\overset{\eqref{hp2:flat_L2 estimate}}{\leq} \Lip(F)^{m+1}\,\eta_2
\end{split}
\end{equation}
we see then that, modulo possibly choosing a smaller $\eta_2(\bS)$, we can apply Lemma \ref{l:zero_excess} and conclude that $T' \mres \bB_{\sfrac12} = \bC\mres \bB_{\sfrac12}$, which proves the statement for $R=1$.\\

When $R < 1$, we can repeat the above proof replacing $T\mres \bB_1$ with $T_R \mres \bB_1$, where $T_R := (\eta_{0,R})_\sharp T$. The assumptions \eqref{hp0:flat_L2 estimate} to \eqref{hp2:flat_L2 estimate} hold for $T_R$ by scaling, and if $P_k$ is a polyhedral approximation of $T_R \mres \bB_1$ we can let $T'_R$ be a representative $\modp$ of $F_{\sharp}(T_R\mres\bB_1)$ so that, setting $T' := (\eta_{0,R^{-1}})_\sharp T'_R$, \eqref{eq.homotopy between T and C} becomes
\begin{equation} \label{final homotopy rescaled}
\hat\flat^p_{\bB_{\sfrac{R}{2}}} (T' - T) \leq C(\bS) \, R^{m+1} \, \int_{\bB_1} \dist(\cdot,\bS) \, d\|T_R\| = C(\bS)\,\int_{\bB_R} \dist(\cdot,\bS) \, d\|T\|\,,
\end{equation}
and the statement follows by arguing as above that $T' \mres \bB_{\bB_{\sfrac{R}{2}}} = \bC \mres \bB_{\sfrac{R}{2}}$ by Lemma \ref{l:zero_excess}.\qed

\section{Proofs of Lemma \ref{lem.L^infty-flat estimate} and Lemma \ref{lem.L^2 controls L^infty}}

\begin{proof}[Proof of Lemma \ref{lem.L^infty-flat estimate}] Let $W,Z,P$ be such that $T-S=W+\partial Z + p\,P$ in $\bB_{3R}$ and
\begin{equation}\label{eq:non lo so 01} \norm{W}(\bB_{2R}) + \norm{Z}(\bB_{2R}) \le 2\,\hat\flat^p_{\bB_{2R}}(T-S)\,. \end{equation}
Since both $T$ and $S$ have finite mass in $\bB_{3R}$, we can assume that $\|\pa Z\|(\bB_{2R}) < \infty$. Pick $q \in (\spt(T)\setminus \spt(S))\cap \bB_R$, and set $d=d(q)$ as in the statement. Observe that by assumption we have $0 < d < 2\,R$. By slicing theory, we may select $\frac{d}{4} < \sigma < \frac{d}{2}$ such that $\M(\langle Z, \mathbf{\varrho}_q, \sigma \rangle) \le 4\,d^{-1}\, \norm{Z}(\bB_{d}(q))$, where $\mathbf{\varrho}_Y(q')=\abs{q'-q}$. Note that $\bB_\sigma(q) \cap \spt(S) =\emptyset$, so that
\[ T\res \bB_\sigma(q) = W\res \bB_\sigma(q) - \langle Z, \mathbf{\varrho}_q, \sigma \rangle + \partial (Z\res \bB_\sigma(q)) +p\,P \res \bB_\sigma(q).  \]
Let us fix a Lipschitz retraction $F: \R^{m+n} \to \Sigma$. Since $\spt(T) \subset \Sigma$ we have 
\begin{equation} \label{competitor} 
T\res \bB_\sigma(q) = F_\sharp\left(W\res\bB_\sigma(q) - \langle Z, \mathbf{\varrho}_q, \sigma \rangle\right) + \partial F_\sharp (Z\res \bB_\sigma(q)) +p \, F_\sharp (P \res \bB_\sigma(q))\,.  \end{equation}

Since $T$ is area minimizing $\modp$, \eqref{competitor} implies that, for some constant $C=C(m)$,
\begin{equation}\label{Haus_flat_final} \begin{split}\frac1C\, \sigma^m \le \M(T\res \bB_\sigma(q)) & \le \M(F_\sharp(W\res \bB_\sigma(q) - \langle Z, \mathbf{\varrho}_q, \sigma \rangle)) \\ & \le \Lip(F)^m \left(\norm{W}(\bB_{2R}) +  \frac{2}{\sigma} \norm{Z}(\bB_{2R})\right)\, \end{split}\end{equation}
where in the first inequality we have used the almost monotonicity of the mass density ratio stemming from the minimality together with the assumption that $\|A_\Sigma\|_{L^\infty} \leq 1$. Plugging \eqref{eq:non lo so 01} into \eqref{Haus_flat_final}, we conclude that
\[ \min\{1,\sigma\}\, \sigma^{m} \le C \, \hat\flat^p_{\bB_{2R}}(T-S)\, \]
which completes the proof.
\end{proof}

\begin{proof}[Proof of Lemma \ref{lem.L^2 controls L^infty}] 
Let $q \in \spt(T) \cap \bB_{\sfrac12}\setminus K$, and set $2\rho:=\dist(q,K)$. Note that $\rho < \sfrac14$, and that $\dist(q', K)>\rho$ for all $q' \in \bB_\rho(q)$. Hence using minimality and the resulting almost monotonicity of the mass density ratio of $T$ we deduce
\begin{equation} \label{conticino}
\frac1C \rho^m \le \M(T\res \bB_\rho(q)) \le \rho^{-2} \int_{\bB_\rho(q)} \dist^2(q',K) \, d\norm{T}(q')\,, 
\end{equation}
which completes the proof.
\end{proof}


\section{Proof of Lemma \ref{l:monot-Jonas}} In order to simplify our notation we write $R$ for $R_1$. For a fixed $a\geq 0$, and for any $0<r<R$, we consider the vector field
\[
W_{a,r} (q):= \left(\frac{1}{\max(r,\abs{q})^{m+a}}-\frac{1}{R^{m+a}}\right)^+ q\,.
\]
We then insert $g^2 W_{a,r}$ in the first variation formula \eqref{e:H in L infty} to derive
\begin{align*}
- &\int_{\bB_R} g^2 W_{a,r} \cdot \vec{H}_T\, d\norm{T} =\frac{m}{r^{m+a}} \int_{\bB_r} g^2 \, d\norm{T} - \frac{m}{R^{m+a}} \int_{\bB_R} g^2 \, d\norm{T}\\&-a \int_{\bB_R\setminus \bB_r} \frac{g^2 (q)}{\abs{q}^{m+a}}\, d\norm{T} (q) + (m+a) \int_{\bB_R\setminus \bB_r} g^2 (q)\frac{\abs{q^\perp}^2}{\abs{q}^{m+a+2}}\, d\norm{T} (q)\\
&+ \int_{\bB_R} W_{a,r}^T \cdot \nabla g^2 \, \norm{T} \, ,
\end{align*}
where $W_{a,r}^T (q)$ denotes the projection on the tangent plane to $T$ at $q$ of the vector $W_{a,r} (q)$.
Observe that $W_{a,r}^T (q)$ is in fact parallel to $q^T$. Now we can use the homogeneity of $g$ and the identity $q= q^T + q^\perp$ to deduce that 
\[
\nabla g^2 (q) \cdot q^T= 2k g^2 (q) - 2g (q) \nabla g (q) \cdot q^\perp \ge \left(2k-\frac\epsilon2\right) g^2 (q) - \frac{2}{\epsilon}\abs{\nabla g (q)}^2 \abs{q^\perp}^2\,.
\]
In particular we may choose $a=2k-\alpha$, $\epsilon = \alpha$ to estimate 
\begin{align*}
- &\int_{\bB_R} g^2 W_{a,r} \cdot \vec{H}_T\, d\norm{T} \ge \frac{m+2k-\sfrac\alpha2}{r^{m+2k-\alpha}} \int_{\bB_r} g^2 \, d\norm{T} - \frac{m+2k-\sfrac\alpha2}{R^{m+2k-\alpha}} \int_{\bB_R} g^2 \, d\norm{T}\\&\frac{\alpha}{2} \int_{\bB_R\setminus \bB_r} \frac{g^2 (q)}{\abs{q}^{m+2k-\alpha}}\, d\norm{T} (q) + (m+2k-\alpha) \int_{\bB_R\setminus \bB_r} g^2\frac{\abs{q^\perp}^2}{\abs{q}^{m+2k+2-\alpha}}\, d\norm{T} (q)\\&- \frac{2}{\alpha}\int_{\bB_R}\left(\frac{1}{\max(r,\abs{q})^{m+2k-\alpha}}-\frac{1}{R^{m+2k-\alpha}}\right)^+ \abs{\nabla g (q)}^2 \abs{q^\perp}^2\,d \norm{T} (q) \,\,.
\end{align*}
To bound the left hand side $|g^2 W_{a,r} \cdot \vec{H}_T| (q)\leq C \|\hat{g}\|_\infty^2\bA R^\alpha |q|^{1-m}$, valid for $|q|\leq R$, and exploit the monotonicity formula to estimate
\[
\int_{\bB_R} |q|^{1-m} d\|T\| (q) \leq C \frac{\|T\| (\bB_R)}{R^m}\, .
\]
We thus conclude
\begin{align*}
\frac{\alpha}{2} \int_{\bB_R\setminus \bB_r} \frac{g^2 (q)}{\abs{q}^{m+2k-\alpha}}\, d\norm{T} (q)\leq
&\frac{m+2k}{R^{m+2k-\alpha}} \int_{\bB_{R}} g^2 \,d\norm{T} + C \bA \|\hat g\|_\infty^2 \frac{\norm{T}(B_{R})}{R^{m-\alpha}}\\
&+ \frac{2}{\alpha}\int_{\bB_R}\frac{\abs{\nabla g (q)}^2 \abs{q^\perp}^2}{\max(r,\abs{q})^{m+2k-\alpha}} \,d \norm{T} (q)\, .
\end{align*}
Letting $r\downarrow 0$ we then conclude \eqref{e.h_k monotonicity}. 

Next recall the "classical" monotonicity formula (which in fact is a particular case of the identities above, where we set $a=0$, $h=1$, $R=\rho$, and let again $r\downarrow0$):  
\begin{equation}\label{e.classical monotonicity}
	\rho^m \Theta_T(0)- \norm{T}(\bB_\rho) + \rho^m \int_{\bB_\rho} \frac{\abs{q^\perp}^2}{\abs{q}^{m+2}}\, d\norm{T} (q) = - \frac{\rho^m}{m} \int_{\bB_\rho} \frac{q^\perp\cdot \vec{H}_T (q)}{\abs{q}^{m}}\, d\norm{T} (q)\,.
\end{equation}
Next recall that
\begin{itemize}
    \item $\Theta_T (0) \geq \Theta_{\bC} (0) = \rho^{-m} \|\bC\| (\bB_\rho)$;
    \item The identities
    \begin{align*}
    \int_{\bB_R} f (|q|)d\mu (q) & = \int_0^R f(t) \frac{d}{dt} (\mu (\bB_t))\, dt\, ,\\
    \int_{\bB_R} F (|q|)d\mu (q) &= \int_0^R f(t) \frac{d}{dt} (t^m \mu (\bB_t))\, dt\, ,
    \end{align*}
    valid for any nonnegative Radon measure $\mu$ such that $\mu (\{0\}) = 0$ (provided we interpret the derivative $\frac{d}{dt} (\mu (\bB_t))$ distributionally as a nonnegative Radon measure $\nu$ on $[0,R]$).  
\end{itemize}
We conclude \eqref{e.classical monotonicity with f} by first differentiating \eqref{e.classical monotonicity} in $\rho$, then multiplying by $f (\rho)$, and finally integrating in $\rho$ between $0$ and $R$.

\section{Quantitative stratification and proof of $\mathcal{S}^{m-2}=\mathcal{S}^{m-2}_\eta$} \label{app:NV}

Recall the definition of the classical stratification 
\[
\mathcal{S}^0 \subset \mathcal{S}^1 \subset \ldots \subset \mathcal{S}^{m-1} \subset \mathcal{S}^m = \spt(T) \setminus \spt^p(\partial T)
\]
of $\spt(T) \setminus \spt^p(\partial T)$ introduced in Section \ref{sec:cones}. Following \cite{NV} (see also \cite{DLHMS}), we give the following definition of a notion of local almost symmetry for an integral varifold $V$. 

\begin{definition} \label{def: local symmetries}
Let $V$ be an $m$-dimensional integral varifold in $\R^{m+n}$. For $k \in \{1,\ldots,m\}$ and $\eta > 0$, we say that $V$ is $(k,\eta)$-almost symmetric in a ball $\bB_s(q)$ if there exists a varifold cone $\bC$ with spine of dimension $k$ such that the varifold distance between $\bC \mres \bB_1(0)$ and $((\eta_{q,s})_\sharp V)\mres \bB_1(0)$ is smaller than $\eta$.
\end{definition}

Let now $T$ be as in Definition \ref{def:am_modp}, and suppose that $\spt^p(\partial T) \cap \bB_2(0) = \emptyset$, so that the associated varifold $\|T\|$ has bounded generalized mean curvature in $\bB_2(0)$. For $k=0,\ldots,m-1$, $\eta > 0$, and $r > 0$ we then introduce the set
\[
\begin{split}
\mathcal{S}^{k,r}_\eta := \Big\lbrace & q \in \bB_1(0) \cap \spt (T) \, \colon \, \mbox{$\|T\|$ is not $(k+1,\eta)$-almost symmetric} \\
& \qquad \qquad \mbox{in $\bB_s(q)$ for all $s \in \left[r,1\right)$} \Big\rbrace\,,
\end{split}
\]
as well as the \emph{quantitative strata}
\[
\mathcal{S}^k_\eta := \bigcap_{r > 0} \mathcal{S}^{k,r}_\eta\,,
\]
so that
\[
\mathcal{S}^k \cap \bB_1(0) = \bigcup_{\eta > 0} \mathcal{S}^k_\eta\,.
\]
In this section we prove the following result, which follows as a simple consequence of the theory developed in the paper.

\begin{proposition} \label{p:strata vs q-strata}
Let $p \geq 3$ be odd, and let $T$ be as in Definition \ref{def:am_modp}. Suppose that $\dim(\Sigma) = m+1$, and that $\spt^p(\partial T) \cap \bB_2(0) = \emptyset$. Then, $\mathcal{S}^{m-2} \cap \bB_1(0) = \mathcal{S}^{m-2}_\eta$ for some $\eta > 0$.
\end{proposition}

\begin{proof}
Suppose that the statement is false, and let, for $h \geq 1$ integer, $q_h \in \bB_1(0) \cap \mathcal{S}^{m-2} \setminus \mathcal{S}^{m-2}_{\eta_h}$, where $\eta_h \to 0^+$. By definition of quantitative strata, there are then radii $r_h \in \left( 0, 1 \right)$ and cones $\bC_h$ with $(m-1)$-dimensional spine such that
\begin{equation} \label{qs:converging}
    \dist_{{\bf var}}(T_h \mres \bB_1(0), \bC_h \mres \bB_1(0)) \leq \eta_h\,,
\end{equation}
where $T_h := (\eta_{q_h,r_h})_\sharp T$ and $\dist_{{\bf var}}$ denotes varifold distance. By the slicing formula $\modp$ and \eqref{qs:converging}, both $\mass(T_h \mres \bB_1(0))$ and $\mass^p(\partial (T_h \mres \bB_1(0)))$ are uniformly bounded in $h$, so that combining \eqref{qs:converging} with \cite[Proposition 5.2]{DLHMS} and Lemma \ref{lem.L^infty-flat estimate} we deduce the existence of a (not relabeled) subsequence such that, when $h \to \infty$, the currents $T_h$ converge, both with respect to the topology induced by $\hat\flat^p_{\bB_1}$ and in the sense of varifolds in $\bB_1(0)$, to a representative $\modp$ current $\bC_0$ which is (the restriction to $\bB_1(0)$ of) an area minimizing cone $\modp$ with no boundary $\modp$ in $\bB_1(0)$ and spine of dimension \emph{at least} $m-1$, and such that the excess of $T_h$ in $\bB_1(0)$ with respect to $\spt(\bC_0)$ converges to zero. 

Let now $q$ be the limit of a (not relabeled) subsequence of $q_h$. If $\rho := \limsup_{h \to \infty} r_h > 0$, then evidently
\begin{equation} \label{qs:first case}
T = q + \bC_0 \qquad \mbox{in $\bB_{\sfrac{\rho}{2}}(q)$}\,.
\end{equation}
In particular, $\bC_0$ cannot be a flat plane, since $q$ is a limit of singular points $q_h$. Hence, $\bC_0$ has $(m-1)$-dimensional spine, and thus, in a neighborhood of $q$, all singular points of $T$ belong to $\mathcal{S}^{m-1} \setminus \mathcal{S}^{m-2}$, contradicting the assumption on $q_h \to q$. 

Therefore, we can assume that $r_h \to 0^+$. Also in this case, we can exclude that $\bC_0$ is a flat plane: indeed, should that happen, White's regularity theorem would readily imply that $T_h$ are regular for all $h$ sufficiently large, a contradiction. Hence, we can assume that $\bC_0$ has $(m-1)$-dimensional spine $V_0$, and, by minimality, that its support is contained in $\pi_0 := T_q\Sigma$.

Now, fix $\delta \in \left( 0, \sfrac{1}{8}\right)$. By Proposition \ref{prop:no-holes}, we then have that for all $h \geq h_0(\delta)$ there is a point $\tilde q_{h} \in \bB_{\delta}(0)$ such that $\Theta_{T_{h}}(\tilde q_{h}) \geq \frac{p}{2}$. For $\delta$ sufficiently small and for all $h$ sufficiently large, the currents $T_h$ (and the manifolds $\Sigma_h = r_h^{-1}(\Sigma - q_h)$) satisfy the Assumptions of Corollary \ref{eps_reg} with $q=\tilde q_h$ and with $\bC_0$ replaced by $(O_h)_\sharp \bC_0$, where $O_h$ is a rotation of $\R^{m+n}$ such that $O_h(\pi_0) = T_{\tilde q_h}\Sigma_h$. In particular, $\sing(T_h) \cap \bB_{\sfrac{1}{10}}(\tilde q_h)$ is a classical free boundary. Rescaling back, we then deduce that, setting $\bar q_h := q_h + r_h\,\tilde q_h$, $\sing (T) \cap \bB_{\sfrac{r_h}{10}}(\bar q_h)$ is a classical free boundary for all $h$ sufficiently large. Since $\abs{\bar q_h - q_h} \leq \delta\, r_h$, up to possibly choosing $\delta$ smaller, we then have that $T$ has, at $q_h$, a unique tangent cone with $(m-1)$-dimensional spine, a contradiction to $q_h \in \mathcal{S}^{m-2}$ which concludes the proof.
\end{proof}


\section{Proof of Theorem \ref{t:p=3}}\label{s:appendimi}

First of all, irrespectively of the codimension of $T$, note that at every point $x\in \mathcal{S}^m\setminus \mathcal{S}^{m-1}$ there is at least one tangent cone which is flat, and which, because $p=3$, has multiplicity $1$. By Allard's regularity Theorem, cf. \cite{Allard72}, every such point is thus regular. Next, at every point $x\in \mathcal{S}^{m-1}\setminus \mathcal{S}^{m-2}$ at least one tangent cone consists of three half $m$-dimensional planes meeting at 120 degrees at an $(m-1)$-dimensional linear subspace. We can thus apply the theory in \cite{Simon} (because the multiplicity on the regular part is always $1$) and thus conclude that $\mathcal{S}^{m-1}\setminus \mathcal{S}^{m-2}$ is locally a classical free boundary. Now, in general codimension, Appendix \ref{app:NV} and \cite{NV} imply that $\mathcal{S}^{m-2}$ is rectifiable and has locally finite $\mathcal{H}^{m-2}$ measure, while in codimension $1$, \cite{Taylor} implies that $\mathcal{S}^{m-2}\setminus \mathcal{S}^{m-3}$ is empty (because there are no codimension $1$ area minimizing cones $\mod$ $3$ with $(m-2)$-dimensional spine. We can thus apply Appendix \ref{app:NV} and the theory in \cite{NV} and conclude that $\mathcal{S}^{m-3}$ is rectifiable and has locally finite $\mathcal{H}^{m-3}$ Hausdorff measure.